\documentclass[oneside, 12pt]{amsart}
\usepackage[utf8]{inputenc}
\usepackage[T1]{fontenc}
\usepackage{amsmath}
\usepackage{amssymb}
\usepackage{amsfonts}
\usepackage{amsthm}
\usepackage{mathtools}
\usepackage{amscd}
\usepackage{amsxtra}
\usepackage{latexsym}
\usepackage{scalerel}
\usepackage{stmaryrd}
\usepackage{mathrsfs}
\usepackage{geometry}
\usepackage{graphicx}
\usepackage[svgnames]{xcolor}
\usepackage{url}
\usepackage{textcomp}
\usepackage{listings}
\usepackage{parskip}
\usepackage{breqn}
\usepackage{textcase}
\usepackage{hyperref}
\usepackage{bookmark}
\usepackage{lmodern}
\usepackage{cmap}
\DeclareMathOperator{\trace}{tr}   
\DeclareMathOperator{\Ric}{Ric}
\DeclareMathOperator{\Spec}{Spec}

\DeclareMathOperator{\ram}{ram}
\DeclareMathOperator{\Aut}{Aut}

\DeclareMathOperator{\ad}{ad}

\DeclareMathOperator{\Fix}{Fix}
\DeclareMathOperator{\Kern}{Ker}
\DeclareMathOperator{\Mat}{Mat}
\DeclareMathOperator{\Eig}{Eig}
\DeclarePairedDelimiter{\norm}{\lVert}{\rVert}

\newcommand{\bbR}{\mathbb{R}}
\newcommand{\C}{\mathrm{C}}
\newcommand{\bbZ}{\mathbb{Z}}
\newcommand{\N}{\mathbb{N}}

\newcommand{\bbC}{\mathbb{C}}

\newcommand{\bbH}{\mathbb{H}}      
\newcommand{\bbO}{\mathbb{O}}

 \newcommand{\frakC}{\mathfrak{C}}

\newcommand{\scrP}{\mathscr{P}}

\newcommand{\scrR}{R}

\newcommand{\dt}{\mathrm{dt}}

\newcommand{\Id}{\mathrm{Id}}
\newcommand{\Iso}{\mathrm{I}}

\newcommand{\Hom}{\mathrm{Hom}}

\newcommand{\rmS}{\mathrm{S}}

\newcommand{\rmR}{\mathrm{R}}

\newcommand{\Eval}{\mathrm{Eval}}

\newcommand{\pwcp}{\mathrm{pw\_cp}}
\newcommand{\bwmin}{\mathrm{bw\_min}}
\newcommand{\Spin}{\mathbf{Spin}}
\newcommand{\Pin}{\mathbf{Pin}}
\newcommand{\Sp}{\mathbf{Sp}}

\newcommand{\Cl}{\mathrm{Cl}}

\newcommand{\Gtwo}{\mathbf{G_2}}

\newcommand{\rmL}{\mathrm{L}}

\newcommand{\pr}{\mathrm{pr}}
\newcommand{\Ogroup}{\mathbf{O}}
\newcommand{\U}{\mathbf{U}}
\newcommand{\SO}{\mathbf{SO}}
\newcommand{\SU}{\mathbf{SU}}

\newcommand{\so}{\mathfrak{so}}

\renewcommand{\i}{\mathrm{i}}
\renewcommand{\j}{\mathrm{j}}
\renewcommand{\k}{\mathrm{k}}
\renewcommand{\d}{\mathrm{d}}     
\renewcommand{\Im}{\mathrm{Im}}   
\renewcommand{\Re}{\mathrm{Re}}    
\newcommand{\K}{\mathrm{K}}
\newcommand{\underK}{\check{\mathrm{K}}}
\newcommand{\overK}{\hat{\mathrm{K}}}

\renewcommand{\#}{\sharp}              

\newcommand{\ext}{\mathrm{ext}}
\newcommand{\g}{\mathfrak{g}}

\newcommand{\h}{\mathfrak h}
\renewcommand{\H}{\mathfrak H}
\newcommand{\n}{\mathfrak{n}}
\renewcommand{\t}{\mathfrak{t}}
\newcommand{\z}{\mathfrak z}
\newcommand{\Z}{\mathfrak Z}
\renewcommand{\v}{\mathfrak v} 
\newcommand{\undermu}{\check{\mu}}
\newcommand{\overmu}{\hat{\mu}}
\newcommand{\oversigma}{\hat{\sigma}}
\newcommand{\branch}{\mathcal{B}}

\DeclareMathOperator{\Endop}{End}
\newcommand{\End}[1]{\Endop\!\left(#1\right)}
\newcommand{\Endp}[1]{\Endop\!\left(#1\right)}

\newcommand{\cbullet}{\mathbin{\bullet}}
\newcommand{\spinprod}{\mathbin{*}}

\theoremstyle{definition}
\newtheorem{definition}{Definition}[section]

\theoremstyle{plain}
\newtheorem{theorem}[definition]{Theorem}
\newtheorem{lemma}[definition]{Lemma}
\newtheorem{proposition}[definition]{Proposition}
\newtheorem{corollary}[definition]{Corollary}

\theoremstyle{remark}
\newtheorem{remark}[definition]{Remark}

\title{Minimal Polynomials of Generalized Heisenberg Groups}
\author{Tillmann Jentsch}
\address{Unidad Cuernavaca del Instituto de Matem\'aticas, Universidad
Nacional Aut\'onoma de M\'exico, Avenida Universidad s/n, Lomas de Chamilpa,
62210 Cuernavaca, Morelos, MEXICO (Guest visitor)}
\email{tilljentsch@gmail.com}


\begin{document}
\sloppy

\begin{abstract}
Every homogeneous Riemannian \(\frakC_0\)-space is associated
with its minimal polynomial. To provide explicit examples, we compute the
minimal polynomials for generalized Heisenberg groups equipped with
their canonical left-invariant metrics.
\end{abstract}

\maketitle

\section{Preliminaries}
A Riemannian manifold is a smooth manifold \(N\) equipped with a
positive definite symmetric \((0,2)\)-tensor field \(g\). The pair
\((N,g)\) admits a unique torsion-free metric connection \(\nabla\) with
\(\nabla g = 0\), the Levi--Civita connection. The Riemann curvature
tensor \(\rmR\) is defined by
\[
  \rmR(X,Y)Z \coloneqq \nabla_X \nabla_Y Z - \nabla_Y \nabla_X Z
  - \nabla_{[X,Y]} Z
\]
for all vector fields \(X,Y,Z\) on \(N\).

For \(k \ge 0\), let \(\nabla^k \rmR\) denote the \(k\)th iterated
covariant derivative of \(\rmR\), defined recursively by
\[
  \nabla^0 \rmR \coloneqq \rmR,\qquad
  \nabla^k \rmR \coloneqq
  \nabla(\nabla^{k-1} \rmR)
  \quad (k \ge 1)
\]
via the induced covariant derivative \(\nabla\) acting on tensor fields
in the standard way. (For instance, locally symmetric spaces are
characterized by \(\nabla \rmR = 0\).) 

Given \(X \in T_pN\), define an
endomorphism \(\scrR^k(X) \in \End{T_pN}\) by
\begin{equation}\label{eq:kth_order_symmetrized_curvature_tensor_1}
  \scrR^k(X)Y \coloneqq (\nabla^{k}_{X,\ldots,X}\,\rmR)(Y,X)X
  \qquad (Y \in T_pN)
\end{equation}
Then the assignment
\begin{equation}
\label{eq:kth_order_symmetrized_curvature_tensor_2}
  \scrR^k|_p \colon T_pN \longrightarrow \End{T_pN},\qquad
  X \longmapsto \scrR^k|_p(X)
\end{equation}
defines an \(\End{T_pN}\)-valued homogeneous polynomial map of degree \(k+2\)
on \(T_pN\). We refer to \(\scrR^k\) as the {\em symmetrized \(k\)th covariant
derivative of the curvature tensor}.

For each \(p \in N\), let \(\scrP^j(T_pN)\) be the vector space of
homogeneous polynomials of degree \(j\) on \(T_pN\), and let
\(\scrP^j\bigl(T_pN;\End{T_pN}\bigr)\) denote the space of homogeneous
polynomials of degree \(j\) on \(T_pN\) with values in \(\End{T_pN}\).
Then
\[
  \scrR^k|_p \in \scrP^{k+2}(T_pN;\End{T_pN})
\]
Hence the symmetrized \(k\)-jet of the curvature tensor,
\[
  \scrR^{\le k}|_p \coloneqq
  (\scrR^0|_p,\,\scrR^1|_p,\,\ldots,\,\scrR^k|_p)
\]
belongs to the graded vector space
\[
  \scrP(T_pN;\End{T_pN}) =
  \bigoplus_{j=0}^{\infty} \scrP^j(T_pN;\End{T_pN})
\]
which is a graded module over the graded \(\bbR\)-algebra
\[
  \scrP(T_pN) = \bigoplus_{i=0}^{\infty} \scrP^i(T_pN)
\]
of polynomial functions on \(T_pN\).

Let \(P = \sum_{i=0}^k a_i\,\lambda^{k-i}\) be a polynomial of degree
\(k\) in \(\lambda\) with coefficients \(a_i \in \scrP(T_pN)\), so
\(P \in \scrP(T_pN)[\lambda]\). We evaluate \(P\) on
\(\scrR^{\le k}|_p\) by
\[
  P\bigl(\scrR^{\le k}|_p\bigr)
  \coloneqq \sum_{i=0}^k a_i\,\scrR^{k-i}|_p
\]
which lies in \(\scrP(T_pN;\End{T_pN})\). This induces the evaluation
map
\begin{equation}\label{eq:Eval_p}
  \Eval_p \colon \scrP(T_pN)[\lambda]
  \longrightarrow \scrP(T_pN;\End{T_pN}), \qquad
  P \longmapsto P\bigl(\scrR^{\le \deg(P)}|_p\bigr)
\end{equation}
characterized as the unique homomorphism of \(\scrP(T_pN)\)-modules such
that \(\Eval_p(\lambda^i) = \scrR^i|_p\) for \(i \ge 0\).

As usual, we call \(P\) \emph{homogeneous} if \(a_i \in \scrP^i(T_pN)\)
for \(i = 0,\ldots,k\). For such \(P\) one has
\[
  P\bigl(\scrR^{\le k}|_p\bigr) \in
  \scrP^{k+2}(T_pN;\End{T_pN})
\]
We say that \(P\) is \emph{monic} if \(a_0 = 1\).

\begin{definition}\label{de:JR}
Let a monic homogeneous polynomial
\(P = \sum_{i=0}^k a_i\,\lambda^{k-i} \in \scrP(T_pN)[\lambda]\)
be given. Following \cite{J}, we say:
\begin{enumerate}
  \item \(P\) is \emph{admissible} if \(\Eval_p(P) = 0\), i.e.,
    \begin{equation}\label{eq:JR_1}
      \text{for all } X \in T_pN\colon\quad
      \scrR^k(X)
      = -\sum_{i=1}^k a_i(X)\,\scrR^{k-i}(X)
    \end{equation}
  \item \(P\) is \emph{minimal} if it is admissible and
    \begin{equation}\label{eq:U(p)_1}
      U(p) \coloneqq \bigl\{\,X \in T_pN \,\big|\,
      \{\scrR^{j}(X)\}_{j=0}^{k-1}
      \text{ is linearly independent in } \End{T_pN}\bigr\}
    \end{equation}
    is nonempty.
\end{enumerate}
\end{definition}

If \(U(p)\) is nonempty, then it is Zariski open and hence open and dense
in the Euclidean topology on \(T_pN\). Therefore, if a minimal polynomial
\(P \in \scrP(T_pN)[\lambda]\) exists, the coefficients \(a_i\) of \(P\)
are uniquely determined by the symmetrized curvature \(k\)-jet at \(p\).
We then define
\[
  P_{\min}(N,g,p) \coloneqq P
\]
and say that the minimal polynomial of \((N,g)\) exists at \(p\). However,
the existence of \(P_{\min}(N,g,p)\) is not automatic without additional
hypotheses.

On the other hand, if the minimal polynomial of a homogeneous Riemannian space
\((N,g)\) exists at some point---and therefore, by the action of \(\Iso(N,g)\),
at every point---then the coefficients assemble into smooth sections
\(a_i \in \Gamma(\scrP^i(N))\) of the vector bundle \(\scrP^i(N) \to N\)
whose fiber over \(p\) is \(\scrP^i(T_pN)\). Let
\[
  \scrP(N) \coloneqq \bigoplus_{i=0}^\infty \scrP^i(N) \to N
\]
denote the associated graded \(\bbR\)-algebra bundle. Then the sections
\(a_i\) combine to a global polynomial section
\[
  P_{\min}(N,g) \coloneqq \sum_{i=0}^k a_i \,\lambda^{k-i}
  \in \Gamma(\scrP(N))[\lambda]
\]
called the minimal polynomial of \((N,g)\). The coefficients \(a_i\) are
\(\Iso(N,g)\)-invariant and satisfy \(a_0 \equiv 1\). For each \(p \in N\),
the fiber
\[
  P_{\min}(N,g)\big|_p \coloneqq \sum_{i=0}^k a_i\big|_p \,\lambda^{k-i}
  \in \scrP(T_pN)[\lambda]
\]
is the minimal polynomial \(P_{\min}(N,g,p)\). Since this is the only
case considered here, we will henceforth work in this setting.

\DeclareRobustCommand{\Czero}{\texorpdfstring{\(\frakC_0\)}{C0}}
\subsection{The minimal polynomial of a homogeneous \Czero-space}
\label{se:JR_on_C0}
Let \(\gamma\colon I \to N\) be a geodesic of \((N,g)\). The Jacobi operator
\(\scrR_\gamma\) is the endomorphism-valued tensor field along \(\gamma\) defined by
\begin{equation}\label{eq:Jacobi_operator}
  \scrR_\gamma(t) \coloneqq
  \rmR(\,\cdot\,,\dot\gamma(t))\,\dot\gamma(t)\colon
  T_{\gamma(t)}N \longrightarrow T_{\gamma(t)}N,\qquad
  X \longmapsto \rmR\bigl(X,\dot\gamma(t)\bigr)\dot\gamma(t)
\end{equation}
cf.~\cite[Ch.~2.8]{BTV}. We define its first covariant derivative along
\(\gamma\) by
\[
  \scrR^{(1)}_\gamma(t) \coloneqq \frac{\nabla}{\d t}\,\scrR_\gamma(t)
\]
that is, for every vector field \(Y(t)\) along \(\gamma\),
\begin{equation}\label{eq:def_Rgamma_1}
  \bigl(\scrR^{(1)}_\gamma(t)\bigr)Y(t)
  \coloneqq
  \frac{\nabla}{\d t}\bigl(\scrR_\gamma(t)Y(t)\bigr)
  -
  \scrR_\gamma(t)\,\frac{\nabla}{\d t}Y(t)
\end{equation}
Along a geodesic this is equivalently
\[
  \scrR^{(1)}_\gamma(t)
  = (\nabla_{\dot\gamma(t)}\rmR)(\,\cdot\,,\dot\gamma(t))\,\dot\gamma(t)
\]
since \(\frac{\nabla}{\d t}\,\dot\gamma(t) \equiv 0\).

\begin{definition}[{\cite[Ch.~2.9]{BTV}}]\label{de:C_0}
A Riemannian manifold \((N,g)\) is called a \(\frakC_0\)-space if, for
every geodesic \(\gamma\colon I \to N\), there exists a parallel
skew-symmetric endomorphism field \(\C_\gamma\) along \(\gamma\) such
that
\begin{equation}\label{eq:def_C_0}
  \scrR^{(1)}_\gamma(t) = [\C_\gamma(t),\,\scrR_\gamma(t)]
  \quad\text{for all } t \in I
\end{equation}
\end{definition}
Here,
\[
  [\C_\gamma(t),\,\scrR_\gamma(t)] \coloneqq \C_\gamma(t)\circ\scrR_\gamma(t)
  - \scrR_\gamma(t)\circ\C_\gamma(t)
\] denotes the usual commutator of endomorphisms. Also \(\C_\gamma\) being parallel means
\[
  \tfrac{\nabla}{\d t}\C_\gamma \equiv 0
\]
By definition, a \emph{homogeneous \(\frakC_0\)-space} \((N,g)\) is a
homogeneous Riemannian manifold that admits a \(\frakC_0\)-structure.
For example, every G.O.\ space is a homogeneous \(\frakC_0\)-space: a
Riemannian manifold \((N,g)\) is called a G.O.\ space if every geodesic
is an orbit of a one-parameter subgroup of \(\Iso(N,g)\).
Furthermore, generalized Heisenberg groups provide additional examples
of homogeneous \(\frakC_0\)-spaces (cf.~\cite[Ch.~2.2]{BTV}); many such
groups are not G.O. To the best of our knowledge, these are the only
\(\frakC_0\)-spaces currently known.

Also recall that a section \(a\) of \(\scrP^i(N)\) (more precisely, the
corresponding symmetric tensor obtained by polarization) is called a
\emph{Killing tensor} if, for every geodesic \(\gamma\colon \bbR \to N\),
the map \(t \mapsto a(\dot\gamma(t))\) is constant; cf.~\cite{DM,HMS}.
The following result is proved in~\cite{J}.

\begin{theorem}\label{th:JR_on_C_0_spaces}
Let \((N,g)\) be a homogeneous Riemannian \(\frakC_0\)-space with
nonvanishing Ricci tensor. Then:
\begin{enumerate}
\item The minimal polynomial \(P_{\min}(N,g)
  \in \Gamma\bigl(\scrP(N)\bigr)[\lambda]\) exists.
\item The degree \(k\) is odd, all odd-indexed coefficients vanish,
  \(a_{2i+1}=0\), and the minimal polynomial is determined entirely by
  its even coefficients:
  \[
    P_{\min}(N,g)
    \;\;=\;\; \sum_{i=0}^{\frac{k-1}{2}} a_{2i}\,\lambda^{\,k-2i}
  \]
  where \(a_{2i} \geq 0\) and \(a_{2i}\big|_{U(p)} > 0\) at an arbitrary point \(p \in N\).
\item Each even coefficient \(a_{2i}\) is a Killing tensor.
\end{enumerate}
\end{theorem}

An initial illustration of Theorem~\ref{th:JR_on_C_0_spaces} appears
in~\cite{JW}, where we exhibited minimal polynomials on naturally
reductive (hence G.O.) spaces whose even coefficients \(a_{2i}\) satisfy
\(a_{2i}(X) = c_i\,\norm{X}^{2i}\) with \(c_i \in \bbR_{>0}\); in
particular, each \(a_{2i}\) is a Killing tensor.

\subsection{Generalized Heisenberg groups}
\label{se:JR_GH}
In this article, we study the family of generalized Heisenberg
groups (also called H-type groups) constructed by Kaplan
\cite{Ka1,Ka2}. All H-type groups are homogeneous
\(\frakC_0\)-spaces \cite[Ch.~3.5]{BTV}, are irreducible
\cite[Ch.~3.1.11]{BTV}, and have nonvanishing Ricci tensor
\cite[Ch.~3.1.7]{BTV}; see also \eqref{eq:Ricci} below.
Therefore, their minimal polynomials exist and have the
properties stated in Theorem~\ref{th:JR_on_C_0_spaces} above.

These groups are parameterized by a pair \((\z,\v)\) consisting of
an \(n\)-dimensional Euclidean vector space \(\z\) with scalar
product \(g_\z\) and an orthogonal Clifford module \(\v\) for
\(\z\); that is, \(\v\) is a nontrivial real module for the
Clifford algebra \(\Cl(\z,g_\z)\), endowed with a scalar product
\(g_\v\) such that
\begin{align}\label{eq:Clifford_module_1}
  \langle Z \cbullet S, \, \tilde S \rangle
    &= -\,\langle S, \, Z \cbullet \tilde S \rangle, \\
\label{eq:Clifford_module_2}
  Z \cbullet Z \cbullet S &= -\,\norm{Z}^2\,S
\end{align}
for all \(S, \tilde S \in \v\) and \(Z \in \z\). By polarization
of~\eqref{eq:Clifford_module_2},
\begin{equation}\label{eq:Clifford_module_3}
  Z_1 \cbullet Z_2 \cbullet S + Z_2 \cbullet Z_1 \cbullet S
  = -\,2\,\langle Z_1, Z_2 \rangle\, S
\end{equation}
for all \(S \in \v\) and \(Z_1, Z_2 \in \z\). In particular, every
orthonormal system \(\{E_1,E_2\}\) in \(\z\) defines a pair
\(\{E_1 \cbullet, E_2 \cbullet\}\) of anticommuting orthogonal
complex structures on \(\v\) \cite[p.~24]{BTV}.

Clifford algebras of \(\z = \bbR^n\) in dimension \(n \le 8\) are given as
follows~\cite{LM}:
\begin{equation*}
\renewcommand{\arraystretch}{1.2}
\begin{array}{c|ccccccccc}
  n & 0 & 1 & 2 & 3 & 4 & 5 & 6 & 7 & 8 \\ \hline
  \Cl(\bbR^n)
    & \bbR
    & \bbC
    & \bbH
    & \bbH \oplus \bbH
    & \Mat_2(\bbH)
    & \Mat_4(\bbC)
    & \Mat_8(\bbR)
    & \Mat_8(\bbR) \oplus \Mat_8(\bbR)
    & \Mat_{16}(\bbR)
\end{array}
\end{equation*}
Moreover, Clifford algebras \(\Cl(\bbR^n)\)  for \(n \ge 9\) can be
calculated from the above table by the famous eightfold periodicity
\(
\Cl(\bbR^{n})
\cong
\Cl(\bbR^8) \otimes_{\bbR} \Cl(\bbR^{n-8})
\).

Adjoint to the Clifford multiplication \(\cbullet\) is the
alternating spinor product \(\spinprod \colon \v \times \v \to \z\)
defined by
\begin{equation}\label{eq:def_spinor_product}
  \langle Z, S \spinprod \tilde S \rangle
  \coloneqq \langle Z \cbullet S, \tilde S \rangle
  = -\,\langle Z, \tilde S \spinprod S \rangle
\end{equation}
for all \(Z \in \z\) and \(S, \tilde S \in \v\). Via the spinor product
the direct sum \(\n \coloneqq \z \oplus \v\) becomes an
orthogonal Lie algebra under the orthogonal sum scalar product
\(g \coloneqq g_\z \oplus g_\v\) and the bracket
\[
  [X, Y] \coloneqq X_\v \spinprod Y_\v \oplus 0
\]
for all \(X = X_\z \oplus X_\v \in \z \oplus \v\) and
\(Y = Y_\z \oplus Y_\v \in \z \oplus \v\). Needless to say, the
Jacobi identity holds trivially, as all iterated brackets vanish.
In particular, the Lie algebra \(\n\) is a central extension of
the abelian Lie algebra \(\v\) with trivial bracket by \(\z\):
\[
  0 \to \z \stackrel{\subset}{\longrightarrow} \n
  \stackrel{\pr}{\longrightarrow} \v \to 0
\]
Therefore, the Lie algebra \(\n\) is 2-step nilpotent and thus the
Lie algebra of a nilpotent Lie group \(N\), which as a manifold is
simply the vector space \(\n\) with the polynomial multiplication
\[
  X \cdot Y \coloneqq
  (X_\z + Y_\z + \tfrac12 X_\v \spinprod Y_\v)
  \oplus (X_\v + Y_\v)
\]
By left translation, the scalar product \(g_\z \oplus g_\v\) induces
a left-invariant metric \(g\), the canonical left-invariant metric
of the H-type group \(N\). In the following, by an H-type group
\((N,g)\) we mean \(N\) equipped with its canonical left-invariant
metric.

Every oriented isometry \(F \colon \z \to \z\) of a Euclidean vector
space \(\z\) lifts along the covering homomorphism
\[
  \Spin(\z, g_\z) \longrightarrow \SO(\z, g_\z)
\]
to an isometry \(\pm F^\ext \colon \v \to \v\) for each orthogonal
Clifford module \(\v\) with the characteristic property that
\[
  F^\ext(Z \cbullet S) = FZ \cbullet F^\ext S
\]
for all \(S \in \v\) and \(Z \in \z\). Then \(F^\ext\) is unique up to
the sign. Consequently, the
automorphism group of the generalized Heisenberg algebra and the
corresponding simply connected Lie group contains at least the
quotient
\[
  \Aut(N,g) \supseteq \Spin(\z, g_\z) \times
  \Ogroup_\Cl(\v, g_\v, \cbullet) \,\big/\,
  \langle (-1, -\Id_\v) \rangle
\]
of the direct product of the spin group of \(\z\) with the group
\(\mathbf{O}_\Cl(\v, g_\v, \cbullet)\) of Clifford-linear isometries
of \(\v\), modulo the two-element central subgroup generated by
\((-1, -\Id_\v)\).

According to standard representation theory, the latter group is
isomorphic to
\[
  \Ogroup_\Cl(\v, g_\v, \cbullet) \cong
  \begin{cases}
    \Ogroup(r, \bbR) & \text{if } n \equiv 0, 6 \pmod{8},\\
    \Ogroup(r_1, \bbR) \times \Ogroup(r_2, \bbR)
      & \text{if } n \equiv 7 \pmod{8},\\
    \mathbf{U}(r, \bbC) & \text{if } n \equiv 1, 5 \pmod{8},\\
    \mathbf{Sp}(r, \bbH) & \text{if } n \equiv 2, 4 \pmod{8},\\
    \mathbf{Sp}(r_1, \bbH) \times \mathbf{Sp}(r_2, \bbH)
      & \text{if } n \equiv 3 \pmod{8},
  \end{cases}
\]
where \(r \in \N_0\) or \(r_1, r_2 \in \N_0\), respectively, are the
multiplicities of the real spinor modules in the Clifford module
\(\v\). For \(n = 0\), \((N,g)\) is a Euclidean space with 
\(P_\min(N,g) = 1\). For the remainder of this article, we may assume that
\(\z\) has positive dimension, i.e. \(n \ge 1\). 

\medskip
A key role in describing the minimal polynomial is 
played by the skew-symmetric linear operator
\[
  \overK(X)\colon \z \to \z
\]
associated with each \(X \in \n\), defined by
\begin{equation}\label{eq:def_overK}
  \big\langle \overK(X)Z_1, Z_2 \big\rangle
  \coloneqq
  \big\langle Z_1 \cbullet X_\z \cbullet X_\v,\,
    Z_2 \cbullet X_\v \big\rangle
\end{equation}
for all \(Z_1, Z_2 \in \z\); cf.\ Section~\ref{se:K}. By
construction, \(\Aut(N,g)\)-equivariance holds:
\begin{equation}\label{eq:K_Aut_equivariance}
  \overK(FX)
  = F|_\z \circ \overK(X) \circ \bigl(F|_\z\bigr)^{-1}
\end{equation}
for all \(F \in \Aut(N,g)\). Moreover, \(\overK(X)X_\z = 0\); hence
\(\Kern \overK(X) \supseteq \bbR X_\z\), and \(\overK(X)\) preserves
the codimension-one subspace
\begin{equation}\label{eq:def_h_X}
  \h_X \coloneqq X_\z^\perp \subseteq \z
\end{equation}
whenever \(X_\z \neq 0\). Since \(\overK(X)\) is skew-symmetric, its
nonzero eigenvalues are \(\pm \i \sqrt{\overmu}\), where
\(\overmu > 0\) is an eigenvalue of the nonnegative operator
\[
  -\overK(X)^2 \coloneqq -\overK(X) \circ \overK(X)
\]
This yields the family
\begin{equation}\label{eq:def_overK_square}
  -\overK^2 \colon \n \longrightarrow \Endp{\z},\quad
  X \longmapsto -\overK(X)^2
\end{equation}
of nonnegative self-adjoint operators.

There is an open dense subset \(\n \setminus \ram(\overK^2)\) (the
complement of the ramification locus) on which the eigenvalue
multiplicities of \(\overK(X)^2\)—and hence the number of distinct
eigenvalues—are constant on each connected component \(U\) of
\(\n \setminus \ram(\overK^2)\). The eigenvalues \(\overmu(X)\) of
\(-\overK(X)^2\) then depend real-analytically on \(X \in U\). The
corresponding real-analytic functions \(\overmu \colon U \to
\bbR_{\ge 0}\) are called \emph{eigenvalue branches}; that is,
\(\overmu(X)\) is an eigenvalue of \(-\overK(X)^2\) for each
\(X \in U\). Since \(\overK(X)X_\z \equiv 0\), there is always the
constant eigenvalue branch \(\overmu \equiv 0\).

We define the rescaled skew-symmetric operator
\(\underK(X)\colon \z \to \z\), characterized by
\begin{equation}\label{eq:def_underK}
  \big\langle \underK(X)Z_1, Z_2 \big\rangle
  \coloneqq
  \frac{\big\langle \overK(X)Z_1, Z_2 \big\rangle}
       {\norm{X_\z}\norm{X_\v}^2}
  =
  \Big\langle
    Z_1 \cbullet \frac{X_\z}{\norm{X_\z}}
    \cbullet \frac{X_\v}{\norm{X_\v}},\,
    Z_2 \cbullet \frac{X_\v}{\norm{X_\v}}
  \Big\rangle
\end{equation}
defined for all \(X \in \n\) that
belong to the open subset \(D \coloneqq \{X \in \n \mid X_\z
\neq 0,\; X_\v \neq 0\}\subseteq \n\). Similar as before,
\begin{equation}\label{eq:underK_square}
  -\underK^2 \colon D
  \longrightarrow \Endp{\z},\quad
  X \longmapsto -\underK(X)^2
\end{equation}
defines a family of nonnegative self-adjoint operators.

If \(\overK \not\equiv 0\), let \(\overmu\) be a nonzero eigenvalue
branch of \(-\overK^2\) defined on some connected component
\(U \subseteq \n \setminus \ram(\overK^2)\). Then \(\overmu(X) > 0\)
for all \(X \in U\), since on \(\n \setminus \ram(\overK^2)\) no
nonzero eigenvalue branch meets the zero branch. In particular,
\(X_\z \neq 0\) and \(X_\v \neq 0\), since otherwise \(\overK(X) = 0\)
by~\eqref{eq:def_overK}. Hence \(\n \setminus \ram(\overK^2) \subseteq D\). 
In this case, we therefore have
\begin{equation}\label{eq:domain}
  D \setminus \ram(\underK^2) = \n \setminus \ram(\overK^2)
\end{equation}
and eigenvalue branches \(\undermu\colon U \to \bbR_{\ge 0}\) of \(-\underK^2\)
correspond to eigenvalue branches \(\overmu\colon U \to \bbR_{\ge 0}\) of
\(-\overK^2\) via
\begin{equation}\label{eq:undermu_overmu}
  \undermu(X)
  = \frac{\overmu(X)}{\norm{X_\z}^2 \norm{X_\v}^4}
\end{equation}
In particular, \(\undermu\) is real-analytic as well.
Moreover, \(\undermu(X) \in [0,1]\) by Proposition~\ref{p:overK},
which justifies the rescaling. The associated eigenspaces
\(\Eig_{\undermu(X)}\) define smooth distributions \(\Eig_\undermu\)
on the trivial bundle \(U \times \z\).

The case \(\overK \equiv 0\) occurs precisely when \(n \le 2\). In this
case, \(\underK \equiv 0\) as well, and the only eigenvalue branch of 
\(-\underK^{2}\) is the constant zero branch. The minimal polynomials
for \(n = 1\) and \(n = 2\) are
\begin{align}
  P_\min(N,g) \;\;&=\;\; \lambda^3 + \norm{X}^2 \lambda
  \label{eq:min_pol_n_=_1}
  \\
  P_\min(N,g) \;\;&=\;\;
  \begin{aligned}[t]
    \lambda^{7}
      &+ \tfrac{1}{4}\bigl(
        49\norm{X_\z}^{2} + 9\norm{X_\v}^{2}
      \bigr)\lambda^{5}
    \\
      &+ \tfrac{1}{2}\Bigl(
        3\norm{X_\v}^{4}
        + 34\norm{X_\z}^{2}\norm{X_\v}^{2}
        + 63\norm{X_\z}^{4}
      \Bigr)\lambda^{3}
    \\
      &+ \tfrac{1}{4}\Bigl(
        \norm{X_\v}^{6}
        + 19\norm{X_\z}^{2}\norm{X_\v}^{4}
        + 99\norm{X_\z}^{4}\norm{X_\v}^{2}
        + 81\norm{X_\z}^{6}
      \Bigr)\lambda
  \end{aligned}
  \label{eq:min_pol_n_=_2}
\end{align}
respectively. These expressions are independent of \(\dim \v\)
(assuming only \(\v \neq \{0\}\)), where \(X = X_\z \oplus X_\v\) and
\(\norm{X}^{2} = \norm{X_\z}^{2} + \norm{X_\v}^{2}\); see
Theorem~\ref{th:main_2} below. In particular, for \(n = 1\) we recover
the minimal polynomial \eqref{eq:min_pol_n_=_1} of the
\((\dim \v + 1)\)-dimensional real Heisenberg group associated with the
complex space \(\v\); see also \cite{JW}. The unique H-type group with
\(n = 2\) and \(\dim \v = 4\) is Kaplan’s example of a G.O.\ space that
is not naturally reductive \cite{Ka2}. In this situation,
\cite[Thm.~3.1]{AN} claims that the minimal polynomial is
\begin{equation}\label{eq:wrong_minimal_polynomial}
  P(X) = \lambda^5 + \tfrac{5}{4}\norm{X}^2 \lambda^3
  + \tfrac{1}{4}\norm{X}^4 \lambda
\end{equation}
However, the proof given there contains an error. The minimal
polynomial for Kaplan’s example is instead the one stated in
\eqref{eq:min_pol_n_=_2}.

\begin{remark}
The polynomial~\eqref{eq:wrong_minimal_polynomial} is obtained if
one replaces Riemannian geodesics by geodesics of the canonical
left-invariant connection. Let \(\nabla^{\mathrm{left}}\) denote the
connection associated with the left-invariant framing, that is, the
connection for which all left-invariant vector fields are
\(\nabla^{\mathrm{left}}\)-parallel. Let \(c\) be a
\(\nabla^{\mathrm{left}}\)-geodesic.

In analogy with~\eqref{eq:Jacobi_operator}, we consider the
endomorphism field
\[
  \widetilde{\scrR}_c(t)
  \coloneqq
  \rmR(\,\cdot\,,\dot c(t))\dot c(t)
  \colon T_{c(t)}N \longrightarrow T_{c(t)}N
\]
For \(j \ge 0\), let
\[
  \widetilde{\scrR}^{(j)}_c(t)
  \coloneqq
  \left(\frac{\nabla}{\d t}\right)^j
  \widetilde{\scrR}_c(t)
  \in \End{T_{c(t)}N}
\]
denote the iterated covariant derivative of
\(\widetilde{\scrR}_c\) with respect to the Levi--Civita connection.
Here \(\tfrac{\nabla}{\d t}\) acts on endomorphism fields as
in~\eqref{eq:def_Rgamma_1}; equivalently,
\[
  \widetilde{\scrR}^{(0)}_c(t)
  \coloneqq
  \widetilde{\scrR}_c(t),
  \qquad
  \widetilde{\scrR}^{(j)}_c(t)
  \coloneqq
  \frac{\nabla}{\d t}
  \widetilde{\scrR}^{(j-1)}_c(t)
  \quad (j \ge 1)
\]
cf.~\cite[p.~65--66]{A}. Then, for Kaplan's \(6\)-dimensional H-type group, one has
\[
  \widetilde{\scrR}^{(5)}_c(t)
  + \frac{5}{4}\norm{\dot c(t)}^2
    \widetilde{\scrR}^{(3)}_c(t)
  + \frac{1}{4}\norm{\dot c(t)}^4
    \widetilde{\scrR}^{(1)}_c(t)
  \equiv 0
\]
Thus the polynomial~\eqref{eq:wrong_minimal_polynomial}
belongs to this auxiliary sequence along \(\nabla^{\mathrm{left}}\)-geodesics,
rather than to the sequence of symmetrized covariant derivatives
along Riemannian geodesics.
\end{remark}

\subsection{A blueprint algorithm to find the minimal
polynomial of a homogeneous \Czero-space}
\label{se:blueprint}

We now return to the original setting of
Theorem~\ref{th:JR_on_C_0_spaces}, where the relevant curves are
Riemannian geodesics. Fix \(p \in N\) and \(X \in T_pN\), and let
\(\gamma\colon \bbR \to N\) be the geodesic with
\(\gamma(0) = p\) and \(\dot\gamma(0) = X\).

For \(j \ge 0\), we consider the iterated covariant derivative
\[
  \scrR^{(j)}_\gamma(t)
  \coloneqq
  \left(\frac{\nabla}{\d t}\right)^j
  \scrR_\gamma(t)
  \in \End{T_{\gamma(t)}N}
\]
of the Jacobi operator \(\scrR_\gamma\) defined
in~\eqref{eq:Jacobi_operator}.

Since the tangent vector field \(\dot\gamma\) is parallel along
\(\gamma\), we have
\begin{equation}
\label{eq:Jacobi_operator_versus_symmetrized_curvature_tensor}
  \scrR^{(j)}_\gamma(t)
  =
  \scrR^j\big|_{\gamma(t)}\bigl(\dot\gamma(t)\bigr)
\end{equation}
for all \(t \in \bbR\), where the right-hand side is defined
via~\eqref{eq:kth_order_symmetrized_curvature_tensor_1}
and~\eqref{eq:kth_order_symmetrized_curvature_tensor_2}.

For a chosen \(\mathfrak{C}_0\)-structure \(\C_{\gamma}\) along \(\gamma\),
we set
\begin{equation}\label{eq:def_C_X}
  C_X \coloneqq \ad\bigl(\C_\gamma(0)\bigr)\colon
  \End{T_pN} \longrightarrow \End{T_pN}
\end{equation}
where
\begin{equation}\label{eq:ad}
  \ad\colon \End{T_pN} \longrightarrow \End{\End{T_pN}}, \quad
  A \longmapsto \ad(A) \coloneqq [\,A,\,\cdot\,]
\end{equation}
denotes the adjoint action.

Since \(\tfrac{\nabla}{\d t}\C_\gamma \equiv 0\), the induced endomorphism field
\(\ad(\C_\gamma(t))\) on \(\End{T_{\gamma(t)}N}\) is parallel with respect to the
connection on \(\End{TN}\) induced by \(\nabla\). In particular, for every
endomorphism field \(A(t)\in\End{T_{\gamma(t)}N}\) along \(\gamma\) we have
\[
  \tfrac{\nabla}{\d t}\bigl(\ad(\C_\gamma(t))A(t)\bigr)
  = \ad(\C_\gamma(t))\,\tfrac{\nabla A}{\d t}(t)
\]
Hence the defining property \eqref{eq:def_C_0} of a \(\mathfrak{C}_0\)-structure implies
\[
  \scrR_{\gamma}^{(j)}(t)
  = \ad\bigl(\C_\gamma(t)\bigr)\bigl(\scrR_{\gamma}^{(j-1)}(t)\bigr)
  = \cdots
  = \ad\bigl(\C_\gamma(t)\bigr)^{j}\bigl(\scrR_{\gamma}(t)\bigr)
\]
by induction on \(j \ge 1\). Evaluating at \(t = 0\) and using
\eqref{eq:Jacobi_operator_versus_symmetrized_curvature_tensor}
and~\eqref{eq:def_C_X}, we obtain
\begin{equation}\label{eq:C_0}
  \scrR^{j}(X)
  = C_X\bigl(\scrR^{j-1}(X)\bigr)
  = \cdots
  = C_X^{j}\bigl(\scrR^0(X)\bigr)
\end{equation}

Furthermore, there exists a largest integer \(k(X)\) such that the set
\[
  \{\,\scrR^0(X),\dotsc,\scrR^{k(X)-1}(X)\,\} \subseteq \End{T_pN}
\]
is linearly independent.
Then the extended set
\[
  \{\,\scrR^0(X),\dotsc,\scrR^{k(X)}(X)\,\}
\]
is linearly dependent, and there exist unique coefficients \(a_i(X) \in \bbR\)
for \(i = 1, \ldots, k(X)\) such that
\begin{equation}\label{eq:P_min_1}
\begin{aligned}
  \scrR^{k(X)}(X)
  &= -\sum_{i = 1}^{k(X)} a_i(X)\,\scrR^{k(X)-i}(X) \\
  &\stackrel{\eqref{eq:C_0}}{=}
    -\sum_{i = 1}^{k(X)} a_i(X)\,C_X^{k(X)-i}\,\scrR^0(X)
\end{aligned}
\end{equation}
The monic polynomial
\[
  P_{\min}(C_X,\scrR^0(X)) \coloneqq
  \lambda^{k(X)} + \sum_{i=1}^{k(X)} a_i(X)\,\lambda^{k(X)-i}
\]
is therefore the \emph{minimal annihilating polynomial}
of \(\scrR^0(X)\) for the linear operator \(C_X\) defined in
\eqref{eq:def_C_X}, cf.~\cite[Sec.~4]{J}.

Since \(P_{\min}(N,g)\big|_p(X)\) is an annihilating polynomial for
\(\scrR^0(X)\) under \(C_X\), we have \(k(X) \le k\)
where \(k \coloneqq  \deg\!\bigl(P_{\min}(N,g)\bigr)\).
Moreover, \(k(X) = k\) if and only if \(X\) belongs to \(U(p)\),
the subset of \(T_pN\) defined in~\eqref{eq:U(p)_1}.
Hence \(U(p)\) can also be described as
\begin{align}
  U(p)
  &= \{\,X \in T_pN \mid k(X) = k\,\} \label{eq:U(p)_2} \\
  &= \bigl\{\,X \in T_pN \mid
       P_{\min}(C_X,\scrR^0(X)) = P_{\min}(N,g)\big|_p(X)\,\bigr\}
       \label{eq:U(p)_3}
\end{align}

From a different point of view, one may start by defining
\(U(p)\) as the subset of \(T_pN\) on which the degree
\(k(X)\) of the minimal annihilating polynomial
\(P_{\min}(C_X,\scrR^0(X))\) in~\eqref{eq:P_min_1}
attains its maximal value \(k\). Then Theorem~\ref{th:JR_on_C_0_spaces}
ensures, via \eqref{eq:U(p)_3}, that the coefficient functions
\(a_i\colon U(p)\to\bbR,\ X\mapsto a_i(X)\) from
\eqref{eq:P_min_1} extend uniquely to polynomial functions
(not just rational functions). Then
\(k = \deg\!\bigl(P_{\min}(N,g)\bigr)\), the degree
of the minimal polynomial of \((N,g)\), and this definition of
\(U(p)\) agrees with the one from~\eqref{eq:U(p)_1}.

\medskip\noindent\textbf{Algorithm.}
Given explicit matrices for \(\C(X) \coloneqq \C_\gamma(0)\) and
\(\scrR(X)\) at \(p\), so that \(C_X = \ad(\C(X))\),
compute \(P_{\min}(N,g)\big|_p\) as follows.
\begin{enumerate}
  \item Initialize \(B_0(X) \coloneqq \scrR^0(X) \in \End{T_pN}\) and
        define recursively
        \[
          B_{i+1}(X) \coloneqq \ad\bigl(\C(X)\bigr)B_i(X)
          = \C(X)B_i(X) - B_i(X)\C(X), \qquad i \ge 0
        \]
        By~\eqref{eq:C_0}, we have \(B_i(X) = \scrR^{i}(X)\) for all \(i\).
  \item To solve a linear system in \(\bbR^{n^2}\), where
        \(n \coloneqq \dim T_pN\), “flatten” the
        iterates \(B_0(X),B_1(X),\ldots,B_k(X)\) via the
        column-stacking map \(\operatorname{vec}\colon
        \bbR^{n \times n}\to\bbR^{n^2}\).
        Use the preceding iterates
        \(B_0(X),B_1(X),\ldots,B_{k-1}(X)\) as columns and
        \(B_k(X)\) as the right-hand side:
        \begin{align*}
          B(X) &\coloneqq \bigl[\operatorname{vec}(B_0(X)),
                   \operatorname{vec}(B_1(X)), \dotsc,
                   \operatorname{vec}(B_{k-1}(X))\bigr]
                   \in \bbR^{n^2 \times k} \\
          b(X) &\coloneqq \operatorname{vec}\bigl(B_k(X)\bigr)
        \end{align*}
        Solve \(B(X)\,\alpha(X) = b(X)\) on \(U(p)\).
        Writing \(\alpha(X) = (\alpha_1(X),\ldots,\alpha_k(X))^T\), we obtain
        \[
          B_k(X) = \sum_{i=1}^{k} \alpha_i(X)\,B_{k-i}(X)
          \quad\text{on } U(p)
        \]
        The coefficients of \(P_{\min}\) are then
        \[
          a_0 = 1, \qquad a_{i}(X) = -\alpha_{k+1-i}(X) \quad (i = 1,\ldots,k)
        \]
        Hence
        \[
          P_{\min}(N,g)\big|_p(X)
          = \lambda^k + \sum_{i=1}^k a_i(X)\,\lambda^{k-i}
        \]
\end{enumerate}\noindent In Maple (dimension-agnostic), this reads:
\begin{verbatim}
> with(LinearAlgebra):

k := ...:     # degree of the minimal polynomial
# the right value for k = 1,2,3,... has to be found by guess and proof:
# for smaller values the following algorithm will not yield a solution

n := ...:     # dimension of N

A := Matrix([[a[1,1](X), ..., a[1,n](X)], ...,
             [a[n,1](X), ..., a[n,n](X)]]);
# A is the C_0-structure at X

B := Array(0..k):
B[0] := Matrix([[b[1,1](X), ..., b[1,n](X)], ...,
                [b[n,1](X), ..., b[n,n](X)]]);
# B[0] is the symmetrized curvature tensor at X

for i from 0 to k - 1 do
    B[i+1] := simplify(A.B[i] - B[i].A):
od:
# B[i] is the i-th symmetrized covariant derivative of the curvature tensor

flatB := Matrix(n^2, k + 1):

for i from 1 to k + 1 do
    for j from 1 to n do
        for l from 1 to n do
            flatB[(j-1)*n + l, i] := B[i-1][j, l]:
        od:
    od:
od:
# flattening: the columns of B[i-1] are stacked into the i-th column of flatB
# columns 1..k are predictors (B[0]..B[k-1]),
# column k+1 is the right-hand side (B[k])

a := LinearSolve(flatB):
# without explicitly providing a separate right-hand side,
# LinearSolve(flatB) treats flatB as an augmented system and expresses
# the last column as a linear combination of the preceding columns

# for convenience, reverse the solution order
# so it matches the notation alpha_1,...,alpha_k
a_help := array(1..k):
for i from 1 to k do
    a_help[i] := - a[k + 1 - i]:
od:
a := a_help:

# minimal annihilating polynomial:
# p = lambda^k + sum_{i=1}^k a_i*lambda^(k-i)
p := collect(lambda^k + sum(a[f]*lambda^(k - f), f = 1 .. k), lambda):

# in the setting of the theorem (g positive definite, Ricci tensor nonzero),
# k is odd and a[1] = a[3] = ... = a[k] = 0
\end{verbatim}

In the case of an H-type group, explicit formulas for both the
symmetrized curvature tensor \(\scrR(X)\) and, above all, for
a natural \(\frakC_0\)-structure \(\C(X)\) are given in~\cite{BTV};
we recall them in Sections~\ref{se:K(X)} and~\ref{se:C0}.
However, the matrices of \(\C(X)\) and \(\scrR(X)\) are block-structured
with respect to the \(X\)-dependent decomposition of \(\n\) into smaller
subspaces described in Lemmas~\ref{le:decomposition_of_n} and
\ref{le:decomposition_of_h}.
In particular, the number and size of the blocks depend on the eigenvalue
branches \(\undermu\) of \(-\underK^2\) defined in \eqref{eq:underK_square}.
Hence, the general blueprint algorithm must be applied blockwise:
one first analyzes the individual blocks.
The output of Maple’s \emph{LinearSolve} routine, which yields the
pointwise factors of the minimal polynomial (for fixed \(X\)), is discussed in
Section~\ref{se:splitting}, where we also provide Maple code that can be used
without further modification.
This leads to a general pointwise factorization of the minimal
polynomial in \(\bbR[\lambda]\), stated in
Theorem~\ref{th:main_1}, whose proof is completed in
Section~\ref{se:proof_of_the_main_theorem_1}.
To determine the minimal polynomial of a specific H-type group, the eigenvalue
branches of \(-\underK^2\) must be computed explicitly.
These results are summarized in Theorem~\ref{th:main_2}.

\section{The main result}
\label{se:main}
First, we determine the general structure of the minimal
polynomial \(P_{\min}\), valid for any H-type group \((N,g)\) 
that is not a Euclidean space. Since this group is a homogeneous space, 
it suffices to work at the single point \(p = e \coloneqq e_N\), the unit element of
\(N\). By definition, \(g|_e\) is the scalar product on \(T_eN
\cong \n = \z \oplus \v\) as described earlier.

For each \(\undermu \in \bbR\) and every \(X \in \n\), we define
\(P_\undermu(X) \in \bbR[\lambda]\) by
\begin{equation}\label{eq:def_P_mu_1}
\begin{aligned}
 P_\undermu(X) \coloneqq
 &\ \lambda^6
  + \left( \tfrac{27}{2} \norm{X_\z}^2 + \tfrac{3}{2} \norm{X_\v}^2
    \right) \lambda^4\\
 &\quad\; + \left( \tfrac{729}{16} \norm{X_\z}^4 + \tfrac{81}{8}
    \norm{X_\v}^2 \norm{X_\z}^2 + \tfrac{9}{16} \norm{X_\v}^4 \right)
    \lambda^2\\
 &\quad\; +  \tfrac{729}{16} \norm{X_\z}^6 + \tfrac{243}{16}
    \norm{X_\z}^4 \norm{X_\v}^2 + \left( \tfrac{27}{16}
    - \tfrac{243}{64} \undermu \right) \norm{X_\z}^2 \norm{X_\v}^4
    + \tfrac{1}{16} \norm{X_\v}^6
\end{aligned}
\end{equation}
For every connected component \(U \subset \n \setminus \ram(\overK^2)\),
let \(\branch_{\mathrm{nc}}(U)\) denote the set of nonconstant
eigenvalue branches \(\undermu \colon U \to [0,1]\) of
\(- \underK^2\) defined in~\eqref{eq:underK_square}. For each
\(\undermu \in \branch_{\mathrm{nc}}(U)\) we define a smooth
function \(P_\undermu \colon U \to \bbR[\lambda]\) by
\begin{equation}\label{eq:def_P_mu_2}
  P_\undermu(X) = P_{\undermu(X)}(X)
\end{equation}
In particular, the product
\[
  \prod_{\undermu \in \branch_{\mathrm{nc}}(U)}\!\!\! P_\undermu
\]
defines a smooth map \(U \to \bbR[\lambda]\). Its degree is
constant, equal to six times the cardinality
\(|\branch_{\mathrm{nc}}(U)|\) on each connected component of
\(\n \setminus \ram(\overK^2)\).

Constant eigenvalue branches \(\nu\) of \(- \underK^2\) arise
from the analytic equation
\[
  \det\!\big(\overK(X)^2 - \nu\norm{X_\z}^2\norm{X_\v}^4\big) = 0
\]
In particular, \(\norm{X_\z}^2\norm{X_\v}^4\nu\) is an eigenvalue
of \(-\overK(X)^2\) for all \(X \in \n\). In this case we refer
to \(\nu\) as a \emph{global eigenvalue} of \(-\underK^2\).
Moreover, its multiplicity is fixed on \(\n \setminus \ram(\overK^2)\).
Indeed, such multiplicities are orders of vanishing of the characteristic polynomial at the
corresponding constant roots, and these orders cannot vary on
different open components without forcing a polynomial identity.

Proposition~\ref{p:overK} and Corollary~\ref{co:mu_=_constant}
show that only \(\nu = 0\) and \(\nu = 1\) can occur, with
multiplicity \(m_0 \in \{1,2\}\) in the case \(\nu = 0\). We set
\begin{align}\label{eq:def_P_0_sharp}
  &\begin{aligned}
    P_0^\sharp(X) \coloneqq \lambda^4
    &+ \left( \tfrac{45}{4}\norm{X_\z}^2 + \tfrac{5}{4}\norm{X_\v}^2
      \right)\lambda^2\\
    &+ \tfrac{81}{4}\norm{X_\z}^4 + \tfrac{9}{2}\norm{X_\z}^2
      \norm{X_\v}^2 + \tfrac{1}{4}\norm{X_\v}^4
  \end{aligned}\\
  \label{eq:def_P_1_sharp}
  &P_1^\sharp(X) \coloneqq \lambda^2 + \tfrac{9}{4}\norm{X_\z}^2
    + \norm{X_\v}^2
\end{align}
and finally
\begin{equation}\label{eq:def_P_n_3}
  P_{\n_3}(X) \coloneqq \lambda^3 + \norm{X}^2 \lambda
\end{equation}
see~\eqref{eq:min_pol_n_=_1}.

\begin{theorem}\label{th:main_1}
  Let \((N,g)\) be the H-type group associated with a Euclidean space \(\z\)
  of positive dimension and a nontrivial orthogonal \(\Cl(\z)\)-module \(\v\).
  There is the following pointwise factorization of the minimal
  polynomial in \(\bbR[\lambda]\):
  \begin{equation}\label{eq:P_min_2}
    P_{\min}(N,g)(X) = P_{\n_3}(X)
    \prod_{\substack{\nu \in \{0,1\}\\ m_\nu\,\ge\,2}}\!\!\!
    P_\nu^\sharp(X)\!\!\!\!
    \prod_{\undermu \in \branch_{\mathrm{nc}}(U(X))}\!\!\!\!\!\!
    P_\undermu(X)
  \end{equation}
  valid for all \(X \in \n \setminus \ram(\overK^2)\), where:
\begin{itemize}
\item \(\nu\) runs over the global eigenvalues \(\nu \in \{0,1\}\),
  with \(\nu = 0\) taken into account only if \(m_0 = 2\).
\item \(\undermu\) runs over \(\branch_{\mathrm{nc}}(U(X))\), the
  set of nonconstant eigenvalue branches of \(-\underK^2\) on the
  connected component \(U(X) \subseteq \n \setminus \ram(\overK^2)\),
\end{itemize}
\end{theorem}

From Theorem~\ref{th:main_1} we can already draw the following
conclusion: on each connected component of \(\n \setminus
\ram(\overK^2)\), the degrees of the factors on the right-hand
side of~\eqref{eq:P_min_2} must sum to \(k\), the degree of
\(P_{\min}(N,g)\). Since constant eigenvalue branches are
automatically global eigenvalues, the number of distinct
nonconstant eigenvalue branches of \(-\underK^2\) is the same on
every connected component of \(\n \setminus \ram(\overK^2)\).
Indeed, in the factorization \eqref{eq:P_min_2}, the degrees
contributed by the constant-eigenvalue factors \(P_\nu^\sharp\)
are fixed, so the remaining degree comes from the product over
\(\undermu\), each contributing \(6\) to the degree; hence their
number \(\ell\) is independent of the connected component.

Consequently, denoting by \(\branch_{\mathrm{nc}}(N)
= \{\undermu_1,\ldots,\undermu_\ell\}\) the set of distinct
nonconstant eigenvalue branches on any connected component of
\(\n \setminus \ram(\overK^2)\), we may briefly write
\[
  P_{\min} = P_{\n_3} \prod_{\substack{\nu \in \{0,1\}\\
  m_\nu \ge 2}} \!\!\!\!P_\nu^\sharp\; \prod_{i=1}^\ell
  P_{\undermu_i}
\]
In the same vein, since \(P_{\min}(N,g)\) is well-defined by
Theorem~\ref{th:JR_on_C_0_spaces}, the product on the
right-hand side of~\eqref{eq:P_min_2}, initially written on
\(\n \setminus \ram(\overK^2)\), extends across
\(\ram(\overK^2)\) as a polynomial in \(\lambda\) with
coefficients in \(\scrP(T_eN)\). Thus, near a point
\(X \in \ram(\overK^2)\), the number of incoming nonconstant
branch factors is independent of the connected component of
\(U \cap (\n \setminus \ram(\overK^2))\) from which \(X\) is
approached, where \(U\) is any sufficiently small connected open
neighborhood of \(X\). Furthermore, the classification shows that not
only the number of distinct nonconstant branches but also their
multiplicities are constant throughout \(\n \setminus
\ram(\overK^2)\) \footnote{However it is not true that \(\n \setminus
  \ram(\overK^2)\) is always connected. For example, for \(n = 3\) and a nonisotypic
\(\Cl(\z)\)-module \(r_1 > 0\) and \(r_2 > 0\) it is not connected.}

These observations set the framework in which the classification
of eigenvalue branches in Section~\ref{se:eigenvalue_branches}
must fit. We will return to this point in
Section~\ref{se:explicit_formulas}, but, in general, the eigenvalue
branches of an arbitrary family of nonnegative self-adjoint
operators need not exhibit the same favorable behavior described
there. The classification also shows that already
\[\prod_{\undermu \in \branch_{\mathrm{nc}}(N)} P_\undermu\] belongs to 
\(\scrP(T_eN)[\lambda]\). Hence the minimal polynomial \(P_{\min}\) of a
generalized Heisenberg group \(N\) associated with a nontrivial orthogonal
\(\Cl(\z)\)-module \(\v\) admits a factorization in
\(\scrP(T_eN)[\lambda]\) into the factors \(\lambda\),
\(\lambda^2+\norm{\,\cdot\,}^2\), and—unless \(\dim\z = 1\)—at
least one among \(\prod_{\undermu \in \branch_{\mathrm{nc}}(N)}
P_\undermu\), \(P_0^\sharp\), and \(P_1^\sharp\). In particular, even after removing the universal factor
\(\lambda\), the minimal polynomial need not be irreducible in
\(\scrP(T_eN)[\lambda]\).

Regarding the positivity assertion \(a_{2i}(X) > 0\) for \(X \in U(e_N)\) claimed by
Theorem~\ref{th:JR_on_C_0_spaces}, the occurrence of the term
\[
  \left(\frac{27}{16} - \frac{243}{64}\undermu(X)\right)\norm{X_\z}^{2}\norm{X_\v}^{4}
\]
in \eqref{eq:def_P_mu_1} makes the inequality \(a_{2i}(X) > 0\) for \(X \in U(e_N)\) less
immediate. However, it is straightforward that the entire constant coefficient
\[
  \frac{729}{16}\norm{X_\z}^{6}
  + \frac{243}{16}\norm{X_\z}^{4}\norm{X_\v}^{2}
  + \left(\frac{27}{16} - \frac{243}{64}\undermu(X)\right)\norm{X_\z}^{2}\norm{X_\v}^{4}
  + \frac{1}{16}\norm{X_\v}^{6}
\]
of \eqref{eq:def_P_mu_1} is strictly positive provided \(X \neq 0\) and \(\undermu(X) < 1\),
which are obvious prerequisites for \(X \in U(e_N)\). Therefore \(P_{\min}(N,g,e_N;X) > 0\), cf.
Corollary~\ref{co:positive}.

\subsection{\texorpdfstring{Overview of the classification of the
eigenvalue branches of \(-\underK^2\)}{Overview of the classification of
the eigenvalue branches of -K^2}}
\label{se:main_result}

Let \(\z\) be a Euclidean vector space of positive dimension \(n\), and
let \((N,g)\) be the H-type group associated with a nontrivial orthogonal
\(\Cl(\z)\)-module \(\v\). Once the eigenvalue branches of
\(-\underK^2\)~\eqref{eq:underK_square} have been determined, we can describe the minimal
polynomial \(P_{\min} \coloneqq P_{\min}(N,g)\) explicitly and compute
its degree using Theorem~\ref{th:main_1}.

Recall that \(\nu = 0\) is always a global eigenvalue since
\(\overK(X)X_\z = 0\) for every \(X \in \n\). Moreover, the multiplicity
satisfies \(m_0 \in \{1,2\}\) with \(n - m_0 \equiv 0 \pmod{2}\)
(cf.~Corollary~\ref{co:mu_=_constant}). Therefore, to apply
Theorem~\ref{th:main_1} to compute the degree \(k\) of \(P_{\min}(N,g)\),
it suffices to determine
\begin{itemize}
\item the cardinality \(\ell \coloneqq
  \#\bigl(\branch_{\mathrm{nc}}(N)\bigr)\), and
\item whether or not \(1\) is a global eigenvalue.
\end{itemize}

Since \(-\underK(X)^2\) acts on \(\z\) with \(\dim \z = n\), the remaining
eigenvalue branches have total multiplicity \(n - m_0\). Because
\(\underK(X)\) is skew-symmetric, each nonzero eigenvalue branch of
\(-\underK(X)^2\) has even multiplicity. Thus, the remaining
multiplicity splits into at most \(\tfrac{n - m_0}{2}\) distinct nonzero
eigenvalue branches.

Therefore, the maximal possible value of \(\ell\) is
\(\ell = \frac{n-2}{2}\) when \(n\) is even and \(\ell = \frac{n-1}{2}\)
when \(n\) is odd. In the maximal case, each branch has multiplicity
\(2\), i.e., \(m_{\undermu_i} = 2\) for all \(i = 1,\dotsc,\ell\), and
\(\nu = 1\) is not a global eigenvalue. If \(n\) is even, then
\begin{equation}\label{eq:P_min_n_even_1}
  P_{\min}
  = P_{\n_3}\,P_0^\sharp\,P_{\undermu_1}\cdots P_{\undermu_{(n-2)/2}}
\end{equation}
where \(P_{\n_3}\), \(P_0^\sharp\), and each \(P_{\undermu_i}\) have
degrees \(3\), \(4\), and \(6\), respectively. Hence
\[
  k = 3 + 4 + 6 \cdot \frac{n-2}{2} = 3n + 1
\]
Similarly, if \(n\) is odd, the minimal polynomial has the form
\begin{equation}\label{eq:P_min_n_odd_1}
  P_{\min}
  = P_{\n_3}\,P_{\undermu_1}\cdots P_{\undermu_{(n-1)/2}}
\end{equation}
and, with no \(P_0^\sharp\) factor, the degrees sum to
\[
  k = 3 + 6 \cdot \frac{n-1}{2} = 3n
\]
This describes the generic situation, which occurs for most orthogonal
\(\Cl(\z)\)-modules and, in particular, for all \(n \ge 10\) according
to Propositions~\ref{p:n_ge_8} and~\ref{p:mult_ge_10}.

Higher-multiplicity eigenvalue branches are exceptional for two reasons.
First, Proposition~\ref{p:m_geq_4} and Corollary~\ref{co:fixed_point_group}
in Section~\ref{se:gradient} relate multiplicities \(m \ge 4\) to the rank
of the common fixed point group \(\Fix(S,\Spin(n-1))\) of a generic tuple of
spinors \(S \in \v\). Because common fixed point groups of spinors are well
understood when \(\dim \z\) is small, this streamlines the classification of
eigenvalue branches and their multiplicities. Second, for \(n \ge 10\), the
fundamental inequality \eqref{eq:dimensions_abschaetzung_1} excludes
eigenvalue branches of multiplicity \(m \ge 4\) \emph{from the outset}.

On the other hand, for \(n \le 9\), there are cases in which eigenvalue
branches of \(-\underK^2\) occur with multiplicity greater than \(2\).
More specifically, nonconstant eigenvalue branches of \(-\underK^2\) of
multiplicity at least \(4\) arise as follows:

\smallskip
\noindent
\emph{Case \(n = 7\).} Let \(\bbO\) and \(\overline{\bbO}\) denote the
two distinct \(8\)-dimensional \(\Cl(7)\)-modules (pinor types), namely
the spaces of positive and negative pinors \cite{Ba,Ha}. They play, for
\(n = 7\), the role that the half-spin representations (positive and
negative spinors) play in the representation theory of \(\Spin(\z)\) for
\(n = 8\); cf. \cite[Ch.~2.3]{Ba} and \cite[Ch.~1,
Proposition~5.12]{LM}.

If \(\v = \bbO \oplus \bbO\) or \(\v = \overline{\bbO} \oplus
\overline{\bbO}\), then \(\nu = 1\) is a global eigenvalue of
\(-\underK^2\) of multiplicity \(2\). Here \(\n \setminus \ram(\overK^2)\)
is connected and there is also a nonconstant
eigenvalue branch \(\undermu\) of multiplicity \(4\), so
\[
  P_{\min} = P_{\n_3}\,P_1^\sharp\,P_{\undermu}
  \qquad\text{and}\qquad
  k = 3 + 2 + 6 = 11
\]
which is smaller than the generic degree \(3 \cdot 7 = 21\).

In the second case, \(\v = \bbO \oplus \overline{\bbO}\),
\(\n \setminus \ram(\overK^2)\) decomposes into two connected components.
On each component, there are two
distinct nonconstant eigenvalue branches \(\undermu_1\) and
\(\undermu_2\) of multiplicities \(2\) and \(4\), respectively. Here,
\[
  P_{\min} = P_{\n_3}\,P_{\undermu_1}\,P_{\undermu_2}
  \qquad\text{and}\qquad
  k = 3 + 6 + 6 = 15
\]
which is again smaller than \(21\)
(see Proposition~\ref{p:n_=_7}).

\emph{Case \(n = 8\).} For an irreducible orthogonal \(\Cl(\z)\)-module \(\v\),
Proposition~\ref{p:n_=_8} shows that \(\n \setminus \ram(\overK^2)\) is connected
and that there is a single nonconstant eigenvalue branch \(\undermu\), with multiplicity \(6\).
Then,
\[
  P_{\min} = P_{\n_3}\,P_0^\sharp\,P_{\undermu}
  \qquad\text{with}\qquad
  k = 3 + 4 + 6 = 13
\]
If \(\v\) is a direct sum of at least two irreducible \(\Cl(\z)\)-modules,
then
\[
  k = 3 \cdot 8 + 1 = 25
\]

\smallskip
\noindent
\emph{Case \(n = 9\).} For an irreducible orthogonal \(\Cl(\z)\)-module,
Proposition~\ref{p:n_=_9} yields three nonconstant eigenvalue branches,
two of multiplicity \(2\) and one of multiplicity \(4\) on each connected
component of \(\n \setminus \ram(\overK^2)\); there are two
such components. It follows that
\[
  P_{\min} = P_{\n_3}\,P_{\undermu_1}P_{\undermu_2}P_{\undermu_3}
  \qquad\text{and}\qquad  k = 3 + 6 + 6 +6  = 21
\]
In the reducible case,
\[
  k = 3 \cdot 9 = 27
\]

For the remaining possible global eigenvalue \(\nu = 1\), a well-studied
condition is the \(J^2\)-condition (cf.~\cite[Ch.~3.1.3]{BTV}), which
asserts that
\[
  \underK^2|_{\h} = -\Id
\]
More precisely, the \(J^2\)-condition means that \(\underK(X)^2 Z = -Z\)
for all \(X \in \n\) such that \(X_\z \neq 0\), \(X_\v \neq 0\)
and all \(Z \in \h_X\), cf.~\eqref{eq:def_h_X}. In this situation, the multiplicity of
\(\nu = 1\) is automatically \(n - 1\), and the only other eigenvalue is
\(\nu = 0\) with multiplicity \(m_0 = 1\). This condition holds
precisely for \(n \in \{3, 7\}\): for \(n = 3\) when \(\v\) is an isotypic
orthogonal \(\Cl(\z)\)-module, and for \(n = 7\) when \(\v\) is an irreducible
orthogonal \(\Cl(\z)\)-module; cf. \cite[Theorem~1.1]{CDKR}.\footnote{Since we do
not rely on~\cite{CDKR}, our classification of the eigenvalue branches of
\(-\underK^2\) yields an independent (albeit somewhat lengthy) proof of
this notable result.}
Here
\[
  P_{\min} = P_{\n_3}\,P_1^\sharp
  \quad\text{with}\quad k = 5
\]
More broadly, the occurrence of \(\nu = 1\) as a global eigenvalue of
\(-\underK^2\) is restricted on structural grounds to dimensions
\(n \le 7\): by Proposition~\ref{p:n_=_8}, \(\nu = 1\) is not a global
eigenvalue for \(n = 8\). By the eightfold periodicity of Clifford
algebras and their modules, the same holds true for all \(n \ge 8\)
(see Proposition~\ref{p:n_ge_8}).

For \(n = 6\) and \(\dim \v = 8\), the group \(\Spin(n)\) acts
transitively on the product
\[
  \rmS^{n-1} \times \rmS^{\dim\v-1} \subseteq \z \oplus \v
\]
of the vector unit sphere \(\rmS^{n-1} \subseteq \z\) and the spinor
unit sphere \(\rmS^{\dim\v-1} \subseteq \v\), exactly as in the
situations where the \(J^2\)-condition holds. Hence all eigenvalue
branches are constant. Here \(\nu = 1\) occurs as a global eigenvalue of
maximal possible multiplicity \(m_1 = 4\), and \(0\) is a global
eigenvalue of multiplicity \(m_0 = 2\) (so the \(J^2\)-condition is not
satisfied). Then \(P_{\min} = P_{\n_3}\,P_0^\sharp\,P_1^\sharp\) and
\[
  k = 3 + 4 + 2 = 9
\]
which is smaller than the generic degree \(k = 3 \cdot 6 + 1 = 19\)
attained whenever \(\dim \v \ge 16\), i.e., \(\v\) is the direct sum of
at least two irreducible \(\Cl(\z)\)-modules, see
Proposition~\ref{p:n_=_4_5_6}.

For \(n = 5\) and \(\dim \v = 8\), the group \(\Spin(n)\)
still acts transitively on the spinor unit sphere
\(\rmS^{7} \subseteq \v\). Moreover, the global eigenvalue
\(\nu = 1\) occurs again, with multiplicity \(m_1 = 2\), while
\(\n \setminus \ram(\overK^2)\) has two connected components,
each carrying a second, nonconstant eigenvalue branch \(\undermu\),
necessarily of multiplicity \(m_{\undermu} = 2\). In this case,
\[
  P_{\min} = P_{\n_3}\,P_1^\sharp\,P_{\undermu}
  \quad\text{with}\quad k = 3 + 2 + 6 = 11
\]
This degree is smaller than the generic degree \(k = 3 \cdot 5 = 15\),
which is attained whenever \(\dim \v \geq 16\).

In the reducible case, the occurrence of the global eigenvalue
\(\nu = 1\) is restricted \emph{a priori} to \(n \in \{3,7\}\) by
Corollary~\ref{co:reducible_Clifford_module}. For \(n = 3\), all cases in
which \(\nu = 1\) occurs satisfy the \(J^2\)-condition. For \(n = 7\), the
global eigenvalue \(\nu = 1\) also appears when \(\v\) is the direct sum
of two copies of the same irreducible \(\Cl(7)\)-representation, as noted
above. This phenomenon is related to a result on the topology of the
space of orthogonal complex structures in dimension \(6\); see
\cite{FGM}. In this case the multiplicity is \(m_1 = 2\), so the
\(J^2\)-condition is not satisfied.

These explanations are distilled
into the following theorem, which constitutes the main result of this
article:

\begin{theorem}\label{th:main_2}
Let \(\z\) be an \(n\)-dimensional Euclidean vector space with \(n\ge 1\), 
let \(\v\) be a nontrivial orthogonal \(\Cl(\z)\)-module, and let \((N,g)\)
be the H-type group associated with the pair \((\z, \v)\). Then the eigenvalue
branches of \(-\underK^2\), the minimal polynomial \(P_{\min}\) of \((N,g)\),
and its degree \(k\) are as follows.
\begin{enumerate}
\item Suppose \(n \le 9\) and \(\v\) is an irreducible \(\Cl(\z)\)-module:
\begin{center}
\begin{tabular}{l|c|c|c|l|l}
\hline
\(n\) & \(\ell\) & \(\nu = 0\) with \(m_0 \ge 2\) & \(\nu = 1\) &
\(P_{\min}\) & \(k\) \\
\hline
1 & 0 & no  & no  & \(P_{\n_3}\) & 3 \\
2 & 0 & yes & no  & \(P_{\n_3}\,P_0^\sharp\) & 7 \\
3 & 0 & no  & yes & \(P_{\n_3}\,P_1^\sharp\) & 5 \\
4 & 1 & yes & no  & \(P_{\n_3}\,P_0^\sharp\,P_{\undermu}\) & 13 \\
5 & 1 & no  & yes & \(P_{\n_3}\,P_1^\sharp\,P_{\undermu}\) & 11 \\
6 & 0 & yes & yes & \(P_{\n_3}\,P_0^\sharp\,P_1^\sharp\) & 9 \\
7 & 0 & no  & yes & \(P_{\n_3}\,P_1^\sharp\) & 5 \\
8 & 1 & yes & no  & \(P_{\n_3}\,P_0^\sharp\,P_{\undermu}\) & 13 \\
9 & 3 & no  & no  &
  \(P_{\n_3}\,P_{\undermu_1}\,P_{\undermu_2}\,P_{\undermu_3}\) & 21 \\
\hline
\end{tabular}
\end{center}
Here \(\ell\) is the number of nonconstant eigenvalue branches of
\(-\underK^2\) on any connected component of
\(\n \setminus \ram(\overK^2)\).

\item Suppose \(\dim \z = 3\) and \(\v\) is isotypic, i.e.,
\(\v = \bigoplus_{i=1}^j \bbH\) or
\(\v = \bigoplus_{i=1}^j \overline{\bbH}\) for some \(j\), where
\(\bbH\) and \(\overline{\bbH}\) are the two distinct irreducible
\(\Cl(3)\)-modules. Then the \(J^2\)-condition holds, and
\(P_{\min} = P_{\n_3}\,P_1^\sharp\) with \(k = 5\), as in the irreducible
case.

\item Suppose \(\dim \z = 7\) and \(\v\) is the direct sum of two
equivalent pinor representations, \(\v \cong \bbO \oplus \bbO\) or
\(\v \cong \overline{\bbO} \oplus \overline{\bbO}\). Then \(\nu = 1\)
occurs as a global eigenvalue of \(-\underK^2\) with multiplicity
\(m_1 = 2\), and there is a nonconstant eigenvalue branch \(\undermu\)
with multiplicity \(m_{\undermu} = 4\) on , so
\(P_{\min} = P_{\n_3}\,P_1^\sharp\,P_{\undermu}\) and \(k = 11\).

\item Suppose \(\dim \z = 7\) and \(\v \cong \bbO \oplus \overline{\bbO}\)
is the direct sum of two different pinor representations. Then there are
two nonconstant eigenvalue branches \(\undermu_1\) and \(\undermu_2\) of
\(-\underK^2\) with multiplicities \(m_{\undermu_1} = 2\) and
\(m_{\undermu_2} = 4\), so
\(P_{\min} = P_{\n_3}\,P_{\undermu_1}\,P_{\undermu_2}\) and \(k = 15\).

\item If none of the above cases applies, then \(\nu = 1\) does not occur
as a global eigenvalue of \(-\underK^2\). The number \(\ell\) of
nonconstant eigenvalue branches on any connected component of
\(\n \setminus \ram(\overK^2)\) is
\[
  \ell = \frac{n - 2}{2} \quad\text{if \(n\) is even}, \qquad
  \ell = \frac{n - 1}{2} \quad\text{if \(n\) is odd}
\]
The minimal polynomials are then given by \eqref{eq:P_min_n_even_1} and
\eqref{eq:P_min_n_odd_1}. Therefore \(k = 3n + 1\) if \(n\) is even, and
\(k = 3n\) if \(n\) is odd.
\end{enumerate}
\end{theorem}

\subsection{\texorpdfstring{Explicit formulas for some eigenvalue
branches of \(-\overK^2\)}{Explicit formulas for some eigenvalue
branches of -K^2}}\label{se:explicit_formulas}

Recall that Theorem~\ref{th:JR_on_C_0_spaces} asserts that the
coefficient functions of \(P_{\min}(N,g)\) are polynomial functions
that extend to \(\Iso(N,g)\)-invariant Killing tensors. We now briefly
explain this assertion, with a view toward the low-dimensional and
pinor cases singled out in Theorem~\ref{th:main_2}.

Since the isometry group \(\Iso(N,g)\) is the semidirect product
\(N \rtimes \Aut(\n,g)\) (see \cite[p.~134]{Ka1}), a polynomial
function \(P \colon \n \to \bbR\) extends to an
\(\Iso(N,g)\)-invariant tensor field on \(N\) if and only if it is
invariant under \(\Aut(\n,g)\). Accordingly, let
\(\scrP(\n)[\lambda]_{\Aut(\n,g)}\) denote the subalgebra of
\(\scrP(\n)[\lambda]\) consisting of the \(\Aut(\n,g)\)-invariant
polynomials, that is, polynomials
\(P = \sum_{i=0}^{k} a_i\lambda^{k-i}\) whose coefficient functions
\(a_i \colon \n \to \bbR\) are invariant under the action of
\(\Aut(\n,g)\).

From the pointwise factorization~\eqref{eq:P_min_2} of the minimal
polynomial and the explicit formulas
\eqref{eq:def_P_mu_1}--\eqref{eq:def_P_n_3} for its factors, it follows
that the even \(\lambda\)-coefficients \(a_{2i}(X)\) of
\(P_{\min}(N,g)\) can be expressed as polynomials in
\(\norm{X_\z}^{2}\), \(\norm{X_\v}^{2}\), and the elementary symmetric
polynomials
\begin{equation}\label{eq:over_sigma_j}
  \oversigma_j(X)
  =
  \sigma_j\bigl(\overmu_1(X), \ldots, \overmu_\ell(X)\bigr)
\end{equation}
in the {\em distinct} eigenvalue branches
\[
  \overmu_i(X)
  =
  \norm{X_\z}^2\norm{X_\v}^4\undermu_i(X),
  \qquad i=1,\ldots,\ell,
\]
of \(-\overK(X)^2\), corresponding to the nonconstant eigenvalue
branches \(\undermu_1,\ldots,\undermu_\ell\) of
\(-\underK(X)^2\), for \(1 \le j \le \ell\) and
\(X \in \n \setminus \ram(\overK^2)\). The argument is routine.
Nevertheless, we give the details in
Section~\ref{se:expression_as_polynomials}, together with closed-form
expressions for the coefficients of the product of those factors in
\eqref{eq:P_min_2} that arise from the distinct nonconstant eigenvalue
branches \(\overmu_1(X), \ldots, \overmu_\ell(X)\); see
\eqref{eq:D_q_l}--\eqref{eq:product_of_the_P_undermu_i}.

Moreover, each \(\oversigma_j(X)\) is automatically
\(\Aut(\n,g)\)-invariant, since \(\overK\) is
\(\Aut(\n,g)\)-equivariant; compare~\eqref{eq:K_Aut_equivariance}.
However, even though \(\overK\) is a polynomial family of operators,
without recourse to Theorem~\ref{th:JR_on_C_0_spaces} it is not
\emph{a priori} clear that these invariant functions are polynomial in
the cases where multiplicities greater than two occur. This has to be
verified through the explicit formulas.

What's more, H-type groups need not be G.O.\ spaces. Thus,
\(\Aut(\n,g)\)-invariance of a polynomial function on \(\n\) does not,
in general, imply that the associated \(\Iso(N,g)\)-invariant symmetric
tensor field is a Killing tensor. Here, however, the tensor fields
induced by the functions \(\oversigma_j\) are Killing for an independent
reason: the eigenvalue branches of \(\underK^2\) are constant along every
geodesic (cf.~Proposition~\ref{p:eigenvalues_of_K_and_tilde_K}), while
\[
  g_{\z}(X) \coloneqq \norm{X_\z}^2,
  \qquad
  g_{\v}(X) \coloneqq \norm{X_\v}^2
\]
are Killing tensors. Consequently, the polynomial combinations appearing
in the functions \(\oversigma_j\) induce Killing tensors. Nevertheless,
we corroborate the Killing property independently in
Section~\ref{se:verification}, providing an additional check on the
explicit formulas for the eigenvalue branches presented below. Since the
space of Killing tensors is closed under polynomial combinations, it
follows that all coefficients of \(P_{\min}(N,g)\) are Killing tensors.

In order to explain why each \(\oversigma_j\) is, in fact, always an
\(\Aut(\n,g)\)-invariant polynomial function on \(\n\), let us first
assume that \(\ell\) is maximal, i.e.,
\[
  \ell = \frac{n-2}{2}
  \quad\text{or}\quad
  \ell = \frac{n-1}{2}
\]
according to whether \(n \coloneqq \dim \z\) is even or odd. In this
case, \(\nu = 1\) is not a global eigenvalue of \(-\underK^2\), and,
away from \(\ram(\underK^2)\), each nonzero eigenvalue of
\(-\underK(X)^2\) has multiplicity \(2\). Consequently, the
characteristic polynomial of \(\overK(X)\) can be written as
\begin{equation}\label{eq:characteristic_polynomial}
  \det\bigl(\lambda \,\Id - \overK(X)\bigr)
  = \sum_{j=0}^\ell \oversigma_j(X) \,\lambda^{n-2j}
\end{equation}
Because \(\overK\) is a polynomial family of \(\Aut(\n,g)\)-invariant
operators by~\eqref{eq:K_Aut_equivariance}, each coefficient
\(\oversigma_j\) is then an \(\Aut(\n,g)\)-invariant polynomial
function.

The simplest nontrivial example is the case \(n = 3\). Let \(\bbH\)
denote the \(4\)-dimensional normed division algebra of quaternions, and
set \(\z \coloneqq \Im(\bbH)\), the subspace of purely imaginary
quaternions. Right multiplication by \(S \in \z\),
\begin{equation}\label{eq:Clifford_multiplication_for_n_=_3_+}
  S \cbullet \tilde{S} \coloneqq \tilde{S} S
\end{equation}
for all \(\tilde{S} \in \bbH\) defines the orthogonal 
\(\Cl(\z)\)-module \(\bbH\). With the opposite sign,
\begin{equation}\label{eq:Clifford_multiplication_for_n_=_3_-}
  S \cbullet \tilde{S} \coloneqq -\,\tilde{S} S
\end{equation}
we obtain the orthogonal \(\Cl(\z)\)-module \(\overline{\bbH}\). 
Here the Clifford algebra
\[
  \Cl(\z) \cong \bbH \oplus \bbH
\]
is the direct sum of two copies of \(\bbH\), reflecting the fact that
\(\bbH\) and \(\overline{\bbH}\) are two nonisomorphic irreducible
\(\Cl(\z)\)-modules.

Let \(\v = \v_+ \oplus \v_-\) be the isotypic decomposition of \(\v\),
with \(\v_+ \cong \bigoplus_{i = 1}^{r_1}\bbH\) and
\(\v_- \cong \bigoplus_{i = 1}^{r_2}\overline{\bbH}\), where
\(r_1 = \tfrac{1}{4}\dim(\v_+)\) and \(r_2 = \tfrac{1}{4}\dim(\v_-)\).
In this case,
\[
  \Aut(\n,g)
  \;=\;
  \Bigl(
    \Spin(\Im(\bbH)) \times \mathbf{Sp}(r_1,\bbH)
    \times \mathbf{Sp}(r_2,\bbH)
  \Bigr)\Big/\bigl\langle(-1,-\Id_{\v})\bigr\rangle
\]
if \(r_1 \neq r_2\). If \(r_1=r_2=r\), then \(\Aut(\n,g)\) has an
additional component generated by the involution
\[
  \tau(S_0\oplus S_+\oplus S_-)\coloneqq (-S_0)\oplus S_-\oplus S_+
\]
and hence
\[
  \Aut(\n,g)
  \;=\;
  \Bigl(
    \bigl(
      \Spin(\Im(\bbH)) \times \mathbf{Sp}(r,\bbH)
      \times \mathbf{Sp}(r,\bbH)
    \bigr)\big/\bigl\langle(-1,-\Id_{\v})\bigr\rangle
  \Bigr)\rtimes \langle\tau\rangle
\]
The unique nonzero eigenvalue branch \(\overmu\) of \(-\overK^2\), given
by
\begin{equation}\label{eq:over_mu_for_n_=_3}
  \overmu(S_0 \oplus S_+ \oplus S_-) \coloneqq
  \norm{S_0}^2 \bigl(\norm{S_+}^2 - \norm{S_-}^2\bigr)^2
\end{equation}
for all \(S_0 \oplus S_+ \oplus S_- \in \z \oplus \v_+ \oplus \v_-\),
defines an \(\Aut(\n,g)\)-invariant polynomial. Moreover, this yields a
Killing tensor because both \(g_{\z}\) and \(g_{\v}\) are Killing
(cf. Corollary~\ref{co:Killing_1}). By the description of the
Levi-Civita connection in \eqref{eq:LC}, H-type subgroups associated
with \(\Cl(\z)\)-submodules are totally geodesic; hence both
\(g_{\v_+}\) and \(g_{\v_-}\) are Killing as well.

The corresponding rescaled eigenvalue branch \(\undermu\) in
\eqref{eq:undermu_overmu} is
\begin{equation}\label{eq:under_mu_for_n_=_3}
  \undermu(S_0 \oplus S_+ \oplus S_-) \coloneqq
  \frac{\bigl(\norm{S_+}^2 - \norm{S_-}^2\bigr)^2}
       {\bigl(\norm{S_+}^2 + \norm{S_-}^2\bigr)^2}
\end{equation}
If \(\v\) is isotypic, then one of the summands \(\v_+\) or \(\v_-\)
is trivial. Hence \eqref{eq:under_mu_for_n_=_3} is identically equal to
\(1\), so the \(J^2\)-condition holds and \(\nu = 1\) is a global
eigenvalue. Otherwise, \eqref{eq:under_mu_for_n_=_3} defines the unique
nonconstant eigenvalue branch \(\undermu\), of multiplicity \(2\), on
each connected component of \(\n \setminus \ram(\overK^2)\).
\footnote{On the natural domain where \(S_0 \neq 0\) and
\(S_+ \oplus S_- \neq 0\), the nonisotypic unramified locus has two
connected components, distinguished by the sign of
\(d \coloneqq \norm{S_+}^2 - \norm{S_-}^2\), as follows from
\eqref{eq:over_mu_for_n_=_3}.}

For \(n = 4\), let \(\bbO\) denote the \(8\)-dimensional normed division
algebra of octonions and let \(\Im(\bbO)\) be the subspace of purely
imaginary octonions. Fix a \(4\)-dimensional subspace
\(\z \subseteq \Im(\bbO)\), whose orthogonal complement
\(\z^\perp \subseteq \Im(\bbO)\) is a \(3\)-plane. Let
\(\{E_1,E_2,E_3\}\) be an oriented orthonormal basis of \(\z^\perp\), and
put \(\v \coloneqq \bbO\), with Clifford action given by right octonionic
multiplication \(Z \cbullet S \coloneqq SZ\) for \(Z \in \z\) and \(S \in \v\). 
Then \(\v\) is an irreducible orthogonal \(\Cl(\z)\)-module, and
\[
  \Aut(\n,g)
  = \bigl(\Pin(\z)\times \mathbf{Sp}(1,\bbH)\bigr)
    \big/\bigl\langle(-1,-\Id_{\v})\bigr\rangle
\]
The endomorphisms
\(I_1 \coloneqq \rmR_{E_1}\rmR_{E_2}\),
\(I_2 \coloneqq \rmR_{E_2}\rmR_{E_3}\), and
\(I_3 \coloneqq \rmR_{E_3}\rmR_{E_1}\) define a quaternionic unitary
structure on \(\v\) commuting with the \(\Cl(\z)\)-action, thereby
realizing the \(\mathbf{Sp}(1,\bbH)\)-factor of \(\Aut(\n,g)\). Hence the
polynomial \(\overmu:\z\oplus\bbO\to\bbR\) defined by
\begin{equation}\label{eq:over_mu_for_n_=_4}
  \overmu(S_0\oplus S_1)
  = \sum_{\alpha=1}^3 \bigl\langle I_\alpha S_1,\; S_1S_0\bigr\rangle^2
\end{equation}
is \(\Aut(\n,g)\)-invariant for all \(S_0 \oplus S_1 \in \z \oplus \bbO\).

It describes a polynomial eigenvalue branch of \(-\overK^{2}\) of multiplicity
\(2\). The corresponding rescaled expression \(\undermu\), obtained by
dividing \eqref{eq:over_mu_for_n_=_4} by
\(\norm{S_0}^{2}\norm{S_1}^{4}\), defines the unique nonconstant
eigenvalue branch of \(-\underK^{2}\) on \(\n \setminus \ram(\overK^{2})\). Thus,
\[
  P_{\min}(N,g) = P_{\n_3}\,P_{0}^\sharp\,P_{\undermu}
  \in \scrP(\n)[\lambda]_{\Aut(\n,g)}
\]
in accordance with Theorem~\ref{th:JR_on_C_0_spaces}.

The remaining cases are handled as follows. Suppose that \(\nu = 1\) is a
global eigenvalue of \(-\underK^2\) with multiplicity \(2\), and that
\(\undermu\) is the only nonconstant eigenvalue branch, with multiplicity
\(m\). Then, even without invoking
Theorem~\ref{th:JR_on_C_0_spaces}, we obtain that
\[
  \overmu(X)
  =
  \frac{1}{m}\Bigl(
    \trace\bigl(-\overK(X)^2\bigr)
    - 2\,\norm{X_\z}^{2}\,\norm{X_\v}^{4}
  \Bigr)
\]
is a polynomial function.

For instance, consider \(n = 5\). Fix a \(5\)-dimensional subspace
\(\z \subseteq \Im(\bbO)\), so its orthogonal complement
\(\z^\perp \subseteq \Im(\bbO)\) is a \(2\)-plane. Let
\(\{E_1, E_2\}\) be an oriented orthonormal basis of \(\z^\perp\), and
set \(\v \coloneqq \bbO\) with Clifford multiplication
\(S \cbullet \tilde S \coloneqq \tilde S S\), given by right octonionic
multiplication as before. Here
\[
  \Aut(\n,g)
  = \bigl(\Pin(\z)\times \U(1,\bbC)\bigr)
    \big/\bigl\langle(-1,-\Id_{\v})\bigr\rangle
\]
and \(J \coloneqq \rmR_{E_1}\rmR_{E_2}\) defines an orthogonal complex
structure on \(\v\) that commutes with the \(\Cl(\z)\)-action, thereby
realizing the \(\U(1,\bbC)\)-factor of \(\Aut(\n,g)\). A nonzero
eigenvalue branch \(\overmu\) of the polynomial family \(-\overK^{2}\) is
given by the \(\Aut(\n,g)\)-invariant polynomial
\begin{equation}\label{eq:over_mu_for_n_=_5}
  \overmu(S_0 \oplus S_1) \coloneqq
  \bigl\langle (S_1 S_0),\; J S_1 \bigr\rangle^{2}
\end{equation}
for \(S_0 \oplus S_1 \in \z \oplus \bbO\). The corresponding rescaled
expression \(\undermu\) yields the unique nonconstant eigenvalue branch of the
family \(-\underK^{2}\) defined in \eqref{eq:underK_square} on each connected
components of \(\n \setminus \ram(\overK^{2})\).\footnote{On
the natural domain \(S_0\neq0\), \(S_1\neq0\), the
unramified locus has two connected components, distinguished
by the sign of \(d \coloneqq \langle (S_1 S_0),\; J S_1 \bigr\rangle\),
as follows from~~\eqref{eq:over_mu_for_n_=_5}.} Therefore,
\[
  P_{\min}(N,g) = P_{\n_3}\,P_{1}^\sharp\,P_{\undermu}
  \in \scrP(\n)[\lambda]_{\Aut(\n,g)}
\]
as required by Theorem~\ref{th:JR_on_C_0_spaces}. The Killing property
for this and the previous example is verified in
Section~\ref{se:verification_for_n_=_4_and_n_=_5}.

For \(n = 7\), there is a decomposition
\[
  \Cl(\Im(\bbO)) \cong \Mat_8(\bbR) \oplus \Mat_8(\bbR)
\]
of \(\Cl(\Im(\bbO))\) as the direct sum of two copies of \(\Mat_8(\bbR)\),
the algebra of \(8 \times 8\) real matrices; see also \cite[p.~283]{Ha}.
Hence there are two distinct irreducible \(\Cl(\Im(\bbO))\)-modules,
\(\bbO\) and \(\overline{\bbO}\), called the spaces of positive and
negative pinors \cite{Ba}. On \(\bbO\) we again use right multiplication
by \(S \in \Im(\bbO)\) to define the structure of an orthogonal 
\(\Cl(\Im(\bbO))\)-module
\begin{equation}\label{eq:Clifford_multiplication_for_n_=_7_1}
  S \cbullet \tilde S \coloneqq \tilde S S
\end{equation}
for \(\tilde S \in \bbO\). For \(\tilde S \in \overline{\bbO}\) we take
the negative sign:
\begin{equation}\label{eq:Clifford_multiplication_for_n_=_7_2}
  S \cbullet \tilde S \coloneqq -\,\tilde S S
\end{equation}
If \(\v\) is the orthogonal direct sum \(\v_1 \oplus \v_2\), where
\(\v_1\) and \(\v_2\) are isomorphic irreducible \(\Cl(\Im(\bbO))\)-modules
(either \(\v_1 \cong \v_2 \cong \bbO\) or
\(\v_1 \cong \v_2 \cong \overline{\bbO}\)), then
\[
  \Aut(\n,g)
  = \Spin(\Im(\bbO)) \times \Ogroup(2)
    \big/ \big\langle (-1, -\Id_{\v}) \big\rangle
\]
Since odd Clifford elements exchange \(\bbO\) and \(\overline{\bbO}\), no
orientation-reversing map of \(\Im(\bbO)\) preserves \(\v_1\oplus\v_2\),
so the \(\Im(\bbO)\)-factor is \(\Spin(\Im(\bbO))\), not
\(\Pin(\Im(\bbO))\). A polynomial eigenvalue branch of \(\overK^2\)
is given by the \(\Aut(\n,g)\)-invariant homogeneous polynomial
\begin{equation}\label{eq:over_mu_for_n_=_7_1}
  \overmu(S_0 \oplus S_1 \oplus S_2) \coloneqq
  \norm{S_0}^2\big(\norm{S_1}^4 + \norm{S_2}^4\big)
  + 2\,\big\langle S_2^*(S_1 S_0), (S_0 S_2^*) S_1 \big\rangle
\end{equation}
where \(^{*}\) denotes octonionic conjugation. The corresponding
rescaled expression ---obtained
by dividing \eqref{eq:over_mu_for_n_=_7_1} by
\(\norm{S_0}^2\big(\norm{S_1}^2 + \norm{S_2}^2\big)^2\)---
yields the unique nonconstant eigenvalue branch \(\undermu\) of \(-\underK^2\)
on \(\n \setminus \ram(\overK^{2})\).

The mixed term \(\big\langle S_2^*(S_1 S_0), (S_0 S_2^*) S_1 \big\rangle\) is
\(\Spin(\Im(\bbO))\)-invariant by the Triality Theorem
\cite[Theorem~14.19]{Ha}; hence \(\overmu\) is
\(\Spin(\Im(\bbO))\)-invariant as well. Both summands
\(\norm{S_0}^2(\norm{S_1}^2 + \norm{S_2}^2)^2\) and
\(\big\langle S_2^*(S_1 S_0), (S_0 S_2^*) S_1 \big\rangle\) induce
Killing tensors on \((N,g)\); thus \(\overmu\) is Killing as well, cf.
Section~\ref{se:verification_for_n_=_7}.
 Consequently, the coefficients
of
\[
  P_{\min}(N,g) = P_{\n_3}\,P_{1}^\sharp\,P_{\undermu}
\]
are \(\Aut(\n,g)\)-invariant polynomial functions on \(\n\) that yield
Killing tensors on \((N,g)\).

When \(-\underK^2\) has nonconstant eigenvalue branches of different
multiplicities, the determinantal identity \eqref{eq:characteristic_polynomial} is no longer
useful for computing the elementary symmetric functions \(\oversigma_j\)
in the \emph{distinct} eigenvalue branches \(\overmu_i\). These cases
must therefore be analyzed separately in order to understand how the
assertions of Theorem~\ref{th:JR_on_C_0_spaces} manifest in this setting.

As a concrete illustration, consider again \(n = 7\), but suppose now
that \(\v_1\) and \(\v_2\) are nonisomorphic irreducible
\(\Cl(\Im(\bbO))\)-modules. In this case the Clifford multiplication is given by
\begin{equation}\label{eq:Clifford_multiplication_for_n_=_7_3}
  S_0 \cbullet (S_+ \oplus S_-) \coloneqq S_+ S_0 \oplus (-S_- S_0)
\end{equation}
for all \(S_0 \oplus S_+ \oplus S_- \in \Im(\bbO) \oplus \bbO \oplus
\overline{\bbO}\). Moreover, the automorphism group \(\Aut(\n,g)\) is the
semidirect product
\[
  \Bigl(
    \bigl(\Spin(\Im(\bbO)) \times \Ogroup(1) \times \Ogroup(1)\bigr)
    \big/ \bigl\langle (-1, -\Id_{\v}) \bigr\rangle
  \Bigr)\rtimes \langle \tau \rangle
\]
where \(\tau(S_0 \oplus S_+ \oplus S_-) \coloneqq (-S_0) \oplus S_- \oplus
S_+\), as in the analogous quaternionic case \(n = 3\).

Without invoking Theorem~\ref{th:JR_on_C_0_spaces}, there is no
\emph{general} argument that \(\oversigma_1\) and \(\oversigma_2\) are
polynomial functions, let alone that they define global Killing tensors
on \((N,g)\). Fortunately, in this case the explicit formulas
\begin{align}
\label{eq:over_mu_for_n_=_7_2}
  \overmu_1(S_0 \oplus S_+ \oplus S_-)
  &= \norm{S_0}^2\bigl(\norm{S_+}^2 - \norm{S_-}^2\bigr)^2 \\[0.3em]
\label{eq:over_mu_for_n_=_7_3}
  \overmu_2(S_0 \oplus S_+ \oplus S_-)
  &= \norm{S_0}^2\bigl(\norm{S_+}^4 + \norm{S_-}^4\bigr)
    - 2\,\big\langle S_-^*(S_+ S_0), (S_0 S_-^*) S_+ \big\rangle
\end{align}
already exhibit the two nonzero eigenvalue branches of \(-\overK^2\) as
\(\Aut(\n,g)\)-invariant polynomial functions. The rescaled expressions
\(\undermu_1\) and \(\undermu_2\)---obtained by dividing
\eqref{eq:over_mu_for_n_=_7_2} and \eqref{eq:over_mu_for_n_=_7_3},
respectively, by \(\norm{S_0}^2\bigl(\norm{S_+}^2 + \norm{S_-}^2\bigr)^2\)---yield the two
nonconstant eigenvalue branches of \(-\underK^2\) on each
connected component of \(\n \setminus \ram(\overK^{2})\).\footnote{On
the natural domain \(S_0\neq0\), \(S_+\oplus S_-\neq0\), the
unramified locus has two connected components, distinguished
by the sign of \(d \coloneqq \norm{S_+}^2 - \norm{S_-}^2\)
(similar as in the nonisotypic case for \(n=3\)).}

In particular, \(\oversigma_1 = \overmu_1 + \overmu_2\) and
\(\oversigma_2 = \overmu_1\overmu_2\) are \(\Aut(\n,g)\)-invariant
polynomials; hence
\[
  P_{\min}(N,g) = P_{\n_3}\,P_{\undermu_1}\,P_{\undermu_2}
  \;\in\; \scrP(\n)[\lambda]_{\Aut(\n,g)}
\]
in complete agreement with Theorem~\ref{th:JR_on_C_0_spaces}. Moreover, a
direct calculation in Section~\ref{se:verification_for_n_=_7} shows that
\(\big\langle S_-^*(S_+ S), (S S_-^*) S_+ \big\rangle\) is a Killing
tensor. Hence \(\overmu_1\) and \(\overmu_2\) are both Killing, and
therefore so are the coefficients of \(P_{\min}(N,g)\).
For \(n = 8\), following \cite[Def.~14.6]{Ha}, set
\(\z \coloneqq \bbO\). We equip
\(\v \coloneqq \bbO^2 = \bbO \oplus \bbO\) with the unique
orthogonal \(\Cl(\bbO)\)-module structure, up to isomorphism, given by
\begin{equation}\label{eq:Clifford_multiplication_for_n_=_8}
  S_0 \cbullet (S_+ \oplus S_-)
  \coloneqq S_- S_0 \oplus (-S_+ S_0^*)
\end{equation}
for \(S_0, S_+, S_- \in \bbO\), where multiplication in each component is
octonionic and \(S_0^*\) denotes octonionic conjugation; see
\cite[Part~II, Ch.~14, p.~275]{Ha}.

The automorphism group is given by
\[
  \Aut(\n,g) \cong \Pin(\z)
  \cong \Spin(\bbO) \rtimes \langle \tau \rangle
\]
where \(\tau \colon \n \to \n\) is the involution defined by
\[
  \tau(S_0 \oplus S_+ \oplus S_-)
  \coloneqq (-S_0^*) \oplus S_- \oplus S_+
\]
In this case, there is a single nonzero polynomial eigenvalue branch
\(\overmu\) of \(-\overK^2\), with multiplicity \(6\). Hence
\begin{equation}\label{eq:trace_formula_over_mu_for_n_=_8}
  \overmu(S_0 \oplus S_+ \oplus S_-)
  =
  -\frac{1}{6}
  \trace\bigl(
    \overK(S_0 \oplus S_+ \oplus S_-)^2
  \bigr)
\end{equation}
is an \(\Aut(\n,g)\)-invariant polynomial function on \(\bbO^3\), and it
yields the unique nonzero eigenvalue branch of \(-\overK^2\) on
\(\n \setminus \ram(\overK^2)\).

More explicitly, this polynomial is given by
\begin{equation}\label{eq:over_mu_for_n_=_8}
  \overmu(S_0 \oplus S_+ \oplus S_-)
  \coloneqq
  4\bigl(
    \norm{S_0}^2\norm{S_+}^2\norm{S_-}^2
    - \langle S_+, S_- S_0 \rangle^2
  \bigr)
\end{equation}
Indeed, \eqref{eq:over_mu_for_n_=_8} is \(\Spin(\bbO)\)-invariant because
the trilinear form
\[
  (S_0,S_+,S_-)
  \longmapsto
  \langle S_+, S_- S_0 \rangle
\]
is \(\Spin(\bbO)\)-invariant by the Triality Theorem
\cite[Theorem~14.19]{Ha}. Moreover, its square is invariant under the
involution \(\tau\), and therefore \eqref{eq:over_mu_for_n_=_8} is
\(\Aut(\n,g)\)-invariant.

Here neither
\(
  \norm{S_0}^2\norm{S_+}^2\norm{S_-}^2
\)
nor
\(
  \langle S_+, S_- S_0 \rangle^2
\)
is a Killing tensor separately; however, their difference
\[
  \norm{S_0}^2\norm{S_+}^2\norm{S_-}^2
  - \langle S_+, S_- S_0 \rangle^2
\]
is Killing. Hence the coefficients of \(P_{\min}(N,g)\) are Killing
tensors as well; see Section~\ref{se:verification_for_n_=_8}.

The corresponding rescaled expression \(\undermu\), obtained by dividing
\eqref{eq:over_mu_for_n_=_8} by
\[
  \norm{S_0}^2
  \bigl(\norm{S_+}^2 + \norm{S_-}^2\bigr)^2
\]
describes the unique nonconstant eigenvalue branch of \(-\underK^2\).
Hence, as predicted by Theorem~\ref{th:JR_on_C_0_spaces},
\[
  P_{\min}(N,g) = P_{\n_3}\,P_0^\sharp\,P_{\undermu}
  \in \scrP(\n)[\lambda]_{\Aut(\n,g)}
\]

For \(n = 9\), following \cite[Part~II, Ch.~14, p.~287]{Ha}, we let
\(\z \coloneqq \bbR \oplus \bbO\) and consider the unique (up to
isomorphism) irreducible orthogonal \(\Cl(\z)\)-module \(\v\), 
realized as the complexification \(\bbO^2 \otimes_{\bbR} \bbC\).
For \(S_0 = r\oplus Z\in\z=\bbR\oplus\bbO\), set
\begin{equation}\label{eq:M_matrix}
  M(S_0)\coloneqq
  \begin{pmatrix}
    r & Z^*\\
    Z & -r
  \end{pmatrix}
\end{equation}
where \(^{*}\colon \bbO \to \bbO\) denotes octonionic conjugation.
Viewing
\[
  S_1=S_+\oplus S_-\in\bbO^2\otimes_{\bbR}\bbC
\]
as a row vector, row--matrix multiplication (using the octonionic
product entrywise) yields
\[
  S_1\,M(S_0)
  = \bigl(r S_+ + S_- Z\bigr)\,\oplus\,\bigl(S_+ Z^* - r S_-\bigr)
\]
for all \(S_0 = r\oplus Z \in \bbR \oplus \bbO\) and
\(S_1 = S_+\oplus S_-\in\bbO^2 \otimes_{\bbR} \bbC\). This defines the
Clifford multiplication\footnote{There is another canonical construction
of an orthogonal \(\Cl(\z)\)-module structure on \(\bbO^2 \otimes_\bbR \bbC\)
whose Clifford multiplication is given in~\eqref{eq:tensor_product_multiplication};
cf.\ Corollary~\ref{co:periodicity}. In that model, the natural splitting
\(\bbO^2 \otimes_\bbR \bbC\) defines a \(\bbZ_2\)-grading as an orthogonal
\(\Cl(\z)\)-module. In the model just defined, this splitting is only
\(\Spin(\z)\)-invariant.}
\begin{equation}\label{eq:Clifford_multiplication_for_n_=_9}
  S_0 \cbullet S_1 \coloneqq \i\,S_1\,M(S_0)
\end{equation}
Moreover,
\[
  \Aut(\n,g) \cong
  \Bigl(
    \bigl(\Spin(9) \times \U(1,\bbC)\bigr)\big/
    \bigl\langle(-1,-\Id_\v)\bigr\rangle
  \Bigr)\rtimes\langle\tau\rangle
\]
where \(\tau\) is the involution defined by
\[
  \tau(r \oplus Z \oplus S_+ \oplus S_-)
  \coloneqq (-r) \oplus (-Z) \oplus \overline{S_+ \oplus S_-}
\]
Here
\[
  \overline{(\tilde S_+ \oplus \tilde S_-)\otimes z}
  \coloneqq
  (\tilde S_+ \oplus \tilde S_-)\otimes \overline z
\]
denotes the standard complex conjugation on the complexified space
\(\v = \bbO^2 \otimes_\bbR \bbC\), and the factor \(\U(1,\bbC)\) acts
naturally by complex multiplication on the \(\bbC\)-factor.

There are three eigenvalue branches, namely \(\overmu_1\) and
\(\overmu_2\), each of multiplicity \(2\), and \(\overmu_3\), of
multiplicity \(4\), of \(-\overK^2\) on each connected component of
\(\n \setminus \ram(\overK^2)\).\footnote{Here, the unramified locus
has two connected components. These are distinguished by the sign of
\[
  d \coloneqq \trace\bigl(G(S_1)JH(S_0,S_1)\bigr)
\]
as follows from~\eqref{eq:over_sigma_2_for_n_=_9}.}
The functions \(\overmu_1\) and \(\overmu_2\) are not polynomial
individually. However, their sum \(\overmu_1+\overmu_2\) and product
\(\overmu_1\overmu_2\) are polynomial. Define
\begin{equation}\label{eq:def_G}
  G(S_1)\coloneqq
  \begin{pmatrix}
    \norm{\Re(S_1)}^{2} &
    \langle \Re(S_1), \Im(S_1)\rangle\\
    \langle \Re(S_1), \Im(S_1)\rangle &
    \norm{\Im(S_1)}^{2}
  \end{pmatrix},
  \qquad
  J\coloneqq
  \begin{pmatrix}
    0 & 1\\
    -1 & 0
  \end{pmatrix}
\end{equation}
and
\begin{equation}\label{eq:def_H}
  H(S_0,S_1)\coloneqq
  \begin{pmatrix}
    -\bigl\langle \Re(S_1),\; \Im(S_1)\,M(S_0)\bigr\rangle &
     \bigl\langle \Re(S_1),\; \Re(S_1)\,M(S_0)\bigr\rangle\\
    -\bigl\langle \Im(S_1),\; \Im(S_1)\,M(S_0)\bigr\rangle &
     \bigl\langle \Im(S_1),\; \Re(S_1)\,M(S_0)\bigr\rangle
  \end{pmatrix}
\end{equation}
(real and imaginary parts taken in the tensor factor \(\bbC\)). Then
\begin{align}
  \overmu_1(S_0 \oplus S_1) + \overmu_2(S_0 \oplus S_1)
  &=
  4\Bigl(
    \norm{S_0}^{2}\,\det G(S_1)
    + \det H(S_0,S_1)
  \Bigr)
  \label{eq:over_sigma_1_for_n_=_9}
  \\
  \overmu_1(S_0 \oplus S_1)\,\overmu_2(S_0 \oplus S_1)
  &=
  4\norm{S_0}^{2}\,
  \bigl(
    \trace\bigl(G(S_1)\,J\,H(S_0,S_1)\bigr)
  \bigr)^{2}
  \label{eq:over_sigma_2_for_n_=_9}
\end{align}
are \(\Aut(\n,g)\)-invariant homogeneous polynomials of degrees
\(6\) and \(12\), respectively, in
\[
  S_0 \oplus S_1
  \in
  (\bbR \oplus \bbO)
  \oplus
  (\bbO^2 \otimes_{\bbR} \bbC)
\]
Conversely, from
\eqref{eq:over_sigma_1_for_n_=_9}--\eqref{eq:over_sigma_2_for_n_=_9},
we recover \(\overmu_1(S_0 \oplus S_1)\) and
\(\overmu_2(S_0 \oplus S_1)\) as the two roots of the resulting
quadratic polynomial, using Vieta's formulas and the quadratic formula.

Moreover, the third eigenvalue branch \(\overmu_3\) is already a
polynomial function. For a row vector \(S = S_+ \oplus S_- \in \bbO^2\), write
\[
  S^* \coloneqq \begin{pmatrix} S_+^*, & S_-^* \end{pmatrix},
  \qquad
  S^\intercal \coloneqq \begin{pmatrix} S_+ \\ S_- \end{pmatrix}
\]
(entrywise octonionic conjugation and transpose). Let
\(J_0\coloneqq\begin{psmallmatrix}0&1\\1&0\end{psmallmatrix}\), set
\[
  \tilde{S}_\pm \coloneqq S_1 \pm S_1\,M(S_0)
  \in \bbO^2 \otimes_{\bbR} \bbC
\]
and regard \(\Re(\tilde S_\pm)\) and \(\Im(\tilde S_\pm)\) as row
vectors in \(\bbO^2\). (In particular, the above notation applies to
\(S=\Re(\tilde S_\pm)\) and \(S=\Im(\tilde S_\pm)\).) Then
\begin{equation}\label{eq:over_mu_3_for_n_=_9_1}
  \begin{split}
    \overmu_3(S_0 \oplus S_1)
    &= \tfrac{1}{2}\bigl(\overmu_1(S_0 \oplus S_1)
      + \overmu_2(S_0 \oplus S_1)\bigr)
    \\
    &\quad - \tfrac{1}{4}
      \bigl\langle
        \Re(\tilde{S}_+)^* J_0\, \Re(\tilde{S}_-)^\intercal,\;
        \Im(\tilde{S}_+)^* J_0\, \Im(\tilde{S}_-)^\intercal
      \bigr\rangle
    \\
    &\quad + \tfrac{1}{4}
      \bigl\langle
        \Re(\tilde{S}_+)^* J_0\, \Im(\tilde{S}_-)^\intercal,\;
        \Im(\tilde{S}_+)^* J_0\, \Re(\tilde{S}_-)^\intercal
      \bigr\rangle
  \end{split}
\end{equation}

In particular,
\(\oversigma_1 = \overmu_1 + \overmu_2 + \overmu_3\),
\(\oversigma_2 = \overmu_1\overmu_2 + (\overmu_1 + \overmu_2)\overmu_3\),
and \(\oversigma_3 = \overmu_1\overmu_2\overmu_3\) are
\(\Aut(\n,g)\)-invariant homogeneous polynomials of degrees \(6\), \(12\),
and \(18\). On the natural domain \(S_0\neq0\), \(S_1\neq0\), the corresponding
three rescaled functions yield three nonconstant eigenvalue branches of
\(-\underK^2\) on each of the two connected component of its unramified locus.
We conclude that
\[
  P_{\min}(N,g)
  = P_{\n_3}\,P_{\undermu_1}\,P_{\undermu_2}\,P_{\undermu_3}
  \in \scrP(\n)[\lambda]_{\Aut(\n,g)}
\]
In Section~\ref{se:verification_for_n_=_9} we show by explicit
calculations that each \(\overmu_i\), \(i=1,2,3\), is constant along
geodesics.

\subsection{Concluding remarks on Theorem \ref{th:main_2}}
\label{se:remarks}
The degree \(k\) of \(P_{\min}(N,g)\) is always an upper bound for
the Singer invariant, \(k_{\mathrm{Singer}}\le k\) according to
\cite[Theorem~1.15]{J}. Therefore, we obtain in particular the
upper bounds \(k_{\mathrm{Singer}}\le 3n+1\) (for \(n\) even) and
\(k_{\mathrm{Singer}}\le 3n\) (for \(n\) odd), valid for every
H-type group modeled on an orthogonal Clifford module \(\v\)
over the Clifford algebra \(\Cl(\z)\) of a fixed Euclidean
vector space \(\z\) with \(\dim \z = n\). Hence, these upper
bounds are independent of \(\dim \v\); e.g., they do not take
into account the number of irreducible copies in which \(\v\)
may decompose as a \(\Cl(\z)\)-module.

In the exceptional cases discussed above, we obtain
even sharper bounds for the Singer invariant. For instance,
suppose that \(\dim \z = 3\) and that \(\v\) is an
isotypic \(\Cl(\z)\)-module, say
\(\v \cong \bbH^{r_1}\). Then
\[
  k_{\mathrm{Singer}} \le 5
\]
independently of the multiplicity \(r_1\).

Likewise, for irreducible Riemannian spaces \((N,g)\), the degree
of the minimal polynomial \(P_{\min}(N,g)\) can be arbitrarily
large.

Finally, following the notation from \cite{JW}, we call
\(P_{\min}(N,g)\) \emph{linear} if all \(a_i\) are multiples of
products of the metric tensor: there exist \(c_i\in\bbR_{>0}\) such that
\[
a_{2i}(X)=c_i\norm{X}^{2i}\quad\text{and}\quad a_{2i-1}=0
\quad \text{for } i=1,\ldots,\frac{k-1}{2}
\]
Thus, inspection of the coefficients \(a_i\) of the minimal
polynomials \(P_{\min}(N,g)\) considered above reveals that H-type
groups do not provide additional examples of linear minimal
polynomials beyond the one associated with Heisenberg groups
for \(n=1\), which was already known \cite{JW}. This supports
our presumption that such a condition occurs only in very few
cases.
\section{Geometric properties of H-type groups}
\label{se:geometric_properties}

Following Chapter~3 of~\cite{BTV}, we first describe the Levi-Civita
connection and the curvature tensor of an H-type group \((N,g)\). More
generally, let \(G\) be an arbitrary Lie group equipped with a
left-invariant metric \(g\). Then the Levi-Civita connection \(\nabla\)
is given by (cf.~\cite[p.~16]{Esch})
\begin{equation}\label{LC_left_invariant}
  \nabla_X Y = \tfrac{1}{2}\bigl([X, Y] - B(X, Y)\bigr)
\end{equation}
for all left-invariant vector fields \(X\) and \(Y\), where
\begin{equation}\label{eq:def_B}
  B(X, Y) \coloneqq \ad(X)^* Y + \ad(Y)^* X
\end{equation}
Here \(\ad\colon \g\to\End{\g}\) denotes the adjoint representation of
the Lie algebra \(\g\), and \(A^*\) denotes the \(g\)-adjoint of
\(A\in\End{\g}\). In particular, \(\nabla_X Y\) is left-invariant. Note
that \(B(X,Y)=B(Y,X)\).

For an H-type group \((N,g)\) with orthogonal Lie algebra \((\n,g|_e)\),
the bracket satisfies \([X,Y] = X_\v \spinprod Y_\v\), where \(\spinprod\)
denotes the alternating spinor product introduced in
\eqref{eq:def_spinor_product}. Thus
\begin{equation}\label{eq:ad(X)*}
  \ad(X)^* Y = Y_\z \cbullet X_\v
\end{equation}
where \(\cbullet\) is the Clifford multiplication. Hence we obtain
\(\nabla_X Y = A(X)Y\) with
\begin{equation}\label{eq:LC}
  A(X)Y \coloneqq \tfrac{1}{2}\bigl(
    X_\v \spinprod Y_\v - X_\z \cbullet Y_\v - Y_\z \cbullet X_\v
  \bigr)
\end{equation}
for all \(X, Y \in \n\) (cf.~\cite[Ch.~3.1.6]{BTV}). Equivalently,
\[
  (\nabla_X Y)_\z = \tfrac{1}{2}(X_\v \spinprod Y_\v),\qquad
  (\nabla_X Y)_\v = -\tfrac{1}{2}\bigl(
    X_\z \cbullet Y_\v + Y_\z \cbullet X_\v
  \bigr)
\]

For the Ricci tensor one computes (cf.~\cite[Ch.~3.1.7]{BTV})
\begin{equation}\label{eq:Ricci}
  \Ric(X, Y) = \tfrac{\dim \v}{4}\, g(X_\z, Y_\z)
  - \tfrac{\dim \z}{2}\, g(X_\v, Y_\v)
\end{equation}
In particular, the Ricci tensor does not vanish.

\subsection{\texorpdfstring{Geodesics of \(N\)}{Geodesics of \(N\)}}
\label{se:geodesics}

Following~\cite[Ch.~3.1.9]{BTV}, we consider the natural left-invariant,
orthogonal framing
\begin{equation}\label{eq:framing}
  N \times \n \longrightarrow TN,\quad (g,X) \longmapsto \rmL_{g*}X
\end{equation}
In this way, the geodesics of \(N\) admit the following description.

\begin{lemma}\label{le:geodesics}
Let \(\gamma \colon \bbR \to N\) be the geodesic with \(\gamma(0) = e_N\)
and \(\dot{\gamma}(0) = X\). In the natural left-invariant
framing~\eqref{eq:framing}, we have \(\dot{\gamma}(t) \cong
(\gamma(t),X(t))\), where
\begin{equation}\label{eq:geodesic}
  X(t)
  = X_\z \oplus \cos\bigl(t \norm{X_\z}\bigr) X_\v
  + \frac{\sin\bigl(t \norm{X_\z}\bigr)}{\norm{X_\z}}\,
    X_\z \cbullet X_\v
\end{equation}
\end{lemma}

\begin{proof}
We have to show that \(\tfrac{\nabla}{\d t}\dot{\gamma}(t) = 0\).
On the one hand,
\[
  \frac{\d}{\dt} X(t)
  = -\norm{X_\z}\sin\bigl(t\norm{X_\z}\bigr)X_\v
    + \cos\bigl(t\norm{X_\z}\bigr)\,X_\z \cbullet X_\v
\]
On the other hand, the algebraic description~\eqref{eq:LC} of the
Levi-Civita connection on left-invariant vector fields yields
\[
  A(X(t))X(t)
  = \tfrac12 \underbrace{X_\v(t) \spinprod X_\v(t)}_{=\,0}
    - X_\z(t) \cbullet X_\v(t)
\]
where
\[
  \begin{aligned}
    X_\z(t) \cbullet X_\v(t)
    &= X_\z \cbullet \cos\bigl(t\norm{X_\z}\bigr)X_\v
      + \frac{\sin\bigl(t\norm{X_\z}\bigr)}{\norm{X_\z}}\,
        X_\z \cbullet X_\z \cbullet X_\v \\
    &\stackrel{\eqref{eq:Clifford_module_2}}{=}
      \cos\bigl(t\norm{X_\z}\bigr)\,X_\z \cbullet X_\v
      - \norm{X_\z}\sin\bigl(t\norm{X_\z}\bigr)\,X_\v \\
    &= \frac{\d}{\dt} X(t)
  \end{aligned}
\]
It follows that
\[
  \tfrac{\nabla}{\d t}\dot{\gamma}(t)
  = \frac{\d}{\dt}X(t) + A(X(t))X(t) = 0
\]
\end{proof}

For an explicit description of these geodesics, see~\cite[Ch.~3]{Ka1}.

Under left translation, the vector spaces \(\z\) and \(\v\) extend to
left-invariant, pointwise orthogonal distributions on \(N\), also
denoted by \(\z\) and \(\v\). Every vector field \(X\) has a unique
pointwise orthogonal decomposition \(X = X_\z \oplus X_\v\), where
\(X_\z\) is a section of \(\z\) and \(X_\v\) is a section of \(\v\).
Similarly, the scalar products \(g_\z\) and \(g_\v\) on \(\z\) and \(\v\)
induce left-invariant symmetric tensor fields \(g_\z\) and \(g_\v\) on
\(N\), with \(g_\z(X,X) = g(X_\z,X_\z)\) and \(g_\v(X,X) = g(X_\v,X_\v)\).

\begin{corollary}\label{co:Killing_1}
Both \(g_\z\) and \(g_\v\) are Killing tensors.
\end{corollary}

\begin{proof}
Along a geodesic as in Lemma~\ref{le:geodesics}, we have
\[
  \begin{aligned}
    \norm{X_\v(t)}^2
    &= \cos^2\bigl(t\norm{X_\z}\bigr)\norm{X_\v}^2
      + \frac{\sin^2\bigl(t\norm{X_\z}\bigr)}{\norm{X_\z}^2}\,
        \underbrace{\norm{X_\z \cbullet X_\v}^2}
          _{=\norm{X_\z}^2\norm{X_\v}^2} \\
    &= \bigl(\cos^2\bigl(t\norm{X_\z}\bigr)
      + \sin^2\bigl(t\norm{X_\z}\bigr)\bigr)\norm{X_\v}^2 \\
    &= \norm{X_\v}^2
  \end{aligned}
\]
which is constant. The same is true for \(\norm{X_\z(t)}^2\), since
\(X_\z(t) \equiv X_\z\).
\end{proof}

\subsection{The symmetrized curvature tensor}
\label{se:K(X)}

Recall that the Lie algebra \(\g\) of a Lie group \(G\) is, by
definition, the vector space of left-invariant vector fields on \(G\),
with Lie bracket defined via the Lie derivative. The following formula
for the symmetrized curvature tensor is well known
(cf.~\cite[p.~17]{Esch}, \cite[p.~54]{CE}), but note that both references
contain a small error.

\begin{lemma}\label{le:sym_curv_general}
Let \(G\) be a Lie group with a left-invariant metric and Lie algebra
\(\g\). Then, for all \(X,Y\in\g\),
\begin{equation}\label{eq:sym_curv}
  \begin{aligned}
    \langle \scrR(X)Y, Y \rangle
    &= -\tfrac{3}{4}\norm{[X,Y]}^2
       - \langle \ad(X)^*X, \ad(Y)^*Y \rangle \\
    &\quad + \tfrac{1}{4}\norm{\ad(X)^*Y + \ad(Y)^*X}^2 \\
    &\quad + \tfrac{1}{2}\bigl(
      \langle [[X,Y],X], Y \rangle + \langle [[Y,X],Y], X \rangle
    \bigr)
  \end{aligned}
\end{equation}
\end{lemma}

\begin{proof}
Let \(X,Y,Z,W\in\g\). Since the metric and the Levi--Civita connection
\(\nabla\) are left-invariant (cf.~\eqref{LC_left_invariant}), we have
\[
  0 = X \cdot \langle \nabla_Y Z, W \rangle
    = \langle \nabla_X \nabla_Y Z, W \rangle
      + \langle \nabla_Y Z, \nabla_X W \rangle
\]
Thus
\[
  \langle \nabla_X \nabla_Y Z, W \rangle
  = - \langle \nabla_Y Z, \nabla_X W \rangle
\]
In particular,
\begin{equation}\label{eq:RXY_identity}
  \langle \rmR(Y,X)X, Y \rangle
  = - \langle \nabla_X X, \nabla_Y Y \rangle
    + \langle \nabla_X Y, \nabla_Y X \rangle
    + \langle \nabla_{[X,Y]} X, Y \rangle
\end{equation}

Furthermore, using~\eqref{LC_left_invariant},
\[
  \begin{aligned}
    \langle \nabla_X X, \nabla_Y Y \rangle
    &= \tfrac{1}{4}\langle B(Y,Y), B(X,X) \rangle \\
    &\stackrel{\eqref{eq:def_B}}{=}
      \tfrac{1}{4}\langle 2\ad(Y)^*Y,\,2\ad(X)^*X \rangle \\
    &= \langle \ad(Y)^*Y, \ad(X)^*X \rangle
  \end{aligned}
\]
\[
  \begin{aligned}
    \langle \nabla_X Y, \nabla_Y X \rangle
    &= \Bigl\langle \tfrac12\bigl([X,Y] - B(X,Y)\bigr),\,
      \tfrac12\bigl([Y,X] - B(Y,X)\bigr) \Bigr\rangle \\
    &\stackrel{\eqref{eq:def_B}}{=}
      -\tfrac{1}{4}\norm{[X,Y]}^2
      + \tfrac{1}{4}\norm{\ad(Y)^*X + \ad(X)^*Y}^2
  \end{aligned}
\]
\[
  \begin{aligned}
    \langle \nabla_{[X,Y]} X, Y \rangle
    &= \tfrac12\Bigl(
      \langle [[X,Y],X], Y \rangle
      - \langle B([X,Y],X), Y \rangle
    \Bigr) \\
    &\stackrel{\eqref{eq:def_B}}{=}
      \tfrac12\Bigl(
        \langle [[X,Y],X], Y \rangle
        - \langle X, [[X,Y],Y] \rangle
        - \norm{[X,Y]}^2
      \Bigr) \\
    &= \tfrac12\Bigl(
         \langle [[X,Y],X], Y \rangle
        + \langle [[Y,X],Y], X \rangle
        - \norm{[X,Y]}^2
      \Bigr)
  \end{aligned}
\]
Substituting these identities into~\eqref{eq:RXY_identity}
yields~\eqref{eq:sym_curv}.
\end{proof}

\begin{corollary}\label{co:sym_curv_H_type}
The symmetrized curvature tensor of an H-type group \((N,g)\) is given,
for all \(X,Y\in\n\), by
\begin{equation}\label{eq:sym_curv_2}
  \begin{aligned}
    \scrR(X)Y
    &= -\tfrac{3}{4}(X_\v \spinprod Y_\v) \cbullet X_\v
       + \tfrac{1}{4}\norm{X_\z}^2 Y_\v
       + \tfrac{3}{4}Y_\z \cbullet X_\z \cbullet X_\v \\
    &\quad + \tfrac{1}{2}\langle X_\z, Y_\z\rangle X_\v
       + \tfrac{1}{4}\norm{X_\v}^2 Y_\z
       - \tfrac{3}{4}Y_\v \spinprod (X_\z \cbullet X_\v)
       + \tfrac{1}{2}\langle X_\v, Y_\v\rangle X_\z
  \end{aligned}
\end{equation}
\end{corollary}
Here \(\cbullet\) and \(\spinprod\) denote the Clifford multiplication
and the alternating spinor product, respectively; see
Section~\ref{se:JR_GH}.

\begin{proof}
First polarize~\eqref{eq:sym_curv} with respect to \(Y\) and use that
\(\n\) is two-step nilpotent, hence \([[\n,\n],\n]=0\). Therefore,
\begin{equation}\label{eq:sym_curv_polarized}
  \begin{aligned}
    \langle \scrR(X)Y, Z \rangle
    &= -\tfrac{3}{4}\langle [X,Y], [X,Z] \rangle \\
    &\quad - \tfrac{1}{2}\langle \ad(X)^*X,\,
      \ad(Y)^*Z + \ad(Z)^*Y \rangle \\
    &\quad + \tfrac{1}{4}\langle \ad(X)^*Y + \ad(Y)^*X,\,
      \ad(X)^*Z + \ad(Z)^*X \rangle
  \end{aligned}
\end{equation}
Next, compute the terms in~\eqref{eq:sym_curv_polarized}:
\[
  \begin{aligned}
    \langle [X,Y], [X,Z] \rangle
    &= \langle X_\v \spinprod Y_\v,\, X_\v \spinprod Z_\v \rangle \\
    &\stackrel{\eqref{eq:def_spinor_product}}{=}
      \langle (X_\v \spinprod Y_\v) \cbullet X_\v,\, Z_\v \rangle \\
    &= \langle (X_\v \spinprod Y_\v) \cbullet X_\v,\, Z \rangle
  \end{aligned}
\]
Moreover, using~\eqref{eq:ad(X)*},
\[
  \begin{aligned}
    \langle \ad(X)^*X,\, \ad(Y)^*Z + \ad(Z)^*Y \rangle
    &= \langle X_\z \cbullet X_\v,\, Z_\z \cbullet Y_\v
      + Y_\z \cbullet Z_\v \rangle \\
    &= \langle Y_\v \spinprod (X_\z \cbullet X_\v),\, Z_\z \rangle
      - \langle Y_\z \cbullet X_\z \cbullet X_\v,\, Z_\v \rangle \\
    &= \langle Y_\v \spinprod (X_\z \cbullet X_\v)
      - Y_\z \cbullet X_\z \cbullet X_\v,\, Z \rangle
  \end{aligned}
\]
Similarly,
\[
  \begin{aligned}
    \langle \ad(X)^*Y,\, \ad(X)^*Z + \ad(Z)^*X \rangle
    &= \langle Y_\z \cbullet X_\v,\, Z_\z \cbullet X_\v
      + X_\z \cbullet Z_\v \rangle \\
    &= \langle X_\v \spinprod (Y_\z \cbullet X_\v),\, Z_\z \rangle
      - \langle X_\z \cbullet Y_\z \cbullet X_\v,\, Z_\v \rangle \\
    &= \langle X_\v \spinprod (Y_\z \cbullet X_\v)
      - X_\z \cbullet Y_\z \cbullet X_\v,\, Z \rangle
  \end{aligned}
\]
and
\[
  \begin{aligned}
    \langle \ad(Y)^*X,\, \ad(X)^*Z + \ad(Z)^*X \rangle
    &= \langle X_\z \cbullet Y_\v,\, Z_\z \cbullet X_\v
      + X_\z \cbullet Z_\v \rangle \\
    &= \langle X_\v \spinprod (X_\z \cbullet Y_\v),\, Z_\z \rangle
      - \langle X_\z \cbullet X_\z \cbullet Y_\v,\, Z_\v \rangle \\
    &\stackrel{\eqref{eq:Clifford_module_2}}{=}
      \langle X_\v \spinprod (X_\z \cbullet Y_\v)
      + \norm{X_\z}^2 Y_\v,\, Z \rangle
  \end{aligned}
\]
Substituting into~\eqref{eq:sym_curv_polarized} gives, for all \(Z\in\n\),
\[
  \begin{aligned}
    \langle \scrR(X)Y, Z \rangle
    &= -\tfrac{3}{4}\langle (X_\v \spinprod Y_\v) \cbullet X_\v,\, Z \rangle \\
    &\quad - \tfrac{1}{2}\langle Y_\v \spinprod (X_\z \cbullet X_\v)
      - Y_\z \cbullet X_\z \cbullet X_\v,\, Z \rangle \\
    &\quad + \tfrac{1}{4}\langle X_\v \spinprod (Y_\z \cbullet X_\v)
      - X_\z \cbullet Y_\z \cbullet X_\v,\, Z \rangle \\
    &\quad + \tfrac{1}{4}\langle X_\v \spinprod (X_\z \cbullet Y_\v)
      + \norm{X_\z}^2 Y_\v,\, Z \rangle
  \end{aligned}
\]
which yields
\begin{equation}\label{eq:sym_curv_1}
  \begin{aligned}
    \scrR(X)Y
    &= -\tfrac{3}{4}(X_\v \spinprod Y_\v) \cbullet X_\v \\
    &\quad - \tfrac{1}{2}\bigl(
      Y_\v \spinprod (X_\z \cbullet X_\v)
      - Y_\z \cbullet X_\z \cbullet X_\v
    \bigr) \\
    &\quad + \tfrac{1}{4}\bigl(
      X_\v \spinprod (Y_\z \cbullet X_\v)
      - X_\z \cbullet Y_\z \cbullet X_\v
    \bigr) \\
    &\quad + \tfrac{1}{4}\bigl(
      X_\v \spinprod (X_\z \cbullet Y_\v)
      + \norm{X_\z}^2 Y_\v
    \bigr)
  \end{aligned}
\end{equation}
by nondegeneracy of the metric.

To obtain~\eqref{eq:sym_curv_2}, note that
\(X_\v \spinprod (X_\z \cbullet Y_\v)\in\z\) by~\eqref{eq:def_spinor_product}.
For any \(Z\in\z\),
\[
  \begin{aligned}
    \langle X_\v \spinprod (X_\z \cbullet Y_\v), Z \rangle
    &= -\langle Y_\v, X_\z \cbullet Z \cbullet X_\v \rangle \\
    &\stackrel{\eqref{eq:Clifford_module_3}}{=}
      \langle Y_\v, Z \cbullet X_\z \cbullet X_\v \rangle
      + 2\,\langle X_\z, Z \rangle\,\langle X_\v, Y_\v \rangle \\
    &= -\langle Z \cbullet Y_\v, X_\z \cbullet X_\v \rangle
      + 2\,\langle X_\z, Z \rangle\,\langle X_\v, Y_\v \rangle \\
    &= -\langle Z, Y_\v \spinprod (X_\z \cbullet X_\v) \rangle
      + 2\,\langle X_\z, Z \rangle\,\langle X_\v, Y_\v \rangle
  \end{aligned}
\]
Hence
\[
  X_\v \spinprod (X_\z \cbullet Y_\v)
  = 2\,\langle X_\v, Y_\v \rangle\,X_\z
    - Y_\v \spinprod (X_\z \cbullet X_\v)
\]
Similarly,
\[
  \begin{aligned}
    \langle X_\v \spinprod (Y_\z \cbullet X_\v), Z \rangle
    &= \langle Y_\z \cbullet X_\v, Z \cbullet X_\v \rangle \\
    &= \langle Y_\z, Z \rangle\,\norm{X_\v}^2
  \end{aligned}
\]
so
\[
  X_\v \spinprod (Y_\z \cbullet X_\v) = \norm{X_\v}^2\,Y_\z
\]
Moreover,
\[
  X_\z \cbullet Y_\z \cbullet X_\v
  = -\,Y_\z \cbullet X_\z \cbullet X_\v
    - 2\,\langle X_\z, Y_\z \rangle\,X_\v
\]
by~\eqref{eq:Clifford_module_3}. Inserting these identities into
\eqref{eq:sym_curv_1} gives~\eqref{eq:sym_curv_2}.
\end{proof}

The formula for \(\scrR(X)Y\) in~\eqref{eq:sym_curv_2} coincides with
\cite[p.~29]{BTV} and will be used below. Also, \(\scrR(X)Y\) is
quadratic in \(X\). Its polarization \(\scrR(X,X')Y\) is given by
\[
  \tfrac12\bigl(\scrR(X+X')Y - \scrR(X)Y - \scrR(X')Y\bigr)
\]
Polarizing~\eqref{eq:sym_curv_2} in \(X\) yields, for all \(X,X',Y\in\n\),
\begin{equation}\label{eq:sym_curv_3}
  \begin{aligned}
    2\,\scrR(X,X')Y
    &= -\tfrac{3}{4}(X'_\v \spinprod Y_\v) \cbullet X_\v
       - \tfrac{3}{4}(X_\v \spinprod Y_\v) \cbullet X'_\v \\
    &\quad + \tfrac{1}{2}\langle X_\z, X'_\z\rangle\,Y_\v
       + \tfrac{3}{4}Y_\z \cbullet X_\z \cbullet X'_\v
       + \tfrac{3}{4}Y_\z \cbullet X'_\z \cbullet X_\v \\
    &\quad + \tfrac{1}{2}\langle X_\z, Y_\z\rangle\,X'_\v
       + \tfrac{1}{2}\langle X'_\z, Y_\z\rangle\,X_\v
       + \tfrac{1}{2}\langle X_\v, X'_\v\rangle\,Y_\z \\
    &\quad - \tfrac{3}{4}\bigl(Y_\v \spinprod (X_\z \cbullet X'_\v)\bigr)
       - \tfrac{3}{4}\bigl(Y_\v \spinprod (X'_\z \cbullet X_\v)\bigr) \\
    &\quad + \tfrac{1}{2}\langle X_\v, Y_\v\rangle\,X'_\z
       + \tfrac{1}{2}\langle X'_\v, Y_\v\rangle\,X_\z
  \end{aligned}
\end{equation}

\section{The characteristic family of linear operators}\label{se:K}

Let \(\n\) be the Lie algebra associated with a nontrivial orthogonal
Clifford module \(\v\) over the positive-dimensional Euclidean vector 
space \(\z\), with alternating spinor product 
\(\spinprod\colon \v \times \v \to \z\).
Since \(\spinprod\) is bilinear, it induces a linear map
\[
  \v \longrightarrow \Hom(\v,\z),\quad
  S \longmapsto S\spinprod
\]
that sends \(S \in \v\) to the linear map
\[
  S\spinprod \colon \v \longrightarrow \z,\quad
  \tilde S \longmapsto S\spinprod\tilde S
\]
namely, the map given by the Lie bracket of \(\n\) restricted to \(\v\):
\begin{equation}\label{eq:spinor_product_2}
  \spinprod = \ad|_\v \colon \v \longrightarrow \Hom(\v,\z),\quad
  S \longmapsto [S,\,\cdot\,]
\end{equation}

\begin{lemma}\label{le:K}
  Let \(S \in \v\) be a fixed vector.
  \begin{enumerate}
    \item We have
      \begin{equation}\label{eq:Kern_spin_product}
        \Kern(S\spinprod)^\perp = \z \cbullet S
      \end{equation}
      where \(\z \cbullet S \coloneqq \{Z \cbullet S \mid Z \in \z\}\).
      Hence, there is an orthogonal decomposition
      \[
        \v = \Kern(S\spinprod) \oplus (\z \cbullet S)
      \]
    \item The map
      \begin{equation}\label{eq:isometry}
        \z \longrightarrow \Kern(S\spinprod)^\perp,\quad
        Z \longmapsto Z \cbullet S
      \end{equation}
      is an isometry once the metric on \(\z\) is rescaled by the
      factor \(\norm{S}^2\).
  \end{enumerate}
\end{lemma}

\begin{proof}
  For (a): From
  \begin{align*}
    \langle Z,\, S_1 \spinprod S_2 \rangle
    \stackrel{\eqref{eq:def_spinor_product}}{=}
    \langle Z \cbullet S_1, S_2 \rangle
    = -\,\langle Z \cbullet S_2, S_1 \rangle
  \end{align*}
  for all \((S_1,S_2) \in \v \times \v\) and \(Z \in \z\), it follows
  that
  \[
    \Kern(S\spinprod) = \{Z \cbullet S \mid Z \in \z\}^\perp
  \]
  and (a) follows. In particular, the map~\eqref{eq:isometry} is
  well-defined, and
  \[
    \langle Z \cbullet S,\, Z \cbullet S \rangle
    = \norm{Z}^2 \norm{S}^2
  \]
  for all \(Z \in \z\), so after rescaling the metric on \(\z\) by the
  factor \(\norm{S}^2\) it is an isometry. This proves (b);
  cf.~also \cite[p.~24]{BTV} and the discussion in \cite[p.~3]{CDKR}.
\end{proof}

By Lemma~\ref{le:K} we can make the following definition:
\begin{definition}[{\cite[Ch.~3.1.12]{BTV}}]\label{de:K}
For every \(X = X_\z \oplus X_\v \in \z \oplus \v = \n\) with
\(X_\v \neq 0\) and every \(Z \in \z\), there exist unique vectors
\(U(X,Z) \in \Kern(X_\v\spinprod)\) and \(\overK(X)Z \in \z\) such that
\begin{equation}\label{eq:de_K_5}
  \norm{X_\v}^2\, Z \cbullet X_\z \cbullet X_\v
  =
  U(X,Z) + \overK(X)Z \cbullet X_\v
\end{equation}
\end{definition}

\begin{lemma}
We have
\begin{equation}\label{eq:Clifford_module_4}
  \langle Z_1 \cbullet S, Z_2 \cbullet S \rangle
  = \norm{S}^2 \langle Z_1, Z_2 \rangle
\end{equation}
for all \(Z_1,Z_2 \in \z\) and \(S \in \v\). Hence
\begin{equation}\label{eq:Clifford_module_5}
  S\spinprod(Z \cbullet S) = \norm{S}^2 Z
\end{equation}
for all \(Z \in \z\) and \(S \in \v\).
\end{lemma}

\begin{proof}
Using the skew-symmetry~\eqref{eq:Clifford_module_1} of
\(Z \cbullet \colon \v \to \v\) and the Clifford relation
\eqref{eq:Clifford_module_3}, we compute
\begin{align*}
  \langle Z_1 \cbullet S_1, Z_2 \cbullet S_2 \rangle
    &\stackrel{\eqref{eq:Clifford_module_1}}{=}
      -\,\langle S_1, Z_1 \cbullet Z_2 \cbullet S_2 \rangle \\
    &\stackrel{\eqref{eq:Clifford_module_3}}{=}
      \langle S_1, Z_2 \cbullet Z_1 \cbullet S_2 \rangle
      + 2 \langle Z_1, Z_2 \rangle \langle S_1, S_2 \rangle \\
    &\stackrel{\eqref{eq:Clifford_module_1}}{=}
      -\,\langle Z_2 \cbullet S_1, Z_1 \cbullet S_2 \rangle
      + 2 \langle Z_1, Z_2 \rangle \langle S_1, S_2 \rangle
\end{align*}
for all \(Z_1,Z_2 \in \z\) and \(S_1,S_2 \in \v\). We conclude that
\[
  \langle Z_1 \cbullet S_1, Z_2 \cbullet S_2 \rangle
  + \langle Z_1 \cbullet S_2, Z_2 \cbullet S_1 \rangle
  = 2 \langle Z_1, Z_2 \rangle \langle S_1, S_2 \rangle
\]
Setting \(S_2 \coloneqq S_1\) and dividing by \(2\), we obtain
\eqref{eq:Clifford_module_4}. Equation~\eqref{eq:Clifford_module_5}
then follows from \eqref{eq:Clifford_module_4} by applying
\eqref{eq:def_spinor_product}.
\end{proof}

\begin{proposition}[{\cite[Ch.~3.1.12]{BTV}}]\label{p:overK}
  The linear endomorphism \(\overK(X) \colon \z \to \z\), defined
  by~\eqref{eq:de_K_5}, is the operator introduced in
  Section~\ref{se:JR_GH}. It satisfies the following properties:
  \begin{enumerate}
  \item The associated bilinear form is
      \(\big\langle Z_1 \cbullet X_\z \cbullet X_\v,\,
      Z_2 \cbullet X_\v \big\rangle\); cf.~\eqref{eq:def_overK}.
  \item \(\overK(X)\) admits the explicit description
      \begin{equation}\label{eq:def_K_2}
        \overK(X) Z = X_\v \spinprod \bigl(Z \cbullet X_\z \cbullet X_\v\bigr)
      \end{equation}
      Equivalently,
      \begin{equation}\label{eq:de_K_4}
        \overK(X) Z = (Z \cbullet X_\v) \spinprod (X_\z \cbullet X_\v)
      \end{equation}
      From this it follows that~\eqref{eq:def_overK} is
      skew-symmetric, i.e.,~\(\overK(X)\) is a skew-symmetric
      endomorphism of \(\z\).
  \item We have \(\overK(X) X_\z = 0\) and
      \(\langle \overK(X) Z, X_\z \rangle = 0\) for all \(Z \in \z\).
      Thus, \(\overK(X)\) restricts to a skew-symmetric endomorphism
      of \(\h_X \coloneqq X_\z^\perp\).
    \item Assume that \(X_\z \neq 0\) and \(X_\v \neq 0\).
      The rescaled operator \(\underK(X)\coloneqq
      \tfrac{1}{\norm{X_\z}\norm{X_\v}^2}\overK(X)\) is characterized by
      \eqref{eq:def_underK}. The eigenvalues \(\undermu\) of
     \(-\underK(X)^2\) satisfy \(0 \le \undermu \le 1\).

  \item The case \(\undermu = 1\) is characterized as follows: assume
      again that \(X_\z \neq 0\) and \(X_\v \neq 0\). Then \(1\) is an
      eigenvalue of \(-\underK(X)^2\), and \(Z\) is a nonzero
      eigenvector if and only if there exists \(\tilde Z \in \h_X\) such
      that
      \begin{equation}\label{eq:mu_=_1}
        Z \cbullet X_\z \cbullet X_\v
        = \tilde Z \cbullet X_\v
      \end{equation}
      In the affirmative case, \(\tilde Z = \norm{X_\z} \underK(X) Z\) holds and
      \(\bigl\{E_1 \coloneqq \tfrac{1}{\norm{Z}}Z,\,
      E_2 \coloneqq \tfrac{1}{\norm{Z}\norm{X_\z}}\tilde Z\bigr\}\) is an
      orthonormal system of \((\Eig_1)_X\).
  \end{enumerate}
\end{proposition}

\begin{proof}
For items (a)–(c), let \(Z_1, Z_2 \in \z\). We compute:
\begin{align*}
  \langle \overK(X) Z_1, Z_2 \rangle
    &\stackrel{\eqref{eq:Clifford_module_4}}{=}
      \frac{1}{\norm{X_\v}^2}
      \big\langle \overK(X) Z_1 \cbullet X_\v,\,
      Z_2 \cbullet X_\v \big\rangle \\
    &\stackrel{\eqref{eq:def_spinor_product}}{=}
      \frac{1}{\norm{X_\v}^2}\big\langle X_\v\spinprod\bigl(\overK(X) Z_1\cbullet X_\v\bigr),\,
      Z_2 \big\rangle \\
    &\stackrel{\eqref{eq:de_K_5}}{=}
      \big\langle X_\v\spinprod\bigl(Z_1 \cbullet X_\z \cbullet X_\v\bigr),\,
      Z_2 \big\rangle -  \frac{1}{\norm{X_\v}^2}\big\langle \underbrace{X_\v\spinprod U(X,Z_1)}_{=0},\,
      Z_2 \big\rangle\\
    &\stackrel{\eqref{eq:def_spinor_product}}{=}
      \big\langle Z_1 \cbullet X_\z \cbullet X_\v,\,
      Z_2 \cbullet X_\v \big\rangle
\end{align*}
where we used that \(U(X,Z_1) \in \Kern(X_\v \spinprod)\).
This proves both descriptions~\eqref{eq:def_overK} and~\eqref{eq:def_K_2}
of \(\overK(X)\). Evaluating~\eqref{eq:def_overK} with
\(Z_1 \coloneqq X_\z\)  gives
\[
  \langle \overK(X) X_\z, Z \rangle
  = \big\langle X_\z \cbullet X_\z \cbullet X_\v,\,
    Z \cbullet X_\v \big\rangle
  = -\norm{X_\z}^2 \big\langle X_\v, Z \cbullet X_\v \big\rangle
  = 0
\]
for all \(Z \in \z\) since \(Z \cbullet\) is skew. Hence \(\overK(X) X_\z = 0\).
With \(Z_2 \coloneqq X_\z\), we likewise obtain
\(\langle \overK(X) Z, X_\z \rangle = 0\) for all \(Z \in \z\).

Thus, for the remaining arguments we may assume \(Z_1, Z_2 \in X_\z^\perp\).
Since \(Z_2 \perp X_\z\), the Clifford relation implies that
\((X_\z \cbullet)(Z_2 \cbullet)=-(Z_2 \cbullet)(X_\z \cbullet)\). Because
both \(X_\z \cbullet\) and \(Z_2 \cbullet\) are skew-symmetric,
their product \(X_\z \cbullet Z_2 \cbullet\) is skew-symmetric as well.
Then
\begin{align*}
  \langle Z_1 \cbullet X_\z \cbullet X_\v,\,
          Z_2 \cbullet X_\v \rangle
  &\stackrel{\eqref{eq:Clifford_module_3}}{=}
    -\,\langle X_\z \cbullet Z_1 \cbullet X_\v,\,
            Z_2 \cbullet X_\v \rangle \\
  &\stackrel{\eqref{eq:Clifford_module_1}}{=}
    \langle Z_2 \cbullet X_\z \cbullet Z_1 \cbullet X_\v,\,
            X_\v \rangle \\
  &= -\,\langle X_\z \cbullet Z_2 \cbullet Z_1 \cbullet X_\v,\,
             X_\v \rangle \\
  &= \langle Z_1 \cbullet X_\v,\,
           X_\z \cbullet Z_2 \cbullet X_\v \rangle \\
  &= \langle X_\z \cbullet Z_2 \cbullet X_\v,\,
           Z_1 \cbullet X_\v \rangle \\
  &= -\,\langle Z_2 \cbullet X_\z \cbullet X_\v,\,
             Z_1 \cbullet X_\v \rangle
\end{align*}

On the one hand, this yields the skew-symmetry of \(\overK(X)\)
by~\eqref{eq:def_overK}; on the other hand, we obtain
\begin{align*}
  \langle \overK(X) Z_1, Z_2 \rangle
  &\stackrel{\eqref{eq:def_overK}}{=}
    \langle Z_1 \cbullet X_\z \cbullet X_\v,\,
            Z_2 \cbullet X_\v \rangle \\
  &= -\langle Z_2 \cbullet X_\z \cbullet X_\v,\,
             Z_1  \cbullet X_\v \rangle \\
  &\stackrel{\eqref{eq:def_spinor_product}}{=}
    -\big\langle Z_2,\,
      (X_\z \cbullet X_\v)\spinprod(Z_1 \cbullet X_\v) \big\rangle \\
  &=  \big\langle Z_2,\,
      (Z_1 \cbullet X_\v)\spinprod(X_\z \cbullet X_\v) \big\rangle
\end{align*}
by alternation of \(\spinprod\) thereby establishing~\eqref{eq:de_K_4}.

For items \textnormal{(d)}–\textnormal{(e)}, the summands on the right-hand
side of \eqref{eq:de_K_5} are orthogonal because
\(\Kern(X_\v\spinprod)\perp (\z\cbullet X_\v)\) by~\eqref{eq:Kern_spin_product}.
Taking norms in \eqref{eq:de_K_5} and using \eqref{eq:Clifford_module_4}
with \(S=X_\v\), we obtain
\[
  \norm{X_\v}^4\,\norm{Z \cbullet X_\z \cbullet X_\v}^2
  = \norm{U(X,Z)}^2 + \norm{(\overK(X)Z)\cbullet X_\v}^2
  = \norm{U(X,Z)}^2 + \norm{X_\v}^2\,\norm{\overK(X)Z}^2
\]
by the Pythagorean theorem. Since \(\norm{Z \cbullet X_\z \cbullet X_\v}^2
= \norm{Z}^2\norm{X_\z}^2\norm{X_\v}^2\), dividing by \(\norm{X_\v}^2\) yields
\[
  \norm{Z}^2\norm{X_\z}^2\norm{X_\v}^4
  = \norm{\overK(X)Z}^2 + \frac{1}{\norm{X_\v}^2}\,\norm{U(X,Z)}^2
\]
In particular,
\[
  \norm{\overK(X) Z}^2 \le \norm{Z}^2 \norm{X_\z}^2 \norm{X_\v}^4
\]

Therefore, the eigenvalues \(\undermu\) of the rescaled operator
\(-\underK(X)^2\) lie in \([0,1]\). Moreover, considering the spectral
decomposition of \(Z \in \z\), we obtain the following characterization
of the \(1\)-eigenspace \(\Eig_{1}\bigl(-\underK(X)^2\bigr)\):
\begin{align*}
  Z \in \Eig_{1}\bigl(-\underK(X)^2\bigr)
  &\Longleftrightarrow
    \norm{\overK(X) Z}^2
    = \norm{Z}^2 \norm{X_\z}^2 \norm{X_\v}^4 \\
    &\Longleftrightarrow
    U(X,Z)=0 \\
    &\stackrel{\eqref{eq:de_K_5}}{\Longleftrightarrow}
    \exists \tilde Z \in \z\colon  Z \cbullet X_\z \cbullet X_\v
    = \tilde Z \cbullet X_\v \\
  &\stackrel{\eqref{eq:de_K_5}}{\Longleftrightarrow}
    \norm{X_\v}^2\, Z \cbullet X_\z \cbullet X_\v
    = \overK(X)Z\cbullet X_\v \\
  &\Longleftrightarrow
    Z \cbullet X_\z \cbullet X_\v
    = \norm{X_\z}\,\underK(X)Z \cbullet X_\v
\end{align*}
Since \(Z \mapsto Z \bullet X_\v\) is injective according to
Lemma~\ref{le:K}~\textnormal{(b)}, this yields~\eqref{eq:mu_=_1}.

The last assertion is clear since
\(\underK(X)\) defines an orthogonal complex structure on
\(\Eig_{1}\bigl(-\underK(X)^2\bigr)\).
\end{proof}

\begin{proposition}[{\cite[p.~34]{BTV}, Lemma}]
\label{p:eigenvalues_of_K_and_tilde_K}
  For each geodesic \(\gamma\) of \(N\), the eigenvalues of both
  \(\underK(\dot{\gamma}(t))^2\) and \(\overK(\dot{\gamma}(t))^2\)
  remain constant.
\end{proposition}

\begin{proof}
  Let \(\gamma\colon \bbR \to N\) be a geodesic with \(\gamma(0) = e_N\)
  and \(\dot{\gamma}(0) = X\). Then \(\dot{\gamma}(t) \cong (\gamma(t),
  X(t))\) in the left-invariant framing, where the function
  \(X\colon \bbR \to \n\) is given by~\eqref{eq:geodesic}. Then,
  \begin{align*}
  Z \cbullet X_\v(t)
  &\stackrel{\eqref{eq:geodesic}}{=} \cos(\norm{X_\z}t)\,Z \cbullet X_\v
     + \frac{1}{\norm{X_\z}}\,
       \sin(\norm{X_\z}t)\,Z \cbullet X_\z \cbullet X_\v \\
  X_\z \cbullet X_\v(t)
  &\stackrel{\eqref{eq:geodesic}}{=} \cos(\norm{X_\z}t)\,X_\z \cbullet X_\v
     - \norm{X_\z}\,\sin(\norm{X_\z}t)\,X_\v
\end{align*}
Hence,
\begin{align*}
  \overK(\dot{\gamma}(t))Z
  &\stackrel{\eqref{eq:de_K_4}}{=}
    (Z \cbullet X_\v(t))\spinprod(X_\z \cbullet X_\v(t))\\
  &=
    \cos(\norm{X_\z}t)^2(Z \cbullet X_\v)\spinprod(X_\z\cbullet X_\v)
    - \sin(\norm{X_\z}t)^2\,\bigl(Z\cbullet X_\z\cbullet X_\v\bigr)\spinprod X_\v\\
  &\quad
    + \sin(\norm{X_\z}t)\cos(\norm{X_\z}t)\Bigl(
      \underbrace{\tfrac{1}{\norm{X_\z}} (Z \cbullet X_\z\cbullet X_\v)
      \spinprod(X_\z\cbullet X_\v)
      - \norm{X_\z}\,(Z\cbullet X_\v)\spinprod X_\v}_{=0}\Bigr)\\
  &\stackrel{\eqref{eq:def_K_2},\eqref{eq:de_K_4}}{=}
    \cos(\norm{X_\z}t)^2\,\overK(X)Z
    + \sin(\norm{X_\z}t)^2\,\overK(X)Z\\
  &= \overK(X)Z
\end{align*}
where the mixed term vanishes by alternation of \(\spinprod\) and
\eqref{eq:Clifford_module_5} applied to \(S=X_\v\) and to
\(S=X_\z \cbullet X_\v\).

From this we see that \(\overK(\dot{\gamma}(t))\) is constant along
\(\gamma\) in the canonical left-invariant framing. The same is true
for
\begin{align*}
  \underK(\dot{\gamma}(t))
  &= \frac{1}{\norm{\dot{\gamma}_\z}\,\norm{\dot{\gamma}_\v}^2}
     \overK(\dot{\gamma}(t)) \\
  &\stackrel{\eqref{eq:geodesic}}{=}
    \frac{1}{\norm{X_\z}\,\norm{X_\v}^2}\overK(\dot{\gamma}(t))
\end{align*}
by Corollary~\ref{co:Killing_1}. In particular, the eigenvalues
of \(\underK(\dot{\gamma}(t))^2\) are constant functions of the parameter \(t\) as well.
\end{proof}

\subsection{Direct sums of Clifford modules}
\label{se:reducible_Clifford_module}
Throughout this subsection consider a Euclidean space \(\z\) of positive
dimension and an orthogonal \(\Cl(\z)\)-module \(\v\).
Furthermore, suppose that \(\v\) admits an orthogonal decomposition
\[
  \v = \v' \oplus \v''
\]
into two nontrivial \(\Cl(\z)\)-submodules \(\v'\) and \(\v''\).
Accordingly, let \(\overK\), \(\overK'\) and \(\overK''\)
denote the family of linear operators defined in~\eqref{eq:def_overK}
for the pairs \((\z,\v)\), \((\z,\v')\) and \((\z,\v'')\).

\begin{lemma}\label{le:reducible_Clifford_module}
Let \(X = Z_0 \oplus S' \oplus S'' \in \z \oplus \v' \oplus \v''\)
with \(Z_0 \in \z\) nonzero.
\begin{enumerate}
\item  We have
\[
  \overK(X) = \overK'(X') + \overK''(X'')
\] where
\(X' \coloneqq Z_0 \oplus S'\) and \(X'' \coloneqq Z_0 \oplus S''\).
\item If \(S' \neq 0\) and \(S'' \neq 0\), then the corresponding
  rescaled operators from~\eqref{eq:def_underK} satisfy
  \begin{equation}\label{eq:underK_vs_underK_1_and_underK_2}
    \underK(X)
    =
    \frac{\norm{S'}^2\,\underK'(X')
          + \norm{S''}^2\,\underK''(X'')}
         {\norm{S'}^2 + \norm{S''}^2}
  \end{equation}
\item If \(S' \neq 0\) and \(S'' \neq 0\), then the eigenspace of
  \(-\underK(X)^2\) for eigenvalue \(1\) is given by
  \begin{equation}\label{eq:common_nu_=_1_eigenspace}
    \Eig_1\bigl(-\underK(X)^2\bigr)
    =
    \bigl\{
      Z \in \Eig_1\bigl(-\underK'(X')^2\bigr)
        \cap \Eig_1\bigl(-\underK''(X'')^2\bigr)
      \,\bigm|\,
      \underK'(X')Z = \underK''(X'')Z
    \bigr\}
  \end{equation}
\end{enumerate}
\end{lemma}

\begin{proof}
For (a), using that \(\v'\) and \(\v''\) are orthogonal
\(\Cl(\z)\)-submodules, we compute for all \(Z_1, Z_2 \in \z\):
\begin{align*}
  \langle \overK(X) Z_1,\, Z_2 \rangle
    &\stackrel{\eqref{eq:def_overK}}{=}
      \langle Z_1 \cbullet Z_0 \cbullet X_\v,\, Z_2 \cbullet X_\v \rangle \\
    &= \langle Z_1 \cbullet Z_0 \cbullet (S' \oplus S''),\,
              Z_2 \cbullet (S' \oplus S'') \rangle \\
    &= \langle Z_1 \cbullet Z_0 \cbullet S',\, Z_2 \cbullet S' \rangle
      + \langle Z_1 \cbullet Z_0 \cbullet S'',\,
          Z_2 \cbullet S'' \rangle \\
    &= \langle \overK'(X') Z_1,\, Z_2 \rangle
      + \langle \overK''(X'') Z_1,\, Z_2 \rangle
\end{align*}
The mixed terms vanish since \(Z_i\cbullet S'\in \v'\),
\(Z_i\cbullet S''\in \v''\) for \(i \in \{1,2\}\) and \(\v'\perp\v''\).
Hence \(\overK(X) = \overK'(X') + \overK''(X'')\), which proves (a).
Comparing~\eqref{eq:def_overK} and~\eqref{eq:def_underK} yields
\eqref{eq:underK_vs_underK_1_and_underK_2}.

For (c), assume \(S' \neq 0\) and \(S'' \neq 0\), and set
\(S \coloneqq S' \oplus S''\).

\smallskip
\emph{``\(\subseteq\)'':}
Let \(Z \in \Eig_1\bigl(-\underK(X)^2\bigr)\). By
Proposition~\ref{p:overK}~\textnormal{(e)} applied to \((\z,\v)\) and \(X\),
there exists \(\tilde Z \in Z_0^\perp\) such that
\[
  Z \cbullet Z_0 \cbullet S = \norm{Z_0}\, \tilde Z \cbullet S
\]
Since \(\v'\) and \(\v''\) are \(\Cl(\z)\)-submodules and
\(\v=\v' \oplus \v''\), this is equivalent to the pair of equalities
\[
  Z \cbullet Z_0 \cbullet S' = \norm{Z_0}\, \tilde Z \cbullet S'\qquad\text{and}\qquad
  Z \cbullet Z_0 \cbullet S'' = \norm{Z_0}\, \tilde Z \cbullet S''
\]
Applying Proposition~\ref{p:overK}~\textnormal{(e)} to \((\z,\v')\) and \(X'\)
(with \(S'\) in place of \(S\)), we obtain
\(Z \in \Eig_1\bigl(-\underK'(X')^2\bigr)\) and \(\tilde Z=\underK'(X')Z\).
Likewise, applying Proposition~\ref{p:overK}~\textnormal{(e)} to \((\z,\v'')\)
and \(X''\), we obtain
\(Z \in \Eig_1\bigl(-\underK''(X'')^2\bigr)\) and \(\tilde Z=\underK''(X'')Z\).
Hence
\[
  Z \in \Eig_1\bigl(-\underK'(X')^2\bigr)
    \cap \Eig_1\bigl(-\underK''(X'')^2\bigr)
  \qquad\text{and}\qquad
  \underK'(X')Z = \underK''(X'')Z
\]
This proves ``\(\subseteq\)'' in~\eqref{eq:common_nu_=_1_eigenspace}.

\smallskip
\emph{``\(\supseteq\)'':}
Conversely, let
\[
  Z \in \Eig_1\bigl(-\underK'(X')^2\bigr)
    \cap \Eig_1\bigl(-\underK''(X'')^2\bigr)
  \quad\text{and}\quad
  \underK'(X')Z = \underK''(X'')Z
\]
Set \(\tilde Z \coloneqq \underK'(X')Z = \underK''(X'')Z\).
By Proposition~\ref{p:overK}~\textnormal{(e)} applied to \((\z,\v')\) and \(X'\),
we have
\[
  Z \cbullet Z_0 \cbullet S' = \norm{Z_0}\, \tilde Z \cbullet S'
\]
and similarly, by Proposition~\ref{p:overK}~\textnormal{(e)} applied to \((\z,\v'')\)
and \(X''\),
\[
  Z \cbullet Z_0 \cbullet S'' = \norm{Z_0}\, \tilde Z \cbullet S''
\]
Combining these yields
\[
  Z \cbullet Z_0 \cbullet (S' \oplus S'')
  = \norm{Z_0}\, \tilde Z \cbullet (S' \oplus S'')
\]
that is,
\[
  Z \cbullet Z_0 \cbullet S = \norm{Z_0}\, \tilde Z \cbullet S
\]
Therefore, Proposition~\ref{p:overK}~\textnormal{(e)} applied to \((\z,\v)\) and \(X\)
implies \(Z \in \Eig_1\bigl(-\underK(X)^2\bigr)\).
This proves ``\(\supseteq\)'' and completes the proof of~\textnormal{(c)}.
\end{proof}

\begin{lemma}\label{le:complex_structures}
Let \(V\) be a Euclidean space and let \(\{J_1,J_2\}\) be a pair of
orthogonal complex structures on \(V\). Then there exists
\(F \in \Ogroup(V)\) such that
\[
  \Kern\bigl(J_1 - F \circ J_2 \circ F^{-1}\bigr) = \{0\}
\]
\end{lemma}

\begin{proof}
Any two orthogonal complex structures on \(V\) are conjugate under the
action of \(\Ogroup(V)\). Since \(-J_2\) is again an orthogonal complex
structure, there exists \(F \in \Ogroup(V)\) such that
\[
  J_1 = F \circ (-J_2) \circ F^{-1}
\]
that is, \(J_1 \circ F = -\,F \circ J_2\). Hence
\[
  F \circ J_2 \circ F^{-1} = -J_1
\]
and therefore \(J_1 - F \circ J_2 \circ F^{-1} = 2J_1\). As
\(J_1^2 = -\Id_V\), the operator \(J_1\) is invertible, so
\(\Kern(J_1) = \{0\}\), and the claim follows.
\end{proof}

\begin{corollary}\label{co:reducible_Clifford_module}
If \(\nu = 1\) is a global eigenvalue of \(-\underK^2\), then both
\((\z,\v')\) and \((\z,\v'')\) satisfy the \(J^2\)-condition.
\end{corollary}

\begin{proof}
Let \(\nu = 1\) and suppose that \(\nu\) is a global eigenvalue of
\(-\underK^2\). Choose \(Z_0 \in \z\), \(S' \in \v'\), and \(S'' \in \v''\),
all nonzero, and set
\[
  X \coloneqq Z_0 \oplus S' \oplus S''
  \in \z \oplus \v' \oplus \v''
\]
Also set \(X' \coloneqq Z_0 \oplus S' \in \z \oplus \v'\) and
\(X'' \coloneqq Z_0 \oplus S'' \in \z \oplus \v''\). Let
\[
  E' \coloneqq \Eig_1\bigl(-\underK'(X')^2\bigr),
  \qquad
  E'' \coloneqq \Eig_1\bigl(-\underK''(X'')^2\bigr)
\]
Then \(J' \coloneqq \underK'(X')|_{E'}\), \(J'' \coloneqq \underK''(X'')|_{E''}\) define
orthogonal complex structures on \(E'\) and \(E''\), respectively.
Set
\[
  E \coloneqq \{Z \in E' \cap E'' \mid J'Z = J''Z\}
\]
On \(E\), the restrictions agree; set
\[
  J \coloneqq J'|_E = J''|_E
\]
Then \(J\) is an orthogonal complex structure on \(E\). By
Lemma~\ref{le:complex_structures}, there exists \(F \in \Ogroup(E)\) such
that
\begin{equation}\label{eq:Kernel_of_intersection_on_E}
  \Kern\bigl(F \circ J \circ F^{-1} - J\bigr) = \{0\}
\end{equation}

Assume, for a contradiction, that \(E'\) is a strict subspace of
\(Z_0^\perp\). Choose the orthogonal decomposition
\[
  E' \cap E'' = E \oplus  E^\perp
\]
where \(E^\perp\) denotes the orthogonal complement of \(E\) in \(E'\cap E''\),
and define \(\hat F \in \Ogroup(E' \cap E'')\) by \(\hat F|_E = F\) and
\(\hat F|_{E^\perp} = \Id_{E^\perp}\).

Extend \(\hat F\) to an element of \(\Ogroup(E')\), and then extend further to an
element \(\tilde F \in \SO(\z)\) such that
\[
  \tilde F Z_0 = Z_0,
  \qquad
  \tilde F(E') = E',
  \qquad
  \tilde F|_{E' \cap E''} = \hat F
\]
This is possible since \(E' \subsetneq Z_0^\perp\), so we can choose the
action of \(\tilde F\) on the nontrivial orthogonal complement of \(E'\) in \(Z_0^\perp\)
and, if necessary, adjust it to obtain \(\det(\tilde F)=+1\).

Because \(\tilde F \in \SO(\z)\), it admits a lift \(\tilde F^{\ext}\)
to the spin covering and thus acts on both \(\z\) (via the vector
representation as \(\tilde F\)) and on \(\v'\) (via the spin representation
of the orthogonal \(\Cl(\z)\)-module \(\v'\)). Then
\[
  \tilde F^{\ext} X' = Z_0 \oplus \tilde F^{\ext} S' \in \z \oplus \v'
\]
since \(\tilde F(Z_0) = Z_0\). By spin-equivariance of \(\underK'\) (cf.~\eqref{eq:K_Aut_equivariance}),
\[
  \underK'\bigl(\tilde F^{\ext} X'\bigr)
  = \tilde F \circ \underK'\bigl(X'\bigr) \circ \tilde F^{-1}
\]
Therefore, \(\Eig_1(-\underK'\bigl(\tilde F^{\ext} X'\bigr)^2)=\tilde F(E')=E'\),
hence \( \tilde J' := \underK'\bigl(\tilde F^{\ext} X'\bigr)|_{E'}\) defines
still an orthogonal complex structure on \(E'\).
Consider
\[
  \tilde X \coloneqq Z_0 \oplus \tilde F^{\ext} S' \oplus S''
  \in \z \oplus \v' \oplus \v''
\]
Applying Lemma~\ref{le:reducible_Clifford_module}~\textnormal{(c)} to
\(\tilde X\), we obtain
\[
  \Eig_1\bigl(-\underK(\tilde X)^2\bigr)
  \stackrel{\eqref{eq:common_nu_=_1_eigenspace}}{=}
  \{Z \in E' \cap E'' \mid \tilde J'Z = J''Z\}
\]
By construction,
\begin{equation}\label{eq:Kernel_of_intersection}
  \Kern\bigl(\tilde J' - J''\bigr) = \{0\}
\end{equation}

To see this, let \(Z \in E' \cap E''\) satisfy \(\tilde J'Z = J''Z\).
Write \(Z = Z_E + Z^\perp\) with \(Z_E \in E\) and \(Z^\perp \in E^\perp\). By construction,
\(\tilde F|_E = F\) and \(\tilde F|_{E^\perp}=\Id\), hence
\(\tilde J'|_E = FJF^{-1}\) and \(\tilde J'|_{E^\perp} = J'|_{E^\perp}\). Using
also \(J''|_E = J\), the equation \(\tilde J'Z = J''Z\) becomes
\[
  (FJF^{-1} - J)Z_E = (J'' - J')Z^\perp
\]
The left-hand side lies in \(E\). The right-hand side is orthogonal to
\(E\), because \(J''-J'\) is skew-symmetric on \(E' \cap E''\) and vanishes on \(E\),
so for every \(Y \in E\),
\[
  \langle (J''-J')Z^\perp, Y \rangle
  = -\langle Z^\perp, (J''-J')Y \rangle
  = 0
\]
Hence both sides must vanish. From \eqref{eq:Kernel_of_intersection_on_E}
we obtain \(Z_E=0\). Then \((J''-J')Z^\perp=0\), so \(Z^\perp \in E\). Since also
\(Z^\perp \in  E^\perp\), it follows that \(Z^\perp=0\). Thus \(Z=0\).

By~\eqref{eq:Kernel_of_intersection},
\[
  \Eig_1\bigl(-\underK(\tilde X)^2\bigr) = \{0\}
\]
This contradicts the assumption that \(\nu=1\) is a global eigenvalue of
\(-\underK^2\). We conclude that
\[
  \Eig_1\bigl(-\underK'(X')^2\bigr) = Z_0^\perp
\]
that is, the pair \((\z,\v')\) satisfies the \(J^2\)-condition.
Interchanging the roles of \(\v'\) and \(\v''\), the same holds for
\((\z,\v'')\).
\end{proof}

\subsection{Tensor products of Clifford modules}
Throughout this subsection fix an orthogonal decomposition
\(\z=\z'\oplus\z''\) with \(\dim(z')\ge 1\). Then there is a
canonical isomorphism of \(\bbZ_2\)-graded algebras
\[
  \Cl(\z) \cong \Cl(\z') \,\hat\otimes_\bbR\, \Cl(\z'')
\]
where \(\hat\otimes_\bbR\) denotes the graded tensor product.

Recall that an orthogonal \(\Cl(\z')\)-module \(\v'\) is called \(\bbZ_2\)-graded if
there exists an orthogonal decomposition \(\v'=\v'_+\oplus\v'_-\) such
that
\[
  Z \cbullet \v'_+ \subseteq \v'_-
  \quad\text{and}\quad
  Z \cbullet \v'_- \subseteq \v'_+
\]
for all \(Z \in \z'\); cf.\ \cite[Ch.~1, \S 5, p.~39]{LM}.

Let \(\v'\) be a nontrivial \(\bbZ_2\)-graded orthogonal \(\Cl(\z')\)-module and let \(\v''\)
be a nontrivial (not necessarily \(\bbZ_2\)-graded) orthogonal \(\Cl(\z'')\)-module. Set
\(\v \coloneqq \v'\otimes_\bbR \v''\) and equip \(\v\) with the tensor
product inner product
\begin{equation}\label{eq:tensor_inner_product}
  \langle S'_1 \otimes S''_1,\; S'_2 \otimes S''_2 \rangle
  \coloneqq
  \langle S'_1, S'_2 \rangle \,\langle S''_1, S''_2 \rangle
\end{equation}
extended bilinearly. Then \(\v\) becomes an orthogonal \(\Cl(\z)\)-module via
\begin{equation}\label{eq:tensor_product_multiplication}
  (Z' \oplus Z'') \cbullet (S' \otimes S'')
  \coloneqq
  (Z' \cbullet S') \otimes S''
  + (-1)^{\deg(S')} S' \otimes (Z'' \cbullet S'')
\end{equation}
for all \(Z' \oplus Z'' \in \z' \oplus \z''\), for homogeneous
\(S' \in \v'\), and for arbitrary \(S'' \in \v''\), where
\(\deg(S') \in \{0,1\}\) denotes the parity of \(S'\) with respect to
the \(\bbZ_2\)-grading on \(\v'\). The sign \((-1)^{\deg(S')}\) arises
because \(Z''\) is odd and, in the graded tensor product, must be
commuted past the homogeneous element \(S'\). It is straightforward to
check that \eqref{eq:Clifford_module_1} and \eqref{eq:Clifford_module_2}
hold. We write \(\v = \v'\,\hat\otimes_\bbR\,\v''\).

To obtain a version of~\eqref{eq:tensor_product_multiplication}
valid for arbitrary \(S' \in \v'\), let \(\Gamma\colon \v' \to \v'\)
be the orthogonal grading involution, defined by \(\Gamma|_{\v'_+} = \Id\) and
\(\Gamma|_{\v'_-} = -\Id\). For homogeneous \(S' \in \v'\) we have
\(\Gamma S' = (-1)^{\deg(S')} S'\). Hence the Clifford
action~\eqref{eq:tensor_product_multiplication} 
extends to arbitrary \(S' \in \v'\) as
\begin{equation}\label{eq:tensor_product_Gamma}
  (Z' \oplus Z'') \cbullet (S' \otimes S'')
  =
  (Z' \cbullet S') \otimes S''
  + (\Gamma S') \otimes (Z'' \cbullet S'')
\end{equation}
for all \(Z' \oplus Z'' \in \z' \oplus \z''\) and all
\(S' \in \v'\), \(S'' \in \v''\).

Consider the families of
skew-symmetric operators \(\underK'\) and \(\underK\) associated with
\((\z',\v')\) and \((\z,\v)\), respectively (see~\eqref{eq:def_underK}).

\begin{lemma}\label{le:simple}
In the above situation, let \(S' \in \v'\) and \(S'' \in \v''\) be
nonzero. Let \(Z_0' \in \z'\) be nonzero and set
\[
  X \coloneqq Z_0' \oplus (S' \otimes S'')
    \in \z \oplus \bigl(\v' \hat\otimes_\bbR \v''\bigr),
  \qquad
  X' \coloneqq Z_0' \oplus S'
    \in \z' \oplus \v'
\]
Then the characteristic endomorphism \(\underK(X)\colon \z \to \z\)
associated with \(\bigl(\z,\v' \hat\otimes_\bbR \v''\bigr)\) and \(X\)
coincides with \(\underK'(X')\colon \z' \to \z'\), the endomorphism
associated with \((\z',\v')\) and \(X'\); that is,
\[
  \underK(X) = \underK'(X') \oplus 0
  \quad\text{as an endomorphism of }\z' \oplus \z''
\]
\end{lemma}

\begin{proof}
\begin{itemize}
\item If \(\{Z'_1,Z'_2\}\) is an orthonormal pair in \(\z'\), then
\begin{align*}
  \norm{X_\z}\norm{X_\v}^2\,
  \langle \underK(X) Z'_1, Z'_2 \rangle
   &\stackrel{\eqref{eq:def_overK},\eqref{eq:def_underK}}{=}
     \langle Z'_1 \cbullet X_\z \cbullet X_\v,
       Z'_2 \cbullet X_\v \rangle \\
   &= \norm{S''}^2\,
      \langle Z'_1 \cbullet Z_0' \cbullet S',
        Z'_2 \cbullet S' \rangle \\
   &= \norm{Z_0'} \norm{S'}^2 \norm{S''}^2\,
      \langle \underK'(X') Z'_1, Z'_2 \rangle
\end{align*}
(using \eqref{eq:def_underK} for \(X' = Z_0' \oplus S'\)). Because
\(\norm{X_\z} = \norm{Z_0'}\) and
\(\norm{X_\v}^2 = \norm{S'}^2 \norm{S''}^2\), we obtain
\(\langle \underK(X) Z'_1, Z'_2 \rangle
 = \langle \underK'(X') Z'_1, Z'_2 \rangle\).

\item If \(Z'_1 \in \z'\) and \(Z''_2 \in \z''\), then
\begin{align*}
  \norm{X_\z}\norm{X_\v}^2\,
  \langle \underK(X) Z'_1, Z''_2 \rangle
   &= \langle Z'_1 \cbullet X_\z \cbullet X_\v,
      Z''_2 \cbullet X_\v \rangle \\
   &\stackrel{\eqref{eq:tensor_product_Gamma}}{=}
      \langle Z'_1 \cbullet Z_0' \cbullet S',\, \Gamma S' \rangle\,
      \langle S'',\, Z''_2 \cbullet S'' \rangle \\
   &= 0
\end{align*}
since \(Z''_2 \cbullet\) is skew-symmetric on \(\v''\). Moreover,
\[
  \norm{X_\v}^2\,\langle \underK(X) Z''_2, Z'_1 \rangle
  = -\,\norm{X_\v}^2\,\langle \underK(X) Z'_1, Z''_2 \rangle
  = 0
\]
\item If \(\{Z''_1,Z''_2\} \subset \z''\) is orthonormal, then
\begin{align*}
  \norm{X_\z}\norm{X_\v}^2\,
  \langle \underK(X) Z''_1, Z''_2 \rangle
   &= \langle Z''_1 \cbullet X_\z \cbullet X_\v,
      Z''_2 \cbullet X_\v \rangle \\
   &\stackrel{\eqref{eq:tensor_product_Gamma}}{=}
      \langle \Gamma(Z_0' \cbullet S'),\, \Gamma S' \rangle\,
      \langle Z''_1 \cbullet S'',\, Z''_2 \cbullet S'' \rangle \\
   &= \langle Z_0' \cbullet S',\, S' \rangle\,
      \langle Z''_1 \cbullet S'',\, Z''_2 \cbullet S'' \rangle \\
   &= 0
\end{align*}
since \(Z_0' \cbullet\) is skew-symmetric on \(\v'\).
\end{itemize}
Hence \(\underK(X) = \underK'(X') \oplus 0\).
\end{proof}

\begin{corollary}\label{co:tensor_product}
If \(\nu = 1\) is a global eigenvalue of \(-\underK^2\), then the same
holds for \(-(\underK')^2\).
\end{corollary}

\begin{proof}
Let \(\nu = 1\). Let \(Z_0' \in \z'\), \(S' \in \v'\), and
\(S'' \in \v''\) be nonzero, and set
\[
  X \coloneqq Z_0' \oplus (S' \otimes S'')
  \in \z \oplus \bigl(\v' \hat\otimes_\bbR \v''\bigr)
\]
By Lemma~\ref{le:simple}, \(\underK(X) = \underK'(X') \oplus 0\) with
\(X' = Z_0' \oplus S'\). Hence
\[
  \Spec\bigl(-\underK(X)^2\bigr)
  = \Spec\bigl(-\underK'(X')^2\bigr)
\]
If \(\nu\) is a global eigenvalue of \(-\underK^2\), then
\(\nu \in \Spec\bigl(-\underK(X)^2\bigr)\) for all such \(X\), and
therefore \(\nu \in \Spec\bigl(-\underK'(X')^2\bigr)\) for all such
\(X'\). Thus \(\nu\) is a global eigenvalue of \(-(\underK')^2\).
\end{proof}

\section{The gradients of the eigenvalue branches of
\texorpdfstring{\(-\underK^2\)}{-K^2}}
\label{se:gradient}

Let \(\z\) be a Euclidean space of positive dimension \(n\), let
\(\v\) be a nontrivial orthogonal \(\Cl(\z)\)-module, and let
\(\n \coloneqq \z \oplus \v\) be the orthogonal Lie algebra associated
with the pair \((\z,\v)\). Recall that the skew-symmetric operator
\(\underK(X)\colon \z \to \z\) defined in~\eqref{eq:def_underK}
satisfies
\[
  \Spec\bigl(-\underK(X)^2\bigr) \subseteq [0,1]
\]
for all \(X \in \n\) with \(X_\z \ne 0\) and \(X_\v \ne 0\), by
Proposition~\ref{p:overK}. Consequently, every eigenvalue branch
\(\undermu\) of the nonnegative self-adjoint family \(-\underK^2\)
defined in~\eqref{eq:underK_square} satisfies
\(0 \le \undermu \le 1\). We now incorporate the gradients
\(\nabla \undermu\) into our analysis. To this end, recall that
\(\h_X = X_\z^\perp \subseteq \z\) denotes the orthogonal complement
of the line spanned by \(X_\z\), for each \(X \in \n\) with
\(X_\z \ne 0\).

Since \(\overK \equiv 0\) for \(n \le 2\), we will assume throughout
this section that \(n \ge 3\). Then each
\(X \in \n \setminus \ram(\overK^2)\) automatically satisfies
\(X_\z \ne 0\) and \(X_\v \ne 0\), cf.~\eqref{eq:domain}.

\begin{lemma}\label{le:var}
  Let \(X \in \n \setminus \ram(\overK^2)\) and let
  \(\undermu\colon U(X) \to [0,1]\) be an eigenvalue branch of
  \(-\underK^2\) of multiplicity \(m\) on the connected component
  \(U(X) \subseteq \n \setminus \ram(\overK^2)\) containing \(X\).
  If \(\undermu(X) = 0\), then \(m \le 2\). Otherwise \(\undermu(X) > 0\),
  \(m\) is even, and there exists an orthonormal basis
  \(E_1,\ldots,E_m\) of \(\h_X\) such that
  \begin{equation}\label{eq:mu_frame_X}
  \begin{aligned}
    \underK(X) E_{2i-1} &= \sqrt{\undermu(X)}\,E_{2i} \\
    \underK(X) E_{2i}   &= -\sqrt{\undermu(X)}\,E_{2i-1}
  \end{aligned}
  \end{equation}
  for \(i = 1,\ldots,\tfrac{m}{2}\). Moreover, for any such basis
  \begin{equation}\label{eq:Pyth_1}
   E_{2i-1} \cbullet E_{2i} \cbullet X_\z \cbullet X_\v
   = \norm{X_\z}\sqrt{\undermu(X)}\,X_\v
    \oplus \frac{\norm{X_\z}\norm{X_\v}^2}{2}
           \bigl(\left.\nabla\!\sqrt{\undermu}\right|_X\bigr)_\v
  \end{equation}
  gives the components of
  \(E_{2i-1} \cbullet E_{2i} \cbullet X_\z \cbullet X_\v\)
  with respect to the orthogonal splitting
  \(\v = \bbR X_\v \oplus X_\v^\perp\). It follows that
  \begin{equation}\label{eq:grad_3}
    1 = \undermu(X) + \frac{\norm{X_\v}^2}{4}
        \norm{\left.\nabla\!\sqrt{\undermu}\right|_X}_\v^2
  \end{equation}
\end{lemma}

\begin{proof}
  Suppose first that \(\undermu(X) > 0\). Then
  \((\Eig_{\undermu})_X \subseteq (\Eig_0)_X^\perp \subseteq X_\z^\perp
  = \h_X\). Since \(\underK(X)\) is skew-symmetric, the operator
  \(J_{\undermu}(X) \coloneqq \tfrac{1}{\sqrt{\undermu(X)}}\,\underK(X)\)
  defines an orthogonal complex structure on \((\Eig_{\undermu})_X\).
  Consequently, \(\underK(X)\) splits \((\Eig_{\undermu})_X\) into
  \(\tfrac{m}{2}\) mutually orthogonal \(2\)-planes, on each of which it
  acts by
  \(\begin{psmallmatrix} 0 & -\sqrt{\undermu(X)} \\ \sqrt{\undermu(X)} &
  0 \end{psmallmatrix}\). Choosing an orthonormal basis
  \(E_1,\ldots,E_m\) of \((\Eig_{\undermu})_X \subseteq \h_X\) with
  \(J_{\undermu}(X) E_{2i-1} = E_{2i}\) for
  \(i = 1,\ldots,\tfrac{m}{2}\), we obtain~\eqref{eq:mu_frame_X}. Since
  \(\undermu\) is real-analytic (hence smooth), it remains strictly
  positive on an open neighborhood \(U\) of \(X\), and we can extend
  \(\{E_1,\ldots,E_m\}\) to a smooth local orthonormal frame with the
  same properties at each point of \(U\).

   Suppose now that \(\undermu(X)=0\) and \(m\ge 3\). Since
  \(X_\z \in (\Eig_0)_X\) and \(\h_X=X_\z^\perp\) has codimension \(1\) in
  \(\z\), we have
  \[
    \dim\bigl((\Eig_0)_X \cap \h_X\bigr)\ge m-1 \ge 2
  \]
  Hence there exists an orthonormal pair
  \(\{E_1,E_2\}\subseteq (\Eig_0)_X\cap\h_X\). In particular, the
  relations in~\eqref{eq:mu_frame_X} hold with \(i=1\), i.e.,
  \(\underK(X)E_1=\underK(X)E_2=0\).

  Since \(X \notin \ram(\underK^2)\), the multiplicity of the eigenvalue
  \(0\) is constant on \(U(X)\). Thus a branch with \(\undermu(X)=0\) cannot
  become positive elsewhere on \(U(X)\), and consequently \(\undermu\equiv 0\)
  on \(U(X)\). Thus there exists an open neighborhood \(U\) of \(X\) and a
  smooth orthonormal \(2\)-frame \(\{E_1,E_2\}\) of \((\Eig_0|_U) \cap (\h|_U)\)
  extending the one at \(X\) and having the same properties at each point of \(U\).
  We will use this choice below to derive a contradiction when \(m \ge 3\).

  With these preliminaries, compute
  \[
  \begin{aligned}
    \frac{1}{\norm{X_\z}\norm{X_\v}^2}\,
    \big\langle
    E_{2i-1}\big|_X \cbullet E_{2i}\big|_X \cbullet X_\z
    \cbullet X_\v, X_\v \big\rangle
    &\stackrel{\eqref{eq:Clifford_module_1},\eqref{eq:Clifford_module_3}}{=}
     \frac{1}{\norm{X_\z}\norm{X_\v}^2}\,
     \big\langle
     E_{2i-1}\big|_X \cbullet X_\z \cbullet X_\v,
     E_{2i}\big|_X \cbullet X_\v \big\rangle \\
    &\stackrel{\eqref{eq:def_underK}}{=}
      \big\langle \underK(X) E_{2i-1}\big|_X, E_{2i}\big|_X \big\rangle
      \stackrel{\eqref{eq:mu_frame_X}}{=}\sqrt{\undermu(X)}
  \end{aligned}
  \]
  which gives
  \begin{equation}\label{eq:grad_1}
  \big\langle E_{2i-1}\big|_X \cbullet E_{2i}\big|_X \cbullet X_\z
  \cbullet X_\v, X_\v \big\rangle
  = \norm{X_\z}\norm{X_\v}^2\sqrt{\undermu(X)}
  \end{equation}

  We now compute
  \(\big\langle E_{2i-1}\big|_X \cbullet E_{2i}\big|_X \cbullet X_\z
  \cbullet X_\v, S \big\rangle\) for \(S \in \v\) with
  \(\langle X_\v, S \rangle = 0\). From
  \[
    \big\langle \nabla\big(\norm{X_\z}\norm{X_\v}^2\big), S \big\rangle
    = 2 \norm{X_\z} \langle X_\v, S \rangle = 0
  \]
  we get
  \[
    \norm{X_\z}\norm{X_\v}^2
    \big\langle \left.\nabla\!\sqrt{\undermu}\right|_X, 0 \oplus S
    \big\rangle
    \stackrel{\eqref{eq:undermu_overmu}}{=}
    \big\langle \left.\nabla\!\sqrt{\overmu}\right|_X, 0 \oplus S
    \big\rangle
  \]
  Here
  \[
    \sqrt{\overmu(X)}
    \stackrel{\eqref{eq:def_underK},\eqref{eq:undermu_overmu},\eqref{eq:mu_frame_X}}{=}
    \big\langle \overK(X) E_{2i-1}\big|_X, E_{2i}\big|_X \big\rangle
  \]
  so with the canonical variation \(X(t) \coloneqq X + (0 \oplus t S)\) we
  have
  \begin{align*}
    \norm{X_\z}\norm{X_\v}^2 \big\langle \left.\nabla\!\sqrt{\undermu}
    \right|_X, 0 \oplus S \big\rangle
    &= \left.\tfrac{d}{dt}\right|_{t=0}
       \big\langle \overK\big(X(t)\big) E_{2i-1}\big|_{X(t)},
       E_{2i}\big|_{X(t)} \big\rangle
    = \Big\langle \left.\tfrac{d}{dt}\right|_{t=0}
       \overK\big(X(t)\big) E_{2i-1}\big|_X, E_{2i}\big|_X \Big\rangle\\
    &\quad + \Big\langle \overK(X) \left.\tfrac{d}{dt}\right|_{t=0}
       E_{2i-1}\big|_{X(t)}, E_{2i}\big|_X \Big\rangle 
    \quad + \Big\langle \overK(X) E_{2i-1}\big|_X,
       \left.\tfrac{d}{dt}\right|_{t=0} E_{2i}\big|_{X(t)} \Big\rangle
  \end{align*}
  The last two terms vanish:
  \[
    \begin{aligned}
     &\Big\langle \overK(X)\left.\tfrac{d}{dt}\right|_{t=0}
     E_{2i-1}\big|_{X(t)}, E_{2i}\big|_X \Big\rangle
     = - \Big\langle \left.\tfrac{d}{dt}\right|_{t=0}
       E_{2i-1}\big|_{X(t)}, \overK(X)E_{2i}\big|_X \Big\rangle\\
      &\stackrel{\eqref{eq:def_underK},\eqref{eq:undermu_overmu},\eqref{eq:mu_frame_X}}{=}
        \sqrt{\overmu(X)}
        \Big\langle \left.\tfrac{d}{dt}\right|_{t=0}
        E_{2i-1}\big|_{X(t)}, E_{2i-1}\big|_X \Big\rangle
     = \tfrac{1}{2}\sqrt{\overmu(X)}
        \left.\tfrac{d}{dt}\right|_{t=0}
        \Big\langle E_{2i-1}\big|_{X(t)}, E_{2i-1}\big|_{X(t)}
        \Big\rangle = 0
    \end{aligned}
  \]
  and similarly
  \(\big\langle \overK(X)E_{2i-1}\big|_X,
  \left.\tfrac{d}{dt}\right|_{t=0}E_{2i}\big|_{X(t)} \big\rangle = 0\).
  Also \(\overK(X)\) is quadratic in \(X_\v\) by~\eqref{eq:def_overK}, so
  \begin{align*}
    \Big\langle \left.\tfrac{d}{dt}\right|_{t=0}\overK\big(X(t)\big)
    E_{2i-1}\big|_X, E_{2i}\big|_X \Big\rangle
    &\quad= \big\langle E_{2i-1}\big|_X \cbullet X_\z \cbullet S,
       E_{2i}\big|_X \cbullet X_\v \big\rangle 
    + \big\langle E_{2i-1}\big|_X \cbullet X_\z \cbullet X_\v,
       E_{2i}\big|_X \cbullet S \big\rangle\\
    &\stackrel{\eqref{eq:Clifford_module_1},\eqref{eq:Clifford_module_3}}{=}
       2 \big\langle E_{2i-1}\big|_X \cbullet E_{2i}\big|_X \cbullet X_\z
       \cbullet X_\v, S \big\rangle
  \end{align*}
  This yields
  \begin{equation}
   \label{eq:grad_2}
    \big\langle E_{2i-1}\big|_X \cbullet E_{2i}\big|_X \cbullet X_\z \cbullet X_\v,
    S \big\rangle
    = \frac{\norm{X_\z}\norm{X_\v}^2}{2}
      \big\langle \left.\nabla\!\sqrt{\undermu}\right|_X, 0 \oplus S
      \big\rangle
  \end{equation}
  for all \(S \in \v\) with \(\langle X_\v, S \rangle = 0\).
  Moreover, the branch \(\undermu\) is constant along the radial line
  \(t \mapsto X_\z \oplus tX_\v\) by~\eqref{eq:def_underK}, hence
  \(\langle \left.\nabla\!\sqrt{\undermu}\right|_X, X_\v \rangle = 0\).
  Together with~\eqref{eq:grad_1} and~\eqref{eq:grad_2} we obtain
  \eqref{eq:Pyth_1}.

  For \(\undermu \equiv 0\), the right-hand side of~\eqref{eq:Pyth_1}
  vanishes. Hence
  \(E_{2i-1}\big|_X \cbullet E_{2i}\big|_X \cbullet X_\z \cbullet X_\v
  = 0\), contradicting
  \begin{equation}\label{eq:Pyth_2}
    \norm{E_{2i-1}\big|_X\cbullet E_{2i}\big|_X \cbullet X_\z \cbullet
   X_\v}^2
    = \norm{X_\z}^2\norm{X_\v}^2
  \end{equation}
  Thus \(m \le 2\). For \(\undermu > 0\), combining~\eqref{eq:Pyth_2}
  with~\eqref{eq:Pyth_1} and applying the Pythagorean theorem gives
  \[
    \norm{X_\z}^2 \norm{X_\v}^2
      = \norm{X_\z}^2 \norm{X_\v}^2 \undermu(X)
         + \frac{\norm{X_\z}^2 \norm{X_\v}^4}{4}
           \norm{\left.\nabla\!\sqrt{\undermu}\right|_X}_\v^2
  \]
  which is exactly~\eqref{eq:grad_3}.
\end{proof}

\begin{corollary}\label{co:mu_=_constant}
  Global eigenvalues of \(-\underK^2\) other than \(0\) and \(1\) do
  not exist. Moreover, the multiplicity of \(\nu = 0\) is \(m_0 = 2\)
  or \(m_0 = 1\) on each connected component of
  \(\n \setminus \ram(\overK^2)\), depending on whether \(n\)
  is even (in which case \(m_0 = 2\)) or odd (then \(m_0 = 1\)).
\end{corollary}

\begin{proof}
  Let \(\nu\) be a global eigenvalue of \(-\underK^2\).
  If \(\nu > 0\), then \(0 = \nabla\!\sqrt{\nu}\), and together with
  \eqref{eq:grad_3} this forces \(\nu = 1\). Hence no global
  eigenvalues with \(0 < \nu < 1\) exist. For \(\nu = 0\),
  Lemma~\ref{le:var} gives \(m_0 \le 2\), while trivially \(m_0 \ge 1\)
  since \(\underK(X)X_\z = 0\). Thus \(m_0 \in \{1,2\}\) on
  \(\n \setminus \ram(\overK^2)\). Moreover, \(m_0 \equiv \dim \z
  \pmod{2}\) because \(\underK(X)\) is skew-symmetric. This completes
  the proof of the statement about \(m_0\).
\end{proof}

For the following definition, let \((0,1)\) denote the open unit
interval and recall that \(0 \le \undermu \le 1\) holds for every
eigenvalue branch \(\undermu\) of \(-\underK^2\).

\begin{definition}\label{de:U_simple}
  Let \(U_{(0,1)} \subseteq \n \setminus \ram(\overK^2)\) be defined by
  \begin{equation}\label{eq:U_simple}
    X \in U_{(0,1)}
    \quad\Longleftrightarrow\quad
    \forall \undermu \in \branch_{\mathrm{nc}}(U(X))\colon
    \undermu(X) < 1
  \end{equation}
\end{definition}

\begin{corollary}\label{co:U_sharp}
  The set \(U_{(0,1)}\) is open and dense in
  \(\n \setminus \ram(\overK^2)\). Moreover, for each
  \(X \in U_{(0,1)}\) and every
  \(\undermu \in \branch_{\mathrm{nc}}(U(X))\) one has
  \(0 < \undermu(X) < 1\) and
  \(\left.\nabla\undermu\right|_X \neq 0\). In particular, for each
  \(X \in U_{(0,1)}\) there is a natural identification
  \begin{equation}\label{eq:Bnc_ident}
    \branch_{\mathrm{nc}}(U(X)) \cong
    \Spec\bigl(-\underK(X)^2\bigr) \cap (0,1)
  \end{equation}
\end{corollary}

\begin{proof}
  Let \(U\) be a connected component of
  \(\n \setminus \ram(\overK^2)\). Eigenvalue branches of
  \(-\underK^2\) are real-analytic, and hence continuous, on \(U\).
  Moreover, \(0 \le \undermu \le 1\) by
  Proposition~\ref{p:overK}~(d). Thus the condition
  \(\undermu(X) < 1\) is open.

  Density follows from the identity theorem. Indeed, if
  \(\undermu\) is a nonconstant branch on \(U\), then
  \[
    \{X \in U \mid \undermu(X) = 1\}
  \]
  has empty interior. Hence its complement is dense in \(U\).
  Since there are only finitely many nonconstant branches on \(U\),
  their common complement is dense. This proves that
  \(U_{(0,1)}\) is open and dense in
  \(\n \setminus \ram(\overK^2)\).

  Let \(X \in U_{(0,1)}\) and
  \(\undermu \in \branch_{\mathrm{nc}}(U(X))\). By definition,
  \(\undermu(X) < 1\). Moreover, \(\undermu(X) > 0\); otherwise
  \(\undermu\) would meet the global eigenvalue \(0\), contradicting
  \(X \notin \ram(\overK^2)\). Hence
  \(0 < \undermu(X) < 1\).

  The identification~\eqref{eq:Bnc_ident} is obtained by evaluating
  branches at \(X\). The previous paragraph shows that every
  nonconstant branch gives an element of
  \(\Spec\bigl(-\underK(X)^2\bigr) \cap (0,1)\). Conversely, let
  \[
    \eta \in \Spec\bigl(-\underK(X)^2\bigr) \cap (0,1)
  \]
  Since \(X \notin \ram(\overK^2)\), the eigenvalue \(\eta\) extends
  uniquely to an eigenvalue branch
  \(\undermu\colon U(X) \to \bbR_{\ge0}\) of \(-\underK^2\) with
  \(\undermu(X) = \eta\). Since \(0 < \eta < 1\), this branch is
  nonconstant by Corollary~\ref{co:mu_=_constant}. This gives the
  claimed natural identification.

  It remains to prove that the gradient does not vanish. Suppose that
  \(X \in U_{(0,1)}\) and
  \(\left.\nabla \undermu\right|_X = 0\). Since
  \(0 < \undermu < 1\) on \(U_{(0,1)}\), the function
  \(\sqrt{\undermu}\) is smooth on \(U_{(0,1)}\), and the chain rule
  gives
  \[
    \nabla \undermu
    =
    2\sqrt{\undermu}\,\nabla\!\sqrt{\undermu}
  \]
  Evaluating at \(X\) and using \(\sqrt{\undermu(X)} > 0\), we obtain
  \[
    \left.\nabla\!\sqrt{\undermu}\right|_X = 0
  \]
  Hence \(\undermu(X) = 1\) by~\eqref{eq:grad_3}, contradicting
  \(\undermu(X) < 1\).
\end{proof}

\subsection{Eigenvalue branches of higher multiplicity}
Continuing with the ideas of Lemma \ref{le:var}, we examine the case
of an eigenvalue branch of multiplicity at least four.

\begin{proposition}\label{p:m_geq_4}
  \begin{enumerate}
    \item Let \(X \in \n \setminus \ram(\overK^2)\). Suppose there is
      an eigenvalue branch \(\undermu\) of multiplicity \(m \ge 4\)
      defined on the connected component \(U(X) \subseteq \n \setminus
      \ram(\overK^2)\) containing \(X\). Then there exists an
      orthonormal system \(\{E_1,\ldots,E_m\}\) of \(\h_X\) with \(m\)
      even such that
      \begin{equation}\label{eq:var_1}
        E_{2i-1}\cbullet E_{2i}\cbullet X_\v
        = E_{2j-1}\cbullet E_{2j}\cbullet X_\v
      \end{equation}
      for all \(1\le i < j\le \tfrac{m}{2}\).
    \item Moreover,
      \begin{equation}\label{eq:var_2}
        \dim\{S \in \v \mid E_1 \cbullet E_2 \cbullet S
          = E_3 \cbullet E_4 \cbullet S\}
        = \frac{1}{2} \dim \v
      \end{equation}
      for every orthonormal system \(\{E_1, E_2, E_3, E_4\}\) of \(\z\).
  \end{enumerate}
\end{proposition}

\begin{proof}
  For the first assertion, since \(m \ge 4\) we have \(\undermu > 0\) by
  Corollary~\ref{co:mu_=_constant}. Thus \(m\) is even and there exists
  an orthonormal system \(\{E_1,\ldots,E_m\}\subseteq \h_X\) satisfying
  \eqref{eq:mu_frame_X}. Applying \eqref{eq:Pyth_1} twice yields
  \begin{equation*}
    E_{2i-1}|_X \cbullet E_{2i}|_X \cbullet X_\z \cbullet X_\v
    \stackrel{\eqref{eq:Pyth_1}}{=}
      \norm{X_\z}\sqrt{\undermu(X)}\,X_\v
      \oplus \frac{\norm{X_\z}\norm{X_\v}^2}{2}
             \bigl(\left.\nabla\!\sqrt{\undermu}\right|_X\bigr)_\v
    \stackrel{\eqref{eq:Pyth_1}}{=}
    E_{2j-1}|_X \cbullet E_{2j}|_X \cbullet X_\z \cbullet X_\v
  \end{equation*}
  Left-multiplying by \(X_\z\) and using~\eqref{eq:Clifford_module_3}
  twice (and \(E_r \perp X_\z\)) gives~\eqref{eq:var_1}.

  For the second assertion, let
  \(\z' \coloneqq \mathrm{span}_\bbR\{E_1, E_2, E_3, E_4\} \subseteq \z\).
  Then \(\Cl(\z')\) is a subalgebra of \(\Cl(\z)\). Hence \(\v\) is an
  orthogonal \(\Cl(\z')\)-module as well, and we may therefore assume
  \(\z = \bbR^4\). In the Euclidean real case one has
  \(\Cl(4) \cong \Mat(2,\bbH)\), so there is a unique irreducible orthogonal
  \(\Cl(4)\)-module, namely \(\bbH^2\) of real dimension \(8\). Thus \(\v\) is 
  isomorphic as an orthogonal \(\Cl(\z')\)-module to an orthogonal direct sum of
  isomorphic copies of \(\bbH^2\), and it suffices to consider the irreducible
  orthogonal \(\Cl(4)\)-module \(\v \coloneqq \bbH^2\).

  Here, the volume element is
  \(\omega \coloneqq E_1 \cbullet E_2 \cbullet E_3 \cbullet E_4\).
    Then \(\omega^2 = \Id\), and there is an orthogonal splitting
  \(\v = \v_+ \oplus \v_-\) into two nonisomorphic
  \(4\)-dimensional \(\Cl^0(4)\cong \Cl(3)\)-submodules, namely the
  eigenspaces of \(\omega\):
  \(\omega|_{\v_\pm} = \pm \Id\)

  Left-multiplying by \(E_1 \cbullet E_2\) and using
\((E_1 \cbullet E_2)^2 = -\Id\) and
\((E_1 \cbullet E_2)(E_3 \cbullet E_4)=\omega\), we obtain
\[
  E_1 \cbullet E_2 \cbullet S
          = E_3 \cbullet E_4 \cbullet S
  \quad\Longleftrightarrow\quad
  \omega S = -S
\]
Thus the subspace in \eqref{eq:var_2} coincides with \(\v_-\), which has
dimension \(\tfrac12 \dim \v\).
\end{proof}

For every orthonormal system \(\{E_1, \ldots, E_m\}\) of \(\z\) with
\(m\) even, the operators \(J_i \coloneqq E_{2i-1} \cbullet E_{2i}
\cbullet\) and \(J_j \coloneqq E_{2j-1} \cbullet E_{2j} \cbullet\) are
commuting orthogonal complex structures on \(\v\) for
\(1 \le i < j \le \tfrac{m}{2}\) according to
\eqref{eq:Clifford_module_3}. Moreover, each \(J_i\) lies in the Lie
algebra of \(\Spin(\z)\) by~\cite[Ch.~1, Prop.~6.1]{LM}, so
\begin{equation}\label{eq:torus}
  \t \coloneqq \mathrm{span}_\bbR\{J_i - J_j \mid 1 \le i < j \le \tfrac{m}{2}\}
\end{equation}
is an Abelian subalgebra of \(\mathfrak{spin}(\z)\); i.e., \(\exp(\t)\)
is a torus of \(\Spin(\z)\). Under the canonical isomorphism
\(\mathfrak{spin}(\z) \cong \so(\z)\), the operator \(J_i\) corresponds
to \(2 E_{2i-1} \wedge E_{2i}\)~\cite[Ch.~1, Prop.~6.2]{LM}, yielding
\(\dim \t = \frac{m}{2} - 1\).

Let \(Z_0\) be a unit vector in \(\z\) and set \(\z' \coloneqq Z_0^\perp\),
the orthogonal complement of \(Z_0\) in \(\z\). For \(S \in \v\), define
\begin{equation}\label{eq:fixed_point_group}
  \Fix(S, \Spin(\z')) \coloneqq \{F \in \Spin(\z') \mid F S = S\}
\end{equation}
referred to as the fixed point group of \(S\) in \(\Spin(\z')\).

\begin{corollary}\label{co:fixed_point_group}
  Suppose that there exists an eigenvalue branch \(\undermu\) of
  \(-\underK^2\) of multiplicity \(m \ge 4\), defined on a connected
  component \(U\) of \(\n \setminus \ram(\overK^2)\). Let
  \(Z_0 \in \z\) be a unit vector and set \(\z' \coloneqq Z_0^\perp\).
  Consider the slice
  \begin{equation}\label{eq:U_sup_v_slice}
    U^{\v} \coloneqq \{ S \in \v \mid Z_0\oplus S \in U \}
  \end{equation}
  Then \(U^{\v}\) is a nonempty open subset of \(\v\), and for every
  \(S \in U^{\v}\) the fixed-point set \(\Fix(S, \Spin(\z'))\) has rank
  at least \(\frac{m}{2} - 1\).
\end{corollary}

\begin{proof}
  Openness of \(U^{\v}\) is immediate from the product topology.
  The set \(\n \setminus \ram(\overK^2)\) is invariant under the
  \(\Spin(\z)\)-action
  \(X_\z \oplus X_\v \mapsto \rho(F^\ext)X_\z \oplus F^\ext X_\v\)
  and under rescalings by positive factors in each component. Since
  \(\Spin(\z)\) and \((\bbR_{>0})^2\) are connected, they preserve the
  connected component \(U\). Hence there exists \(X = X_\z \oplus X_\v \in U\)
  with \(\norm{X_\z}=1\) and \(F^\ext \in \Spin(\z)\) such that
  \(\rho(F^\ext)X_\z = Z_0\). Then \(Z_0 \oplus F^\ext X_\v \in U\),
  proving \(U^{\v}\neq\emptyset\).

  For \(S \in U^{\v}\), set \(X \coloneqq Z_0 \oplus S\), so \(X \in U\).
  By Proposition~\ref{p:m_geq_4}, there is an orthonormal system
  \(\{E_1,\ldots,E_m\}\subseteq \z'=\h_X\) such that
  \(J_i S = J_j S\) for all \(1\le i<j\le \tfrac{m}{2}\), where
  \(J_i \coloneqq E_{2i-1}\cbullet E_{2i}\cbullet\).
  Hence \((J_i-J_j)S=0\) for all \(i<j\), and therefore every element of
  the Abelian subalgebra \(\t\) in~\eqref{eq:torus} annihilates \(S\).
  Consequently, \(\exp(\t)\subseteq \Fix(S,\Spin(\z'))\).
  Since \(\dim\t = \tfrac{m}{2}-1\), the fixed-point group has rank at least
  \(\tfrac{m}{2}-1\).
\end{proof}

The above corollary will be helpful for small \(n = \dim \z\), where
fixed-point groups of spinors are well understood. By contrast, the
following corollary is useful when \(n\) is large, in particular for
\(n \ge 10\).

\begin{corollary}\label{co:m_geq_4}
  Let \(\z\) be a Euclidean vector space of
  dimension \(n \ge 6\) and \(\v\) be an orthogonal \(\Cl(\z)\)-module.
  Suppose that
  \begin{equation}\label{eq:dimensions_abschaetzung_1}
     \dim \v > 8n - 28
  \end{equation}
  holds. Then all the nonzero eigenvalue branches of \(-\underK^2\)
  have multiplicity \(2\).
\end{corollary}

\begin{proof}
  We prove the inverse implication, namely, that the existence of an
  eigenvalue branch of \(-\underK^2\) of multiplicity at least \(4\)
  implies
  \begin{equation}\label{eq:dimensions_abschaetzung_2}
    \dim \v \;\le\;  8n - 28
  \end{equation}
  Choose a fixed orthonormal system \(\{E_0, \ldots, E_4\} \subseteq
  \z\), set
  \[
    \z' \coloneqq E_0^\perp \subseteq \z \quad\text{and}\quad
    \tilde{\z} \coloneqq \{E_0, \ldots, E_4\}^\perp \subseteq \z'.
  \]
  Then \(\tilde{\z} \subseteq \z' \subseteq \z\) is a chain of
  subspaces, giving inclusions \(\SO(\tilde{\z}) \subseteq \SO(\z')
  \subseteq \SO(\z)\). Passing to spin coverings yields
  \(\Spin(\tilde{\z}) \subseteq \Spin(\z') \subseteq \Spin(\z)\).
  Furthermore, \(J_1 \coloneqq E_1 \cbullet E_2 \cbullet\) and
  \(J_2 \coloneqq E_3 \cbullet E_4 \cbullet\) are orthogonal complex
  structures on \(\v\) such that, by Proposition~\ref{p:m_geq_4},
  \begin{equation}\label{eq:var_3}
    \Kern(J_1 - J_2) \coloneqq \{S \in \v \mid J_1 S = J_2 S\}
  \end{equation}
  has dimension \(\tfrac12 \dim \v\).

  We claim that \(\Kern(J_1 - J_2)\) is invariant under the action of
  \(\Spin(\tilde{\z})\): recall that Clifford multiplication is
  \(\Spin(\z)\)-equivariant, with \(\Spin(\z)\) acting on \(\z\) by
  the vector representation and on \(\v\) by the spin representation.
  If \(S \in \Kern(J_1 - J_2)\), then
  \begin{equation}\label{eq:var_4}
    E_1 \cbullet E_2 \cbullet S = E_3 \cbullet E_4 \cbullet S.
  \end{equation}
  For \(F \in \SO(\tilde{\z})\), we have \(F E_i = E_i\) for
  \(i = 1,\ldots,4\), hence any lift \(F^\ext\) of \(F\) to the spin
  covering satisfies
  \[
    E_1 \cbullet E_2 \cbullet F^\ext S
      = F^\ext(E_1 \cbullet E_2 \cbullet S)
      \stackrel{\eqref{eq:var_4}}{=}
      F^\ext(E_3 \cbullet E_4 \cbullet S)
      = E_3 \cbullet E_4 \cbullet F^\ext S.
  \]
  This shows \(F^\ext S \in \Kern(J_1 - J_2)\) and hence that
  \(\Kern(J_1 - J_2)\) is invariant under the action of
  \(\Spin(\tilde{\z})\).

  Consider also the map
  \begin{equation}\label{eq:surjective_map}
    \Spin(\z') \times \Kern(J_1 - J_2) \;\longrightarrow\; \v,\qquad
    (F^\ext, S) \longmapsto F^\ext S.
  \end{equation}
  Now suppose that some nonzero eigenvalue branch of \(-\underK^2\)
  has multiplicity at least \(4\) on a connected component \(U \subseteq
  \n \setminus \ram(\overK^2)\). We claim that the image
  of~\eqref{eq:surjective_map} contains the subset \(U^{\v} \subseteq
  \v\) defined in~\eqref{eq:U_sup_v_slice}.

  Let \(S \in U^{\v}\) and set \(X \coloneqq E_0 \oplus S \in U\).
  By Proposition~\ref{p:m_geq_4}, there exists an orthonormal system
  \(\{E_1', E_2', E_3', E_4'\} \subseteq \z'\) such that
  \begin{equation}\label{eq:var_5}
    E_1' \cbullet E_2' \cbullet S = E_3' \cbullet E_4' \cbullet S.
  \end{equation}
  Since \(\dim \z' = n - 1 \ge 5\), \(\SO(\z')\) acts transitively on
  orthonormal \(4\)-frames in \(\z'\). Hence there exists
  \(F \in \SO(\z')\) with \(F E_i' = E_i\) for
  \(i = 1,\ldots,4\). Then, for any such \(F\) and a lift \(F^\ext \in
  \Spin(\z')\) to the spin covering,
  \begin{align*}
    E_1 \cbullet E_2 \cbullet F^\ext S
      &= F E_1' \cbullet F E_2' \cbullet F^\ext S \\
      &= F^\ext(E_1' \cbullet E_2' \cbullet S) \\
      &\stackrel{\eqref{eq:var_5}}{=}
        F^\ext(E_3' \cbullet E_4' \cbullet S) \\
      &= E_3 \cbullet E_4 \cbullet F^\ext S.
  \end{align*}
  This shows that \(\tilde{S} \coloneqq F^\ext S\) lies in
  \(\Kern(J_1 - J_2)\). Since we can conversely write
  \(S = (F^\ext)^{-1} \tilde{S}\), we conclude that \(S\) belongs to
  the image of~\eqref{eq:surjective_map}.

  Moreover, since \(\Kern(J_1 - J_2)\) is \(\Spin(\tilde{\z})\)–
  invariant, the group \(\Spin(\tilde{\z})\) acts from the right on
  \(\Spin(\z') \times \Kern(J_1 - J_2)\) by
  \begin{equation}\label{eq:action_1}
    (F^\ext, S)H \coloneqq (F^\ext \circ H, H^{-1} S), \quad H \in \Spin(\tilde{\z}).
  \end{equation}
  Because \((F^\ext \circ H)(H^{-1} S) = F^\ext S\),  the map~\eqref{eq:surjective_map}
  descends to the quotient space
  \begin{equation}\label{eq:quotient_space}
    \bigl(\Spin(\z') \times \Kern(J_1 - J_2)\bigr) / \Spin(\tilde{\z}).
  \end{equation}

  Finally, combining the facts that~\eqref{eq:action_1} is a free
  action, that \(\dim U^{\v} = \dim \v\), and that the latter is twice
  the dimension of \(\Kern(J_1 - J_2)\) from~\eqref{eq:var_3}, the
  main corollary to Sard's theorem implies
  \[
   \frac{\dim \v}{2} \;\le\;  \dim(\so(n - 1)) - \dim(\so(n - 5))
  \]
  Moreover,
  \begin{align*}
    \dim(\so(n - 1)) - \dim(\so(n - 5))
      &= \frac{(n - 1)(n - 2)}{2} - \frac{(n - 5)(n - 6)}{2} \\
      &= \frac{(11 - 3)n - (30 - 2)}{2} \\
      &= 4n - 14.
  \end{align*}
  Multiplying by \(2\) yields \(\dim \v \le  8 n - 28\),
  i.e.~\eqref{eq:dimensions_abschaetzung_2}.
\end{proof}

Using the classification of irreducible Clifford modules (in
particular, the eightfold periodicity of Clifford algebras and their
representations), it is straightforward that
\eqref{eq:dimensions_abschaetzung_2} is automatically true for all
\(n \ge 10\); see Proposition~\ref{p:mult_ge_10} of
Section~\ref{se:n_ge_10}.

\section{\texorpdfstring{Classification of the eigenvalue branches of \(-\underK^2\)}%
  {Classification of the eigenvalue branches of (-K^2)}}\label{se:eigenvalue_branches}

In this section we establish the classification of the number and type of
eigenvalue branches of the family \(-\underK^2\) of nonnegative
self-adjoint operators defined in~\eqref{eq:underK_square}, associated
with a pair \((\z, \v)\) where \(\z\) is an \(n\)-dimensional Euclidean space with \(n\ge 1\)
and \(\v\) is a nontrivial orthogonal \(\Cl(\z)\)-module. This classification leads to
Theorem~\ref{th:main_2}. 

By Corollary~\ref{co:mu_=_constant}, the global eigenvalue \(0\) 
has multiplicity \(m_0 \in \{1,2\}\) with  \(m_0 \equiv n \pmod{2}\),
and the only other possibility for a global eigenvalue is \(1\).
Hence it remains to determine the nonzero eigenvalue branches and to
check whether or not \(1\) is a global eigenvalue. 

In case \(n \in \{1,2\}\) there is only the global eigenvalue \(0\). 
Thus we may start our analysis with \(n =3\).

\subsection{\texorpdfstring{Eigenvalue branches of \(-\underK^2\) for
\(n = 3\)}{Eigenvalue branches of (-K^2) for dim z = 3}}%
\label{se:n_=_3}

The \(4\)-dimensional Euclidean space \(\bbR^4\) carries the structure of
the associative, normed division algebra \(\bbH\) of quaternions. For
\(S, \tilde{S} \in \Im(\bbH)\), the product decomposes as
\[
  S \tilde{S}
  = -\langle S, \tilde{S} \rangle 1_{\bbH} + \Im_{\bbH}(S \tilde{S})
\]
and the imaginary part defines the standard vector cross product in
dimension \(3\):
\begin{equation}\label{eq:quaternionic_vector_cross_product}
  \Im(\bbH) \times \Im(\bbH) \to \Im(\bbH),\quad
  (S, \tilde{S}) \longmapsto
  S \times \tilde{S} \coloneqq \Im_{\bbH}(S \tilde{S})
\end{equation}

\begin{proposition}\label{p:n_=_3}
  Suppose that \(\dim \z = 3\), and let \(\v\) be a nontrivial
  orthogonal \(\Cl(\z)\)-module. The nonzero eigenvalue branches of the
  family \(-\underK^2\) of nonnegative self-adjoint operators defined
  in~\eqref{eq:underK_square} and associated with the pair \((\z,\v)\)
  are as follows:
  \begin{itemize}
  \item If \(\v\) is an isotypic \(\Cl(\z)\)-module, then the
    \(J^2\)-condition holds. In this case, \(\nu \coloneqq 1\) is a
    global eigenvalue of \(-\underK^2\) of multiplicity \(m_\nu = 2\).
  \item Otherwise, there is a nonconstant eigenvalue branch
    \(\undermu\) of multiplicity \(m_{\undermu} = 2\), described
    by~\eqref{eq:under_mu_for_n_=_3}.
  \end{itemize}
\end{proposition}

\begin{proof}
Set \(\z \coloneqq \Im(\bbH)\). For each nontrivial orthogonal \(\Cl(\z)\)-module
\(\v\), we write
\[
  \overK_{\v} \coloneqq \overK_{(\z,\v)},
  \qquad
  \underK_{\v} \coloneqq \underK_{(\z,\v)}
\]
for the operator families from \eqref{eq:def_overK} and
\eqref{eq:def_underK}.

Assume first that \(\v \in \{\bbH,\overline{\bbH}\}\), where \(\bbH\) and
\(\overline{\bbH}\) denote the two irreducible \(\Cl(\z)\)-modules
(cf.\ Section~\ref{se:explicit_formulas}). Using the Clifford
multiplication formulas \eqref{eq:Clifford_multiplication_for_n_=_3_+}
and \eqref{eq:Clifford_multiplication_for_n_=_3_-}, together with
\eqref{eq:def_K_2} and associativity of \(\bbH\), we obtain
\begin{subequations}\label{eq:overK_quaternions}
\begin{align}
  \overK_{\bbH}(S_0 \oplus S)\tilde S
  &= \Im_{\bbH}\bigl(S^* S\, S_0 \tilde S\bigr)
    = \norm{S}^2\,(S_0 \times \tilde S)
  \label{eq:overK_quaternions_+}\\
  \overK_{\overline{\bbH}}(S_0 \oplus S)\tilde S
  &= -\,\Im_{\bbH}\bigl(S^* S\, S_0 \tilde S\bigr)
   = -\,\norm{S}^2\,(S_0 \times \tilde S)
  \label{eq:overK_quaternions_-}
\end{align}
\end{subequations}
for all \(S_0,\tilde S \in \Im(\bbH)\) and \(S \in \bbH\), where
\(S_0 \times \tilde S \coloneqq \Im_{\bbH}(S_0\tilde S)\) is the vector
cross product from \eqref{eq:quaternionic_vector_cross_product}.

Now let \(\v=\v_+\oplus\v_-\) be the decomposition into isotypic
components, where \(\v_+\) is a direct sum of copies of \(\bbH\) and
\(\v_-\) is a direct sum of copies of \(\overline{\bbH}\). For
\(S_0 \oplus S_+ \oplus S_- \in \Im(\bbH)\oplus\v_+\oplus\v_-\),
Lemma~\ref{le:reducible_Clifford_module}~(a) gives
\begin{align*}
  \overK_{\v}(S_0 \oplus S_+ \oplus S_-)
  &= \overK_{\v_+}(S_0 \oplus S_+) + \overK_{\v_-}(S_0 \oplus S_-) \\
  &\stackrel{\eqref{eq:overK_quaternions}}{=}
     \bigl(\norm{S_+}^2 - \norm{S_-}^2\bigr)\, S_0 \times \Box
     \colon \Im(\bbH) \to \Im(\bbH)
\end{align*}
In particular, for \(S_0 \neq 0\), 
\[
  \underK_{\v}(S_0 \oplus S_+ \oplus S_-)
  =
  \Bigl(\frac{\norm{S_+}^2 - \norm{S_-}^2}{\norm{S_+}^2 + \norm{S_-}^2}\Bigr)
  \frac{S_0}{\norm{S_0}} \times \Box
  \colon \Im(\bbH) \to \Im(\bbH)
\]
The map \(S \mapsto \frac{S_0}{\norm{S_0}} \times S\) is an orthogonal complex
structure on \(S_0^\perp\), hence \(-\underK_{\v}^2\) has the eigenvalue
\[
  \undermu(S_0 \oplus S_+ \oplus S_-)
  \coloneqq
  \frac{\bigl(\norm{S_+}^2 - \norm{S_-}^2\bigr)^2}
       {\bigl(\norm{S_+}^2 + \norm{S_-}^2\bigr)^2}
\]
with multiplicity \(m_{\undermu}=2\). Moreover, \(\undermu \equiv 1\) if and
only if \(\v_+=\{0\}\) or \(\v_-=\{0\}\), i.e., if and only if \(\v\) is
isotypic. In that case \(\nu = 1\) is a global eigenvalue of \(-\underK^2\)
with multiplicity \(m_\nu = 2\). Otherwise, \(\undermu\) is nonconstant and
this yields \eqref{eq:under_mu_for_n_=_3}.
\end{proof}

\subsection{\texorpdfstring{Eigenvalue branches of \(-\underK^2\) for
\(4 \le n \le 9\)}{Eigenvalue branches of (-K^2) for
4 ≤ n ≤ 9}}%
\label{se:4_leq_n_leq_9}

Recall that \(\bbR^8\) carries the structure of the \(8\)-dimensional
normed division algebra \(\bbO\) of octonions; see~\cite{Ba}. As with
multiplication in \(\bbH\), the imaginary part of the octonionic product
\begin{equation}\label{eq:octonionic_product}
  S\tilde{S}
  = -\langle S, \tilde{S} \rangle 1_{\bbO} + \Im_{\bbO}(S\tilde{S})
\end{equation}
of two purely imaginary octonions \(S\) and \(\tilde{S}\) defines the
vector cross product in dimension \(7\); cf.\ also~\cite[Ch.~2]{Gr}:
\begin{equation}\label{eq:octonionic_vector_cross_product}
  \Im(\bbO) \times \Im(\bbO) \to \Im(\bbO),\quad
  (S, \tilde{S}) \longmapsto
  S \times \tilde{S} \coloneqq \Im_{\bbO}(S \tilde{S})
\end{equation}

\begin{proposition}\label{p:n_=_7}
  Suppose that \(\dim \z = 7\), and let \(\v\) be a nontrivial
  orthogonal \(\Cl(\z)\)-module. The nonzero eigenvalue branches of the
  family \(-\underK^2\) of nonnegative self-adjoint operators defined in
  \eqref{eq:underK_square} and associated with the pair \((\z,\v)\) are
  as follows:
  \begin{itemize}
  \item If \(\v\) is an irreducible \(\Cl(\z)\)-module, then the
    \(J^2\)-condition holds. In this case, \(\nu \coloneqq 1\) is a
    global eigenvalue of \(-\underK^2\) of multiplicity \(m_\nu = 6\).

  \item If \(\v\) splits as the orthogonal direct sum
    \(\v_1 \oplus \v_2\) of two isomorphic irreducible
    \(\Cl(\z)\)-submodules, then \(\nu \coloneqq 1\) is a global
    eigenvalue of \(-\underK^2\) of multiplicity \(m_\nu = 2\).
    In addition, there is a nonconstant eigenvalue branch \(\undermu\)
    of multiplicity \(m_{\undermu} = 4\), obtained by rescaling the
    polynomial function defined in~\eqref{eq:over_mu_for_n_=_7_1}.

  \item If \(\v\) splits as the orthogonal direct sum
    \(\v_1 \oplus \v_2\) of two nonisomorphic irreducible
    \(\Cl(\z)\)-submodules, then there are two nonconstant eigenvalue
    branches \(\undermu_1\) and \(\undermu_2\), of multiplicities
    \(m_{\undermu_1} = 2\) and \(m_{\undermu_2} = 4\), respectively,
    obtained by rescaling the polynomial functions defined
    in~\eqref{eq:over_mu_for_n_=_7_2}
    and~\eqref{eq:over_mu_for_n_=_7_3}.

  \item In all other cases, \(\v\) splits as the orthogonal direct sum
    of at least three irreducible \(\Cl(\z)\)-submodules, and there are
    three distinct nonconstant eigenvalue branches, each of multiplicity
    \(2\).
  \end{itemize}
\end{proposition}

\begin{proof}
In the following, we fix \(\z \coloneqq \Im(\bbO)\) and let \(\v\)
be a nontrivial orthogonal \(\Cl(\z)\)-module. For brevity, set
\[
  \overK_{\v} \coloneqq \overK_{(\z,\v)},
  \qquad
  \underK_{\v} \coloneqq \underK_{(\z,\v)}
\]
as in \eqref{eq:def_overK} and \eqref{eq:def_underK}.

Assume first that \(\v \in \{\bbO,\overline{\bbO}\}\), where \(\bbO\) and
\(\overline{\bbO}\) denote the two irreducible \(\Cl(\z)\)-modules
(cf.\ Section~\ref{se:explicit_formulas}). In particular, we
write \(\overK_{\bbO}\coloneqq \overK_{(\z,\bbO)}\) and
\(\underK_{\bbO}\coloneqq \underK_{(\z,\bbO)}\). Then
\begin{align}\label{eq:overK_bbO}
  \overK_{\bbO}(S_0 \oplus S_1)\tilde{S}
  &\stackrel{\eqref{eq:Clifford_multiplication_for_n_=_7_1},
             \eqref{eq:def_K_2}}{=}
  \Im_{\bbO}\bigl(S_1^*\bigl((S_1 S_0)\tilde{S}\bigr)\bigr)
\end{align}
for all \(S_0 \oplus S_1 \in \z \oplus \bbO\) and \(\tilde{S} \in \z\).
If we replace \(\bbO\) by \(\overline{\bbO}\), then the above formulas
change sign by \eqref{eq:Clifford_multiplication_for_n_=_7_2}, and hence
\(\overK_{\overline{\bbO}} = -\,\overK_{\bbO}\).

By \(\Spin(7)\)-equivariance of \(\underK_{\bbO}\) (cf.\
\eqref{eq:K_Aut_equivariance}), the spectrum of \(-\underK_{\bbO}(X)^2\)
is constant along \(\Spin(7)\)-orbits. Since \(\Spin(7)\) acts
transitively on \(\rmS^6 \times \rmS^7 \subseteq \Im(\bbO)\times\bbO\),
it follows that the eigenvalue branches of \(-\underK_{\bbO}^2\) are
constant. By Corollary~\ref{co:mu_=_constant}, \(-\underK_{\bbO}^2\) has
the global eigenvalue \(\nu = 1\) with multiplicity \(m_\nu = 6\).

This is also clear directly from \eqref{eq:overK_bbO}. For \(S_0 \neq 0\),
\begin{equation}\label{eq:K_=_J_1}
  \underK_{\bbO}(S_0 \oplus 1_{\bbO})
  \stackrel{\eqref{eq:octonionic_vector_cross_product},
            \eqref{eq:overK_bbO}}{=}
  \frac{S_0}{\norm{S_0}} \times \Box
  \colon \Im(\bbO) \longrightarrow \Im(\bbO)
\end{equation}
defines an orthogonal complex structure on \(S_0^\perp \subseteq
\Im(\bbO)\). Hence \(-\underK_{\bbO}(S_0 \oplus 1_{\bbO})^2\bigl|_{S_0^\perp}=\Id\)
has the single eigenvalue \(1\) of multiplicity \(6\). Since
\(\underK_{\overline{\bbO}} = -\,\underK_{\bbO}\), the same conclusion
holds for \(-\underK_{\overline{\bbO}}^2\). This clarifies the
irreducible case.

Moreover, with these preliminaries in place, we handle the case
\(\v = \v_1 \oplus \v_2\), where \(\v_1\) and \(\v_2\) are irreducible
\(8\)-dimensional \(\Cl(\Im(\bbO))\)-modules. Fix the standard models
\(\bbO\) and \(\overline{\bbO}\), and identify each \(\v_i\) with one of
these. Put
\[
  \varepsilon_i \coloneqq
  \begin{cases}
    +1, & \v_i = \bbO,\\
    -1, & \v_i = \overline{\bbO}
  \end{cases}
\]
so that \(\overK_{\v_i} = \varepsilon_i\,\overK_{\bbO}\) and
\(\underK_{\v_i} = \varepsilon_i\,\underK_{\bbO}\).

By Lemma~\ref{le:reducible_Clifford_module}~(a), we have
\[
  \overK_{\v}(S_0 \oplus S_1 \oplus S_2)
  = \varepsilon_1\,\overK_{\bbO}(S_0 \oplus S_1)
    + \varepsilon_2\,\overK_{\bbO}(S_0 \oplus S_2)
\]
for all \(S_0 \oplus S_1 \oplus S_2 \in \z \oplus \bbO^2\).

For the rescaled family, Part~(b) of Lemma~\ref{le:reducible_Clifford_module}
yields, whenever \((S_1,S_2)\neq (0,0)\),
\begin{equation}\label{eq:K_6_7_eps}
  \underK_{\v}(S_0 \oplus S_1 \oplus S_2)
  =
  \frac{\norm{S_1}^2\,\varepsilon_1\,\underK_{\bbO}(S_0 \oplus S_1)
        + \norm{S_2}^2\,\varepsilon_2\,\underK_{\bbO}(S_0 \oplus S_2)}
       {\norm{S_1}^2 + \norm{S_2}^2}
\end{equation}

For simplicity assume \(S_0 = E_7\), so that \(S_0^\perp \cong \bbR^6\).
Assume \(S_1,S_2 \neq 0\) (otherwise we are in the irreducible case).
Since \(\Spin(6)\) acts transitively on the unit sphere in the spin
representation, there exists \(F \in \Spin(6)\) such that
\(S_1\) is a multiple of \(F^{\mathrm{ext}} S_2\). By
\(\Spin(6)\)-equivariance of \(\underK_{\bbO}\) (cf.\ \eqref{eq:K_Aut_equivariance}),
it follows that
\[
  \underK_{\bbO}(E_7 \oplus S_1)|_{\bbR^6}
  = F \circ \underK_{\bbO}(E_7 \oplus S_2)|_{\bbR^6} \circ F^{-1}
\]
Hence \(J_1 \coloneqq \underK_{\bbO}(E_7 \oplus S_1)|_{\bbR^6}\) and
\(J_2 \coloneqq \underK_{\bbO}(E_7 \oplus S_2)|_{\bbR^6}\) define equally
oriented orthogonal complex structures on \(\bbR^6\).

By the topological argument given after Proposition~2.1 in~\cite{FGM}, any two
equally oriented orthogonal complex structures \(J_1\) and \(J_2\) on \(\bbR^6\)
share a common complex line. That is, there exists a \(2\)-dimensional subspace
\(W \subseteq \bbR^6\) such that \(J_i(W) \subseteq W\) for \(i = 1,2\) and
\(J_1|_{W} = J_2|_{W}\). Hence, on \(W\),
\[
  \underK_{\v}(E_7 \oplus S_1 \oplus S_2)\bigl|_{W}
  =
  \begin{cases}
    \pm\,J_1|_{W}, & \text{if } \v_1 \cong \v_2,\\[4pt]
    \pm\,\dfrac{\norm{S_1}^2 - \norm{S_2}^2}{\norm{S_1}^2 + \norm{S_2}^2}\,J_1|_{W},
      & \text{otherwise,}
  \end{cases}
\]
and therefore
\begin{align*}
  \nu(X) &\coloneqq 1 && \text{if } \v_1 \cong \v_2,\\
  \undermu_1(X) &\coloneqq
    \frac{(\norm{S_1}^2 - \norm{S_2}^2)^2}{(\norm{S_1}^2 + \norm{S_2}^2)^2}
    && \text{otherwise}
\end{align*}
defines, in each case, an eigenvalue branch of \(-\underK_{\v}^2\) of multiplicity
at least \(2\). In the case \(\v_1 = \bbO\) and \(\v_2 = \overline{\bbO}\), this yields
\eqref{eq:over_mu_for_n_=_7_2}. Note the similarity with the nonconstant eigenvalue
branch~\eqref{eq:under_mu_for_n_=_3} in the case \(n = 3\).

In order to show that there is a second nonconstant eigenvalue branch \(\undermu_2\)
of multiplicity \(4\), note that \(J_1|_{W^\perp}\) and \(J_2|_{W^\perp}\) induce the
same complex orientation on \(W^\perp\). The space of orientation-preserving
orthogonal complex structures on \(W^\perp \cong \bbR^4\) is isomorphic to
\(\SO(4)/\U(2)\), and hence has dimension \(2\). Therefore it can be identified with
the \(2\)-sphere of purely imaginary unit quaternions acting by left multiplication
as orientation-preserving orthogonal complex structures on \(\bbR^4 \cong \bbH\)
(see, for example,~\cite{GM}). Consequently, there exist purely imaginary unit
quaternions \(w_1\) and \(w_2\) such that
\[
  J_1|_{W^\perp} = \rmL_{w_1},
  \qquad
  J_2|_{W^\perp} = \rmL_{w_2},
\]
where \(\rmL_{w}\) denotes left multiplication by \(w\). Then the restriction of
\(\norm{S_1}^2 J_1 \pm \norm{S_2}^2 J_2\) to \(W^\perp\) is given by left multiplication
with \(\norm{S_1}^2 w_1 \pm \norm{S_2}^2 w_2\), which is again a purely imaginary
quaternion. Hence \(-\underK_{\v}(X)^2\) has a second eigenvalue \(\undermu_2(X)\)
of multiplicity \(m_{\undermu_2} = 4\), characterized by
\begin{equation}\label{eq:under_mu_for_n_=_7_4}
  4\,\undermu_2(X)
  = -\trace\bigl(\underK_{\v}(X)^2|_{W^\perp}\bigr)
  =
  \begin{cases}
    -\trace\bigl(\underK_{\v}(X)^2\bigr) - 2, &
      \text{if } \v_1 \cong \v_2,\\[4pt]
    -\trace\bigl(\underK_{\v}(X)^2\bigr) - 2\,\undermu_1(X), &
      \text{otherwise.}
  \end{cases}
\end{equation}
Moreover,
\begin{equation}\label{eq:under_mu_for_n_=_7_5}
  -\trace\bigl(\overK_{\v}(S_0 \oplus S_1 \oplus S_2)^2\bigr)
  = 6\,\norm{S_0}^2\bigl(\norm{S_1}^4 + \norm{S_2}^4\bigr)
    \mp 2\,\trace\bigl(\overK_{\bbO}(S_0 \oplus S_1)\circ
    \overK_{\bbO}(S_0 \oplus S_2)\bigr),
\end{equation}
where the upper sign is taken if \(\v_1 \cong \v_2\) and the lower sign otherwise.

We claim that
\begin{equation}\label{eq:under_mu_for_n_=_7_6}
  -\trace\bigl(\overK_{\bbO}(S_0 \oplus S_1)\circ \overK_{\bbO}(S_0 \oplus S_2)\bigr)
  = 2\,\norm{S_0}^2 \norm{S_1}^2 \norm{S_2}^2
    + 4\,\bigl\langle S_2^*(S_1 S_0),\, (S_0 S_2^*)S_1 \bigr\rangle
\end{equation}
for all \(S_0 \oplus S_1 \oplus S_2 \in \Im(\bbO) \oplus \bbO^2\).

For this, first assume that \(S_2 = 1_{\bbO}\). Then a direct
calculation shows that
\begin{equation}\label{eq:under_mu_for_n_=_7_7}
  -\trace\bigl(\overK_{\bbO}(S_0 \oplus S_1)\circ \overK_{\bbO}(S_0 \oplus 1_{\bbO})\bigr)
  = 2\,\norm{S_0}^2\norm{S_1}^2 + 4\,\langle S_1 S_0, S_0 S_1 \rangle
\end{equation}
as follows.

The subalgebra of \(\bbO\) generated by \(\Im_\bbO(S_0)\) and
\(\Im_\bbO(S_1)\) is contained in a quaternionic subalgebra. Hence,
after replacing \(\bbH\) by this quaternionic subalgebra, we may assume
\(\mathrm{span}_\bbR\{S_0, S_1\} \subseteq \bbH \subseteq \bbO\).
Fix a unit \(u \in \bbH^{\perp}\). Right multiplication by \(u\) gives
an isometric isomorphism \(\bbH \to \bbH^{\perp}\), \(h \mapsto h\,u\).
Thus every \(x \in \bbO\) can be written uniquely as \(x = a + b\,u\)
with \(a, b \in \bbH\). Via the identification
\(a + b\,u \leftrightarrow a \oplus b\) (so \(u \leftrightarrow 0 \oplus 1\)),
we view \(\bbO\) as the vector space \(\bbH \oplus \bbH\), with
\begin{equation}\label{eq:Cayley_Dickson}
  (a \oplus b)(c \oplus d) = (ac - d^{*} b) \oplus (da + b\,c^{*}),
  \qquad
  (a \oplus b)^{*} = a^{*} \oplus (-b)
\end{equation}
where \(^{*}\) on \(a, b, c, d\) denotes quaternionic conjugation on
\(\bbH\); we use the same symbol \(^{*}\) for the induced octonionic
conjugation on \(\bbO\). This is the Cayley–Dickson construction.

Next, let \(\{E_1, E_2, E_3, E_4\}\) be an orthonormal system of \(\bbH\)
(for example, \(\{1, \i, \j, \k\}\)). Then
\(\{0 \oplus E_1, 0 \oplus E_2, 0 \oplus E_3, 0 \oplus E_4\}\) is an
orthonormal system of \(\bbH^\perp\). Set
\(E_i \coloneqq 0 \oplus E_{i-4}\) for \(i = 5, 6, 7, 8\). Thus
\(\{E_1, \ldots, E_8\}\) is an orthonormal basis of \(\bbO\), and
\begin{align*}
  \sum_{i=1}^8 \bigl\langle S_0 E_i,\,
    S_1^*\bigl((S_1 S_0)E_i\bigr) \bigr\rangle
  &= \sum_{i=1}^8 \bigl\langle S_1(S_0 E_i),\,
    (S_1 S_0)E_i \bigr\rangle
\end{align*}
By associativity of \(\bbH\),
\begin{align*}
  \sum_{i=1}^4 \bigl\langle S_1(S_0 E_i),\,
    (S_1 S_0)E_i \bigr\rangle
  &= \sum_{i=1}^4 \bigl\langle (S_1 S_0)E_i,\,
    (S_1 S_0)E_i \bigr\rangle \\
  &= 4\,\norm{S_1 S_0}^2 \\
  &= 4\,\norm{S_0}^2\norm{S_1}^2
\end{align*}
where, for the first equality, we used that right multiplication
\(\rmR_{E_i}\) is an isometry of \(\bbH\). For \(i = 5, \ldots, 8\),
\begin{align*}
  (S_1 S_0)E_i
    &= (S_1 S_0 \oplus 0)(0 \oplus E_{i-4})
     = 0 \oplus E_{i-4} S_1 S_0, \\
  S_1(S_0 E_i)
    &= (S_1 \oplus 0)\bigl(0 \oplus E_{i-4} S_0\bigr)
     = 0 \oplus E_{i-4} S_0 S_1
\end{align*}
Therefore,
\[
  \sum_{i=5}^8 \bigl\langle S_1(S_0 E_i),\,
    (S_1 S_0)E_i \bigr\rangle
  = \sum_{i=1}^4 \bigl\langle E_i S_0 S_1,\,
    E_i S_1 S_0 \bigr\rangle
  = 4\,\langle S_1 S_0,\, S_0 S_1 \rangle
\]
since left multiplication \(\rmL_{E_i}\) is an isometry of \(\bbH\) as
well. We conclude that
\begin{equation}\label{eq:trace}
  \sum_{i=1}^8 \bigl\langle S_0 E_i,\,
    S_1^*\bigl((S_1 S_0)E_i\bigr) \bigr\rangle
  = 4\,\norm{S_1}^2\norm{S_0}^2
    + 4\,\langle S_1 S_0,\, S_0 S_1 \rangle
\end{equation}
Moreover,~\eqref{eq:trace} holds for every orthonormal basis of \(\bbO\),
since it computes the trace of a bilinear form. Hence, we may repeat the
previous computation in a second orthonormal basis
\(\{E_1', \ldots, E_8'\}\) of \(\bbO\) with \(E_1' = 1_{\bbO}\) and
\(S_0 = \norm{S_0}\,E_2'\). Then
\begin{equation}\label{eq:Beitrag_von_i_=_1_und_2}
  \sum_{i=1}^2 \bigl\langle S_0 E_i',\,
    S_1^*\bigl((S_1 S_0)E_i'\bigr) \bigr\rangle
  = 2\,\norm{S_1}^2\norm{S_0}^2
\end{equation}
Also \(\{E_3',\dots,E_8'\}\) is an orthonormal basis of
\(\z' \coloneq S_0^\perp \subseteq \Im(\bbO)\), and for
every \(S \in \z'\) we have \(\Re(S_0 S)
\stackrel{\eqref{eq:octonionic_product}}{=} -\langle S_0, S\rangle = 0\),
so \(S_0 S \in \Im(\bbO)\). In particular,
\(\Im_{\bbO}(S_0 E_i') = S_0 E_i'\) for \(i=3,\dots,8\). Therefore,
\begin{align*}
  -\trace\bigl(\overK_{\bbO}(S_0 \oplus S_1)
      \circ \overK_{\bbO}(S_0 \oplus 1_{\bbO})\bigr)
    &\stackrel{\eqref{eq:overK_bbO}}{=}
      \sum_{i=3}^8 \bigl\langle
        \Im_\bbO(S_0 E_i'),\,
        \Im_\bbO\bigl(S_1^*((S_1 S_0)E_i')\bigr) \bigr\rangle \\
    &= \sum_{i=3}^8 \bigl\langle S_0 E_i',\,
        S_1^*\bigl((S_1 S_0)E_i'\bigr) \bigr\rangle \\
    &\stackrel{\eqref{eq:Beitrag_von_i_=_1_und_2}}{=}
      \sum_{i=1}^8 \bigl\langle S_0 E_i',\,
        S_1^*\bigl((S_1 S_0)E_i'\bigr) \bigr\rangle
      - 2\,\norm{S_1}^2\norm{S_0}^2 \\
    &\stackrel{\eqref{eq:trace}}{=}
      2\,\norm{S_1}^2\norm{S_0}^2
      + 4\,\langle S_1 S_0,\, S_0 S_1 \rangle
\end{align*}
which yields~\eqref{eq:under_mu_for_n_=_7_7}.

Now recall the triality description of \(\Spin(8)\). Via its three
irreducible \(8\)-dimensional representations on \(\bbO\) (the vector
representation and the two half-spin representations), each
\(g \in \Spin(8)\) corresponds to a triple
\((g_0, g_+, g_-) \in \SO(\bbO) \times \SO(\bbO) \times \SO(\bbO)\)
characterized by
\begin{equation}\label{eq:triality}
  g_+(S_1 S_2) = g_-(S_1)g_0(S_2)
\end{equation}
for all \(S_1, S_2 \in \bbO\). Conversely, any triple \((g_0, g_+, g_-)\)
satisfying~\eqref{eq:triality} arises from a unique \(g \in \Spin(8)\);
see the Triality Theorem \cite[Thm.~14.19]{Ha}.

In this model, the subgroup \(\Spin(7) \subseteq \Spin(8)\) is
characterized by the condition \(g_+ = g_-\) or, equivalently,
\(g_0(1_{\bbO}) = 1_{\bbO}\) \cite[Cor.~14.64]{Ha}. If, moreover,
\((g_0, g_+, g_-)\) represents an element of \(\Gtwo\), the stabilizer of
\(1_{\bbO}\) in the spin representation of \(\Spin(7)\), then also
\(g_+(1_{\bbO}) = 1_{\bbO}\) and \(g_+ = g_-\) hold. In particular,
\(g_0\) and \(g_+\) preserve the decomposition
\(\bbO = \bbR 1_{\bbO} \oplus \Im(\bbO)\), and hence commute with the
octonionic conjugation \(\mathrm{c}\). Therefore \(g_0' = g_0\) and
\(g_+' = g_+\), and since \(g_+ = g_-\) we also have \(g_-' = g_-\).

Furthermore, there exists an involution \(\alpha\) of \(\Spin(8)\)
mapping \((g_0, g_+, g_-)\) to \((g_-', g_+', g_0')\), where
\(g' \coloneqq \mathrm{c} \circ g \circ \mathrm{c}\)
\cite[p.~279]{Ha}. Thus
\begin{align*}
  \langle S_1 S_0,\, S_0 S_1 \rangle
  &\stackrel{g_+ \,\in\, \SO(\bbO)}{=}
    \bigl\langle g_+(S_1 S_0),\, g_+(S_0 S_1) \bigr\rangle \stackrel{\Gtwo}{=}
    \bigl\langle g_+(S_1 S_0),\, g_+'(S_0 S_1) \bigr\rangle \\
  &\overset{(g_0, g_+, g_-)}{=}
    \bigl\langle g_-(S_1)g_0(S_0),\, g_+'(S_0 S_1) \bigr\rangle \overset{(g_-', g_+', g_0')}{=}
    \bigl\langle g_-(S_1)g_0(S_0),\, g_0'(S_0)g_-'(S_1) \bigr\rangle \\
  &\quad\stackrel{\Gtwo}{=}
    \bigl\langle g_+(S_1)g_0(S_0),\, g_0(S_0)g_+(S_1) \bigr\rangle
\end{align*}
which shows that \(\langle S_1 S_0,\, S_0 S_1 \rangle\) is \(\Gtwo\)-invariant.

We next show that the right-hand side of~\eqref{eq:under_mu_for_n_=_7_6}
defines a \(\Spin(7)\)-invariant polynomial in
\(S_0 \oplus S_1 \oplus S_2 \in \Im(\bbO) \oplus \bbO^2\) and that it
agrees with~\eqref{eq:under_mu_for_n_=_7_7} when \(S_2 \coloneqq 1_{\bbO}\):
The agreement is immediate, since substituting \(S_2 \coloneqq 1_{\bbO}\)
into the right-hand side of~\eqref{eq:under_mu_for_n_=_7_6} yields
\(2\,\norm{S_0}^2\norm{S_1}^2 + 4\,\langle S_1 S_0,\, S_0 S_1 \rangle\).
To verify \(\Spin(7)\)-invariance, recall the \emph{triality
automorphism} \(\tau\) of \(\Spin(8)\), an outer automorphism of order
three (\(\tau^3 = \Id\)) that cyclically permutes the vector and the two
half-spin representations:
\begin{equation}\label{eq:tau}
  (g_0, g_+, g_-) \overset{\tau}{\longmapsto} (g_+', g_-', g_0)
  \overset{\tau}{\longmapsto} (g_-, g_0', g_+')
\end{equation}
so that \(\tau^2\) maps \((g_0, g_+, g_-)\) to
\((g_-, g_0', g_+')\); cf.\ \cite[p.~279]{Ha}.

If \((g_0, g_+, g_-)\in \Spin(7)\), then \(g_0(1_{\bbO})=1_{\bbO}\).
Since \(g_0\in \SO(\bbO)\), it follows that \(g_0\) preserves
\(1_{\bbO}^{\perp}=\Im(\bbO)\), and hence commutes with the octonionic
conjugation \(\mathrm{c} \coloneqq (\,\cdot\,)^*\). In particular, \(g_0' = g_0\).
Moreover \(g_+=g_-\) for \(\Spin(7)\). Thus
\begin{align*}
  \bigl\langle S_2^*(S_1 S_0),\, (S_0 S_2^*) S_1 \bigr\rangle
  &\stackrel{g_0'\,\in\,\SO(\bbO)}{=}
    \bigl\langle g_0'(S_2^*(S_1 S_0)),\,
      g_0'\bigl((S_0 S_2^*) S_1\bigr) \bigr\rangle \\
  &\overset{(g_-, g_0', g_+')}{=}
    \bigl\langle g_+'(S_2^*)\,g_-(S_1 S_0),\,
      g_+'(S_0 S_2^*)\,g_-(S_1) \bigr\rangle \\
  &\stackrel{\mathrm{def.\;'} }{=}
    \bigl\langle g_+(S_2)^*\,g_-(S_1 S_0),\,
      g_+(S_2 S_0^*)^*\,g_-(S_1) \bigr\rangle \\
  &\stackrel{\Spin(7)}{=}
    \bigl\langle g_+(S_2)^*\,g_+(S_1 S_0),\,
      g_+(S_2 S_0^*)^*\,g_+(S_1) \bigr\rangle \\
  &\overset{(g_0, g_+, g_-)}{=}
    \bigl\langle g_+(S_2)^*\bigl(g_-(S_1)g_0(S_0)\bigr),\,
      \bigl(g_-(S_2)g_0(S_0^*)\bigr)^*\,g_+(S_1) \bigr\rangle \\
  &\stackrel{\mathrm{def.\;'} }{=}
    \bigl\langle g_+(S_2)^*\bigl(g_-(S_1)g_0(S_0)\bigr),\,
      \bigl(g_0'(S_0)g_-(S_2)^*\bigr)g_+(S_1) \bigr\rangle \\
  &\stackrel{\Spin(7)}{=}
    \bigl\langle g_+(S_2)^*\bigl(g_+(S_1)g_0(S_0)\bigr),\,
      \bigl(g_0(S_0)g_+(S_2)^*\bigr)g_+(S_1) \bigr\rangle
\end{align*}
where the equalities marked with \((g_0, g_+, g_-)\) and
\((g_-, g_0', g_+')\) use that these triples represent elements of
\(\Spin(8)\), whereas those marked with \(\Spin(7)\) hold because
\((g_0, g_+, g_-)\in \Spin(7)\). The steps labelled “\(\mathrm{def.\;'}\)”
use the definition \(g' \coloneqq \mathrm{c} \circ g \circ \mathrm{c}\).
Thus the right-hand side of~\eqref{eq:under_mu_for_n_=_7_6} is
\(\Spin(7)\)-invariant.

Since \eqref{eq:under_mu_for_n_=_7_6} holds for \(S_2=1_{\bbO}\) and both
sides are \(\Spin(7)\)-invariant, transitivity of the \(\Spin(7)\)-action
on the unit sphere \(\rmS^7 \subseteq \bbO\) (via the spin representation)
implies that \eqref{eq:under_mu_for_n_=_7_6} holds for all \(S_2 \in \rmS^7\).
Because both sides of \eqref{eq:under_mu_for_n_=_7_6} are homogeneous of
degree \(2\) in \(S_2\), the identity holds for all \(S_2 \in \bbO\).
Substituting~\eqref{eq:under_mu_for_n_=_7_5} and~\eqref{eq:under_mu_for_n_=_7_6}
into~\eqref{eq:under_mu_for_n_=_7_4}, we conclude that
\eqref{eq:over_mu_for_n_=_7_1} and~\eqref{eq:over_mu_for_n_=_7_3} hold.

Finally, suppose that \(\v\) splits as
\(\v = \v_1 \oplus \v_2 \oplus \v_3 \oplus \v_4\), where the first three
summands are irreducible \(\Cl(\Im(\bbO))\)-modules (isomorphic to
\(\bbO\) or \(\overline{\bbO}\)), while \(\v_4\) is arbitrary, possibly
trivial. Fix \(S_0 \coloneqq E_7 \in \Im(\bbO)\) and identify
\(\Spin(6) \cong \Spin(S_0^\perp) \subseteq \Spin(7)\) with its action on
each summand.

Assume, for contradiction, that there is an eigenvalue branch of
multiplicity at least four on some connected component
\(U \subseteq \n \setminus \ram(\overK_{\v}^2)\). By
Corollary~\ref{co:fixed_point_group}, there exists a nonempty open subset
\(U^{\v} \subseteq \v\) such that for each \(S \in U^{\v}\) the rank of
\[
  \Fix(S, \Spin(6)) \coloneqq \{\,F \in \Spin(6) \mid F S = S\,\}
\]
is at least one. Let \(\pi \colon \v \to \v_1 \oplus \v_2 \oplus \v_3\)
denote the projection onto the first three components. Since \(\pi\) is
open and \(U^{\v}\) is nonempty and open, \(\pi(U^{\v})\) is a nonempty
open subset of \(\v_1 \oplus \v_2 \oplus \v_3\). Hence \(\pi(U^{\v})\)
meets the open dense subset of \(\v_1 \oplus \v_2 \oplus \v_3\)
consisting of triples with trivial common stabilizer in \(\Spin(6)\).

To see that the common stabilizer is generically trivial, identify
\(\v_1\), \(\v_2\), and \(\v_3\) with the standard real spin
representation \(\bbO\) of \(\Spin(6)\), so that
\(\v_1 \oplus \v_2 \oplus \v_3 \cong \bbO^3\). View \(\bbO\) as the
underlying real vector space of \(\bbC^4\) with the standard
\(\SU(4)\)-action, and recall that \(\Spin(6) \cong \SU(4)\) acts
transitively on \(\rmS^7 \subseteq \bbO\). Hence we may assume that
\(S_1 \coloneqq 1_{\bbO}\). Then \(\Fix(S_1,\Spin(6)) \cong \SU(3)\),
which fixes the complex line \(\bbC 1_{\bbO}\) and acts by the standard
representation on its orthogonal complement \(\bbC^3\). For a generic
second spinor \(S_2\) with nonzero projection to \(\bbC^3\), its
stabilizer inside \(\SU(3)\) is \(\SU(2)\). The residual \(\SU(2)\) acts
on the remaining orthogonal complement \(\bbC^2\), and this action is
free on the unit sphere \(\rmS^3 \subseteq \bbC^2\). Hence, for a
generic third spinor \(S_3\) with nonzero projection to \(\bbC^2\), the
common stabilizer is trivial.

In other words, in the generic case the isotropy chain is
\[
  \Spin(6) \stackrel{\Fix(S_1)}{\longrightarrow} \SU(3)
  \stackrel{\Fix(S_2)}{\longrightarrow} \SU(2)
  \stackrel{\Fix(S_3)}{\longrightarrow} \{e\}
\]
Thus we may choose \(S = S_1 \oplus S_2 \oplus S_3 \oplus S_4 \in U^{\v}\)
such that \((S_1,S_2,S_3)\) has trivial common stabilizer in \(\Spin(6)\).
Therefore,
\[
  \Fix(S,\Spin(6))
  \subseteq \Fix(\{S_1,S_2,S_3\},\Spin(6)) = \{e\}
\]
a contradiction.

Moreover, the direct sum of two irreducible \(\Cl(\Im(\bbO))\)-modules,
as considered above, does not satisfy the \(J^2\)-condition (even though
\(\nu = 1\) is a global eigenvalue in the isotypic two-summand case).
Therefore, if \(\v\) contains at least three irreducible summands, write
\(\v = \v' \oplus \v''\), where \(\v' \coloneqq \v_1 \oplus \v_2\) and
\(\v''\) is the direct sum of the remaining summands. Since \((\z,\v')\)
fails the \(J^2\)-condition, Corollary~\ref{co:reducible_Clifford_module}
implies that \(\nu = 1\) is not a global eigenvalue of \(-\underK_{\v}^2\).
Also \(m_0 = 1\), and the previous implies that the remaining muliplicity \(6\)
splits up into three distinct nonconstant eigenvalue branches of multiplicity
\(2\), each of multiplicity \(2\).
\end{proof}

\subsubsection{\texorpdfstring{Eigenvalue branches of \(-\underK^2\) for
\(4 \le n \le 6\)}{Eigenvalue branches of (-K^2) for 4 \le n \le 6}}
\label{se:4_leq_n_leq_6}

Here we obtain an orthogonal \(\Cl(\z)\)-module structure on \(\v \coloneqq \bbO\)
by restricting the Clifford multiplication%
~\eqref{eq:Clifford_multiplication_for_n_=_7_1} to \(\z \times \v\),
where \(\z\) is any \(n\)-dimensional subspace of \(\Im(\bbO)\). By the
uniqueness (up to isomorphism) of irreducible \(8\)-dimensional
\(\Cl(\z)\)-modules for \(4 \le n \le 6\), these are the only possibilities
up to isomorphism (and hence independent of the choice of \(\z\) up to isomorphism
The corresponding Clifford algebras are 
\(\Cl(\bbR^4) \cong \Mat_2(\bbH)\), \(\Cl(\bbR^5) \cong \Mat_4(\bbC)\), and
\(\Cl(\bbR^6) \cong \Mat_8(\bbR)\), namely the algebras of quaternionic \(2 \times 2\),
complex \(4 \times 4\), and real \(8 \times 8\) matrices, respectively.\footnote{In particular, 
for \(n = 4\) under an isometric identification \(\z \cong \bbR^4\),
this model of an irreducible orthogonal \(\Cl(\z)\)-module becomes isomorphic to the
irreducible orthogonal \(\Cl(\bbR^4)\)-module \(\bbH^2\) considered in the proof 
of Proposition~\ref{p:m_geq_4}~(b).}

\begin{proposition}\label{p:n_=_4_5_6}
Suppose that \(4 \le n \le 6\) and let \(\v\) be a nontrivial orthogonal
  \(\Cl(\z)\)-module. Nonzero eigenvalue branches of the family
  \(-\underK^2\) of nonnegative self-adjoint operators defined in
  \eqref{eq:underK_square} and associated with the pair \((\z, \v)\) are
  as follows:
  \begin{enumerate}
  \item For \(n = 4\), there is exactly one nonconstant eigenvalue
    branch \(\undermu\), of multiplicity \(m_{\undermu} = 2\), 
    obtained by rescaling~\eqref{eq:over_mu_for_n_=_4}.
  \item For \(n = 5\):
    \begin{itemize}
    \item If \(\v\) is an irreducible \(\Cl(\z)\)-module, then \(1\) is a
      global eigenvalue of \(-\underK^2\) of multiplicity \(2\). In
      addition, there exists a nonconstant eigenvalue branch
      \(\undermu\), also of multiplicity \(m_{\undermu} = 2\), obtained
      by rescaling~\eqref{eq:over_mu_for_n_=_5}.
    \item Otherwise, if \(\v\) splits as the orthogonal direct sum of at
      least two irreducible \(\Cl(\z)\)-submodules, there are two distinct
      nonconstant eigenvalue branches \(\undermu_1\) and \(\undermu_2\), 
      each of multiplicity \(2\).
    \end{itemize}
  \item For \(n = 6\):
    \begin{itemize}
    \item If \(\v\) is an irreducible \(\Cl(\z)\)-module, then \(1\) is a
      global eigenvalue of \(-\underK^2\) of multiplicity \(4\).
    \item Otherwise, there are two nonconstant eigenvalue branches \(\undermu_1\) and
      \(\undermu_2\), each of multiplicity \(2\).
    \end{itemize}
  \end{enumerate}
\end{proposition}

\begin{proof}
Assume first that \(\v\) is an irreducible \(\Cl(\z)\)-module. As noted
earlier, we may assume that \(\v = \bbO\) and that \(\z\) is an
\(n\)-dimensional subspace of \(\Im(\bbO) \cong \bbR^7\).  Set
\[
  \overK_{\z} \coloneqq \overK_{(\z,\bbO)},\qquad
  \underK_{\z} \coloneqq \underK_{(\z,\bbO)},\qquad
  \overK_{\bbO} \coloneqq \overK_{(\Im(\bbO),\bbO)},\qquad
  \underK_{\bbO} \coloneqq \underK_{(\Im(\bbO),\bbO)}
\]
Then \(\overK_{\z}(S_0 \oplus S_1) \colon \z \to \z\), the linear operator
defined in~\eqref{eq:def_overK} and associated with the pair
\((\z,\bbO)\) and \(S_0 \oplus S_1 \in \z \oplus \bbO\), is obtained by
restricting \(\overK_{\bbO}(S_0 \oplus S_1)\) from ~\eqref{eq:overK_bbO}
to \(\z\) and projecting back to \(\z\):
\begin{equation*}
  \bigl\langle \overK_{\z}(S_0 \oplus S_1)S,\,\tilde S \bigr\rangle
    = \bigl\langle \overK_{\bbO}(S_0 \oplus S_1)S,\,\tilde S \bigr\rangle
\end{equation*}
for all \(S_0 \oplus S_1 \in \z \oplus \bbO\) and \(S, \tilde S \in \z\).
Recall also that \(\underK_{\bbO}(S_0 \oplus S_1)S_0 = 0\), and that the
restriction
\[
  J := \underK_{\bbO}(S_0 \oplus S_1)|_{S_0^\perp} \colon S_0^\perp \to S_0^\perp
\] defines an orthogonal complex structure on the orthogonal
complement \(S_0^\perp \subseteq \Im(\bbO)\) of \(S_0\) in \(\Im(\bbO)\).

For \(n = 6\), let \(\z \coloneqq \bbR^6\). Since \(\SO(6)\)
acts transitively on \(\rmS^5 \subseteq \z\)---and hence also \(\Spin(6)\)
via the vector representation---we may assume \(S_0 = E_6\).
Because \(\Spin(5)\) still acts transitively
on \(\rmS^7 \subseteq \v\), we can argue as in the case \(n = 7\)
to conclude that the eigenvalue branches of \(\underK_{\z}\), and hence
of \(-\underK_{\z}^2\), are constant on
\[
  \Spin(6)/\Spin(5) \times \Spin(5)/\SU(2)
  = \rmS^5 \times \rmS^7 \subseteq \z \oplus \v
\]
By Corollary~\ref{co:mu_=_constant}, \(-\underK_{\z}^2\) has the global eigenvalues \(0\) and
\(1\) with multiplicities \(m_0 = 2\) and \(m_1 = 4\), respectively.

For \(n = 5\), consider the complex vector space \((S_0^\perp,J)\). Then
\(\z' \coloneqq S_0^\perp \cap \z\) is a \(4\)-dimensional real subspace,
so \((\z')^\perp \subseteq S_0^\perp\) is a real \(2\)-plane. Consider its
complex span in \((S_0^\perp,J)\),
\[
  \bbC(\z')^\perp \coloneqq (\z')^\perp + J(\z')^\perp
\]
Then \(\dim_{\bbC}\bigl(\bbC(\z')^\perp\bigr) \le 2\). Since
\(\dim_{\bbC}(S_0^\perp)=3\), the orthogonal complement
\[
  \bigl(\bbC(\z')^\perp\bigr)^\perp \subseteq S_0^\perp
\]
has complex dimension \(\ge 1\). Moreover,
\(\bigl(\bbC(\z')^\perp\bigr)^\perp \subseteq \bigl((\z')^\perp\bigr)^\perp = \z'\),
because \((\z')^\perp \subseteq \bbC(\z')^\perp\). Thus \(\z'\) contains a
\(J\)-invariant real \(2\)-plane, i.e.\ a complex line. Therefore \(\nu = 1\)
is a global eigenvalue of \(-\underK_{\z}^2\) with multiplicity \(m_1 \ge 2\).

To show that \(m_1 = 2\) and to identify the second, nonconstant
eigenvalue branch, remember that \(\underK_{\bbO}(S_0 \oplus S_1)S_0 = 0\) and
\(\underK_{\bbO}(S_0 \oplus S_1)\bigl|_{S_0^\perp} = J\). Since \(J\) is an
orthogonal complex structure on \(S_0^\perp\), we have \(J^2 = -\Id\) on
\(S_0^\perp\), and hence
\[
  -\trace\bigl(\underK_{\bbO}(S_0 \oplus S_1)^2\bigr) = 6
\]
Let us now compute this trace by means of an orthonormal basis
\(\{E_1, \ldots, E_7\}\) of \(\Im(\bbO)\) with \(S_0 = \norm{S_0}\,E_7\) and
\(\z' = \mathrm{span}_\bbR\{E_3, \ldots, E_6\}\). Using Parseval’s
identity and skew-symmetry,
\begin{align*}
  6 = -\trace\bigl(\underK_{\bbO}(S_0 \oplus S_1)^2\bigr)
  &= \sum_{i=1}^7 \bigl\langle \underK_{\bbO}(S_0 \oplus S_1) E_i,\,
      \underK_{\bbO}(S_0 \oplus S_1)E_i \bigr\rangle \\
  &= \sum_{i=1}^6 \bigl\langle J E_i,\,
     J E_i \bigr\rangle  = \sum_{\substack{i,j=1 \\ i \ne j}}^6
      \bigl\langle J E_i,\,
        E_j \bigr\rangle^2
\end{align*}
Moreover,
\begin{align*}
  \sum_{j=3}^6 \bigl\langle J E_1,\,
    E_j \bigr\rangle^2
  = 1 - \bigl\langle J E_1,\,
  E_{2} \bigr\rangle^2 = 1 - \bigl\langle J E_2,\,
    E_{1} \bigr\rangle^2 =
  \sum_{j=3}^6 \bigl\langle J E_2,\,
    E_j \bigr\rangle^2
\end{align*}
 Hence,
\begin{align*}
  -\trace\bigl(\underK_{\bbO}(S_0 \oplus S_1)^2\bigr)
  &= 2\,\bigl\langle J E_1,\,
       E_2 \bigr\rangle^2
     + 2\sum_{i=1}^2 \sum_{j=3}^6
       \bigl\langle J E_i,\,
         E_j \bigr\rangle^2  + \sum_{\substack{i,j=3 \\ i \ne j}}^6
       \bigl\langle J E_i,\,
         E_j \bigr\rangle^2 \\
  &= 4 - 2\,\bigl\langle J E_1,\,
       E_2 \bigr\rangle^2
     + \sum_{\substack{i,j=3 \\ i \ne j}}^6
       \bigl\langle J E_i,\,
         E_j \bigr\rangle^2
\end{align*}

Therefore
\begin{align*}
  -\trace\bigl(\underK_{\z}(S_0 \oplus S_1)^2\bigr)
  &= \sum_{\substack{i,j=3 \\ i \ne j}}^6
       \bigl\langle J E_i,\,
         E_j \bigr\rangle^2 \\
  &= 2 + 2\,\bigl\langle J E_1,\,
         E_2 \bigr\rangle^2
\end{align*}
On the other hand, since \(\underK_{\z}(S_0 \oplus S_1)\) is skew-symmetric on the
\(5\)-dimensional real vector space \(\z\), its nonzero eigenvalues occur
in purely imaginary conjugate pairs. Hence \(-\underK_{\z}(S_0 \oplus S_1)^2\)
has spectrum of the form
\[
  \{\,0,\; \nu,\; \nu,\; \undermu(S_0 \oplus S_1),\; \undermu(S_0 \oplus S_1)\,\}
\]
(counted with multiplicity), where \(\nu = 1\) is the global eigenvalue.
Consequently,
\[
  -\trace\bigl(\underK_{\z}(S_0 \oplus S_1)^2\bigr)
  = 2\,\nu + 2\,\undermu(S_0 \oplus S_1)
  = 2 + 2\,\undermu(S_0 \oplus S_1)
\]
Hence
\begin{align*}
  \undermu(S_0 \oplus S_1)
  &= \bigl\langle J E_1,\,
      E_2 \bigr\rangle^2 = \frac{1}{\|S_0\|^2\|S_1\|^4} \bigl\langle \overK_{\bbO}(S_0 \oplus S_1)E_1,\,
      E_2 \bigr\rangle^2 \\
  &\stackrel{\eqref{eq:overK_bbO}}{=}
     \frac{\bigl\langle (S_1 S_0)E_1,\, S_1 E_2 \bigr\rangle^2}{\|S_0\|^2\|S_1\|^4}
\end{align*}
Clearing the denominator yields~\eqref{eq:over_mu_for_n_=_5}. 
For \(n = 4\), an analogous calculation yields~\eqref{eq:over_mu_for_n_=_4}.

We now treat the reducible case and claim that there are no eigenvalue
branches of multiplicity greater than \(2\). Suppose that \(\v\) splits
as \(\v = \v_1 \oplus \v_2 \oplus \v_3\) into three
\(\Cl(\z)\)-submodules, where \(\v_1\) and \(\v_2\) are irreducible and
\(\v_3\) is arbitrary (possibly trivial). 
By uniqueness of irreducible \(\Cl(\z)\)-modules for \(4 \le n \le 6\), we have
\(\v_1 \cong \v_2 \cong \bbO\). Assume, for the sake of contradiction, that there
exists an eigenvalue branch of multiplicity at least \(4\) on some
connected component \(U \subseteq \n \setminus \ram(\overK^2)\).
Set \(\z' \coloneqq S_0^\perp \cap \z\). By
Corollary~\ref{co:fixed_point_group}, there exists a nonempty open
subset \(U^{\v} \subseteq \v\) such that the fixed-point subgroup
\[
  \Fix(S, \Spin(\z')) \coloneqq \{\,F \in \Spin(\z') \mid F S = S\,\}
\]
has rank at least one for each \(S \in U^{\v}\).
Let \(\pi\colon \v=\v_1\oplus\v_2\oplus\v_3 \to \v_1\oplus\v_2\)
denote the projection onto the first two summands. Since \(\pi\) is
open and \(U^{\v}\subseteq \v\) is nonempty and open, \(\pi(U^{\v})\)
is a nonempty open subset of \(\v_1\oplus\v_2\). Hence \(\pi(U^{\v})\)
meets the open dense subset of \(\v_1 \oplus \v_2 \) consisting of pairs
\((S_1,S_2)\) with trivial common stabilizer in \(\Spin(\z')\).

To see that this stabilizer is generically trivial, choose a
\(5\)-dimensional subspace \(\tilde\z \subseteq S_0^\perp\) with
\(\z' \subseteq \tilde\z\). Choose an orthonormal basis
\(E_1,\dots,E_6\) of \(S_0^\perp\) such that
\(\z' = \mathrm{span}_\bbR\{E_1,\dots,E_{n-1}\}\) and set
\(\tilde\z \coloneqq \mathrm{span}_\bbR\{E_1,\dots,E_5\}\). Then
\(\Spin(\tilde\z) \cong \Spin(5)\).

Consider the unit sphere \(\rmS^7 \subseteq \bbO\). The stabilizer of a
unit spinor \(S_1\) in \(\Spin(5)\cong \Sp(2)\) is \(\Sp(1)\cong \SU(2)\),
and this \(\Sp(1)\) fixes precisely the quaternionic line \(\bbH S_1\).
Therefore \(S_2 \notin \bbH S_1\) implies that the common stabilizer of
\((S_1,S_2)\) in \(\Spin(5)\) is trivial, i.e.,
\[
  \Spin(5) \stackrel{\Fix(S_1)}{\longrightarrow} \Sp(1)
  \stackrel{\Fix(S_2)}{\longrightarrow} \{e\}
\]
is the generic isotropy chain. Choose
\(S = S_1 \oplus S_2 \oplus S_3 \in U^{\v}\) such that \((S_1,S_2)\) has
trivial common stabilizer in \(\Spin(5)\). Since
\(\Spin(\z') \subseteq \Spin(\tilde\z) \cong \Spin(5)\), we obtain
\[
  \Fix(S,\Spin(\z'))
  \subseteq \Fix(\{S_1,S_2\},\Spin(\z'))
  \subseteq \Fix(\{S_1,S_2\},\Spin(5)) = \{e\}
\]
a contradiction.

Finally, Corollary~\ref{co:reducible_Clifford_module} rules out the case
that \(\nu = 1\) is a global eigenvalue of \(-\underK^2\) for reducible
\(\v\), since for \(4 \le n \le 6\) no nontrivial \(\Cl(\z)\)-summand
satisfies the \(J^2\)-condition. Therefore \(\z\) splits into \(\tfrac{n - m_0}{2}\) 
distinct nonconstant eigenvalue branches of multiplicity \(2\) where 
\(m_0 \in \{1,2\}\) with \(m_0 \equiv n \pmod{2}\) is the multiplicity
of the global eigenvalue \(0\). This completes the argument for \(4 \le n \le 6\).
\end{proof}

\subsubsection{\texorpdfstring{Eigenvalue branches of \(-\underK^2\) for \(n = 8\)}
  {Eigenvalue branches of (-K^2) for n = 8}}
\label{se:n_=_8}

For \(\dim \z = 8\), the Clifford algebra \(\Cl(\z)\) is isomorphic to
\(\Mat_{16}(\bbR)\), the algebra of \(16 \times 16\) real matrices. Hence
\(\Cl(\z)\) has, up to isomorphism, a unique \(16\)-dimensional irreducible
representation. In the model \(\z \coloneqq \bbO\) and \(\v \coloneqq \bbO^2\),
the Clifford multiplication is described by
\eqref{eq:Clifford_multiplication_for_n_=_8}.

The natural splitting \(\bbO^2 = \bbO \oplus \bbO\) is a
\(\mathbb{Z}_2\)-grading. In particular, each summand \(\bbO \oplus 0\) and
\(0 \oplus \bbO\) is \(\Spin(8)\)-stable, where we put
\(\Spin(8) \coloneqq \Spin(\bbO)\). This yields the \(\Spin(8)\)-modules
of positive and negative (right-handed and left-handed) spinors; cf.\
\cite[Ch.~1, \S~8]{LM}.

\begin{proposition}\label{p:n_=_8}
  Suppose that \(\dim \z = 8\), and let \(\v\) be a nontrivial
  orthogonal \(\Cl(\z)\)-module. The nonzero eigenvalue branches of the
  family \(-\underK^2\) of nonnegative self-adjoint operators defined in
  \eqref{eq:underK_square} and associated with the pair \((\z,\v)\) are
  as follows:
  \begin{itemize}
  \item If \(\v\) is an irreducible \(\Cl(\z)\)-module, then there is a
  single nonconstant eigenvalue branch \(\undermu\), of multiplicity
  \(m_{\undermu} = 6\), obtained by rescaling the polynomial function
  defined in~\eqref{eq:over_mu_for_n_=_8}.
  \item Otherwise, there are three distinct nonconstant eigenvalue
    branches, each of multiplicity \(2\).
  \end{itemize}
\end{proposition}

\begin{proof}
Assume first that \(\v\) is an irreducible \(\Cl(\z)\)-module.
As noted earlier, we may assume that \(\z = \bbO\) and \(\v = \bbO^2\).
Let \(S_0 \oplus S_+ \oplus S_- \in \z \oplus \v = \bbO^3\). First, consider the special case
\(S_0 = 1_\bbO\). For every orthonormal pair \(\{Z_1, Z_2\} \subseteq
\Im(\bbO) \cong \bbR^7\), we have
\begin{equation}\label{eq:tilde_K_bbO_oplus_bbO}
\begin{aligned}
  \langle \overK(1_\bbO \oplus S_+ \oplus S_-) Z_1,\, Z_2 \rangle
  &\stackrel{\eqref{eq:def_overK}}{=}
    \langle Z_2 \cbullet (S_+ \oplus S_-),\,
      Z_1 \cbullet 1_\bbO \cbullet (S_+ \oplus S_-) \rangle \\
  &\stackrel{\eqref{eq:Clifford_multiplication_for_n_=_8}}{=}
    \langle Z_2 \cbullet (S_+ \oplus S_-),\,
      Z_1 \cbullet (S_- \oplus -\,S_+) \rangle \\
  &= \langle S_- Z_2 \oplus S_+ Z_2,\, -\,S_+ Z_1 \oplus S_- Z_1 \rangle \\
  &= \langle S_+ Z_2,\, S_- Z_1 \rangle - \langle S_- Z_2,\, S_+ Z_1 \rangle \\
  &= \langle Z_2,\, S_+^*(S_- Z_1) \rangle - \langle Z_2,\, S_-^*(S_+ Z_1) \rangle
\end{aligned}
\end{equation}
where we used \(\langle AB,\, C\rangle = \langle B,\, A^* C\rangle\), since
left multiplication \(\rmL_{A^*}\) by \(A^*\) is the adjoint of left
multiplication \(\rmL_A\) by \(A\). Thus
\begin{equation}\label{eq:K_on_ImO_1}
  \langle \overK(1_\bbO \oplus S_+ \oplus S_-) Z_1,\, Z_2 \rangle
  =
  \langle (\rmL_{S_+^*}\rmL_{S_-} - \rmL_{S_-^*}\rmL_{S_+})Z_1,\, Z_2 \rangle
\end{equation}

In this model, we identify
\(\Spin(7) \coloneqq \Spin(\Im(\bbO))\) with the stabilizer of
\(1_{\bbO} \in \z = \bbO\). This group acts transitively on the unit
sphere of the first summand
\(\bbO \oplus 0 \subseteq \v = \bbO^2\), via the spin representation on
\(\bbO \cong \bbR^8\). Hence we may assume that
\(S_+ \in \bbR_{>0} 1_{\bbO}\). For the remainder of this calculation,
we further specialize to \(S_+ = 1_{\bbO}\).
 
Then
\begin{equation}\label{eq:S_+_=_1_n=_8}
  \rmL_{S_+^*}\,\rmL_{S_-} - \rmL_{S_-^*}\,\rmL_{S_+}
  = \rmL_{S_- - S_-^*} = 2\,\rmL_{\Im_\bbO(S_-)}
\end{equation}
where \(\Im_\bbO(S_-)\) denotes the imaginary part of \(S_- \in \bbO\).
Substituting~\eqref{eq:S_+_=_1_n=_8} into~\eqref{eq:K_on_ImO_1} gives
\begin{equation}\label{eq:K_on_ImO_2}
  \overK(1_\bbO \oplus 1_\bbO \oplus S_-)\bigl|_{\Im(\bbO)}
  = 2\,\Im_\bbO(S_-) \times \Box
  \;\colon\; \Im(\bbO) \longrightarrow \Im(\bbO)
\end{equation}

On \(\Im(\bbO)\), the map \(S \mapsto \Im_\bbO(S_-)\times S\) has kernel
\(\mathrm{span}_\bbR\{\Im_\bbO(S_-)\}\). On the orthogonal complement
\(\Im_\bbO(S_-)^\perp\), it satisfies
\[
  \bigl(\Im_\bbO(S_-)\times \Box\bigr)^2
  = -\,\norm{\Im_\bbO(S_-)}^{2}\,\Id
\]
and hence \(\Im_\bbO(S_-)\times \Box\) is a scaled orthogonal complex structure
on \(\Im_\bbO(S_-)^\perp\), with scale factor \(\norm{\Im_\bbO(S_-)}\).

Moreover, \(\overK(1_\bbO \oplus 1_\bbO \oplus S_-)\,1_\bbO = 0\) (equivalently,
\(\overK(X)S_0 = 0\) for all \(X\)). Hence
\(-\overK(1_\bbO \oplus 1_\bbO \oplus S_-)^2\) has a single nonzero eigenvalue
\begin{equation}\label{eq:n_=_8_4}
  \overmu(1_\bbO \oplus 1_\bbO \oplus S_-) = 4\norm{\Im_\bbO(S_-)}^2
\end{equation}
on \(\z = \bbO\), with multiplicity \(6\), and the zero eigenvalue has
multiplicity \(2\). Passing to the rescaled operator \(\underK\) introduces
the factor \((1 + \norm{S_-}^2)^{-2}\), so
\begin{equation}\label{eq:n_=_8_5}
  \undermu(1_\bbO \oplus 1_\bbO \oplus S_-)
  = 4\,\frac{\norm{\Im_\bbO(S_-)}^2}{(1 + \norm{S_-}^2)^2}
\end{equation}
is nonconstant. In particular, \(\nu = 1\) is not a global eigenvalue of
\(-\underK^2\) in the irreducible case.

So far, we have treated the special case \(S_0 = 1_\bbO\) and \(S_+ = 1_\bbO\).
For arbitrary \(S_+ \in \bbO\), set
\begin{equation}\label{eq:n_=_8_6}
  \overmu(1_\bbO \oplus S_+ \oplus S_-)
  \coloneqq 4\bigl(\norm{S_+}^2 \norm{S_-}^2 - \langle S_+, S_- \rangle^2\bigr)
\end{equation}
This expression is \(\Spin(\Im(\bbO))\)-invariant, since the two half-spin
representations coincide on \(\Spin(\Im(\bbO)) \subseteq \Spin(\bbO)\)
\cite[p.~282]{Ha}. Moreover, \eqref{eq:n_=_8_4} is \(\Gtwo\)-invariant and
agrees with \eqref{eq:n_=_8_6} when \(S_+ = 1_\bbO\). Hence \eqref{eq:n_=_8_6}
is the unique \(\Spin(\Im(\bbO))\)-invariant extension of \eqref{eq:n_=_8_4}\!,
because this group acts transitively on the first \(\bbO\)-factor via the
positive spin representation.

Finally, to allow arbitrary \(S_0 \in \bbO\) as well, consider the
\(\Spin(\bbO)\)-invariant polynomial \(\overmu\) described by
\eqref{eq:over_mu_for_n_=_8}. It coincides with \eqref{eq:n_=_8_6} for
\(S_0 = 1_\bbO\), and therefore yields a homogeneous polynomial eigenvalue
branch of \(-\overK^2\) on \(\z = \bbO\), with multiplicity \(6\).
Rescaling gives the corresponding nonconstant eigenvalue branch \(\undermu\)
of \(-\underK^2\) on \(U\).

Assume now that \(\v\) is reducible, so it contains at least two irreducible
\(\Cl(\z)\)-submodules. Choose two such summands \(\v_1,\v_2\) and write
\(\v=\v_1\oplus\v_2\oplus\v_3\), where \(\v_3\) is a (possibly trivial)
complement. By uniqueness of the irreducible \(\Cl(\z)\)-module structure, we have
\(\v_1 \cong \v_2 \cong \bbO^2\), with the \(\Cl(\bbO)\)-module structure on
each summand \(\bbO^2\) as above.

Assume, for contradiction, that there is an eigenvalue branch of multiplicity
at least four on some connected component \(U \subseteq \n \setminus \ram(\overK^2)\).
Let \(\z' \coloneqq \Im(\bbO)\), so \(\Spin(\z') \cong \Spin(7)\).
By Corollary~\ref{co:fixed_point_group}, there exists a nonempty open subset
\(U^{\v} \subseteq \v\) such that for every \(S \in U^{\v}\) the fixed-point group
\[
  \Fix(S, \Spin(\z')) \coloneqq \{\,F \in \Spin(\z') \mid F S = S\,\}
\]
has positive rank. Let \(\pi\colon \v \to \v_1 \oplus \v_2\) denote the
projection onto the first two summands. Since \(\pi\) is open and \(U^{\v}\)
is nonempty and open, \(\pi(U^{\v})\) is a nonempty open subset of
\(\v_1 \oplus \v_2\).

On the other hand, for a generic tuple
\[
  (S_+ \oplus S_-) \oplus (\tilde S_+ \oplus \tilde S_-)
  \in \bbO^2 \oplus \bbO^2
\]
the isotropy chain in \(\Spin(7)\) is
\[
  \Spin(7)
  \stackrel{\Fix(S_+ \oplus S_-)}{\longrightarrow} \SU(3)
  \stackrel{\Fix(\tilde S_+ \oplus \tilde S_-)}{\longrightarrow} \{e\}
\]
and in particular the common stabilizer is trivial. Hence \(\pi(U^{\v})\)
meets the open dense subset of \(\v_1 \oplus \v_2\) consisting of pairs with
trivial common stabilizer in \(\Spin(\z')\). Choose \(S \in U^{\v}\) such that
\(\pi(S)\) has trivial common stabilizer. Then
\[
  \Fix(S, \Spin(\z'))
  \subseteq \Fix(\pi(S), \Spin(\z')) = \{e\}
\]
a contradiction. Therefore no eigenvalue branch of multiplicity \(\ge 4\) can
occur in the reducible case.

Finally, in the irreducible case, \(\nu = 1\) is not a global eigenvalue
of \(-\underK^2\), so the \(J^2\)-condition does not hold. Hence, if
\(\v\) is reducible, Corollary~\ref{co:reducible_Clifford_module} rules
out the possibility that \(\nu = 1\) is a global eigenvalue of
\(-\underK^2\). Moreover, \(0\) is a global eigenvalue of multiplicity
\(m_0 = 2\) by Corollary~\ref{co:mu_=_constant}. Therefore the remaining
multiplicity \(6\) splits into three distinct nonconstant eigenvalue branches
of multiplicity \(2\). This completes the argument for \(n = 8\).
\end{proof}

\subsubsection{\texorpdfstring{Eigenvalue branches of \(-\underK^2\) for
  \(n = 9\)}{Eigenvalue branches of (-\underK^2) for n = 9}}
\label{se:n_=_9}

Assume \(\dim \z = 9\). Up to isomorphism, there is a unique irreducible
orthogonal \(\Cl(\z)\)-module \(\v\) of real dimension \(32\). A concrete model is
the orthogonal direct sum \(\z = \bbR \oplus \bbO\) and \(\v = \bbO^2 \otimes_{\bbR} \bbC\),
with Clifford multiplication as in \eqref{eq:Clifford_multiplication_for_n_=_9}.
The subgroup \(\Spin(\bbO) \subseteq \Spin(\bbR \oplus \bbO)\) acts on
both \(\Re(S_+)\) and \(\Im(S_+)\) via the positive half-spin
representation, and on both \(\Re(S_-)\) and \(\Im(S_-)\) via the
negative half-spin representation. Hence the spin representation
commutes with taking real and imaginary parts of spinors; cf.\
\cite[Lemma~14.77]{Ha}.

\begin{proposition}\label{p:n_=_9}
  Suppose that \(\dim \z = 9\) and let \(\v\) be a nontrivial
  orthogonal \(\Cl(\z)\)-module. Nonzero eigenvalue branches of the family
  \(-\underK^2\) of nonnegative self-adjoint operators defined in
  \eqref{eq:underK_square} and associated with the pair \((\z, \v)\) are
  as follows:
  \begin{itemize}
    \item If \(\v\) is irreducible, then \(-\underK^2\) has three
          nonconstant eigenvalue branches: \(\undermu_1\) and
          \(\undermu_2\) (each of multiplicity \(2\)), and \(\undermu_3\)
          (of multiplicity \(4\)), see~\eqref{eq:over_sigma_1_for_n_=_9}--\eqref{eq:over_mu_3_for_n_=_9_1}.
    \item Otherwise, \(-\underK^2\) has four distinct nonconstant
          eigenvalue branches, each of multiplicity \(2\).
  \end{itemize}
\end{proposition}
\begin{proof}
Assume first \(\v\) irreducible; then we use the canonical model
\(\z \coloneqq \bbR \oplus \bbO\) and \(\v \coloneqq \bbO^2\otimes_{\bbR}\bbC
= \bbO\otimes_{\bbR}\bbC \oplus \bbO\otimes_{\bbR}\bbC\) mentioned before.
Let
\[
 S_0 \oplus S_+ \oplus S_-\in \z \oplus \v = (\bbR \oplus \bbO) \oplus
 \bbO\otimes_{\bbR}\bbC \oplus \bbO\otimes_{\bbR}\bbC
\]
We first treat the special case \(S_0 = 1_\bbR\), where we identify
\(1_\bbR \in \bbR\) with the base point \(1 \oplus 0 \in \bbR \oplus
\bbO\). Write \(S_\pm = \Re(S_\pm) + \i\,\Im(S_\pm) \in \bbO \oplus \i\,\bbO\)
for \(\pm \in \{+,-\}\). In order to compute \(\overK(1_\bbR \oplus S_+ \oplus S_-)\), we use
\eqref{eq:Clifford_multiplication_for_n_=_9} to obtain
\begin{equation}\label{eq:multiplication_by_E_9}
  1_\bbR \cbullet (S_+ \oplus S_-) = \i\,S_+ \oplus -\,\i\,S_-
\end{equation}
Therefore, for all \(Z_1,Z_2 \in \bbO\),
\begin{align*}
  \langle \overK(1_\bbR \oplus S_+ \oplus S_-)Z_1,\,Z_2 \rangle
  &\stackrel{\eqref{eq:def_overK}}{=}
    \langle Z_2 \cbullet (S_+ \oplus S_-),\,
            Z_1 \cbullet 1_\bbR \cbullet (S_+ \oplus S_-) \rangle \\
  &\stackrel{\eqref{eq:multiplication_by_E_9}}{=}
    \langle Z_2 \cbullet (S_+ \oplus S_-),\,
            Z_1 \cbullet (\i\,S_+ \oplus -\,\i\,S_-) \rangle
\end{align*}
Furthermore,
\begin{align*}
  Z_2 \cbullet (S_+ \oplus S_-)
    &= \i\,S_- Z_2 \oplus \i\,S_+ Z_2^* \\
  Z_1 \cbullet (\i\,S_+ \oplus -\,\i\,S_-)
    &= S_- Z_1 \oplus -\,S_+ Z_1^*
\end{align*}
It follows that
\begin{align*}
  \langle \overK(1_\bbR \oplus S_+ \oplus S_-)Z_1,\,Z_2 \rangle
  &= \langle \i\,S_- Z_2 \oplus \i\,S_+ Z_2^*,\,
            S_- Z_1 \oplus -\,S_+ Z_1^* \rangle \\
  &= \langle \i\,S_- Z_2,\,S_- Z_1 \rangle
     - \langle \i\,S_+ Z_2^*,\,S_+ Z_1^* \rangle
\end{align*}
Moreover,
\begin{align*}
  \langle \i\,S_- Z_2,\,S_- Z_1 \rangle
  &= \langle \Re(S_-)Z_2,\,\Im(S_-)Z_1 \rangle
     - \langle \Im(S_-)Z_2,\,\Re(S_-)Z_1 \rangle \\
  &= \langle Z_2,\,\Re(S_-)^*\Im(S_-)Z_1 \rangle
     - \langle Z_2,\,\Im(S_-)^*\Re(S_-)Z_1 \rangle \\
  \langle \i\,S_+ Z_2^*,\,S_+ Z_1^* \rangle
  &= \langle \Re(S_+)Z_2^*,\,\Im(S_+)Z_1^* \rangle
     - \langle \Im(S_+)Z_2^*,\,\Re(S_+)Z_1^* \rangle \\
  &= \langle Z_2^*,\,\Re(S_+)^*(\Im(S_+)Z_1^*) \rangle
     - \langle Z_2^*,\,\Im(S_+)^*(\Re(S_+)Z_1^*) \rangle \\
  &= \langle Z_2,\,(Z_1\,\Im(S_+)^*)\Re(S_+) \rangle
     - \langle Z_2,\,(Z_1\,\Re(S_+)^*)\Im(S_+) \rangle
\end{align*}
We conclude that
\begin{align*}
  \overK(1_\bbR \oplus S_+ \oplus S_-)
  &= \rmR_{\Im(S_+)} \circ \rmR_{\Re(S_+)^*}
     - \rmR_{\Re(S_+)} \circ \rmR_{\Im(S_+)^*} \\
  &\quad
     + \rmL_{\Re(S_-)^*} \circ \rmL_{\Im(S_-)}
     - \rmL_{\Im(S_-)^*} \circ \rmL_{\Re(S_-)}
\end{align*}
Define
\begin{align}
  \overK_+(1_\bbR \oplus S_+)
  &\coloneqq
    \rmR_{\Im(S_+)} \circ \rmR_{\Re(S_+)^*}
    - \rmR_{\Re(S_+)} \circ \rmR_{\Im(S_+)^*}
    \label{eq:tilde_K_+_1} \\
  \overK_-(1_\bbR \oplus S_-)
  &\coloneqq
    \rmL_{\Re(S_-)^*} \circ \rmL_{\Im(S_-)}
    - \rmL_{\Im(S_-)^*} \circ \rmL_{\Re(S_-)}
    \label{eq:tilde_K_-_1}
\end{align}
where \(\rmR_S\) and \(\rmL_S\) denote right and left multiplication by
\(S\), respectively. By definition,
\[
  \overK(1_\bbR \oplus S_+ \oplus S_-)
  = \overK_+(1_\bbR \oplus S_+) + \overK_-(1_\bbR \oplus S_-)
\]

\smallskip

\noindent
By \cite[Theorem 14.69]{Ha} there exists \(g=(g_0,g_+,g_-) \in \Spin(\bbO)\)
with \(\Re(g_+\,S_+) = \norm{\Re(S_+)}1_\bbO\) and \(\Re(g_-\,S_-) =  \norm{\Re(S_-)}1_\bbO\).
Therefore, let us assume, for the moment, that both \(\Re(S_+) = \norm{\Re(S_+)}1_\bbO\)
and \(\Re(S_-) = \norm{\Re(S_-)}1_\bbO\) hold. Then we can rewrite \(\overK_\pm\) as
single multiplication operators. Namely,
\begin{align}
  \overK_+(1_\bbR \oplus S_+)
  &= \rmR_{\Re(S_+)^*\Im(S_+) - \Im(S_+)^*\Re(S_+)}
     \notag \\
  &= 2\,\rmR_{a(S_+)}
     \label{eq:Kplus_as_Ra}
\end{align}
where
\begin{equation}\label{eq:def_aSplus}
  a(S_+) \coloneqq \norm{\Re(S_+)}\,\Im_{\bbO}\bigl(\Im(S_+)\bigr)
  \in \Im(\bbO)
\end{equation}
and similarly
\begin{align}
  \overK_-(1_\bbR \oplus S_-)
  &= \rmL_{\Re(S_-)^*\Im(S_-) - \Im(S_-)^*\Re(S_-)}
     \notag \\
  &= 2\,\rmL_{b(S_-)}
     \label{eq:Kminus_as_Lb}
\end{align}
with
\begin{equation}\label{eq:def_bSminus}
  b(S_-) \coloneqq \norm{\Re(S_-)}\,\Im_{\bbO}\bigl(\Im(S_-)\bigr)
  \in \Im(\bbO)
\end{equation}
Consequently,
\begin{equation}\label{eq:K_as_Ra_plus_Lb}
  \overK(1_\bbR \oplus S_+ \oplus S_-)
  = 2\,\rmR_{a(S_+)} + 2\,\rmL_{b(S_-)}
\end{equation}

The following lemma isolates the algebraic mechanism behind the
appearance of three eigenvalue branches in the irreducible case.

\begin{lemma}[Three-Branch Lemma]\label{le:three_branches}
Let \(a,b \in \Im(\bbO)\). Set \(\overK_+ \coloneqq 2\,\rmR_a\) and
\(\overK_- \coloneqq 2\,\rmL_b\). Then the operator
\(-(\overK_+ + \overK_-)^2\) on \(\bbO\) has eigenvalues
\[
  4\bigl(\norm{a} \pm \norm{b}\bigr)^2
\]
each with multiplicity \(2\), and
\[
  4\norm{a - b}^2
\]
with multiplicity \(4\). In particular, for generic \(a,b\) these give
three distinct eigenvalues.
\end{lemma}

\begin{proof}
Choose a quaternionic subalgebra \(\bbH \subseteq \bbO\) containing both
\(a\) and \(b\). (For instance, the subalgebra generated by \(a\) and
\(b\) is associative and hence contained in some \(\bbH\).) Let
\(u \in \Im(\bbO)\) be a unit with \(u \perp \bbH\). Then
\(\bbO = \bbH \oplus \bbH u\), corresponding to the Cayley--Dickson
description~\eqref{eq:Cayley_Dickson} with \(u \leftrightarrow 0 \oplus 1\).
Hence for \(x,y,h \in \bbH\),
\[
  (x + y u)h = (x h) + (y h^*)u,
  \qquad
  h(x + y u) = (h x) + (y h)u
\]
In particular, for \(y \in \bbH\),
\[
  \rmR_a(y u) = (y a^*)u,
  \qquad
  \rmL_b(y u) = (y b)u
\]
where \(\rmR_a\) and \(\rmL_b\) denote right and left multiplication by
\(a\) and \(b\), respectively.

Using \((\rmR_a + \rmL_b)^2 = \rmR_{a^2} + \rmL_{b^2}
  + \rmR_a\rmL_b + \rmL_b\rmR_a\) together with
\(a^2 = -\norm{a}^2\) and \(b^2 = -\norm{b}^2\), one obtains on \(\bbO\)
\[
  -\bigl(\rmR_a + \rmL_b\bigr)^2
  = \bigl(\norm{a}^2 + \norm{b}^2\bigr)I
    - \bigl(\rmR_a\rmL_b + \rmL_b\rmR_a\bigr)
\]
On \(\bbH\), associativity implies \(\rmR_a\rmL_b = \rmL_b\rmR_a\). Thus
\(\rmR_a\) and \(\rmL_b\) commute on \(\bbH\), and they satisfy
\(\rmR_a^2 = -\norm{a}^2 I\) and \(\rmL_b^2 = -\norm{b}^2 I\). Hence
\(\bbH\) splits into a direct sum of common real \(2\)-planes on which
\(\rmR_a\) and \(\rmL_b\) act as commuting complex structures scaled by
\(\norm{a}\) and \(\norm{b}\). On such a plane, the eigenvalues of
\(-(\rmR_a + \rmL_b)^2\) are \((\norm{a} \pm \norm{b})^2\), and each sign
occurs with multiplicity \(2\) on \(\bbH\).

On \(\bbH u\),
\[
  (\rmL_b\rmR_a + \rmR_a\rmL_b)(y u)
  = \bigl(y(a^* b + b a^*)\bigr)u
\]
Since \(a,b \in \Im(\bbH)\), one has \(a^* = -a\) and \(b^* = -b\), and
therefore \(a^* b + b a^* = 2\langle a,b\rangle\). Consequently,
\[
  (\rmL_b\rmR_a + \rmR_a\rmL_b)\big|_{\bbH u}
  = 2\langle a,b\rangle\,I
\]
and thus
\[
  -\bigl(\rmR_a + \rmL_b\bigr)^2\big|_{\bbH u}
  = \bigl(\norm{a}^2 + \norm{b}^2 - 2\langle a,b\rangle\bigr)I
  = \norm{a - b}^2\,I
\]
Finally, since \(\overK_+ = 2\,\rmR_a\) and \(\overK_- = 2\,\rmL_b\), we
have
\(-(\overK_+ + \overK_-)^2 = 4\bigl(-(\rmR_a + \rmL_b)^2\bigr)\), which
multiplies all eigenvalues by \(4\).
\end{proof}
Applying Lemma~\ref{le:three_branches} directly to \eqref{eq:K_as_Ra_plus_Lb}
with \(a \coloneqq a(S_+)\) and \(b \coloneqq b(S_-)\), this yields the
eigenvalues of \(-\overK(1_\bbR \oplus S_+ \oplus S_-)^2\) on \(\bbO\) as
\begin{align}
  \overmu_1(1_\bbR \oplus S_+ \oplus S_-)
  &\coloneqq 4\bigl(\norm{a(S_+)} + \norm{b(S_-)}\bigr)^2
  \label{eq:def_mu_1_n_=_9} \\
  \overmu_2(1_\bbR \oplus S_+ \oplus S_-)
  &\coloneqq 4\bigl(\norm{a(S_+)} - \norm{b(S_-)}\bigr)^2
  \label{eq:def_mu_2_n_=_9} \\
  \overmu_3(1_\bbR \oplus S_+ \oplus S_-)
  &\coloneqq 4\,\norm{a(S_+) - b(S_-)}^2
  \label{eq:def_mu_3_n_=_9}
\end{align}
under the slice condition \(\Re(S_+) = \norm{\Re(S_+)}1_\bbO\) and \(\Re(S_-) =  \norm{\Re(S_-)}1_\bbO\).
Here \(\overmu_1\) and \(\overmu_2\) have multiplicity \(2\), and
\(\overmu_3\) has multiplicity \(4\). 

\smallskip

\noindent
The auxiliary quantities \(a(S_+)\) and \(b(S_-)\) were introduced only on the normalized slice
\(\Re(S_+) = \norm{\Re(S_+)}1_\bbO\) and \(\Re(S_-) = \norm{\Re(S_-)}1_\bbO\);
off the slice~\eqref{eq:def_mu_1_n_=_9}--\eqref{eq:def_mu_3_n_=_9} fail to hold.
To obtain global expressions for the eigenvalue branches,
we now rewrite \(\overmu_1\), \(\overmu_2\), and \(\overmu_3\)
in a manifestly \(\Spin(\bbO)\)-invariant form.

For this, consider the Gramians
\begin{equation}\label{eq:Gram}
  \overmu_\pm(1_\bbR \oplus S_\pm)
  \coloneqq 4\,\norm{\Re(S_\pm)\wedge \Im(S_\pm)}^2
\end{equation}
for \(\pm \in \{+,-\}\). More explicitly,
\begin{align}
  \overmu_+(1_\bbR \oplus S_+)
  &= 4\bigl(
      \norm{\Re(S_+)}^2 \norm{\Im(S_+)}^2
      - \langle \Re(S_+),\,\Im(S_+) \rangle^2
    \bigr)
    \label{eq:over_mu_+} \\
  \overmu_-(1_\bbR \oplus S_-)
  &= 4\bigl(
      \norm{\Re(S_-)}^2 \norm{\Im(S_-)}^2
      - \langle \Re(S_-),\,\Im(S_-) \rangle^2
    \bigr)
    \label{eq:over_mu_-}
\end{align}
are \(\Spin(\bbO)\)-invariant terms.
Under the condition \(\Re(S_+) = \norm{\Re(S_+)}1_\bbO\) and \(\Re(S_-) =  \norm{\Re(S_-)}1_\bbO\), we have
\[
  \sqrt{\overmu_+(1_\bbR \oplus S_+)} = 2\norm{a(S_+)},
  \qquad
  \sqrt{\overmu_-(1_\bbR \oplus S_-)} = 2\norm{b(S_-)}
\]
Therefore \eqref{eq:def_mu_1_n_=_9} and \eqref{eq:def_mu_2_n_=_9} can be
rewritten in a \(\Spin(\bbO)\)-invariant form as
\begin{align}
  \overmu_1(1_\bbR \oplus S_+ \oplus S_-)
  &=
    \Bigl(
      \sqrt{\overmu_+(1_\bbR \oplus S_+)}
      + \sqrt{\overmu_-(1_\bbR \oplus S_-)}
    \Bigr)^2
    \label{eq:over_mu_1} \\
  \overmu_2(1_\bbR \oplus S_+ \oplus S_-)
  &=
    \Bigl(
      \sqrt{\overmu_+(1_\bbR \oplus S_+)}
      - \sqrt{\overmu_-(1_\bbR \oplus S_-)}
    \Bigr)^2
    \label{eq:over_mu_2}
\end{align}

For later use, we remark that by the binomial identities,
the elementary symmetric functions are
\begin{align}
\overmu_1(1_\bbR \oplus S_+ \oplus S_-)
     + \overmu_2(1_\bbR \oplus S_+ \oplus S_-)
     &= 2\Bigl(\overmu_+(1_\bbR \oplus S_+) + \overmu_-(1_\bbR \oplus
     S_-)\Bigr)
     \label{eq:n_=_9_binom_1}\\
\overmu_1(1_\bbR \oplus S_+ \oplus S_-)\overmu_2(1_\bbR \oplus S_+
\oplus S_-)
     &= \Bigl(\overmu_+(1_\bbR \oplus S_+) - \overmu_-(1_\bbR \oplus
     S_-)\Bigr)^2
     \label{eq:n_=_9_binom_2}
\end{align}

In order to obtain an explicit \(\Spin(\bbO)\)-invariant expression
for the third branch \(\overmu_3\), define the mixed term
\begin{equation}\label{eq:over_mu_mix}
  \overmu_{\mathrm{mix}}(S_+,S_-)
  \coloneqq 4\bigl\langle \Re(S_+)^*\Im(S_-),\,
    \Im(S_+)^*\Re(S_-)
  \bigr\rangle
  - 4\bigl\langle
    \Re(S_+)^*\Re(S_-),\, \Im(S_+)^*\Im(S_-)
  \bigr\rangle
\end{equation}
We claim that \(\overmu_{\mathrm{mix}}\) is \(\Spin(\bbO)\)-invariant:
Recall the triality description of \(\Spin(\bbO)\) by triples
\((g_0,g_+,g_-)\) satisfying~\eqref{eq:triality}, and the triality
automorphism \(\tau\) from~\eqref{eq:tau}. For \(g=(g_0,g_+,g_-)\),
the triple \( (g_-,g_0',g_+') = \tau^2 g\) also lies in \(\Spin(\bbO)\).
Using the triality relation, \(g_+'(x^*) = g_+(x)^*\),
and \(g_0'\in\SO(\bbO)\), we obtain
\begin{align*}
  \bigl\langle
    \Re(S_+)^*\Re(S_-),\,\Im(S_+)^*\Im(S_-)
  \bigr\rangle
  &=
  \bigl\langle
    \Re(g_+S_+)^*\Re(g_-S_-),\, \Im(g_+S_+)^*\Im(g_-S_-)
  \bigr\rangle
\end{align*}
and similarly
\[
  \bigl\langle
  \Re(S_+)^*\Im(S_-),\, \Im(S_+)^*\Re(S_-)
  \bigr\rangle
  =
  \bigl\langle
    \Re(g_+S_+)^*\Im(g_-S_-),\,\Im(g_+S_+)^*\Re(g_-S_-)
  \bigr\rangle
\]
because the \(\Spin(\bbO)\)-action commutes with taking real and
imaginary parts. Therefore \(\overmu_{\mathrm{mix}}\) is \(\Spin(\bbO)\)-invariant.

\begin{lemma}\label{le:mu3_explicit_formula}
The third eigenvalue branch \(\overmu_3\) from
\eqref{eq:def_mu_3_n_=_9} admits the equivalent expression
\begin{equation}\label{eq:over_mu_3_=_over_mu_mix}
  \overmu_3(1_\bbR \oplus S_+ \oplus S_-)
  = \overmu_+(1_\bbR \oplus S_+)
    + \overmu_-(1_\bbR \oplus S_-)
    + \overmu_{\mathrm{mix}}(S_+,S_-)
\end{equation}
\end{lemma}

\begin{proof}
Since the involved terms are \(\Spin(\bbO)\)-invariant, it suffices to
verify \eqref{eq:over_mu_3_=_over_mu_mix} on the slice
\(\Re(S_+) \in \bbR_{>0}\,1_{\bbO}\) and
\(\Re(S_-) \in \bbR_{>0}\,1_{\bbO}\). On this slice, expanding
\(\norm{a(S_+)-b(S_-)}^2\) gives
\begin{equation}\label{eq:over_mu_3_on_slice}
\begin{aligned}
\overmu_3(1_{\bbR}\oplus S_+\oplus S_-)
  &\stackrel{\eqref{eq:def_mu_3_n_=_9}}{=} 4\norm{a(S_+) - b(S_-)}^2\\
  &\stackrel{\eqref{eq:def_aSplus}\eqref{eq:def_bSminus}}{=}
    4\Bigl\lVert
      \norm{\Re(S_+)}\,\Im_{\bbO}\bigl(\Im(S_+)\bigr)
      - \norm{\Re(S_-)}\,\Im_{\bbO}\bigl(\Im(S_-)\bigr)
    \Bigr\rVert^2\\
  &\stackrel{\eqref{eq:over_mu_+},\eqref{eq:over_mu_-}}{=}
    \overmu_+(1_\bbR \oplus S_+)
    + \overmu_-(1_\bbR \oplus S_-)\\
  &\qquad\qquad\qquad\qquad
    - 8\,\norm{\Re(S_+)}\norm{\Re(S_-)}
      \bigl\langle
        \Im_{\bbO}\bigl(\Im(S_+)\bigr),\,
        \Im_{\bbO}\bigl(\Im(S_-)\bigr)
      \bigr\rangle
\end{aligned}
\end{equation}
Moreover, on this slice we have
\[
  \Im_{\bbO}\bigl(\Im(S_+)\bigr)
  = \tfrac12\bigl(\Im(S_+) - \Im(S_+)^*\bigr),
  \qquad
  \Im_{\bbO}\bigl(\Im(S_-)\bigr) \perp 1_{\bbO}
\]
so we may replace \(\Im(S_-)\) by \(\Im_{\bbO}(\Im(S_-))\) inside the
inner product. Thus
\begin{equation*}
\begin{aligned}
  \overmu_3(1_\bbR \oplus S_+ \oplus S_-)
  &- \overmu_+(1_\bbR \oplus S_+)
   - \overmu_-(1_\bbR \oplus S_-)\\
  &\stackrel{\eqref{eq:over_mu_3_on_slice}}{=}
   - 8\,\norm{\Re(S_+)}\norm{\Re(S_-)}
     \bigl\langle
       \Im_{\bbO}\bigl(\Im(S_+)\bigr),\,
       \Im_{\bbO}\bigl(\Im(S_-)\bigr)
     \bigr\rangle\\
  &= - 8\,\norm{\Re(S_+)}\norm{\Re(S_-)}
     \Bigl\langle
       \Im_{\bbO}\bigl(\Im(S_+)\bigr),\,
       \Im(S_-)
     \Bigr\rangle\\
  &= - 4\,\norm{\Re(S_+)}\norm{\Re(S_-)}
     \Bigl\langle
       \Im(S_+) - \Im(S_+)^*,\,
       \Im(S_-)
     \Bigr\rangle\\
  &= 4\norm{\Re(S_+)}\norm{\Re(S_-)}
     \Bigl(
       \bigl\langle \Im(S_-),\,\Im(S_+)^*\bigr\rangle
       - \bigl\langle \Im(S_+),\,\Im(S_-)\bigr\rangle
     \Bigr)\\
  &= 4\bigl\langle
       \Re(S_+)^*\Im(S_-),\,\Im(S_+)^*\Re(S_-)
     \bigr\rangle
     - 4\bigl\langle
       \Re(S_+)^*\Re(S_-),\,\Im(S_+)^*\Im(S_-)
     \bigr\rangle\\
  &\stackrel{\eqref{eq:over_mu_mix}}{=}
    \overmu_{\mathrm{mix}}(S_+,S_-)
\end{aligned}
\end{equation*}
This proves \eqref{eq:over_mu_3_=_over_mu_mix}.
\end{proof}

\smallskip

\noindent
So far, we have worked at the base point \(S_0 = 1_\bbR\).
Since \eqref{eq:over_sigma_1_for_n_=_9}, \eqref{eq:over_sigma_2_for_n_=_9},
and \eqref{eq:over_mu_3_for_n_=_9_1} are \(\Spin(\bbR\oplus\bbO)\)-invariant
and homogeneous of degree \(2\) in \(S_0\), it suffices to verify
that they restrict to the base-point identities \eqref{eq:n_=_9_binom_1},
\eqref{eq:n_=_9_binom_2}, and \eqref{eq:over_mu_3_=_over_mu_mix} when \(S_0=1_\bbR\).
Indeed, for any \(S_0\neq 0\) there exists \(g\in\Spin(\bbR\oplus\bbO)\) with
\(g\,S_0 = \|S_0\|\,1_\bbR\), and homogeneity reduces further to \(S_0=1_\bbR\).

\smallskip

\noindent
For this, note that \(M(1_\bbR)=\mathrm{diag}(1,-1)\) holds
for the matrix defined in~\eqref{eq:M_matrix} (with \(S_0 = 1_\bbR\)).
Hence, writing \(S = S_+ \oplus S_-\), we have
\[
  SM(1_\bbR) = (S_+ \oplus S_-)\,\mathrm{diag}(1,-1)
  = S_+ \oplus (-\,S_-)
\]
Thus,
\begin{align*}
  2\bigl(\norm{\Re(S_+)}^2 \norm{\Im(S_+)}^2
    &+ \norm{\Re(S_-)}^2 \norm{\Im(S_-)}^2\bigr)\\
   &= \bigl(\norm{\Re(S_+)}^2 + \norm{\Re(S_-)}^2\bigr)
     \bigl(\norm{\Im(S_+)}^2 + \norm{\Im(S_-)}^2\bigr) \\
   &\qquad\qquad\qquad\qquad + \bigl(\norm{\Re(S_+)}^2 - \norm{\Re(S_-)}^2\bigr)
     \bigl(\norm{\Im(S_+)}^2 - \norm{\Im(S_-)}^2\bigr) \\
  &= \norm{\Re(S)}^2 \norm{\Im(S)}^2
     + \langle \Re(S), \Re(SM(1_{\bbR})) \rangle\,
       \langle \Im(S), \Im(SM(1_{\bbR})) \rangle 
\end{align*}
Furthermore, using the first two binomial formulas,
\begin{align*}
  2\bigl(\langle \Re(S_+), \Im(S_+) \rangle^2
          &+ \langle \Re(S_-), \Im(S_-) \rangle^2\bigr)
  = \bigl(\langle \Re(S_+), \Im(S_+) \rangle
     + \langle \Re(S_-), \Im(S_-) \rangle\bigr)^2 \\
   &\qquad\qquad\qquad\qquad\qquad + \bigl(\langle \Re(S_+), \Im(S_+) \rangle
     - \langle \Re(S_-), \Im(S_-) \rangle\bigr)^2 \\
  &= \langle \Re(S), \Im(S) \rangle^2
     + \langle \Re(S), \Im(SM(1_{\bbR})) \rangle^2
\end{align*}
Substituting this into~\eqref{eq:n_=_9_binom_1} yields the following
alternative description of \(\overmu_1 + \overmu_2\):
\begin{equation*}
\begin{aligned}
  (\overmu_1 + \overmu_2)&(1_\bbR \oplus S)
  \stackrel{\eqref{eq:over_mu_+},\eqref{eq:over_mu_-}}{=} 4\Bigl(
  \underbrace{\norm{\Re(S)}^2 \norm{\Im(S)}^2
  - \langle \Re(S), \Im(S) \rangle^2}_{=\det(G(S))}\\
  &\qquad+ \underbrace{\langle \Re(S), \Re(SM(1_{\bbR})) \rangle\,
         \langle \Im(S), \Im(SM(1_{\bbR})) \rangle 
     - \langle \Re(S), \Im(SM(1_{\bbR})) \rangle^2
     }_{=\det(H(1_{\bbR},S))}\Bigr)
\end{aligned}
\end{equation*}
for all \(S \in \bbO^2 \otimes_\bbR \bbC\). Here we used 
that \(M(1_{\bbR}) = \begin{psmallmatrix}1&0\\[2pt]0&-1\end{psmallmatrix}\) is real and symmetric, so
\begin{align*}
\det(H(1_{\bbR},S)) &\stackrel{\eqref{eq:def_H}}{=} \langle \Re(S), \Re(SM(1_{\bbR})) \rangle\,
         \langle \Im(S), \Im(SM(1_{\bbR})) \rangle \\
                  &- \langle \Re(S),\,\Im( SM(1_{\bbR})) \rangle\,\langle \Im(S),\, \Re(SM(1_{\bbR}) )\rangle \\
  &= \langle \Re(S), \Re(SM(1_{\bbR})) \rangle\,
         \langle \Im(S), \Im(SM(1_{\bbR})) \rangle 
     - \langle \Re(S), \Im(SM(1_{\bbR})) \rangle^2
\end{align*}
Hence, the binomial relation \eqref{eq:n_=_9_binom_1}
rewrites as
\[
  (\overmu_1+\overmu_2)(1_\bbR\oplus S)
  = 4\bigl(\det G(S) + \det H(1_\bbR,S)\bigr)
\]
which is \eqref{eq:over_sigma_1_for_n_=_9} at \(S_0=1_\bbR\) and \(S_1 = S\).

Likewise, 
\begin{equation}\label{eq:n_=_9_9}
\begin{aligned}
  &\overmu_1(1_\bbR \oplus S_+ \oplus S_-)\,
   \overmu_2(1_\bbR \oplus S_+ \oplus S_-) \\
  &=
    16\Bigl(
      \norm{\Re(S_+)}^2 \norm{\Im(S_+)}^2
      - \norm{\Re(S_-)}^2 \norm{\Im(S_-)}^2 \\
  &\qquad
      - \langle \Re(S_+), \Im(S_+) \rangle^2
      + \langle \Re(S_-), \Im(S_-) \rangle^2
    \Bigr)^2
\end{aligned}
\end{equation}
Here,
\begin{align*}
  2\Bigl(
    \norm{\Re(S_+)}^2 \norm{\Im(S_+)}^2
    &- \norm{\Re(S_-)}^2 \norm{\Im(S_-)}^2
  \Bigr) \\
  &= \bigl(
       \norm{\Re(S_+)}^2 + \norm{\Re(S_-)}^2
     \bigr)\bigl(
       \norm{\Im(S_+)}^2 - \norm{\Im(S_-)}^2
     \bigr) \\
  &\quad
     + \bigl(
         \norm{\Re(S_+)}^2 - \norm{\Re(S_-)}^2
       \bigr)\bigl(
         \norm{\Im(S_+)}^2 + \norm{\Im(S_-)}^2
       \bigr) \\
  &= \norm{\Re(S)}^2 \,
     \langle \Im(S), \Im(SM(1_{\bbR})) \rangle
     + \norm{\Im(S)}^2 \,
       \langle \Re(S), \Re(SM(1_{\bbR})) \rangle \\
\end{align*}
and, by the third binomial formula,
\begin{align*}
  \langle \Re(S_+), \Im(S_+) \rangle^2
  &- \langle \Re(S_-), \Im(S_-) \rangle^2 \\
  &= \bigl(
       \langle \Re(S_+), \Im(S_+) \rangle
       + \langle \Re(S_-), \Im(S_-) \rangle
     \bigr)\\
     &\qquad\qquad\times\bigl(
       \langle \Re(S_+), \Im(S_+) \rangle
       - \langle \Re(S_-), \Im(S_-) \rangle
     \bigr) \\
  &= \langle \Re(S), \Im(S) \rangle \,
    \langle \Re(S), \Im(SM(1_{\bbR})) \rangle
\end{align*}
Therefore,
\begin{equation*}
\begin{aligned}
  \bigl(
     \overmu_+(1_\bbR \oplus S_+)
     &- \overmu_-(1_\bbR \oplus S_-)
   \bigr)^2 \\
  &= 4\Bigl(
       \norm{\Im(S)}^2 \,
       \langle \Re(S), \Re(SM(1_\bbR)) \rangle
       + \norm{\Re(S)}^2 \,
         \langle \Im(S), \Im(SM(1_\bbR)) \rangle \\
  &\qquad\qquad\qquad\qquad\qquad\qquad
       - 2\,\langle \Re(S), \Im(S) \rangle \,
           \langle \Re(S), \Im(SM(1_\bbR)) \rangle
    \Bigr)^2\\
  &= 4\,\bigl(\trace \bigl(G(S)\,J\,H(1_{\bbR},S)\bigr)\bigr)^{2}
\end{aligned}
\end{equation*}
Here we used that
\begin{align*}
 -\trace(G(S)J H(1_{\bbR},S)) &=  \norm{\Re(S)}^{2}\bigl\langle \Im(S),\; \Im(S M(1_{\bbR}))\bigr\rangle\\
 &\quad- \bigl\langle \Re(S),\; \Im(SM(1_{\bbR}))\bigr\rangle
      \bigl\langle \Re(S),\Im(S)\bigr\rangle\\
 &\quad- \bigl\langle \Re(S), \Im(S)\bigr\rangle \underbrace{\bigl\langle \Im(S),\;
 \Re(S M(1_{\bbR}))\bigr\rangle}_{=\bigl\langle \Im(S M(1_{\bbR})),\; \Re(S)\bigr\rangle}\\
 &\quad+ \norm{\Im(S)}^{2} \bigl\langle \Re(S),\; \Re(S\,M(1_{\bbR}))\bigr\rangle
\end{align*}
Therefore \eqref{eq:n_=_9_binom_2} becomes
\[
  (\overmu_1\overmu_2)(1_\bbR\oplus S)
  = 4\bigl(\trace\bigl(G(S)\,J\,H(1_\bbR,S)\bigr)\bigr)^2
\]
which is \eqref{eq:over_sigma_2_for_n_=_9} at \(S_0=1_\bbR\) and \(S_1 = S\).

\smallskip

\noindent
Finally, in order to show that~\eqref{eq:over_mu_3_for_n_=_9_1}
corresponds to~\eqref{eq:over_mu_3_=_over_mu_mix}, note that
\[
  \tilde S_+ \coloneqq S + SM(1_\bbR) = 2(S_+ \oplus 0),
  \qquad
  \tilde S_- \coloneqq S - SM(1_\bbR) = 2(0 \oplus S_-)
\]
Since \(J_0 = \begin{psmallmatrix}0&1\\[2pt]1&0\end{psmallmatrix}\)
swaps the two components, we obtain
\[
  \Re(\tilde S_+)^*J_0\,\Re(\tilde S_-)^\intercal
  = 4\,\Re(S_+)^*\Re(S_-),
  \qquad
  \Im(\tilde S_+)^*J_0\,\Im(\tilde S_-)^\intercal
  = 4\,\Im(S_+)^*\Im(S_-)
\]
and similarly
\[
  \Re(\tilde S_+)^*J_0\,\Im(\tilde S_-)^\intercal
  = 4\,\Re(S_+)^*\Im(S_-),
  \qquad
  \Im(\tilde S_+)^*J_0\,\Re(\tilde S_-)^\intercal
  = 4\,\Im(S_+)^*\Re(S_-)
\]
Specializing \eqref{eq:over_mu_3_for_n_=_9_1} at \(S_0 = 1_\bbR\)
and \(S_1 = S = S_+\oplus S_-\)
and using the above identities for \(\tilde S_\pm\) yields
\begin{equation}\label{eq:over_mu_3_via_J_0}
\begin{aligned}
  \overmu_3(1_\bbR \oplus S)
  &=
    \frac{1}{2}\Bigl(
      \overmu_1(1_\bbR \oplus S)
      + \overmu_2(1_\bbR \oplus S)
    \Bigr)\\
  &\qquad
    - \frac{1}{4}\bigl\langle
      \Re(\tilde{S}_+)^* J_0\, \Re(\tilde{S}_-)^\intercal,\;
      \Im(\tilde{S}_+)^* J_0\, \Im(\tilde{S}_-)^\intercal
    \bigr\rangle\\
  &\qquad
    + \frac{1}{4}\bigl\langle
      \Re(\tilde{S}_+)^* J_0\, \Im(\tilde{S}_-)^\intercal,\;
      \Im(\tilde{S}_+)^* J_0\, \Re(\tilde{S}_-)^\intercal
    \bigr\rangle\\
  &=
    \frac{1}{2}\Bigl(
      \overmu_1(1_\bbR \oplus S)
      + \overmu_2(1_\bbR \oplus S)
    \Bigr)\\
  &\qquad
    - 4\bigl\langle
      \Re(S_+)^*\Re(S_-),\, \Im(S_+)^*\Im(S_-)
    \bigr\rangle\\
  &\qquad
    + 4\bigl\langle
      \Re(S_+)^*\Im(S_-),\, \Im(S_+)^*\Re(S_-)
    \bigr\rangle
\end{aligned}
\end{equation}
Moreover, \eqref{eq:over_mu_3_via_J_0} is equivalent to
\eqref{eq:over_mu_3_=_over_mu_mix} by \eqref{eq:n_=_9_binom_1} and
\eqref{eq:over_mu_mix}.

In the reducible case, an eigenvalue branch of \(-\underK^2\) of
multiplicity greater than two on some connected component of
\(\n \setminus \ram(\overK^2)\) is impossible, since for a generic
\(4\)-tuple
\[
  \bigl(
    (S_+, S_-),\,
    (\widetilde S_+, \widetilde S_-),\,
    (S^\sharp_+, S^\sharp_-),\,
    (S'_+, S'_-)
  \bigr)
\]
of pairs of oppositely handed spinors, the common fixed point group
\(\Fix(\,\cdot\,, \Spin(8))\) is trivial
\[
  \Spin(8)
  \stackrel{\Fix(S_+\oplus S_-)}{\longrightarrow} \Gtwo
  \stackrel{\Fix(\widetilde S_+\oplus \widetilde S_-)}{\longrightarrow}
  \SU(2)
  \stackrel{\Fix(S_+^\sharp\oplus S_-^\sharp)}{\longrightarrow} \{e\}
\]
in fact already for the first three pairs. Likewise,
\[
  8 \cdot 9 - 28 = 44 < 64
\]
while \(\dim \v \ge 64\) when \(\v\) contains at least two copies of
\(\bbO^2 \otimes_\bbR \bbC\). Thus the fundamental inequality
\eqref{eq:dimensions_abschaetzung_1} holds. Hence
Corollary~\ref{co:m_geq_4} provides us with a second argument showing
that there are no eigenvalue branches of multiplicity greater than two
if \(\v\) is reducible.

Moreover, \(\nu = 1\) is not a global eigenvalue of \(-\underK^2\) in the
irreducible case, so the \(J^2\)-condition does not hold. Hence, if
\(\v\) is reducible, Corollary~\ref{co:reducible_Clifford_module} rules
out the possibility that \(\nu = 1\) is a global eigenvalue of
\(-\underK^2\). Moreover, \(0\) is a global eigenvalue of multiplicity
\(m_0 = 1\) by Corollary~\ref{co:mu_=_constant}. Therefore the remaining
multiplicity splits into four distinct nonconstant eigenvalue branches
of multiplicity \(2\). This completes the argument for \(n = 9\).
\end{proof}

\subsection{\texorpdfstring{Eigenvalue branches of \(-\underK^2\) for
    \(\dim \z \ge 10\)}{Eigenvalue branches of (-K^2) for dim(z) >= 10}}
\label{se:n_ge_10}
The eightfold periodicity of real Clifford algebras asserts that
\[
\Cl(\bbR^{n})
\cong
\Cl(\bbR^8) \hat\otimes_{\bbR} \Cl(\bbR^{n-8})
\]
as graded algebras for all \(n \ge 8\) where \(\hat\otimes_{\bbR}\) denotes the graded tensor product
(see \cite[Ch.~I, Thm.~4.3]{LM}). Here \(\Cl(\bbR^8)\) is (as an ungraded \(\bbR\)-algebra)
isomorphic to \(\Mat_{16}(\bbR)\).

At the level of Clifford modules, given an \(\Cl(\bbR^{n})\)-module \(\v\),
there exists an \(\Cl(\bbR^{n-8})\)-module \(\v'\) such that
\(\v \cong \bbR^{16} \otimes_{\bbR} \v'\) as real vector spaces;
cf.~\cite[Ch.~I, Thm.~5.8]{LM}. However, the precise statement of
the following corollary is not made explicit in the aforementioned reference.

\begin{corollary}[{
  Corollary to \protect\cite[Ch.~I, Thm.~4.3]{LM} and
  \protect\cite[Ch.~I, Thm.~5.8]{LM}}
]
\label{co:periodicity}
Let \(\z\) be a Euclidean vector space with \(\dim \z \ge 8\), and fix an
orthogonal splitting \(\z \cong \bbO \oplus \z'\). Let \(\v\) be an orthogonal
\(\Cl(\z)\)-module. Then there exists an orthogonal \(\Cl(\z')\)-module \(\v'\),
unique up to isomorphism, and an isomorphism of orthogonal \(\Cl(\z)\)-modules
\[
  \v \cong \bbO^{2} \,\hat\otimes_{\bbR}\, \v'
\]
where \(\bbO^{2}\) carries the \(\bbZ_2\)-graded orthogonal \(\Cl(\bbO)\)-module
structure from \eqref{eq:Clifford_multiplication_for_n_=_8} via the canonical orthogonal
splitting \(\bbO^2 = (\bbO\oplus 0)\oplus(0\oplus \bbO)\), and the right-hand side is 
the orthogonal \(\Cl(\z))\)-module from \eqref{eq:tensor_inner_product},
\eqref{eq:tensor_product_multiplication} under the canonical identification 
\(\Cl(\z) \cong \Cl(\bbO)\,\hat\otimes_{\bbR}\,\Cl(\z')\).
\end{corollary}

\begin{proof}
By decomposing \(\v\) into irreducible summands, it suffices to assume
that \(\v\) is an irreducible orthogonal \(\Cl(\z)\)-module. Let \(n \coloneqq \dim \z\ge 8\);
then \(\dim \z' = n - 8 \ge 0\). The dimension \(d_n\) of a minimal
(hence irreducible) \(\Cl(\bbR^n)\)-module satisfies \(d_{n} = 16\,d_{n-8}\)
for all \(n \ge 8\) by \cite[Ch.~I, Thm.~5.8]{LM}. Moreover, there is,
up to isomorphism, a unique irreducible orthogonal \(\Cl(\bbO)\)-module of
dimension \(16\), namely \(\bbO^2\); see Section~\ref{se:n_=_8}. Thus
the natural splitting \(\bbO^2 = \bbO \oplus \bbO\) yields a natural
\(\bbZ_2\)-grading as an orthogonal \(\Cl(\bbO)\)-module.

Moreover, any \(\Cl(\z')\)-module \(\v'\) of minimal
dimension \(d_{n-8}\) gives a \(\Cl(\z)\)-module
\(\v \coloneqq \bbO^2 \hat\otimes_\bbR \v'\) of minimal dimension
\(d_{n} = 16 d_{n-8}\) (with the inner product defined
by~\eqref{eq:tensor_inner_product} and the orthogonal Clifford module
structure defined by~\eqref{eq:tensor_product_multiplication}).

If \(n \not\equiv 3 \pmod{4}\), then there is only one irreducible
\(\Cl(\z)\)-module \(\v\) and one irreducible
\(\Cl(\z')\)-module \(\v'\) (up to isomorphism). Hence, by uniqueness,
\(\v \cong \bbO^{2} \,\hat\otimes_{\bbR}\, \v'\).

If \(n \equiv 3 \pmod{4}\), then there are two different irreducible
\(\Cl(\z)\)-modules \(\v_+\) and \(\v_-\) of dimension \(d_{n}\), and there
are two different irreducible \(\Cl(\z')\)-modules \(\v'_+\) and \(\v'_-\)
of dimension \(d_{n-8}\). These can be distinguished by the action of the
volume elements \(\omega_{n}\) and \(\omega_{n-8}\): we have
\(\omega_{n} \cbullet = \pm \Id\) on \(\v_\pm\) and
\(\omega_{n-8} \cbullet = \pm \Id\) on \(\v'_\pm\) \cite[Prop.~5.9]{LM}.

Similarly, \(\omega_8^2 = 1\), and the natural decomposition
\(\bbO^2 = \bbO \oplus \bbO\) is the decomposition into the \(\pm 1\)-eigenspaces
of \(\omega_8 \cbullet\). Therefore, for all \(S_+ \oplus S_- \in \bbO \oplus \bbO\)
and all \(S' \in \v'\), we have
\[
  (\omega_8 \hat\otimes \omega_{n-8})\cbullet\bigl((S_+ \oplus S_-)\otimes S'\bigr)
  = (\omega_8 \cbullet S_+) \otimes (\omega_{n-8} \cbullet S')
    \;\oplus\;
    \bigl((-1)^{n-8}(\omega_8 \cbullet S_-)\otimes(\omega_{n-8} \cbullet S')\bigr)
\]
according to~\eqref{eq:tensor_product_multiplication}. Here the
factor \((-1)^{n-8}\) arises because \(\omega_{n-8} = e_1\cdots e_{n-8}\) 
is a product of \(n-8\) unit vectors \(e_i \in \z'\), and 
iterating~\eqref{eq:tensor_product_multiplication} contributes a sign
\((-1)^{\deg(S_\pm)}\) each time a vector from \(\z'\) acts. Since \(S_+\) is even
and \(S_-\) is odd with respect to the \(\bbZ_2\)-grading on \(\bbO^2\), the total
sign on the odd summand is \((-1)^{n-8}\). In particular, if \(n \equiv 3 \pmod 4\),
then \(n-8\) is odd and hence \((-1)^{n-8} = -1\). Furthermore, \(\omega_8 \cbullet S_+ =
S_+\) and \(\omega_8 \cbullet S_- = -S_-\). Therefore,
\[
  (\omega_8 \hat\otimes \omega_{n-8})\cbullet
  \bigl((S_+ \oplus S_-)\otimes S'\bigr)
  = (S_+ \oplus S_-)\otimes (\omega_{n-8} \cbullet S')
\]

On the other hand, if we orient \(\z = \bbO \oplus \z'\) by the
direct-sum orientation, then the volume elements satisfy
\begin{align*}
  \omega_n = e_1\cdots e_n
  &= (e_1\cdots e_8)(e_9\cdots e_n) \\
  &= \omega_8 \hat\otimes \omega_{n-8}
\end{align*}
Thus the above shows that \(\omega_n \cbullet = \Id \otimes \omega_{n-8} \cbullet\).
Hence \(\v_\pm \coloneqq \bbO^2 \hat\otimes_\bbR \v'_\pm\) yields
the two different irreducible \(\Cl(\z)\)-modules \(\v_\pm\) in the case
\(n \equiv 3 \pmod{4}\). This finishes the proof.
\end{proof}

By Proposition~\ref{p:n_=_8}, \(1\) is not a global eigenvalue of
\(-\underK^2\) for \(n = 8\). By eightfold periodicity, the same holds
for all \(n \ge 8\):

\begin{proposition}\label{p:n_ge_8}
Suppose that \(\dim \z \ge 8\). Then \(1\) is not a global eigenvalue
of \(-\underK^2\).
\end{proposition}

\begin{proof}
Fix an orthogonal splitting \(\z = \bbO \oplus \z'\). By the
eightfold periodicity of Clifford algebras together with
Corollary~\ref{co:periodicity}, there exists an orthogonal \(\Cl(\z')\)-module
\(\v'\) such that \(\v \cong \bbO^2 \hat\otimes_\bbR \v'\).
Furthermore, by Proposition~\ref{p:n_=_8}, \(1\) is not a global
eigenvalue of \(-\underK_{\bbO^2}^2\). By
Corollary~\ref{co:tensor_product}, it cannot be a global eigenvalue
of \(-\underK^2\).
\end{proof}

In particular, for \(n = 9\) we recover the fact that \(1\) is not a global eigenvalue,
in accordance with Proposition~\ref{p:n_=_9} (recall that for \(\dim \z = 9\)
there are three nonconstant eigenvalue branches of \(-\underK^2\) in the
irreducible case, and four nonconstant eigenvalue branches otherwise).

\begin{proposition}\label{p:mult_ge_10}
Suppose that \(n \coloneqq \dim \z \ge 10\). Then every nonzero eigenvalue
branch \(\undermu\) of \(-\underK^2\) on \(\n \setminus \ram(\overK^2)\) has
multiplicity \(m_{\undermu}=2\).
\end{proposition}

\begin{proof}
Write \(n = 8n_1 + n_2\) with \(n_2 \in \{0,\dots,7\}\). By eightfold
periodicity, every nontrivial \(\Cl(\bbR^n)\)-module has real dimension at
least \(16^{\,n_1} d_{n_2}\), where
\[
  d_0 = 1,\quad d_1 = 2,\quad d_2 = d_3 = 4,\quad d_4 = d_5 = d_6 = d_7 = 8
\]
In particular,
\[
  \dim \v \ge
  \begin{cases}
    64  & 10 \le n \le 11\\
    128 & 12 \le n \le 15\\
    256 & 16 \le n \le 17
  \end{cases}
  \qquad\text{and}\qquad
  8n - 28 \le
  \begin{cases}
    60  & 10 \le n \le 11\\
    92  & 12 \le n \le 15\\
    108 & 16 \le n \le 17
  \end{cases}
\]
Hence \(\dim \v > 8n - 28\) for \(10 \le n \le 17\), establishing
\eqref{eq:dimensions_abschaetzung_1}. Replacing \(n\) by \(n+8\) multiplies
\(\dim\v\) by \(16\) while increasing \(8n-28\) by \(64\), so the inequality
persists for all \(n \ge 10\) by iteration.

By Corollary~\ref{co:m_geq_4},  every nonzero eigenvalue branch
of \(- \underK^2\) has multiplicity \(2\).
\end{proof}

\section{\texorpdfstring{The natural \(\mathfrak{C}_0\)-structure}
{The natural C0-structure}}\label{se:C0}
Let \(\z\) be a Euclidean space of positive dimension, let \(\v\)  be a 
nontrivial orthogonal \(\Cl(\z)\)-module and \(\n \coloneqq \z \oplus \v\) be the 
orthogonal Lie algebra associated with the pair \((\z,\v)\). Following 
\cite[Ch.~3.3]{BTV}, we consider the following subspaces of \(\n\)
\begin{align}\label{eq:def_n_3_X}
  (\n_3)_X &\coloneqq \bbR X_\z \oplus \bbR X_\v
    \oplus \bbR X_\z \cbullet X_\v \\
  \notag
           &\rotatebox[origin=c]{270}{$\subseteq$} \\
  \label{eq:def_hat_z_X}
  \Z_X &\coloneqq \z \oplus \bigl(\z \cbullet X_\v
    + \z \cbullet X_\z \cbullet X_\v\bigr) \subseteq \n\\
  \notag
           &\rotatebox[origin=c]{90}{$\subseteq$} \\
  \label{eq:def_H_X}
  \H_X &\coloneqq \h_X \oplus \bigl(\h_X \cbullet X_\v
    + \h_X \cbullet X_\z \cbullet X_\v\bigr)
\end{align}
defined for each \(X = X_\z \oplus X_\v \in \n = \z \oplus \v\) such
that \(X_\z \neq 0\) and \(X_\v \neq 0\). The direct-sum decompositions
indicated above hold because
\(\z \cbullet X_\v + \z \cbullet X_\z \cbullet X_\v \subseteq \v\) and
\(\langle X_\z \cbullet X_\v, X_\v \rangle = 0\).
Also note that \((\n_3)_X\) is the Lie algebra of an embedded,
totally geodesic Heisenberg subgroup \(N^3 \subseteq N\).

\begin{lemma}\label{le:decomposition_of_n}
\begin{enumerate}
\item We have
\begin{equation}\label{eq:z_X_perp}
  (\Z_X)^\perp
  = \Kern(X_\v \spinprod) \cap \Kern\bigl((X_\z \cbullet X_\v)\spinprod\bigr)
  \subseteq \v
\end{equation}
i.e., \((\Z_X)^\perp\) is the intersection of the kernels of
\(X_\v \spinprod\colon \v \to \z\) and \((X_\z \cbullet X_\v)\spinprod\colon
\v \to \z\), see~\eqref{eq:spinor_product_2}. Hence, there is an orthogonal
decomposition
\begin{equation}\label{eq:decomposition_of_n_1}
  \n = \Z_X \oplus
  \bigl(\Kern(X_\v \spinprod) \cap
  \Kern\bigl((X_\z \cbullet X_\v)\spinprod\bigr)\bigr)
\end{equation}

\item There is an orthogonal direct-sum decomposition
\begin{equation}\label{eq:decomposition_of_z}
  \Z_X = (\n_3)_X \oplus \H_X
\end{equation}
\end{enumerate}
In particular, for each \(X \in \n = \z \oplus \v\) with \(X_\z \neq 0\)
and \(X_\v \neq 0\) there is a pointwise orthogonal direct-sum
decomposition
\begin{equation}\label{eq:decomposition_of_n_2}
  \n = (\n_3)_X \oplus \H_X \oplus (\Z_X)^\perp
\end{equation}
with \((\Z_X)^\perp\) given by \eqref{eq:z_X_perp}.
\end{lemma}

\begin{proof}
  Since \(\z \subseteq \Z_X\), the orthogonal complement
  \((\Z_X)^\perp\) lies in \(\v\). Moreover, from~\eqref{eq:Kern_spin_product}
  we have
  \begin{align*}
    \z \cbullet X_\v + \z \cbullet X_\z \cbullet X_\v
      &= \Kern(X_\v \spinprod)^\perp +
         \Kern\bigl((X_\z \cbullet X_\v)\spinprod\bigr)^\perp \\
      &= \Bigl(\Kern\bigl(X_\v \spinprod\bigr) \cap
         \Kern\bigl((X_\z \cbullet X_\v)\spinprod\bigr)\Bigr)^\perp
  \end{align*}
  Taking orthogonal complements yields~\eqref{eq:decomposition_of_n_1}.

For~\eqref{eq:decomposition_of_z}, the defining equations
\eqref{eq:def_n_3_X} and \eqref{eq:def_H_X} show that
\(\Z_X = (\n_3)_X + \H_X\). To see that this sum is orthogonal, and hence
direct, let \(Z \in \h_X\). Then \(\langle X_\z, Z\rangle = 0\) by
definition of \(\h_X\). Moreover, using the skew-symmetry
\eqref{eq:Clifford_module_2} and the metric compatibility
\eqref{eq:Clifford_module_4}, we obtain
\begin{align*}
  \langle X_\v,\, Z \cbullet X_\v \rangle
  &\stackrel{\eqref{eq:Clifford_module_2}}{=} 0 =  \langle X_\z \cbullet X_\v,\, X_\v \rangle\\
  \langle X_\v,\, Z \cbullet X_\z \cbullet X_\v \rangle
  &\stackrel{\eqref{eq:Clifford_module_2}}{=}
    -\,\langle Z \cbullet X_\v,\, X_\z \cbullet X_\v \rangle \\
  &\stackrel{\eqref{eq:Clifford_module_4}}{=}
    -\,\langle X_\z, Z \rangle\,\norm{X_\v}^2 = 0 \\
  \langle X_\z \cbullet X_\v,\, Z \cbullet X_\v \rangle
  &\stackrel{\eqref{eq:Clifford_module_4}}{=}
    \langle X_\z, Z \rangle\,\norm{X_\v}^2 = 0 \\
  \langle X_\z \cbullet X_\v,\, Z \cbullet X_\z \cbullet X_\v \rangle
  &\stackrel{\eqref{eq:Clifford_module_2}}{=} 0
\end{align*}
This proves~\eqref{eq:decomposition_of_z} and~\eqref{eq:decomposition_of_n_2}.

\end{proof}
Following \cite{BTV}, for every \(X = X_\z \oplus X_\v \in \n\) with
\(X_\z \neq 0\) and \(X_\v \neq 0\), we consider the operator
\(\C(X)\colon \n \to \n\) defined with respect to the direct-sum
decomposition \(\n = (\n_3)_X \oplus \H_X \oplus (\Z_X)^\perp\)
by
\begin{equation}\label{eq:C_0_structure_1}
  Y = Y_\z \oplus Y_\v \longmapsto
  \begin{cases}
    -\dfrac{3}{2}\, X_\z \cbullet Y_\v
    - \dfrac{1}{2}\, Y_\z \cbullet X_\v
    + \dfrac{1}{2}\, X_\v \spinprod Y_\v
      & \text{if } Y \in \H_X \\
      \dfrac{1}{2}\, X_\z \cbullet Y_\v
    - \dfrac{1}{2}\, Y_\z \cbullet X_\v
    + \dfrac{1}{2}\, X_\v \spinprod Y_\v
      & \text{if } Y \in (\n_3)_X \oplus (\Z_X)^\perp
  \end{cases}
\end{equation}
and extending linearly to all of \(\n\). Since \(X \in (\n_3)_X\), it is
immediate from the definition that \(\C(X)X = 0\).

\begin{theorem}[{\cite[Ch.~3.5]{BTV}}]\label{th:C0}
For every geodesic \(\gamma\colon \bbR \to N\), the map
\begin{equation}\label{eq:C_0_structure_2}
  \bbR \longrightarrow \End{\gamma^*TN},\qquad t \longmapsto \C_\gamma(t)
\end{equation}
with \(\C_\gamma(t)\in \End{T_{\gamma(t)}N}\) characterized by
\begin{equation}\label{eq:C_0_structure_3}
  \C_\gamma(t)\circ\rmL_{\gamma(t)\,*}^{-1}
  = \rmL_{\gamma(t)\,*}
  \C\!\left(\rmL_{\gamma(t)\,*}^{-1}\dot{\gamma}(t)\right)
\end{equation}
where the right-hand side is defined by means of~\eqref{eq:C_0_structure_1},
defines a parallel skew-symmetric endomorphism field along \(\gamma\)
such that \(\scrR^{(1)}_\gamma =
[\C_\gamma(t),\,\scrR_\gamma(t)]\) for all \(t\in\bbR\).
\end{theorem}

Appendix~\ref{se:C0_proof} provides a proof of Theorem~\ref{th:C0} from
a slightly different point of view than that in \cite{BTV}. For brevity,
we restrict our attention to the open dense subset where
\(X_\v \neq 0\) and \(X_\z \neq 0\) (which is enough for the
purpose of this article). How to extend the above definition
of \(\C(X)\) to the cases \(X_\z = 0\) or \(X_\v = 0\) is
explained in \cite[p.~44--47]{BTV}.

\subsection{The pointwise factors of the minimal polynomial}
\label{se:splitting}

In what follows, we assume \(X_\z \neq 0\) and \(X_\v \neq 0\).
By Proposition~\ref{p:overK}, every eigenvalue \(\undermu\) of
\(-\underK(X)^2\) satisfies \(0 \le \undermu \le 1\). In addition,
\(\underK(X)X_\z = 0\). Moreover, \(\h_X\) decomposes into eigenspaces of
\(-\underK(X)^2\):
\begin{equation}\label{eq:h-splitting}
  \h_X
  = \bigoplus_{\undermu \in \Spec\bigl(-\underK(X)^2\bigr)}
    (\h_\undermu)_X
\end{equation}

By definition, \(\underK(X)^2|_{(\h_\undermu)_X} = -\undermu\,\Id\) for
each \(\undermu \in \Spec\bigl(-\underK(X)^2\bigr)\). Set
\begin{align}
  \label{eq:def_H_mu_X}
  (\H_\undermu)_X &\coloneqq
  (\h_\undermu)_X
  \oplus \bigl((\h_\undermu)_X \cbullet X_\v\bigr)
  \oplus \bigl((\h_\undermu)_X \cbullet X_\z \cbullet X_\v\bigr),
  \quad\text{for } 0 \le \undermu < 1 \\[0.3em]
  \label{eq:def_H_1_X}
  (\H_1)_X &\coloneqq
  (\h_1)_X \oplus \bigl((\h_1)_X \cbullet X_\v\bigr)
\end{align}

\begin{lemma}\label{le:decomposition_of_h}
The right-hand sides of~\eqref{eq:def_H_mu_X} and~\eqref{eq:def_H_1_X}
are direct sums, and
\begin{equation}\label{eq:hat_h-split}
  \H_X
  = \bigoplus_{\undermu \in \Spec\bigl(-\underK(X)^2\bigr)}
    (\H_\undermu)_X
\end{equation}
\end{lemma}

\begin{proof}
  By definition, \(X_\z \perp \h_X\). Also, since the map
  \(\z \to \v\), \(Z \mapsto Z \cbullet X_\v\), is injective, both
  \[
    \bigoplus_{\undermu} (\h_\undermu)_X \cbullet X_\v
    \quad\text{and}\quad
    \bigoplus_{\undermu < 1} (\h_\undermu)_X \cbullet X_\z \cbullet X_\v
  \]
  are direct sums. To show that~\eqref{eq:hat_h-split} is a direct sum as
  well, first note that the intersections are trivial: suppose there
  exist \(Z_1 \in \bigoplus_{\undermu \neq 1} (\h_\undermu)_X\) and
  \(Z_2 \in \h_X\) such that
  \(Z_1 \cbullet X_\z \cbullet X_\v = Z_2 \cbullet X_\v\). By setting
  \(Z \coloneqq Z_1\) and \(\tilde Z \coloneqq Z_2\), we see
  that~\eqref{eq:mu_=_1} holds. In particular, \(Z_1 \in (\h_1)_X\) by
  Proposition~\ref{p:overK}. Since, on the other hand,
  \(Z_1 \in \bigoplus_{\undermu \neq 1} (\h_\undermu)_X\), we conclude
  \(Z_1 = Z_2 = 0\).

  To see that the sum is \(\H_X\), use~\eqref{eq:def_H_X}
  and~\eqref{eq:h-splitting} to write
  \[
    \H_X
    = \bigoplus_{\undermu} (\h_\undermu)_X
      \oplus
      \bigoplus_{\undermu} \bigl((\h_\undermu)_X \cbullet X_\v\bigr)
      \oplus
      \bigoplus_{\undermu < 1}
      \bigl((\h_\undermu)_X \cbullet X_\z \cbullet X_\v\bigr)
  \]
  where the last sum omits \(\undermu = 1\) by~\eqref{eq:mu_=_1}.
  Grouping the terms yields
  \(\H_X = (\H_1)_X \oplus \bigoplus_{\undermu < 1} (\H_\undermu)_X\), as
  claimed.
\end{proof}

Moreover, the subspace \(X_\z^\perp\) and the eigenvalues \(\undermu(X)\)
(and hence also the eigenspaces \((\h_\undermu)_X\)) depend only on the
\(2\)-plane \(\mathrm{span}_\bbR\{X_\z, X_\v\} \subseteq \n\) generated
by \(X_\z\) and \(X_\v\). The same holds for the splittings
\eqref{eq:decomposition_of_n_2}, \eqref{eq:h-splitting}, and
\eqref{eq:hat_h-split}. In particular, all preceding algebraic
constructions are invariant under rescalings \(X_\z \mapsto a X_\z\) and
\(X_\v \mapsto b X_\v\) with \(a,b \in \bbR\setminus\{0\}\).

We now show that the splittings~\eqref{eq:decomposition_of_n_2}
and~\eqref{eq:hat_h-split} are invariant under both \(\scrR(X)\), as
given in~\eqref{eq:sym_curv_2}, and \(\C(X)\), as defined
in~\eqref{eq:C_0_structure_1}.

First, consider the summands \((\H_\undermu)_X\). Let
\(Z \in (\h_\undermu)_X\). Using~\eqref{eq:sym_curv_2} and
\eqref{eq:def_K_2} (cf.~\cite[p.~39]{BTV}) we find
\begin{align}\label{eq:R_X_5}
  \scrR(X)Z
    &= \tfrac34\, Z \cbullet X_\z \cbullet X_\v
       + \tfrac14\, \norm{X_\v}^2 Z \\[0.3em]
  \label{eq:R_X_6}
  \scrR(X)(Z \cbullet X_\v)
    &= \tfrac14\bigl(\norm{X_\z}^2 - 3 \norm{X_\v}^2\bigr)
       Z \cbullet X_\v
       - \tfrac34\, \overK(X) Z \\[0.3em]
  \label{eq:R_X_7}
  \scrR(X)(Z \cbullet X_\z \cbullet X_\v)
    &= \tfrac34\, \norm{X_\v}^2 \norm{X_\z}^2 Z
       - \tfrac34\, \overK(X) Z \cbullet X_\v
       + \tfrac14\, \norm{X_\z}^2 Z \cbullet X_\z \cbullet X_\v
\end{align}
Since \((\h_\undermu)_X\) is invariant under \(\overK(X)\), it follows
that \((\H_\undermu)_X\) is invariant under \(\scrR(X)\).

Similarly, from~\eqref{eq:C_0_structure_1} we obtain
(cf.~\cite[pp.~57--58]{BTV})
\begin{align}\label{eq:C_X_5}
  \C(X)Z
    &= -\tfrac12 Z \cbullet X_\v \\[0.3em]
  \label{eq:C_X_6}
  \C(X)(Z \cbullet X_\v)
    &\stackrel{\eqref{eq:Clifford_module_3},
    \eqref{eq:Clifford_module_5}}{=}
      \tfrac32 Z \cbullet X_\z \cbullet X_\v
      + \tfrac12 \norm{X_\v}^2 Z \\[0.3em]
  \label{eq:C_X_7}
  \C(X)(Z \cbullet X_\z \cbullet X_\v)
    &\stackrel{\eqref{eq:def_K_2}}{=}
      \tfrac12 \overK(X) Z
      - \tfrac32 \norm{X_\z}^2 Z \cbullet X_\v
\end{align}
for each \(Z \in (\h_\undermu)_X\). Therefore \((\H_\undermu)_X\), and
hence \(\H_X\), is invariant under \(\C(X)\).

Let \(\{Z_1, \dotsc, Z_{m_\undermu}\}\) be a basis of
\((\h_\undermu)_X\). Since the sums in~\eqref{eq:def_H_mu_X} are direct by
Lemma~\ref{le:decomposition_of_h}, a basis of \((\H_\undermu)_X\) for
\(\undermu < 1\) is
\[
  \bigl\{
    Z_1, \dotsc, Z_{m_\undermu},\,
    Z_1 \cbullet X_\v, \dotsc, Z_{m_\undermu} \cbullet X_\v,\,
    Z_1 \cbullet X_\z \cbullet X_\v, \dotsc,
    Z_{m_\undermu} \cbullet X_\z \cbullet X_\v
  \bigr\}
\]
If \(\undermu = 1\), then a basis of \((\H_1)_X\) is
\[
  \bigl\{
    Z_1, \dotsc, Z_{m_1},\,
    Z_1 \cbullet X_\v, \dotsc, Z_{m_1} \cbullet X_\v
  \bigr\}
\]

A decomposition of \((\H_\undermu)_X\) into smaller subspaces, each
invariant under both \(\C(X)\) and \(\scrR(X)\), is obtained as follows.
Fix \(X \in \n\) with \(X_\v \neq 0\) and \(X_\z \neq 0\), and assume
\(0 < \undermu < 1\).

For each nonzero \(Z \in (\h_\undermu)_X\), set
\begin{equation}\label{eq:q_2}
  \H^Z_X \coloneqq \mathrm{span}_\bbR\bigl\{
    Z,\ \underK(X)Z,\ Z \cbullet X_\v,\ \underK(X)Z \cbullet X_\v,\ 
    Z \cbullet X_\z \cbullet X_\v,\ 
    \underK(X)Z \cbullet X_\z \cbullet X_\v
  \bigr\}
\end{equation}
Then \(\H^Z_X\) is a \(6\)-dimensional subspace of \((\H_\undermu)_X\)
that is invariant under both \(\scrR(X)\) and \(\C(X)\). The operator
\[
  J \coloneqq \tfrac{1}{\sqrt{\undermu}}\,
  \underK(X)\big|_{(\h_\undermu)_X}
\]
is a complex structure on \((\h_\undermu)_X\).

Hence, writing \(m_\undermu \coloneqq \dim\bigl((\h_\undermu)_X\bigr)\),
we obtain
\[
  \dim\bigl((\H_\undermu)_X\bigr) = 3m_\undermu
  \quad\text{and}\quad
  \dim\bigl((\H_\undermu)_X\bigr) \equiv 0 \pmod 6
\]
Moreover, for any \(J\)-complex basis
\(\{Z_1,\dotsc,Z_{m_\undermu/2}\}\) of \((\h_\undermu)_X\),
\[
  (\H_\undermu)_X \cong
  \bigoplus_{i=1}^{m_\undermu/2} \H^{Z_i}_X
\]

The corresponding \(6 \times 6\) matrices with respect to the basis
\[
  \bigl(
    Z,\ \underK(X)Z,\ Z \cbullet X_\v,\ \underK(X)Z \cbullet X_\v,\ 
    Z \cbullet X_\z \cbullet X_\v,\ 
    \underK(X)Z \cbullet X_\z \cbullet X_\v
  \bigr)
\]
are
\begin{align}\label{eq:R_mu_matrix}
  \scrR_\undermu(z,v) &\coloneqq \tfrac14
  \begin{pmatrix}
    v^2 & 0 & 0 & 3\undermu v^2 z & 3 v^2 z^2 & 0 \\
    0 & v^2 & -3 v^2 z & 0 & 0 & 3 v^2 z^2 \\
    0 & 0 & z^2 - 3 v^2 & 0 & 0 & 3 \undermu v^2 z \\
    0 & 0 & 0 & z^2 - 3 v^2 & -3 v^2 z & 0 \\
    3 & 0 & 0 & 0 & z^2 & 0 \\
    0 & 3 & 0 & 0 & 0 & z^2
  \end{pmatrix}
  \\[1ex]
  \label{eq:C_mu_matrix}
  \C_\undermu(z,v) &\coloneqq \tfrac12
  \begin{pmatrix}
    0  & 0  & v^2 & 0   & 0      & -\undermu v^2 z \\
    0  & 0  & 0   & v^2 & v^2 z  & 0 \\
    -1 & 0  & 0   & 0   & -3 z^2 & 0 \\
    0  & -1 & 0   & 0   & 0      & -3 z^2 \\
    0  & 0  & 3   & 0   & 0      & 0 \\
    0  & 0  & 0   & 3   & 0      & 0
  \end{pmatrix}
\end{align}
where \(z \coloneqq \norm{X_\z}\) and \(v \coloneqq \norm{X_\v}\). Also,
for every triple \((\undermu,z,v) \in \bbR^3\), define
\(P_\undermu(z,v) \in \bbR[\lambda]\) by
\begin{equation}\label{eq:P_mu_param}
\begin{aligned}
  P_\undermu(z,v) \coloneqq\;& \lambda^6
    + \Bigl(\tfrac{27}{2} z^2 + \tfrac{3}{2} v^2\Bigr) \lambda^4
    + \Bigl(\tfrac{729}{16} z^4 + \tfrac{81}{8} z^2 v^2
       + \tfrac{9}{16} v^4\Bigr) \lambda^2 \\
    &\quad+ \tfrac{729}{16} z^6 + \tfrac{243}{16} z^4 v^2
       + \Bigl(\tfrac{27}{16} - \tfrac{243}{64} \undermu\Bigr) z^2 v^4
       + \tfrac{1}{16} v^6
\end{aligned}
\end{equation}
in accordance with~\eqref{eq:def_P_mu_1}, and set
\begin{equation}\label{eq:def_tilde_P_mu}
  \tilde P_\undermu(z,v) \coloneqq \lambda\,P_\undermu(z,v)
\end{equation}

\begin{lemma}\label{le:maple_P_mu}
There exists an open dense subset \(\tilde U \subseteq \bbR^3\) such that
\(\tilde P_\undermu(z,v)\) is the minimal annihilating polynomial of
\(\scrR_\undermu(z,v)\) with respect to the endomorphism
\[
  \ad\bigl(\C_\undermu(z,v)\bigr)\colon \Mat_6(\bbR) \longrightarrow
  \Mat_6(\bbR)
\]
if and only if \((\undermu,z,v) \in \tilde U\). In particular,
\begin{equation}\label{eq:Eval_q_mu}
  \tilde P_\undermu(z,v)\bigl(\ad(\C_\undermu(z,v))\bigr)\,
  \scrR_\undermu(z,v) = 0
\end{equation}
for all \((\undermu,z,v) \in \bbR^3\).
\end{lemma}

\begin{proof}
\begin{verbatim}
# case 0 < mu < 1
with(LinearAlgebra):
n := 6:  # block size
C := (1/2)*Matrix([
  [0, 0, v^2, 0, 0, -mu*v^2*z],
  [0, 0, 0, v^2, v^2*z, 0],
  [-1, 0, 0, 0, -3*z^2, 0],
  [0, -1, 0, 0, 0, -3*z^2],
  [0, 0, 3, 0, 0, 0],
  [0, 0, 0, 3, 0, 0]
]):
k := 7:  # target degree (guess-and-check)
R := array(0..k):
R[0] := (1/4)*Matrix([
  [v^2, 0, 0, 3*mu*v^2*z, 3*v^2*z^2, 0],
  [0, v^2, -3*v^2*z, 0, 0, 3*v^2*z^2],
  [0, 0, z^2-3*v^2, 0, 0, 3*mu*v^2*z],
  [0, 0, 0, z^2-3*v^2, -3*v^2*z, 0],
  [3, 0, 0, 0, z^2, 0],
  [0, 3, 0, 0, 0, z^2]
]):
for i from 0 to k - 1 do
  R[i+1] := simplify(C . R[i] - R[i] . C):
od:

# Build augmented system [vec(R_0) ... vec(R_(k-1)) | vec(R_k)]
flatR := Matrix(n^2, k+1):
for i from 1 to k+1 do
  for j from 1 to n do
    for l from 1 to n do
      flatR[(j-1)*n+l, i] := R[i-1][j,l]:
    od:
  od:
od:

# Solve for (a_1,...,a_k) in  R_k + a_1 R_(k-1) + ... + a_k R_0 = 0
a := LinearSolve(flatR);     # last column is RHS
a_help := array(1..k):
for i from 1 to k do
    a_help[i] := - a[k + 1 - i]:
od:
a := a_help:

# Annihilating polynomial
p := collect(lambda^k + sum(a[f]*lambda^(k - f),
                            f = 1 .. k), lambda);
\end{verbatim}

We obtain
\begin{equation}\label{eq:tilde_P_mu}
  \begin{aligned}
    p \coloneqq \lambda^{7}
    &+ \Bigl(\tfrac{27}{2} z^{2}+\tfrac{3}{2} v^{2}\Bigr)\lambda^{5}
     + \Bigl(\tfrac{729}{16} z^{4}+\tfrac{81}{8} v^{2}z^{2}
              +\tfrac{9}{16} v^{4}\Bigr)\lambda^{3}\\
    &+ \Bigl(\tfrac{729}{16} z^{6}+\tfrac{243}{16} z^{4}v^{2}
              + \tfrac{27}{16} v^{4}z^{2}+\tfrac{1}{16} v^{6}
              - \tfrac{243}{64} \undermu v^{4}z^{2}\Bigr)\lambda
  \end{aligned}
\end{equation}
which is exactly \(\tilde P_\undermu(z,v)\) from
\eqref{eq:def_tilde_P_mu}. An optional certificate is
\begin{verbatim}
simplify( R[k] + sum(a[f]*R[k - f], f = 1 .. k) );  # -> zero 6x6
\end{verbatim}
\end{proof}

Suppose that \(\undermu = 0\). Let \(X \in \n\) such that \(X_\z \neq 0\)
and \(X_\v \neq 0\). For every nonzero \(Z \in (\h_0)_X\),
\begin{equation}\label{eq:q_1}
  \H^Z_X \coloneqq \mathrm{span}_\bbR\bigl\{
    Z,\ Z \cbullet X_\v,\ Z \cbullet X_\z \cbullet X_\v
  \bigr\}
\end{equation}
is a \(3\)-dimensional subspace of \((\H_0)_X\).
By~\eqref{eq:R_X_5}--\eqref{eq:R_X_7} and~\eqref{eq:C_X_5}--\eqref{eq:C_X_7},
\(\H^Z_X\) is invariant under both \(\scrR(X)\) and \(\C(X)\). Let
\(m_0 \coloneqq \dim\bigl((\h_0)_X\bigr)\). Then
\(\dim\bigl((\H_0)_X\bigr) = 3m_0\), and in particular
\(\dim\bigl((\H_0)_X\bigr) \equiv 0 \pmod 3\), with the direct sum
decomposition
\[
  (\H_0)_X \cong \bigoplus_{i=1}^{m_0} \H^{Z_i}_X
\]

With respect to the basis
\(\bigl(Z,\ Z \cbullet X_\v,\ Z \cbullet X_\z \cbullet X_\v\bigr)\), the
corresponding \(3 \times 3\) matrices are
\begin{align}\label{eq:R_0_sharp_matrix}
  \scrR_0^\sharp(z,v) &\coloneqq \tfrac14
  \begin{pmatrix}
    v^2 & 0 & 3 v^2 z^2 \\
    0 & z^2 - 3 v^2 & 0 \\
    3 & 0 & z^2
  \end{pmatrix}
  \\[0.7ex]
  \label{eq:C_0_sharp_matrix}
  \C_0^\sharp(z,v) &\coloneqq \tfrac12
  \begin{pmatrix}
    0 & v^2 & 0 \\
    -1 & 0 & -3 z^2 \\
    0 & 3 & 0
  \end{pmatrix}
\end{align}
where we set \(z \coloneqq \norm{X_\z}\) and \(v \coloneqq \norm{X_\v}\).

For each \((z,v) \in \bbR^2\), define \(P_0^\sharp(z,v) \in \bbR[\lambda]\)
by
\begin{equation}\label{eq:P_0_sharp_param}
\begin{aligned}
  P_0^\sharp(z,v) \coloneqq\;& \lambda^4
    + \Bigl(\tfrac{45}{4} z^2 + \tfrac{5}{4} v^2\Bigr) \lambda^2
    + \tfrac{81}{4} z^4 + \tfrac{9}{2} z^2 v^2 + \tfrac14 v^4
\end{aligned}
\end{equation}
in accordance with~\eqref{eq:def_P_0_sharp}. Also let
\begin{equation}\label{eq:tilde_P_0_sharp}
  \tilde P_0^\sharp(z,v) \coloneqq \lambda\,P_0^\sharp(z,v)
\end{equation}

\begin{lemma}\label{le:maple_P_0_sharp}
There exists an open dense subset \(\tilde U_0 \subseteq \bbR^2\) such
that \(\tilde P_0^\sharp(z,v)\) is the minimal annihilating polynomial
of \(\scrR_0^\sharp(z,v)\) with respect to the linear operator
\[
  \ad\bigl(\C_0^\sharp(z,v)\bigr)\colon
  \Mat_3(\bbR) \longrightarrow \Mat_3(\bbR)
\]
if and only if \((z,v) \in \tilde U_0\). In particular,
\begin{equation}\label{eq:Eval_q_0}
  \tilde P_0^\sharp(z,v)\bigl(\ad(\C_0^\sharp(z,v))\bigr)\,
  \scrR_0^\sharp(z,v) = 0
\end{equation}
for all \((z,v) \in \bbR^2\).
\end{lemma}

\begin{proof}
\begin{verbatim}
# case mu = 0
with(LinearAlgebra):

n := 3:  # block size
k := 5:  # target degree (guess-and-check)

C := (1/2)*Matrix([
  [0, v^2, 0],
  [-1, 0, -3*z^2],
  [0, 3, 0]
]):

R := array(0..k):

R[0] := (1/4)*Matrix([
  [v^2, 0, 3*v^2*z^2],
  [0, z^2-3*v^2, 0],
  [3, 0, z^2]
]):

for i from 0 to k - 1 do
  R[i+1] := simplify(C . R[i] - R[i] . C):
od:

# Build augmented system [vec(R_0) ... vec(R_(k-1)) | vec(R_k)]
flatR := Matrix(n^2, k+1):

for i from 1 to k+1 do
  for j from 1 to n do
    for l from 1 to n do
      flatR[(j-1)*n + l, i] := R[i-1][j, l]:
    od:
  od:
od:

# Solve for (a_1,...,a_k) in R_k + a_1 R_(k-1) + ... + a_k R_0 = 0
a := LinearSolve(flatR):  # last column is RHS

# Reverse order and flip sign for the polynomial coefficients
a_help := array(1..k):
for i from 1 to k do
  a_help[i] := -a[k + 1 - i]:
od:
a := a_help:

# Annihilating polynomial
p := collect(lambda^k + sum(a[f]*lambda^(k - f), f = 1 .. k), lambda);
\end{verbatim}

\[
  \begin{aligned}
   p \coloneqq \lambda^{5}
     + \Bigl(\tfrac{45}{4} z^{2}+\tfrac{5}{4} v^{2}\Bigr) \lambda^{3}
     + \Bigl(\tfrac{81}{4} z^{4}+\tfrac{9}{2} v^{2} z^{2}+\tfrac{1}{4} v^{4}\Bigr)
       \lambda
  \end{aligned}
\]
which is \(\tilde P_0^\sharp(z,v)\) from~\eqref{eq:tilde_P_0_sharp}.
\begin{verbatim}
# optional certificate:
simplify( R[k] + sum(a[f]*R[k - f], f = 1 .. k) );  # -> zero 3x3
\end{verbatim}
\end{proof}

Suppose that \(\undermu = 1\). Let \(X \in \n\) with \(X_\z \neq 0\) and
\(X_\v \neq 0\). For every nonzero \(Z \in (\h_1)_X\),
\begin{equation}\label{eq:q_3}
  \H^Z_X \coloneqq \mathrm{span}_\bbR\bigl\{
    Z,\ \underK(X)Z,\ Z \cbullet X_\v,\ \underK(X)Z \cbullet X_\v
  \bigr\}
\end{equation}
is a \(4\)-dimensional subspace of \((\H_1)_X\) that is invariant under
both \(\scrR(X)\) and \(\C(X)\). Hence
\(J \coloneqq \underK(X)\big|_{(\h_1)_X}\) is an orthogonal complex
structure, so \(\dim\bigl((\H_1)_X\bigr) \equiv 0 \pmod 4\) and
\[
  (\H_1)_X \cong \bigoplus_{i=1}^{m_1/2} \H^{Z_i}_X
\]
for any \(J\)-complex basis \(\{Z_1,\dotsc,Z_{m_1/2}\}\) of
\((\h_1)_X\). By~\eqref{eq:R_X_5}--\eqref{eq:R_X_7} and
\eqref{eq:C_X_5}--\eqref{eq:C_X_7} in combination with the
characterization of the \(1\)-eigenspace \(\Eig_{1}(-\underK(X)^2)\)
given by~\eqref{eq:mu_=_1}, the corresponding \(4 \times 4\) matrices
(with respect to the basis
\(\bigl(Z,\ \underK(X)Z,\ Z \cbullet X_\v,\ \underK(X)Z \cbullet X_\v\bigr)\))
are
\begin{align}\label{eq:R_1_sharp_matrix}
  \scrR_1^\sharp(z,v) &\coloneqq \tfrac14
  \begin{pmatrix}
    v^2 & 0   & 0            & 3 v^2 z \\
    0   & v^2 & -3 z v^2     & 0 \\
    0   & -3z & z^2 - 3 v^2  & 0 \\
    3z  & 0   & 0            & z^2 - 3 v^2
  \end{pmatrix}
  \\[0.7ex]
  \label{eq:C_1_sharp_matrix}
  \C_1^\sharp(z,v) &\coloneqq \tfrac12
  \begin{pmatrix}
    0 & 0 & v^2 & 0 \\
    0 & 0 & 0   & v^2 \\
    -1& 0 & 0   & -3z \\
    0 & -1& 3z  & 0
  \end{pmatrix}
\end{align}
where \(z \coloneqq \norm{X_\z}\) and \(v \coloneqq \norm{X_\v}\).

For each \((z,v) \in \bbR^2\), define
\begin{equation}\label{eq:P_1_sharp_param}
  P_1^\sharp(z,v) \coloneqq \lambda^2 + \tfrac94 z^2 + v^2
\end{equation}
in accordance with~\eqref{eq:def_P_1_sharp}, and set
\begin{equation}\label{eq:tilde_P_1_sharp}
  \tilde P_1^\sharp(z,v) \coloneqq \lambda\,P_1^\sharp(z,v)
\end{equation}

\begin{lemma}\label{le:maple_P_1_sharp}
There exists an open dense subset \(\tilde U_{1} \subseteq \bbR^2\) such
that \(\tilde P_1^\sharp(z,v)\) is the minimal annihilating polynomial
of \(\scrR_1^\sharp(z,v)\) with respect to the linear operator
\[
  \ad\bigl(\C_1^\sharp(z,v)\bigr)\colon
  \Mat_4(\bbR) \longrightarrow \Mat_4(\bbR)
\]
if and only if \((z,v) \in \tilde U_{1}\). In particular,
\begin{equation}\label{eq:Eval_q_1}
  \tilde P_1^\sharp(z,v)\bigl(\ad(\C_1^\sharp(z,v))\bigr)\,
  \scrR_1^\sharp(z,v) = 0
\end{equation}
for all \((z,v) \in \bbR^2\).
\end{lemma}

\begin{proof}
\begin{verbatim}
# case mu = 1
with(LinearAlgebra):

n := 4:  # block size
k := 3:  # target degree (guess-and-check)

C := (1/2)*Matrix([
  [0, 0, v^2, 0],
  [0, 0, 0, v^2],
  [-1, 0, 0, -3*z],
  [0, -1, 3*z, 0]
]):

R := array(0..k):
R[0] := (1/4)*Matrix([
  [v^2, 0, 0, 3*z*v^2],
  [0, v^2, -3*z*v^2, 0],
  [0, -3*z, z^2-3*v^2, 0],
  [3*z, 0, 0, z^2-3*v^2]
]):
for i from 0 to k - 1 do
  R[i+1] := simplify(C . R[i] - R[i] . C):
od:

# Build augmented system [vec(R_0) ... vec(R_(k-1)) | vec(R_k)]
flatR := Matrix(n^2, k+1):

for i from 1 to k+1 do
  for j from 1 to n do
    for l from 1 to n do
      flatR[(j-1)*n + l, i] := R[i-1][j, l]:
    od:
  od:
od:

# Solve for (a_1,...,a_k) in R_k + a_1 R_(k-1) + ... + a_k R_0 = 0
a := LinearSolve(flatR):  # last column is RHS
# Reverse order and flip sign for the polynomial coefficients
a_help := array(1..k):
for i from 1 to k do
  a_help[i] := -a[k + 1 - i]:
od:
a := a_help:

# Annihilating polynomial
p := collect(lambda^k + sum(a[f]*lambda^(k - f), f = 1 .. k), lambda);
\end{verbatim}

We obtain the following result:
\[
  p \coloneqq \lambda^{3} + \Bigl(\tfrac{9}{4} z^{2} + v^{2}\Bigr)\lambda
\]
\begin{verbatim}
# optional certificate:
simplify( R[k] + sum(a[f]*R[k - f], f = 1 .. k) );  # -> zero 4x4
\end{verbatim}
\end{proof}

We also want to show that \((\n_3)_X\), defined in~\eqref{eq:def_n_3_X},
is invariant under \(\scrR(X)\) and \(\C(X)\). Here,
\eqref{eq:sym_curv_2} implies that
\begin{align}
  \label{eq:R_X_1}
  \scrR(X)X_\v
    &= \tfrac14\bigl(\norm{X_\z}^2 X_\v - \norm{X_\v}^2 X_\z\bigr) \\
  \label{eq:R_X_2}
  \scrR(X)\bigl(X_\z \cbullet X_\v\bigr)
    &= \tfrac14\bigl(\norm{X_\z}^2 - 3\norm{X_\v}^2\bigr)
      \bigl(X_\z \cbullet X_\v\bigr) \\
  \label{eq:R_X_3}
  \scrR(X)X_\z
    &= \tfrac14\bigl(\norm{X_\v}^2 X_\z - \norm{X_\z}^2 X_\v\bigr)
\end{align}

Similarly, from~\eqref{eq:C_0_structure_1} we directly see that
\begin{align}
  \label{eq:C_X_1}
  \C(X)X_\v
    &= \tfrac12 \bigl(X_\z \cbullet X_\v\bigr) \\
  \label{eq:C_X_2}
  \C(X)\bigl(X_\z \cbullet X_\v\bigr)
    &\stackrel{\eqref{eq:Clifford_module_2},
    \eqref{eq:Clifford_module_5}}{=}
      \tfrac12\bigl(\norm{X_\v}^2 X_\z - \norm{X_\z}^2 X_\v\bigr) \\
  \label{eq:C_X_3}
  \C(X)X_\z
    &= -\tfrac12 \bigl(X_\z \cbullet X_\v\bigr)
\end{align}

Thus, by~\eqref{eq:R_X_1}--\eqref{eq:R_X_3} and
\eqref{eq:C_X_1}--\eqref{eq:C_X_3}, \((\n_3)_X\) is invariant under both
\(\scrR(X)\) and \(\C(X)\). On the open set where
\(z \coloneqq \norm{X_\z} > 0\) and \(v \coloneqq \norm{X_\v} > 0\), the
corresponding \(3 \times 3\) matrices in the basis
\(\bigl(X_\v,\ X_\z \cbullet X_\v,\ X_\z\bigr)\) are
\begin{align}\label{eq:R_n_3_matrix}
  \scrR_{\n_3}(z,v) &\coloneqq \tfrac14
  \begin{pmatrix}
    z^2 & 0 & -z^2 \\
    0 & z^2 - 3 v^2 & 0 \\
    -v^2 & 0 & v^2
  \end{pmatrix}
  \\[0.7ex]
  \label{eq:C_n_3_matrix}
  \C_{\n_3}(z,v) &\coloneqq \tfrac12
  \begin{pmatrix}
    0 & -z^2 & 0 \\
    1 & 0 & -1 \\
    0 & v^2 & 0
  \end{pmatrix}
\end{align}

For each \((z,v) \in \bbR^2\), define
\begin{equation}\label{eq:P_n_3_param}
  P_{\n_3}(z,v) \coloneqq \lambda^{3} + \bigl(z^{2} + v^{2}\bigr)\lambda
\end{equation}
Since \(\norm{X}^2 = \norm{X_\z}^2 + \norm{X_\v}^2\), this agrees
with~\eqref{eq:def_P_n_3}.

\begin{lemma}\label{le:maple_P_n_3}
There exists an open dense subset \(\tilde U_{\n_3} \subseteq \bbR^2\)
such that \(P_{\n_3}(z,v)\), given by~\eqref{eq:P_n_3_param}, is the
minimal annihilating polynomial of \(\scrR_{\n_3}(z,v)\) with respect to
the linear operator
\[
  \ad\bigl(\C_{\n_3}(z,v)\bigr)\colon
  \Mat_3(\bbR) \longrightarrow \Mat_3(\bbR)
\]
if and only if \((z,v) \in \tilde U_{\n_3}\). In particular,
\begin{equation}\label{eq:Eval_q_n3}
  P_{\n_3}(z,v)\bigl(\ad(\C_{\n_3}(z,v))\bigr)\,
  \scrR_{\n_3}(z,v) = 0
\end{equation}
for all \((z,v) \in \bbR^2\).
\end{lemma}

\begin{proof}
\begin{verbatim}
# case n_3
with(LinearAlgebra):
n := 3:  # block size
k := 3:  # target degree (guess-and-check)
C := (1/2)*Matrix([
  [0, -z^2, 0],
  [1, 0, -1],
  [0, v^2, 0]
]):  # C_0-structure
R := array(0..k):
R[0] := (1/4)*Matrix([
  [z^2, 0, -z^2],
  [0, z^2-3*v^2, 0],
  [-v^2, 0, v^2]
]):  # symmetrized curvature tensor

for i from 0 to k - 1 do
  R[i+1] := simplify(C . R[i] - R[i] . C):
od:

# Build augmented system [vec(R_0) ... vec(R_(k-1)) | vec(R_k)]
flatR := Matrix(n^2, k+1):
for i from 1 to k+1 do
  for j from 1 to n do
    for l from 1 to n do
      flatR[(j-1)*n + l, i] := R[i-1][j, l]:
    od:
  od:
od:

# Solve for (a_1,...,a_k) in R_k + a_1 R_(k-1) + ... + a_k R_0 = 0
a := LinearSolve(flatR):  # last column is RHS

# Reverse order and flip sign for the polynomial coefficients
a_help := array(1..k):
for i from 1 to k do
  a_help[i] := -a[k + 1 - i]:
od:
a := a_help:

# Minimal annihilating polynomial
p := collect(lambda^k + sum(a[f]*lambda^(k - f), f = 1 .. k), lambda);
\end{verbatim}
This yields the following output:
\[
  p \coloneqq \lambda^{3} + \bigl(z^{2} + v^{2}\bigr)\lambda
\]
\begin{verbatim}
# optional certificate
simplify( R[3] + (z^2+v^2)*R[1] );     # -> zero 3x3
\end{verbatim}
This shows that \(P_{\n_3}(z,v)\bigl(\ad(\C_{\n_3}(z,v))\bigr)\,\scrR_{\n_3}(z,v) = 0\),
as claimed in~\eqref{eq:Eval_q_n3}.
\end{proof}

Finally, we consider the remaining summand \((\Z_X)^\perp\).
From~\eqref{eq:sym_curv_2} we obtain
\begin{equation}\label{eq:R_X_4}
  \scrR(X)S = \tfrac14\, \norm{X_\z}^2\, S
\end{equation}
for \(S \in (\Z_X)^\perp\).
From~\eqref{eq:C_0_structure_1} we also obtain, for \(S \in (\Z_X)^\perp\)
(so \(X_\v \spinprod S = 0\)),
\begin{equation}\label{eq:C_X_4}
  \C(X)S = \tfrac12 X_\z \cbullet S
\end{equation}
Moreover,
\[
  \begin{aligned}
    (X_\z \cbullet S) \spinprod X_\v
      &= S \spinprod (X_\z \cbullet X_\v) \\
    (X_\z \cbullet S) \spinprod (X_\z \cbullet X_\v)
      &= -\norm{X_\z}^2\, (S \spinprod X_\v)
  \end{aligned}
\]
so both expressions vanish for \(S \in (\Z_X)^\perp\), and hence
\(X_\z \cbullet S \in (\Z_X)^\perp\).

Hence there is a unique monic polynomial \(P_{\Z^\perp}(X)\) of smallest
degree such that
\begin{equation}\label{eq:local_minimal_p}
  P_{\Z^\perp}\bigl(X;\,\ad(\C(X))\bigr)\scrR(X)\big|_{(\Z_X)^\perp} = 0
\end{equation}
i.e., \(P_{\Z^\perp}(X)\) is the minimal annihilating polynomial of
\(\scrR(X)\big|_{(\Z_X)^\perp} \in \End{(\Z_X)^\perp}\) with
respect to the derivation
\[
  \ad\bigl(\C(X)\big|_{(\Z_X)^\perp}\bigr)\colon
  \End{(\Z_X)^\perp} \longrightarrow
  \End{(\Z_X)^\perp}
\]

\begin{lemma}\label{le:unimportant}
We have \(\ad(\C(X))\scrR(X) = 0\) on \((\Z_X)^\perp\). It follows that
\(P_{\Z^\perp}(X) = \lambda\) for each \(X \in \n\) with \(X_\v \neq 0\)
and \(X_\z \neq 0\), provided \((\Z_X)^\perp \neq 0\).
\end{lemma}

\begin{proof}
By~\eqref{eq:C_X_4} and~\eqref{eq:R_X_4},
\[
  \scrR(X)\big|_{(\Z_X)^\perp}
  = \tfrac14 \norm{X_\z}^2 \Id
\]
so \(\C(X)\big|_{(\Z_X)^\perp}\) and \(\scrR(X)\big|_{(\Z_X)^\perp}\)
commute. Thus
\[
  \ad(\C(X))\scrR(X)
  = \C(X)\circ \scrR(X) - \scrR(X)\circ \C(X)
  = 0
\]
on \((\Z_X)^\perp\), as claimed.
\end{proof}

Because \(P_{\n_3}(X)\) is divisible by \(P_{\Z^\perp}(X) = \lambda\),
this last factor does not influence the construction of the minimal
polynomial \(P_{\min}(N,g)\).

\begin{lemma}\label{le:gcd}
Fix pairwise distinct numbers \(\undermu_1,\dotsc,\undermu_\ell \in (0,1)\).
There exists an open dense set
\(U(\undermu_1,\dotsc,\undermu_\ell) \subseteq \bbR^2\) such that, for
every \((z,v) \in U(\undermu_1,\dotsc,\undermu_\ell)\), the polynomials
\begin{equation}\label{eq:set_of_polynomials_1}
  \mathcal{S}(\undermu_1,\dotsc,\undermu_\ell;z,v)
  \coloneqq
  \{P_{\undermu_i}(z,v)\}_{i=1}^\ell
  \cup
  \{P_0^\sharp(z,v),\,P_1^\sharp(z,v),\,P_{\n_3}(z,v)\}
  \subseteq \bbR[\lambda]
\end{equation}
(see~\eqref{eq:P_mu_param}, \eqref{eq:P_0_sharp_param},
\eqref{eq:P_1_sharp_param}, and~\eqref{eq:P_n_3_param}) are pairwise
coprime.
\end{lemma}

\begin{proof}
Let \(z > 0\) and \(v > 0\) be fixed. Two polynomials
\(P_{\undermu_{i_1}}(z,v)\) and \(P_{\undermu_{i_2}}(z,v)\) as in
\eqref{eq:P_mu_param} differ only in their constant term. More precisely,
\[
  P_{\undermu_{i_2}}(z,v) - P_{\undermu_{i_1}}(z,v)
  = \frac{243}{64}\,(\undermu_{i_1} - \undermu_{i_2})\, z^{2} v^{4}
\]
Since this is a nonzero scalar for \(i_1 \neq i_2\), we have
\(\gcd\bigl(P_{\undermu_{i_1}}(z,v),P_{\undermu_{i_2}}(z,v)\bigr) = 1\)
whenever \(z \neq 0\) and \(v \neq 0\).

We claim that \(P_\nu^\sharp(z,v) \mid P_{\nu}(z,v)\) for each
\(\nu \in \{0,1\}\).

\emph{Case \(\nu = 0\).} Set
\[
  p \coloneqq P_0(z,v)
\]
so that
\begin{equation*}
\begin{aligned}
  p = {} & \lambda^{6}
  + \left(\tfrac{27}{2} z^{2}+\tfrac{3}{2} v^{2}\right)\lambda^{4}
  + \left(\tfrac{729}{16} z^{4}+\tfrac{81}{8} z^{2}v^{2}
    + \tfrac{9}{16} v^{4}\right)\lambda^{2} \\
  &\quad + \left(\tfrac{729}{16} z^{6}+\tfrac{243}{16} z^{4} v^{2}
    + \tfrac{27}{16} z^{2} v^{4}+\tfrac{1}{16} v^{6}\right)
\end{aligned}
\end{equation*}
Moreover,
\begin{equation}\label{eq:P_0_1}
  q \coloneqq P_0^\sharp(z,v)
  = \frac{\bigl(\lambda^{2}+9 z^{2}+v^{2}\bigr)\,
    \bigl(4 \lambda^{2}+9 z^{2}+v^{2}\bigr)}{4}
\end{equation}
A direct computation yields
\[
  \frac{p}{q} = \lambda^{2}+\tfrac{9}{4}z^{2}+\tfrac{1}{4}v^{2}
\]
and hence
\begin{equation}\label{eq:P_0}
  P_0(z,v) = \bar P_0(z,v)\,P_0'(z,v)^{2}
\end{equation}
with
\[
  \bar P_0(z,v) = \lambda^2 + 9 z^2 + v^2,\qquad
  P_0'(z,v) = \lambda^2 + \tfrac{9}{4} z^2 + \tfrac{1}{4} v^2
\]
In particular, \(P_0^\sharp(z,v) \mid P_0(z,v)\).

\emph{Case \(\nu = 1\).} Set
\[
  p \coloneqq P_1(z,v)
\]
so that
\begin{equation*}
\begin{aligned}
  p = {} & \lambda^{6}
  + \left(\tfrac{27}{2} z^{2}+\tfrac{3}{2} v^{2}\right)\lambda^{4}
  + \left(\tfrac{729}{16} z^{4}+\tfrac{81}{8} z^{2}v^{2}
    + \tfrac{9}{16} v^{4}\right)\lambda^{2} \\
  &\quad + \left(\tfrac{729}{16} z^{6}+\tfrac{243}{16} z^{4} v^{2}
    - \tfrac{135}{64} z^{2} v^{4}+\tfrac{1}{16} v^{6}\right)
\end{aligned}
\end{equation*}
Moreover,
\[
  q \coloneqq P_1^\sharp(z,v) = \lambda^2 + \tfrac{9}{4} z^2 + v^2
\]
and a direct computation gives
\begin{equation}\label{eq:p_over_q}
  \frac{p}{q}
  = \lambda^{4}
    + \frac{\left(45 z^{2}+2 v^{2}\right)\lambda^{2}}{4}
    + \frac{81\left(z^{2}-\tfrac{v^{2}}{18}\right)^{2}}{4}
\end{equation}
Thus
\begin{equation}\label{eq:P_1}
  P_1(z,v) = P_1^\sharp(z,v)\,P_1'(z,v)
\end{equation}
with
\[
  P_1'(z,v)
  = \lambda^4
    + \left(\tfrac{45}{4} z^2 + \tfrac{1}{2} v^2\right)\lambda^2
    + \tfrac{1}{16} v^4 - \tfrac{9}{4} v^2 z^2 + \tfrac{81}{4} z^4
\]
cf.~\eqref{eq:P_1_sharp_param}. In particular, \(P_1^\sharp(z,v)\) divides
\(P_1(z,v)\).

Now fix \(i\) with \(0 < \undermu_i < 1\). Then \(P_{\undermu_i}(z,v)\) is
distinct from \(P_0(z,v)\) and \(P_1(z,v)\). In particular, as shown
before,
\[
  \gcd\bigl(P_{\undermu_i}(z,v),P_\nu(z,v)\bigr)=1
  \quad\text{for }\nu \in \{0,1\}
\]
and
\[
  \gcd\bigl(P_0(z,v),P_1(z,v)\bigr)=1
\]
whenever \(z \neq 0\) and \(v \neq 0\). Since \(P_\nu^\sharp(z,v)\mid
P_\nu(z,v)\) for \(\nu \in \{0,1\}\), it follows that
\[
  \gcd\bigl(P_{\undermu_i}(z,v),P_\nu^\sharp(z,v)\bigr)=1
  \quad\text{for }\nu \in \{0,1\}
\]
and
\[
  \gcd\bigl(P_0^\sharp(z,v),P_1^\sharp(z,v)\bigr)=1
\]
whenever \(z \neq 0\) and \(v \neq 0\).

Next, consider coprimality of \(P_{\undermu}(z,v)\) with
\(P_{\n_3}(z,v)=\lambda(\lambda^2+v^2+z^2)\). These two polynomials share
a factor if and only if \(\lambda_0 = 0\) or
\(\lambda_0 = \pm \mathrm{i}\sqrt{v^2+z^2}\) is a root of
\(P_{\undermu}(z,v)\).

If \(\lambda_0=0\) is a root, then
\[
  0=P_{\undermu}(z,v)\big|_{\lambda=0}
   =\tfrac{729}{16}z^6+\tfrac{243}{16}z^4v^2
    +\left(\tfrac{27}{16}-\tfrac{243}{64}\undermu\right)z^2v^4
    +\tfrac{1}{16}v^6
\]
Solving for \(\undermu\) gives
\[
  \undermu=\frac{4(9z^2+v^2)^3}{243\,v^4 z^2}\;\ge\;1
\]
with equality if and only if \(v^2=18z^2\).

If \(\lambda_0^2 = -\,(v^2 + z^2)\) (with \(z\neq 0\) and \(v\neq 0\)),
then
\[
  0 = P_{\undermu}(z,v)\big|_{\lambda^2=-v^2-z^2}
   =\tfrac{25}{2}z^6-15z^4v^2-\tfrac{243}{64}
     \left(\undermu-\tfrac{32}{27}\right)v^4 z^2
\]
Solving for \(\undermu\) yields
\[
  \undermu=\frac{32\,(5z^2-3v^2)^2}{243\,v^4}
\]
This defines a proper algebraic subset of \(\bbR^2\), as do the
equations \(z=0\), \(v=0\), and \(v^2=18z^2\).

Finally, \(P_{\n_3}(z,v)\) is coprime to \(P_0^\sharp(z,v)\) and
\(P_1^\sharp(z,v)\) under the same conditions on \(\z\) and \(\v\), since
\(P_\nu^\sharp(z,v) \mid P_\nu(z,v)\) for \(\nu \in \{0,1\}\).

Therefore, the conditions
\[
  z\neq 0,\quad v\neq 0,\quad v^2\neq 18z^2,\quad
  \frac{(5z^2-3v^2)^2}{v^4}\neq \frac{243}{32}\,\undermu_i
  \quad (i=1,\ldots,\ell)
\]
define an open dense subset of \(\bbR^2\). On this set, all polynomials
from~\eqref{eq:set_of_polynomials_1} are pairwise coprime, as claimed.
\end{proof}

\subsection{Proof of Theorem~\ref{th:main_1}}
\label{se:proof_of_the_main_theorem_1}

Using the results of the previous section, one can multiply the
factors~\eqref{eq:def_P_mu_1}--\eqref{eq:def_P_n_3} to obtain
\(P_{\min}(N,g)\). To make this rigorous, we transfer the results of
Lemmas~\ref{le:maple_P_mu}--\ref{le:gcd} from the parameter spaces to
the geometric setting.

Throughout this section, we work on the open and dense subset
\(U_{(0,1)} \subseteq \n \setminus \ram(\overK^2)\) defined in
\eqref{eq:U_simple}. For each \(X \in U_{(0,1)}\), we have the natural
identification
\[
  \Spec\bigl(-\underK(X)^2\bigr) \cap (0,1)
  \cong
  \branch_{\mathrm{nc}}(U(X))
\]
from~\eqref{eq:Bnc_ident}, while the remaining eigenvalues in
\(\Spec\bigl(-\underK(X)^2\bigr) \cap \{0,1\}\) correspond to global
eigenvalues. Set
\[
  \begin{aligned}
    \mathcal{S}(X) \coloneqq {}&
    \bigl\{\,P_{\undermu}(X)\ \bigm|\ \undermu \in
      \Spec\bigl(-\underK(X)^2\bigr) \cap (0,1)\,\bigr\} \\
    &\;\cup\;
    \bigl\{\,P_\nu^\sharp(X)\ \bigm|\ \nu \in
      \Spec\bigl(-\underK(X)^2\bigr) \cap \{0,1\},\ 
      m_\nu \ge 2\,\bigr\} \\
    &\;\cup\; \bigl\{\,P_{\n_3}(X)\,\bigr\}
    \;\subseteq\; \bbR[\lambda]
  \end{aligned}
\]
Then the cardinality of \(\mathcal{S}(X)\) does not depend on
\(X \in U_{(0,1)}\), and the polynomials in \(\mathcal{S}(X)\) depend
continuously on \(X\).

\begin{definition}\label{de:U_min}
Define \(U_{\bwmin} \subseteq U_{(0,1)}\) as the set of all \(X\) such
that all polynomials in \(\mathcal{S}(X)\) are the blockwise minimal
annihilating polynomials of \(\scrR(X)\) with respect to the linear
operator
\[
  \ad(\C(X))\colon \End{\n} \longrightarrow \End{\n}
\]
More explicitly, we require that \(X_\z \neq 0\) and \(X_\v \neq 0\), and
that the following conditions all hold:
\begin{itemize}
  \item For each
    \(\undermu \in \Spec\bigl(-\underK(X)^2\bigr)\cap(0,1)\), the minimal
    annihilating polynomial of
    \(\scrR(X)\big|_{(\H_\undermu)_X} \in \End{(\H_\undermu)_X}\)
    with respect to the linear operator
    \[
      \ad\bigl(\C(X)\big|_{(\H_\undermu)_X}\bigr)\colon
      \End{(\H_\undermu)_X} \longrightarrow \End{(\H_\undermu)_X}
    \]
    is \(\lambda\,P_{\undermu}(X)\)
    (see~\eqref{eq:def_P_mu_1} and~\eqref{eq:def_H_mu_X}).
  \item For each global eigenvalue \(\nu \in \{0,1\}\) of multiplicity
    \(m_{\nu} \ge 2\), the minimal annihilating polynomial of
    \(\scrR(X)\big|_{(\H_\nu)_X} \in \End{(\H_\nu)_X}\)
    with respect to the linear operator
    \[
      \ad\bigl(\C(X)\big|_{(\H_\nu)_X}\bigr)\colon
      \End{(\H_\nu)_X} \longrightarrow \End{(\H_\nu)_X}
    \]
    is \(\lambda\,P_\nu^\sharp(X)\)
    (see~\eqref{eq:def_P_0_sharp}, \eqref{eq:def_P_1_sharp},
    \eqref{eq:def_H_mu_X}, and~\eqref{eq:def_H_1_X}).
  \item The minimal annihilating polynomial of
    \(\scrR(X)\big|_{(\n_3)_X} \in \End{(\n_3)_X}\)
    with respect to the linear operator
    \[
      \ad\bigl(\C(X)\big|_{(\n_3)_X}\bigr)\colon
      \End{(\n_3)_X} \longrightarrow \End{(\n_3)_X}
    \]
    is \(P_{\n_3}(X)\) (see~\eqref{eq:def_P_n_3}
    and~\eqref{eq:def_n_3_X}).
\end{itemize}
\end{definition}

Here it is natural to assume that \(X_\z \neq 0\) and \(X_\v \neq 0\) for
every \(X \in U_{\bwmin}\), since otherwise \(\C(X)\) is not defined in
this article. Recall that this condition is automatically satisfied on
\(\n \setminus \ram(\overK^2)\) for \(n \ge 3\).

\begin{proposition}\label{p:minimal_polynomial_1}
The set \(U_{\bwmin}\) defined above is open and dense in \(U_{(0,1)}\).
\end{proposition}

\begin{proof}
Since finite intersections of open dense subsets are again open and
dense, it suffices to treat each condition separately and show that it
defines an open dense subset of \(U_{(0,1)}\) with the claimed property.
Then their intersection \(U_{\bwmin}\) is again open and dense.

Let us start with the last item. In order to take advantage of
Lemma~\ref{le:maple_P_n_3}, consider
\[
  \pi \colon \n \longrightarrow \bbR^2,\qquad
  Y \longmapsto \bigl(\norm{Y_\z},\,\norm{Y_\v}\bigr)
\]
On the open dense subset
\[
  W \coloneqq \bigl\{\,Y \in \n \bigm| \norm{Y_\z}\neq 0,\ \norm{Y_\v}\neq 0\,\bigr\}
\]
of \(\n\), this map is differentiable and has full rank \(2\); hence it
is a submersion there. By Lemma~\ref{le:maple_P_n_3} there exists an
open dense subset \(\tilde U_{\n_3}\subseteq \bbR^2\) such that
\(P_{\n_3}(z,v)\) is the minimal annihilating polynomial of
\(\scrR_{\n_3}(z,v)\) with respect to the linear operator
\[
  \ad\bigl(\C_{\n_3}(z,v)\bigr)\colon
  \Mat_3(\bbR) \longrightarrow \Mat_3(\bbR)
\]
if and only if \((z,v) \in \tilde U_{\n_3}\). Moreover, by the
correspondence between the polynomials~\eqref{eq:def_P_n_3}
and~\eqref{eq:P_n_3_param}, and using the matrix representations
\eqref{eq:R_n_3_matrix}--\eqref{eq:C_n_3_matrix} of \(\scrR(X)\) and
\(\C(X)\) restricted to \((\n_3)_X\), the minimal annihilating
polynomial of \(\scrR(X)\big|_{(\n_3)_X} \in \End{(\n_3)_X}\) with respect
to the linear operator
\[
  \ad\bigl(\C(X)\big|_{(\n_3)_X}\bigr)\colon
  \End{(\n_3)_X} \longrightarrow \End{(\n_3)_X}
\]
is \(P_{\n_3}(X)\) if and only if \(X \in \pi^{-1}(\tilde U_{\n_3})\).
Thus \(U_{(0,1)} \cap \pi^{-1}(\tilde U_{\n_3})\) is the subset of
\(U_{(0,1)}\) satisfying the last condition. Furthermore, since \(\pi(W)\)
is open and \(\tilde U_{\n_3}\) is dense in \(\bbR^2\), the intersection
\(\tilde U_{\n_3}\cap \pi(W)\) is open and dense in \(\pi(W)\); hence
\(\pi^{-1}(\tilde U_{\n_3})\) is open and dense in \(W\). Therefore
\(U_{(0,1)} \cap \pi^{-1}(\tilde U_{\n_3})\) is open and dense in
\(U_{(0,1)}\).

Next suppose there is a global eigenvalue \(\nu \in \{0,1\}\) with
\(m_\nu \ge 2\). According to Lemmas~\ref{le:maple_P_0_sharp}
and~\ref{le:maple_P_1_sharp}, there exists an open dense subset
\(\tilde U_\nu \subseteq \bbR^2\) such that
\(\tilde P_\nu^\sharp(z,v) \coloneqq \lambda\,P_\nu^\sharp(z,v)\)
(see~\eqref{eq:tilde_P_0_sharp} and~\eqref{eq:tilde_P_1_sharp}) is the
minimal annihilating polynomial of \(\scrR_\nu^\sharp(z,v)\) with respect
to the linear operator
\[
  \ad\bigl(\C_\nu^\sharp(z,v)\bigr)\colon
  \Mat_{i_\nu}(\bbR) \longrightarrow \Mat_{i_\nu}(\bbR)
\]
if and only if \((z,v) \in \tilde U_\nu\), where \(i_\nu \coloneqq 3\)
(for \(\nu = 0\)) or \(i_\nu \coloneqq 4\) (for \(\nu = 1\)). Again,
\(\pi^{-1}(\tilde U_{\nu}) \subseteq W\) is open and dense. By the
correspondence between the polynomials~\eqref{eq:def_P_0_sharp},
\eqref{eq:def_P_1_sharp} and~\eqref{eq:P_0_sharp_param},
\eqref{eq:P_1_sharp_param}, and using the matrix representations
\eqref{eq:R_0_sharp_matrix}--\eqref{eq:C_0_sharp_matrix} and
\eqref{eq:R_1_sharp_matrix}--\eqref{eq:C_1_sharp_matrix}, respectively,
of \(\scrR(X)\) and \(\C(X)\) restricted to \((\H_\nu)_X\), the minimal
annihilating polynomial of \(\scrR(X)\big|_{(\H_\nu)_X} \in
\End{(\H_\nu)_X}\) with respect to the linear operator
\[
  \ad\bigl(\C(X)\big|_{(\H_\nu)_X}\bigr)\colon
  \End{(\H_\nu)_X} \longrightarrow \End{(\H_\nu)_X}
\]
is \(\lambda\,P_\nu^\sharp(X)\) for each \(X \in \pi^{-1}(\tilde
U_{\nu})\). Thus \(U_{(0,1)} \cap \pi^{-1}(\tilde U_{\nu})\) is the open
and dense subset of \(U_{(0,1)}\) satisfying the second condition.

Furthermore, openness and density are local properties, so for the first
item it suffices to consider the connected component
\(U(X)\subseteq \n \setminus \ram(\overK^2)\) of some \(X \in U_{(0,1)}\)
(cf.~Definition~\ref{de:U_simple} and Corollary~\ref{co:U_sharp}). Then
there is a natural identification
\[
  \Spec\bigl(-\underK(X)^2\bigr) \cap (0,1)
  \cong
  \branch_{\mathrm{nc}}(U(X))
\]
The family of linear operators \(\underK\) is constant along the
\(2\)-plane \(\mathrm{span}_\bbR\{X_\z,X_\v\} \subseteq \n\) by
\eqref{eq:def_underK}, and so is each
\(\undermu \in \branch_{\mathrm{nc}}(U(X))\). Moreover,
\(\left.\nabla\undermu\right|_X\neq 0\) because \(0<\undermu(X)<1\)
(cf.~Corollary~\ref{co:U_sharp}). Thus the gradients of \(\undermu\),
\(\norm{\,\cdot\,}_\z\), and \(\norm{\,\cdot\,}_\v\) are linearly
independent at \(X\).

After shrinking to a neighborhood \(W_\undermu\) of \(X\) that is
contained in \(U(X) \cap U_{(0,1)}\), the map
\[
  \pi_\undermu \colon W_\undermu \longrightarrow \bbR^3,\qquad
  Y \longmapsto \bigl(\undermu(Y),\norm{Y_\z},\norm{Y_\v}\bigr)
\]
has full rank \(3\) everywhere on \(W_\undermu\). Hence \(\pi_\undermu\)
is a submersion and \(\pi_\undermu(W_\undermu)\) is open in \(\bbR^3\);
in particular, \(\pi_\undermu\) is an open map.

By Lemma~\ref{le:maple_P_mu} there exists an open dense subset
\(\tilde U \subseteq \bbR^3\) such that
\(\tilde P_\undermu(z,v) \coloneqq \lambda\,P_\undermu(z,v)\)
(see~\eqref{eq:def_tilde_P_mu}) is the minimal annihilating polynomial
of \(\scrR_\undermu(z,v)\) with respect to the linear operator
\[
  \ad\bigl(\C_\undermu(z,v)\bigr)\colon
  \Mat_6(\bbR) \longrightarrow \Mat_6(\bbR)
\]
for all \((\undermu,z,v) \in \tilde U\). Therefore
\(\pi_\undermu^{-1}(\tilde U)\) is open and dense in \(W_\undermu\),
since \(\pi_\undermu\) is a submersion and \(\tilde U\) is dense in
\(\bbR^3\).

By the correspondence between the polynomials~\eqref{eq:def_P_mu_1}
and~\eqref{eq:P_mu_param}, and using the matrix representations
\eqref{eq:R_mu_matrix}--\eqref{eq:C_mu_matrix} of \(\scrR(X)\) and
\(\C(X)\) restricted to \((\H_\undermu)_X\), the minimal annihilating
polynomial of \(\scrR(X)\big|_{(\H_\undermu)_X} \in \End{(\H_\undermu)_X}\)
with respect to the linear operator
\[
  \ad\bigl(\C(X)\big|_{(\H_\undermu)_X}\bigr)\colon
  \End{(\H_\undermu)_X} \longrightarrow \End{(\H_\undermu)_X}
\]
is \(\lambda\,P_\undermu(X)\) for each \(X \in \pi_\undermu^{-1}(\tilde
U)\).

These considerations together imply that \(U_{\bwmin}\) is open and
dense in \(U_{(0,1)}\).
\end{proof}

\begin{definition}\label{de:U_pwcp}
Define
\begin{equation}\label{eq:U_pwcp}
  U_\pwcp
  \coloneqq
  \bigl\{\, X \in U_{(0,1)}
    \bigm| \mathcal{S}(X) \text{ is pairwise coprime in } \bbR[\lambda]
  \,\bigr\}
\end{equation}
\end{definition}

\begin{proposition}\label{p:minimal_polynomial_2}
The set \(U_\pwcp\) is open and dense in \(U_{(0,1)}\).
\end{proposition}

\begin{proof}
Pairwise coprimeness is equivalent to the nonvanishing of finitely many
resultants, which are polynomial in the coefficients. Hence the
coprimeness locus is open, and therefore \(U_\pwcp\) is open in
\(U_{(0,1)}\).

To see that \(U_\pwcp\) is dense in \(U_{(0,1)}\), fix
\(X \in U_{(0,1)}\). In order to include the case \(\K \equiv 0\) for
\(n = 2\), assume moreover explicitly that \(X_\z \neq 0\) and \(X_\v \neq 0\). Then
\[
  \Phi\colon \mathrm{span}_\bbR\{X_\z, X_\v\}\setminus
  \bigl( (\bbR X_\z \oplus \{0\}) \cup (\{0\}\oplus \bbR X_\v) \bigr)
  \longrightarrow \bbR_{>0}^2,\qquad
  \Phi(Y) = \bigl(\norm{Y_\z},\norm{Y_\v}\bigr)
\]
is a fourfold covering map. Also let
\[
  \{\undermu_1,\dotsc,\undermu_\ell\}
  \coloneqq \Spec\bigl(-\underK(X)^2\bigr) \cap (0,1)
\]
Then \(\ell = \bigl|\branch_{\mathrm{nc}}(U(X))\bigr|\), the number of
nonconstant branches on \(U(X)\), and by~\eqref{eq:Bnc_ident} each
\(\undermu_i\) extends uniquely to a nonconstant branch. Since
\(\underK(X)\) depends only on the two-plane
\(\mathrm{span}_\bbR\{X_\z, X_\v\} \subseteq \n\), these nonconstant
branches are nevertheless constant along this two-plane.

By Lemma~\ref{le:gcd} there is an open dense subset
\(U(\undermu_1,\dotsc,\undermu_\ell) \subseteq \bbR^2\) on which the
specialized polynomials~\eqref{eq:P_mu_param},
\eqref{eq:P_0_sharp_param}, \eqref{eq:P_1_sharp_param}, and
\eqref{eq:P_n_3_param} are pairwise coprime. Then
\(\Phi^{-1}\bigl(U(\undermu_1,\dotsc,\undermu_\ell)\bigr) \subseteq
U_\pwcp\), by correspondence with~\eqref{eq:def_P_mu_1},
\eqref{eq:def_P_0_sharp}, \eqref{eq:def_P_1_sharp}, and
\eqref{eq:def_P_n_3}. Since \(\Phi\) is a covering map,
\(\Phi^{-1}\bigl(U(\undermu_1,\dotsc,\undermu_\ell)\bigr)\) is dense in
\(\mathrm{span}_\bbR\{X_\z, X_\v\}\). Therefore every
\(X \in U_{(0,1)}\) with \(X_\z \neq 0\) and \(X_\v \neq 0\) lies in the
closure of \(U_\pwcp\). The claim follows.
\end{proof}

\begin{corollary}\label{co:U_min_cap_U_pwcp}
Let \(X \in U_\pwcp \cap U_\bwmin\). The minimal annihilating polynomial
\(P_{\min}(\C,X)\) of \(\scrR(X) \in \End{\n}\) for the operator
\[
  \ad(\C(X))\colon \End{\n}\longrightarrow \End{\n}
\]
is
\begin{equation}\label{eq:local_minimal_polynomial_2}
  P(X) \coloneqq
  P_{\n_3}(X)
  \prod_{\undermu \in \Spec\bigl(-\underK(X)^2\bigr)\cap (0,1)}
    \!\!\!\!\!\!P_\undermu(X)
  \prod_{\substack{\nu \in \Spec\bigl(-\underK(X)^2\bigr)\cap\{0,1\}\\
    m_\nu \ge 2}}\!\!\!\!\!\!
    P_\nu^\sharp(X)
\end{equation}
where the factors \(P_\undermu(X)\), \(P_\nu^\sharp(X)\) and  \(P_{\n_3}(X)\) are defined
by~\eqref{eq:def_P_mu_1}--\eqref{eq:def_P_n_3}.
\end{corollary}

\begin{proof}
Consider the splittings~\eqref{eq:decomposition_of_n_2} and
\eqref{eq:hat_h-split} defined for every \(X \in \n\) such that
\(\norm{X_\z}\neq 0\) and \(\norm{X_\v}\neq 0\). Each block is invariant
under \(\scrR(X)\) and \(\C(X)\). By restricting \(\scrR(X)\) and
\(\C(X)\) to each block, we obtain, for each block, a minimal
annihilating polynomial of \(\scrR(X)\) with respect to the linear
operator \(\ad(\C(X))\). Then, on the one hand, \(P_{\min}(\C,X)\)
divides the least common multiple of these blockwise minimal annihilating
polynomials. On the other hand, each blockwise minimal annihilating
polynomial divides \(P_{\min}(\C,X)\). Thus \(P_{\min}(\C,X)\) coincides
with this least common multiple for every \(X \in \n\) such that
\(\norm{X_\z}\neq 0\) and \(\norm{X_\v}\neq 0\).

For \(X \in U_{\bwmin}\), the blockwise minimal annihilating polynomials
of \(\scrR(X)\) with respect to the linear operator \(\ad(\C(X))\) are
precisely the polynomials defined in
\eqref{eq:def_P_mu_1}--\eqref{eq:def_P_n_3}, that is (up to the
negligible factor \(\lambda\)) the factors
\(\{P_{\n_3}(X),\,P_\undermu(X),\,P_\nu^\sharp(X)\}\) in
\eqref{eq:local_minimal_polynomial_2}, with \(\undermu\) and \(\nu\) as
indicated in the subscripts.

If moreover \(X \in U_\pwcp\), then these factors are pairwise coprime,
so the least common multiple is the product itself. This shows that
\eqref{eq:local_minimal_polynomial_2} holds.
\end{proof}

\paragraph{Proof of Theorem~\ref{th:main_1}.}
By Corollary~\ref{co:U_min_cap_U_pwcp} and~\eqref{eq:Bnc_ident}, the
formula~\eqref{eq:P_min_2} holds for all
\(X \in U_\bwmin\cap U_\pwcp\cap U(e_N)\). This set is open and dense in
\(\n \setminus \ram(\overK^2)\), so~\eqref{eq:P_min_2} extends to all of
\(\n \setminus \ram(\overK^2)\) by continuity.\qed

\begin{remark}
From the proofs, one directly reads off that
\(U(e_N)=U_\bwmin\cap U_\pwcp\) for \(n \ge 3\). For \(n = 2\) the same
assertion still holds. For \(n = 1\) we simply have
\[
  U(e_N)=\bigl\{\,X\in\n \bigm| X_\v\neq 0\,\bigr\}
\]
\end{remark}

\appendix
\section{\texorpdfstring{Explicit verification that~\eqref{eq:over_mu_for_n_=_4},
\eqref{eq:over_mu_for_n_=_5}, \eqref{eq:over_mu_for_n_=_7_1},
\eqref{eq:over_mu_for_n_=_7_2}, \eqref{eq:over_mu_for_n_=_7_3},
\eqref{eq:over_mu_for_n_=_8}, \eqref{eq:over_sigma_1_for_n_=_9}, and
\eqref{eq:over_sigma_2_for_n_=_9} describe Killing tensors}%
{Explicit verification: the \(n=4,5,7,8,9\) cases describe Killing tensors}}
\label{se:verification}

To ensure consistency with the theory, we explicitly verify (without
invoking Proposition~\ref{p:eigenvalues_of_K_and_tilde_K}) that
\eqref{eq:over_mu_for_n_=_4}, \eqref{eq:over_mu_for_n_=_5},
\eqref{eq:over_mu_for_n_=_7_1}, \eqref{eq:over_mu_for_n_=_7_2},
\eqref{eq:over_mu_for_n_=_7_3}, \eqref{eq:over_mu_for_n_=_8},
\eqref{eq:over_sigma_1_for_n_=_9}, and \eqref{eq:over_sigma_2_for_n_=_9}
indeed define Killing tensors.

Recall that parallel transport with respect to the left-invariant
framing along a geodesic \(\gamma\colon \bbR \to N\) with
\(\gamma(0) = e_N\) and \(\dot{\gamma}(0) = X\) is given by
\eqref{eq:geodesic}.

\subsection{\texorpdfstring{For~\eqref{eq:over_mu_for_n_=_4}
and~\eqref{eq:over_mu_for_n_=_5}}{For \eqref{eq:over_mu_for_n_=_4} and
\eqref{eq:over_mu_for_n_=_5}}}
\label{se:verification_for_n_=_4_and_n_=_5}
For \(n = 5\), let \(\z \coloneqq \bbR^5 \subseteq \bbR^7 = \Im(\bbO)\)
and \(\v \coloneqq \bbO\), and consider the geodesic \(\gamma\) through
\(e_N\) in the direction \(\dot \gamma(0) = S_0 \oplus S_1 \in
\bbR^5 \oplus \bbO\). By~\eqref{eq:geodesic},
\begin{equation}\label{eq:S_1(t)_0}
  S_1(t) = \cos(t)S_1 + \sin(t)S_1 S_0
\end{equation}
while \(S_0(t) \equiv S_0\) is constant.

We claim that \(\langle (S_1(t) S_0) E_6,\, S_1(t) E_7 \rangle\) is
constant. This follows by a direct calculation: after reparameterizing
we may assume \(\norm{S_0} = 1\). Then
\begin{align*}
  \langle (S_1(t) S_0) E_6,\, S_1(t) E_7 \rangle
   &= \cos^2\!t\,\langle (S_1 S_0) E_6,\, S_1 E_7 \rangle \\
   &\quad + \cos t\,\sin t\Big(
        \underbrace{\langle (S_1 S_0) E_6,\,(S_1 S_0) E_7 \rangle}_{=\,0}
        + \underbrace{\langle ((S_1 S_0) S_0) E_6,\, S_1 E_7 \rangle}
          _{=\,-\langle S_1 E_6,\, S_1 E_7 \rangle = 0}
      \Big) \\
   &\quad + \sin^2\!t\,\langle ((S_1 S_0) S_0) E_6,\,(S_1 S_0) E_7 \rangle \\
   &= \Bigl(\cos^2\!t\,\langle (S_1 S_0) E_6,\, S_1 E_7 \rangle
       - \sin^2\!t\,\underbrace{\langle S_1 E_6,\,(S_1 S_0) E_7 \rangle}
         _{=\,-\langle S_1 E_7,\,(S_1 S_0) E_6 \rangle}\Bigr) \\
   &= (\sin^2\!t + \cos^2\!t)\,\langle (S_1 S_0) E_6,\, S_1 E_7 \rangle \\
   &= \langle (S_1 S_0) E_6,\, S_1 E_7 \rangle
\end{align*}
and hence
\[
  \langle (S_1(t) S_0) E_6,\, S_1(t) E_7 \rangle^2
  = \langle (S_1(0) S_0) E_6,\, S_1(0) E_7 \rangle^2
\]
for all \(t\). Thus \(\overmu\) as defined in~\eqref{eq:over_mu_for_n_=_5}
is a Killing tensor.

The fact that~\eqref{eq:over_mu_for_n_=_4} defines a Killing tensor for
\(n = 4\) follows by the same argument.

\subsection{\texorpdfstring{For~\eqref{eq:over_mu_for_n_=_7_1},
\eqref{eq:over_mu_for_n_=_7_2}, and~\eqref{eq:over_mu_for_n_=_7_3}}%
{For \eqref{eq:over_mu_for_n_=_7_1}, \eqref{eq:over_mu_for_n_=_7_2},
and \eqref{eq:over_mu_for_n_=_7_3}}}
\label{se:verification_for_n_=_7}

For \(n = 7\), consider first \(\z \coloneqq \Im(\bbO)\) and
\(\v \coloneqq \v_1 \oplus \v_2 \coloneqq \bbO \oplus \bbO\) (with
Clifford multiplication given by~\eqref{eq:Clifford_multiplication_for_n_=_7_1}
on both summands), or \(\v \coloneqq \v_1 \oplus \v_2 \coloneqq
\overline{\bbO} \oplus \overline{\bbO}\) (with the sign of the Clifford
multiplication reversed on both summands as
in~\eqref{eq:Clifford_multiplication_for_n_=_7_2}). We claim that both
summands \(\norm{S_0}^2(\norm{S_1}^2+\norm{S_2}^2)\) and
\[
  2\,\bigl\langle S_2^*(S_1 S_0), (S_0 S_2^*) S_1 \bigr\rangle
\]
from~\eqref{eq:over_mu_for_n_=_7_1} are Killing tensors.

For the first summand this follows as for \(n = 3\); cf.
Section~\ref{se:explicit_formulas}. For the second summand, let
\(S_0 \oplus S_1 \oplus S_2 \in \Im(\bbO) \oplus \bbO^2\) and consider
the geodesic \(\gamma\) through \(e_N\) with initial velocity
\(S_0 \oplus S_1 \oplus S_2\). After reparameterizing, assume
\(\norm{S_0} = 1\).

By~\eqref{eq:geodesic}, we have
\begin{align}
  S_1(t) &= \cos t\,S_1 \pm \sin t\,S_1 S_0 \label{eq:S_1(t)_1}\\
  S_2(t) &= \cos t\,S_2 \pm \sin t\,S_2 S_0 \label{eq:S_2(t)_1}
\end{align}
We must show that
\[
  \bigl\langle S_1(t)^*(S_2(t)S_0),\,(S_0S_1(t)^*)S_2(t)\bigr\rangle
\]
is constant in \(t\). We have
\begin{align}
  S_0S_1(t)^* &= \bigl(-\cos t\,S_1S_0 \pm \sin t\,S_1\bigr)^*
  \label{eq:S_1(t)_2}\\
  S_2(t)S_0 &= \cos t\,S_2S_0 \mp \sin t\,S_2
  \label{eq:S_2(t)_2}
\end{align}
where \eqref{eq:S_2(t)_2} uses alternativity:
\((S_2S_0)S_0 = S_2(S_0^2) = -S_2\). A direct calculation yields
\begin{align*}
  &\bigl\langle S_1(t)^*(S_2(t)S_0),\,(S_0S_1(t)^*)S_2(t)\bigr\rangle \\
  &\quad= \cos^4 t\,\langle S_1^*(S_2S_0), (S_0S_1^*)S_2\rangle \\
  &\qquad - 2\sin^2 t\cos^2 t\,\langle S_1^*S_2, (S_0S_1^*)(S_2S_0)\rangle \\
  &\qquad + \sin^4 t\,\langle (S_0S_1^*)S_2, S_1^*(S_2S_0)\rangle
\end{align*}
Moreover,
\begin{equation}\label{eq:identity_3}
  \langle S_1^*S_2, (S_0S_1^*)(S_2S_0)\rangle
  = -\,\langle S_1^*(S_2S_0), (S_0S_1^*)S_2\rangle
\end{equation}
which follows from the octonionic identity
\[
  \langle xy, zw\rangle + \langle xw, zy\rangle
  = 2\langle x, z\rangle\langle y, w\rangle
\]
together with \(\langle S_1^*, S_0S_1^*\rangle = 0\) and
\(\langle S_2, S_2S_0\rangle = 0\) (skewness of left multiplication
\(L_{S_0}\) and right multiplication \(R_{S_0}\) for imaginary \(S_0\)).
Using~\eqref{eq:identity_3}, we obtain
\begin{align*}
  &\bigl\langle S_1(t)^*(S_2(t)S_0),\,(S_0S_1(t)^*)S_2(t)\bigr\rangle \\
  &\quad= (\cos^2 t + \sin^2 t)^2\,
    \langle (S_0S_1^*)S_2, S_1^*(S_2S_0)\rangle \\
  &\quad= \langle (S_0S_1^*)S_2, S_1^*(S_2S_0)\rangle
\end{align*}
Hence the second summand is constant along \(\gamma\), and thus \(\overmu\)
in~\eqref{eq:over_mu_for_n_=_7_1} is a Killing tensor.

Now suppose that \(\v_1\) and \(\v_2\) have opposite signs, i.e.,
\(\v = \v_1 \oplus \v_2 \coloneqq \bbO \oplus \overline{\bbO}\). As
before, both \(\norm{\cdot}_{\v_1}^2\) and \(\norm{\cdot}_{\v_2}^2\) are
Killing tensors, revealing~\eqref{eq:over_mu_for_n_=_7_2} as Killing.

To verify that~\eqref{eq:over_mu_for_n_=_7_3} also defines a Killing
tensor, set
\begin{align}
  S_1(t) &= \cos t\,S_1 + \sin t\,S_1S_0 \label{eq:S_1(t)_3}\\
  S_2(t) &= \cos t\,S_2 - \sin t\,S_2S_0 \label{eq:S_2(t)_3}\\
  S_0S_1(t)^* &= \bigl(-\cos t\,S_1S_0 + \sin t\,S_1\bigr)^*
  \label{eq:S_1(t)_4}\\
  S_2(t)S_0 &= \cos t\,S_2S_0 + \sin t\,S_2 \label{eq:S_2(t)_4}
\end{align}
The same computation then gives
\[
  \bigl\langle S_1(t)^*(S_2(t)S_0),\,(S_0S_1(t)^*)S_2(t)\bigr\rangle
  = \langle (S_0S_1^*)S_2,\,S_1^*(S_2S_0)\rangle
\]
so~\eqref{eq:over_mu_for_n_=_7_3} defines a Killing tensor as well.

\subsection{\texorpdfstring{For~\eqref{eq:over_mu_for_n_=_8}}%
{For \eqref{eq:over_mu_for_n_=_8}}}
\label{se:verification_for_n_=_8}
Let \(n = 8\), \(\z \coloneqq \bbO\), and \(\v \coloneqq \bbO^2\) with
Clifford multiplication given by~\eqref{eq:Clifford_multiplication_for_n_=_8}.
We show that~\eqref{eq:over_mu_for_n_=_8} is constant along geodesics
\(\gamma\) through \(e_N\). Since \(\Spin(\bbO)\) acts transitively on
the unit sphere \(\rmS^7 \subseteq \bbO\) via the vector representation,
we may assume that the velocity vector
\[
  \dot\gamma(0) \coloneqq S_0 \oplus S_+ \oplus S_-
  \in \bbO \oplus \bbO^2
\]
satisfies \(S_0 = 1_{\bbO}\).

By~\eqref{eq:geodesic}, with \(c \coloneqq \cos(t)\) and
\(s \coloneqq \sin(t)\),
\[
  S_+(t) = c\,S_+ + s\,S_-, \qquad
  S_-(t) = c\,S_- - s\,S_+^* .
\]
Set \(A \coloneqq \norm{S_+}^2\), \(B \coloneqq \norm{S_-}^2\), and
\(C \coloneqq \langle S_+, S_- \rangle\). Using
\(\norm{S_+^*} = \norm{S_+}\) and
\(\langle S_-, S_+^* \rangle = \langle S_+, S_- \rangle\), we obtain
\[
  \begin{aligned}
    \norm{S_+(t)}^2 &= c^2 A + s^2 B + 2 s c\, C,\\
    \norm{S_-(t)}^2 &= c^2 B + s^2 A - 2 s c\, C,\\
    \langle S_+(t), S_-(t) \rangle
      &= s c\,(B - A) + (c^2 - s^2) C
  \end{aligned}
\]
Hence
\[
  \begin{aligned}
    &\norm{S_+(t)}^2 \norm{S_-(t)}^2
      - \langle S_+(t), S_-(t) \rangle^2 \\
    &\qquad= (c^4 + 2 s^2 c^2 + s^4)\,(A B - C^2)
     \;=\; (c^2 + s^2)^2\,(A B - C^2)
     \;=\; A B - C^2
  \end{aligned}
\]
The expression is constant in \(t\), so~\eqref{eq:over_mu_for_n_=_8}
defines a Killing tensor.
\subsection{\texorpdfstring{For~\eqref{eq:over_sigma_1_for_n_=_9},
\eqref{eq:over_sigma_2_for_n_=_9}, and~\eqref{eq:over_mu_3_for_n_=_9_1}}%
{For \eqref{eq:over_sigma_1_for_n_=_9}, \eqref{eq:over_sigma_2_for_n_=_9},
and \eqref{eq:over_mu_3_for_n_=_9_1}}}
\label{se:verification_for_n_=_9}

Let \(n = 9\), \(\z \coloneqq \bbR \oplus \bbO\), and
\(\v \coloneqq \bbO^2 \otimes_\bbR \bbC\) with Clifford multiplication
given by~\eqref{eq:Clifford_multiplication_for_n_=_9}. First we show that
\eqref{eq:over_sigma_1_for_n_=_9} and~\eqref{eq:over_sigma_2_for_n_=_9}
are constant along geodesics \(\gamma\) through \(e_N\). Since
\(\Spin(\bbR \oplus \bbO)\) acts transitively on the \(8\)-sphere
\(\rmS^8 \subseteq \z\) via the vector representation, we may assume that
the velocity vector
\[
  \dot\gamma(0) \coloneqq S_0 \oplus S_+ \oplus S_-
  \in \bbR \oplus \bbO \oplus \bbO^2 \oplus \i\,\bbO^2
\]
satisfies \(S_0 \coloneqq 1_\bbR \coloneqq 1 \oplus 0\). By
\eqref{eq:n_=_9_binom_1} and~\eqref{eq:n_=_9_binom_2}, it suffices to
show that \eqref{eq:over_mu_+} and~\eqref{eq:over_mu_-} are constant.

For this, write \(c \coloneqq \cos(t)\) and \(s \coloneqq \sin(t)\). By
\eqref{eq:geodesic},
\[
  \begin{aligned}
    \Re(S_+(t)) &= c\,\Re(S_+) - s\,\Im(S_+), &
    \Im(S_+(t)) &= s\,\Re(S_+) + c\,\Im(S_+),\\
    \Re(S_-(t)) &= c\,\Re(S_-) + s\,\Im(S_-), &
    \Im(S_-(t)) &= -s\,\Re(S_-) + c\,\Im(S_-)
  \end{aligned}
\]
Set \(A_\pm \coloneqq \norm{\Re(S_\pm)}^2\),
\(B_\pm \coloneqq \norm{\Im(S_\pm)}^2\), and
\(C_\pm \coloneqq \langle \Re(S_\pm), \Im(S_\pm) \rangle\). Then
\[
  \begin{aligned}
    \norm{\Re(S_+(t))}^2 &= c^2 A_+ + s^2 B_+ - 2 s c\, C_+,\\
    \norm{\Im(S_+(t))}^2 &= s^2 A_+ + c^2 B_+ + 2 s c\, C_+,\\
    \langle \Re(S_+(t)), \Im(S_+(t)) \rangle
      &= s c\,(A_+ - B_+) + (c^2 - s^2)\,C_+
  \end{aligned}
\]
and
\[
  \begin{aligned}
    \norm{\Re(S_-(t))}^2 &= c^2 A_- + s^2 B_- + 2 s c\, C_-,\\
    \norm{\Im(S_-(t))}^2 &= s^2 A_- + c^2 B_- - 2 s c\, C_-,\\
    \langle \Re(S_-(t)), \Im(S_-(t)) \rangle
      &= - s c\,(A_- - B_-) + (c^2 - s^2)\,C_-
  \end{aligned}
\]
Consequently, for each sign,
\[
  \begin{aligned}
    &\norm{\Re(S_\pm(t))}^2 \norm{\Im(S_\pm(t))}^2
      - \langle \Re(S_\pm(t)), \Im(S_\pm(t)) \rangle^2 \\
    &\qquad= (c^2 + s^2)^2\,(A_\pm B_\pm - C_\pm^2)
     \;=\; A_\pm B_\pm - C_\pm^2
  \end{aligned}
\]
is constant in \(t\) by the Pythagorean identity \(c^2 + s^2 = 1\).

Therefore~\eqref{eq:over_mu_+} (and likewise~\eqref{eq:over_mu_-}) is
constant along geodesics. Hence the sum and the difference are also
constant:
\[
  \begin{aligned}
    &\bigl(\norm{\Re(S_+(t))}^2 \norm{\Im(S_+(t))}^2
       \pm \norm{\Re(S_-(t))}^2 \norm{\Im(S_-(t))}^2\bigr) \\
    &\quad - \bigl(\langle \Re(S_+(t)), \Im(S_+(t)) \rangle^2
       \pm \langle \Re(S_-(t)), \Im(S_-(t)) \rangle^2\bigr) \\
    &= \bigl(A_+ B_+ - C_+^2\bigr) \pm \bigl(A_- B_- - C_-^2\bigr)
  \end{aligned}
\]
By~\eqref{eq:n_=_9_binom_1} and~\eqref{eq:n_=_9_binom_2}, both
\(\tilde{\sigma}_1(1 \oplus S(t))\) and
\(\tilde{\sigma}_2(1 \oplus S(t))\) are constant in \(t\). This proves
that~\eqref{eq:over_sigma_1_for_n_=_9} and~\eqref{eq:over_sigma_2_for_n_=_9}
are Killing tensors.

Next, we show the same for~\eqref{eq:over_mu_3_for_n_=_9_1}. For this,
we may work with~\eqref{eq:over_mu_mix} and~\eqref{eq:over_mu_3_=_over_mu_mix}.
Since~\eqref{eq:over_mu_+} and~\eqref{eq:over_mu_-} are constant along
geodesics, it remains to show that the following expression is constant:
\[
  \begin{aligned}
    E(t) \coloneqq {}\,&
      \bigl\langle \Re(S_+(t))^* \Re(S_-(t)),\,
                   \Im(S_+(t))^* \Im(S_-(t)) \bigr\rangle \\
    &- \bigl\langle \Re(S_+(t))^* \Im(S_-(t)),\,
                   \Im(S_+(t))^* \Re(S_-(t)) \bigr\rangle
  \end{aligned}
\]

Set \(A(t) \coloneqq \Re(S_+(t))\), \(C(t) \coloneqq \Im(S_+(t))\),
\(B(t) \coloneqq \Re(S_-(t))\), and \(D(t) \coloneqq \Im(S_-(t))\).
Furthermore, put \(A_0 \coloneqq A(0)\), \(C_0 \coloneqq C(0)\),
\(B_0 \coloneqq B(0)\), and \(D_0 \coloneqq D(0)\). Hence
\[
  E(t) = \langle A(t)^* B(t),\, C(t)^* D(t) \rangle
       - \langle A(t)^* D(t),\, C(t)^* B(t) \rangle
\]

By~\eqref{eq:geodesic}, the geodesic equation decouples so that the
pairs \((A(t), C(t))\) and \((B(t), D(t))\) rotate with opposite angles:
\[
  \binom{A(t)}{C(t)} =
  \begin{pmatrix}\cos t & -\sin t\\ \sin t & \cos t\end{pmatrix}
  \binom{A_0}{C_0},\qquad
  \binom{B(t)}{D(t)} =
  \begin{pmatrix}\cos t & \ \sin t\\ -\sin t & \ \cos t\end{pmatrix}
  \binom{B_0}{D_0}
\]
Define
\[
  E(A,B,C,D) \coloneqq
  \langle A^* B,\, C^* D \rangle - \langle A^* D,\, C^* B \rangle
\]
This is alternating in the pairs \((A,C)\) and \((B,D)\). Therefore,
\(E\) factors through a unique alternating map
\[
  \tilde E \colon \Lambda^2(\bbO) \times \Lambda^2(\bbO) \to \bbR,\qquad
  \tilde E(A \wedge C,\, B \wedge D) = E(A,B,C,D)
\]
Since both rotation matrices have determinant \(1\), we get
\[
  A(t) \wedge C(t) = A_0 \wedge C_0,\qquad
  B(t) \wedge D(t) = B_0 \wedge D_0
\]
Hence
\[
  E(t) = \tilde E\bigl(A(t) \wedge C(t),\, B(t) \wedge D(t)\bigr)
       = \tilde E\bigl(A_0 \wedge C_0,\, B_0 \wedge D_0\bigr)
       = E(0)
\]
Thus~\eqref{eq:over_mu_3_for_n_=_9_1} defines a Killing tensor as well.

\section{Proof of Theorem~\ref{th:C0}}
\label{se:C0_proof}

The proof of the following lemma uses the same arguments as in
\cite[Ch.~3.5]{BTV}. For \(X \in \n\) with \(X_\z \neq 0\)
and \(X_\v \neq 0\), let \(\gamma \colon \bbR \to N\) be the geodesic
with \(\gamma(0) = e_N\) and \(\dot\gamma(0) = X\).

\begin{lemma}\label{le:C_is_parallel}
The endomorphism field \(\C_\gamma\) defined by
\eqref{eq:C_0_structure_1}--\eqref{eq:C_0_structure_3} is parallel
along \(\gamma\); that is,
\[
  \frac{\nabla}{\d t}\,\C_\gamma \equiv 0
\]
\end{lemma}

\begin{proof}
We will first show that
\[
  \left.\frac{\nabla}{\d t}\right|_{t=0} \C_\gamma = 0
\]
By~\eqref{eq:geodesic},
\[
  \dot\gamma(t) = \bigl(\gamma(t), X(t)\bigr)
\]
with
\begin{equation}\label{eq:canonical_trivialization}
  X(t) =
  \cos\bigl(t\norm{X_\z}\bigr)\,X_\v
  +
  \frac{\sin\bigl(t\norm{X_\z}\bigr)}{\norm{X_\z}}\,
  X_\z \cbullet X_\v
  +
  X_\z
\end{equation}
in the canonical trivialization \(TN \cong N \times \n\) induced by
left translation.

In particular, the ordinary derivative of the second component of
\(\dot\gamma(t)\) in \eqref{eq:canonical_trivialization} satisfies
\begin{equation}\label{eq:canonical_trivialization_1}
  \frac{\mathrm d}{\mathrm dt}\,X(t) =
  -\norm{X_\z}\,\sin\bigl(t\norm{X_\z}\bigr)\,X_\v
  +
  \cos\bigl(t\norm{X_\z}\bigr)\,X_\z \cbullet X_\v
\end{equation}
With these preliminaries, suppose that
\(Y \in \H_X^\perp = (\n_3)_X \oplus (\Z_X)^\perp\)
(see~\eqref{eq:def_n_3_X}, \eqref{eq:z_X_perp}
and~\eqref{eq:decomposition_of_n_2}). Then
\begin{align*}
 \Bigl(\frac{\nabla}{\d t}\Big|_{t=0} \C_\gamma(t)\Bigr)Y
  &=
  \frac{\nabla}{\d t}\Big|_{t=0}
    \Bigl(\C_\gamma(t)Y \Bigr) - \C(X)\nabla_X Y\\
  &\stackrel{\eqref{eq:C_0_structure_1}-\eqref{eq:C_0_structure_3}}{=}
     \frac12\,\left.\tfrac{\nabla}{\d t}\right|_{t=0}
     \Bigl(\gamma(t),
           X(t)_\z \cbullet Y_\v
           - Y_\z \cbullet X(t)_\v
           + X(t)_\v \spinprod Y_\v\Bigr) \\
  &\quad
     - \frac12\Bigl(
         X_\z \cbullet (\nabla_X Y)_\v
         - (\nabla_X Y)_\z \cbullet X_\v
         + X_\v \spinprod (\nabla_X Y)_\v
       \Bigr)
\end{align*}
Here the term
\[
  \left.\tfrac{\nabla}{\d t}\right|_{t=0}
  \Bigl(\gamma(t),
        X(t)_\z \cbullet Y_\v
        - Y_\z \cbullet X(t)_\v
        + X(t)_\v \spinprod Y_\v\Bigr)
\]
splits into the algebraic component
\[
  \frac12\Bigl(
    X_\v \spinprod (X_\z \cbullet Y_\v - Y_\z \cbullet X_\v)
    - X_\z \cbullet (X_\z \cbullet Y_\v - Y_\z \cbullet X_\v)
    - (X_\v \spinprod Y_\v) \cbullet X_\v
  \Bigr)
\]
(which is \(A(X)X\) from~\eqref{eq:LC}) and the
usual derivative
\[
  \left.\tfrac{\mathrm d}{\mathrm dt}\right|_{t=0}
  \Bigl(
    X(t)_\z \cbullet Y_\v
    - Y_\z \cbullet X(t)_\v
    + X(t)_\v \spinprod Y_\v
  \Bigr)
  =
  -\Bigl(
    Y_\z \cbullet X_\z \cbullet X_\v
    + Y_\v \spinprod (X_\z \cbullet X_\v)
  \Bigr)
\]
obtained by inserting \(t = 0\) into
\eqref{eq:canonical_trivialization_1}. Thus,
\begin{align*}
  &\left.\Bigl(\tfrac{\nabla}{\d t}\right|_{t=0} \C_\gamma(t)\Bigr)
  Y
  + \frac12\bigl(
      Y_\z \cbullet X_\z \cbullet X_\v
      + Y_\v \spinprod (X_\z \cbullet X_\v)
    \bigr) \\
  &\qquad=\;
    \frac14\Bigl(
      X_\v \spinprod (X_\z \cbullet Y_\v - Y_\z \cbullet X_\v)
      - X_\z \cbullet (X_\z \cbullet Y_\v - Y_\z \cbullet X_\v)
      - (X_\v \spinprod Y_\v) \cbullet X_\v
    \Bigr) \\
  &\qquad\quad
    - \frac12\Bigl(
        X_\z \cbullet (\nabla_X Y)_\v
        - (\nabla_X Y)_\z \cbullet X_\v
        + X_\v \spinprod (\nabla_X Y)_\v
      \Bigr) \\
  &\qquad\stackrel{\eqref{eq:LC}}{=}\;
    \frac14\Bigl(
      X_\v \spinprod (X_\z \cbullet Y_\v - Y_\z \cbullet X_\v)
      - X_\z \cbullet (X_\z \cbullet Y_\v - Y_\z \cbullet X_\v)
      - (X_\v \spinprod Y_\v) \cbullet X_\v
    \Bigr) \\
  &\qquad\quad
    - \frac14\Bigl(
        - X_\z \cbullet (X_\z \cbullet Y_\v + Y_\z \cbullet X_\v)
        - (X_\v \spinprod Y_\v) \cbullet X_\v
        - X_\v \spinprod (X_\z \cbullet Y_\v + Y_\z \cbullet X_\v)
      \Bigr) \\
  &\qquad=\;
    \frac12\Bigl(
      X_\v \spinprod (X_\z \cbullet Y_\v)
      + X_\z \cbullet Y_\z \cbullet X_\v
    \Bigr)
\end{align*}
Thus,
\[
  \Bigl(\left.\tfrac{\nabla}{\d t}\right|_{t=0}\C_\gamma(t)\Bigr)
 Y
  =
  \frac12\Bigl(
    X_\v \spinprod (X_\z \cbullet Y_\v)
    + X_\z \cbullet Y_\z \cbullet X_\v
    - Y_\z \cbullet X_\z \cbullet X_\v
    - Y_\v \spinprod (X_\z \cbullet X_\v)
  \Bigr)
\]
To prove that this vanishes, observe that
\[
  X_\z \cbullet Y_\z \cbullet X_\v - Y_\z \cbullet X_\z \cbullet X_\v
  =
  \frac{\langle Y_\z, X_\z \rangle}{\norm{X_\z}^2}\,
  \bigl(
    X_\z \cbullet X_\z \cbullet X_\v
    - X_\z \cbullet X_\z \cbullet X_\v
  \bigr)
  = 0
\]
since \(Y_\z\) lies in the real line spanned by \(X_\z\), that is,
\(Y_\z \in \bbR X_\z\), for each
\(Y \in (\n_3)_X \oplus (\Z_X)^\perp\). It remains to show that
\[
  Y_\v \spinprod (X_\z \cbullet X_\v)
  - X_\v \spinprod (X_\z \cbullet Y_\v)
  = 0
\]
for all \(Y \in \H_X^\perp\). Let \(Z \in \z\). If \(Z \in \h_X =
X_\z^\perp\), then
\begin{align*}
  \bigl\langle
    X_\v \spinprod (X_\z \cbullet Y_\v)
    - Y_\v \spinprod (X_\z \cbullet X_\v),
    Z
  \bigr\rangle
  &\stackrel{\eqref{eq:def_spinor_product}}{=}
    \langle X_\z \cbullet Y_\v, Z \cbullet X_\v \rangle
    - \langle X_\z \cbullet X_\v, Z \cbullet Y_\v \rangle \\
  &\stackrel{\eqref{eq:Clifford_module_1}}{=}
   - \langle Y_\v, X_\z \cbullet Z \cbullet X_\v \rangle
   + \langle Z \cbullet X_\z \cbullet X_\v, Y_\v \rangle \\
  &\stackrel{\eqref{eq:Clifford_module_3}}{=}
   2\langle Y_\v, \underbrace{Z \cbullet X_\z \cbullet X_\v}_{\in \v}
   \rangle = 2\langle Y, Z \cbullet X_\z \cbullet X_\v
   \rangle  = 0
\end{align*}
because \( Z \cbullet X_\z \cbullet X_\v \in \H_X\) and \(Y \in \H_X^\perp\),
so the inner product vanishes. If
\(Z = X_\z\), then
\[
  \bigl\langle
    X_\v \spinprod (X_\z \cbullet Y_\v)
    - Y_\v \spinprod (X_\z \cbullet X_\v),
    Z
  \bigr\rangle
  \stackrel{\eqref{eq:def_spinor_product}}{=}
  \langle X_\z \cbullet Y_\v, X_\z \cbullet X_\v \rangle
  - \langle X_\z \cbullet X_\v, X_\z \cbullet Y_\v \rangle
  = 0
\]
Since \(\z = \h_X \oplus \bbR X_\z\), it follows that
\[
  X_\v \spinprod (X_\z \cbullet Y_\v)
  - Y_\v \spinprod (X_\z \cbullet X_\v)
  = 0
\]
by the nondegeneracy of the metric. Hence
\[
  \Bigl(\left.\tfrac{\nabla}{\d t}\right|_{t=0}\C_\gamma(t)\Bigr) Y = 0
\]

Next, let \(Y \in \H_X\) (see~\eqref{eq:def_H_X}). Then,
similar to the previous case,
\begin{align*}
  \Bigl(\frac{\nabla}{\d t}\Big|_{t=0} \C_\gamma(t)\Bigr)Y
  &=
  \frac{\nabla}{\d t}\Big|_{t=0}
    \Bigl(\C_\gamma(t)Y \Bigr) - \C(X)\nabla_X Y \\
  &=
  \frac12\,\frac{\nabla}{\d t}\Big|_{t=0}
  \Bigl(
    \gamma(t),
    -3\,X(t)_\z \cbullet Y_\v
    - Y_\z \cbullet X(t)_\v
    + X(t)_\v \spinprod Y_\v
  \Bigr) \\
  &\quad
  - \frac12\Bigl(
      -3\,X_\z \cbullet (\nabla_X Y)_\v
      - (\nabla_X Y)_\z \cbullet X_\v
      + X_\v \spinprod (\nabla_X Y)_\v
    \Bigr)
\end{align*}
The algebraic part~\eqref{eq:LC} of
\[
  \frac{\nabla}{\d t}\Big|_{t=0}
  \Bigl(
    \gamma(t),
    -3\,X(t)_\z \cbullet Y_\v
    - Y_\z \cbullet X(t)_\v
    + X(t)_\v \spinprod Y_\v
  \Bigr)
\]
is
\begin{align*}
  \frac12\Bigl(
    X_\v \spinprod \bigl(-3\,X_\z \cbullet Y_\v - Y_\z \cbullet X_\v\bigr)
    - X_\z \cbullet \bigl(-3\,X_\z \cbullet Y_\v - Y_\z \cbullet X_\v\bigr)
    - (X_\v \spinprod Y_\v) \cbullet X_\v
  \Bigr)
\end{align*}
while the derivative from~\eqref{eq:canonical_trivialization_1} is the
same as in the case \(Y \in \H_X^\perp\):
\[
  \frac{\mathrm d}{\mathrm dt}\Big|_{t=0}
  \Bigl(
    -3\,X(t)_\z \cbullet Y_\v
    - Y_\z \cbullet X(t)_\v
    + X(t)_\v \spinprod Y_\v
  \bigr)
  =
  -\bigl(
    Y_\z \cbullet X_\z \cbullet X_\v
    + Y_\v \spinprod (X_\z \cbullet X_\v)
  \Bigr)
\]
We obtain
\begin{align*}
  &\Bigl(\frac{\nabla}{\d t}\Big|_{t=0} \C_\gamma(t)\Bigr)Y
    + \frac12\Bigl(
      Y_\z \cbullet X_\z \cbullet X_\v
      + Y_\v \spinprod (X_\z \cbullet X_\v)
    \Bigr) \\
  &= \frac14\Bigl(
        - X_\v \spinprod \bigl(3\,X_\z \cbullet Y_\v + Y_\z \cbullet X_\v\bigr)
        + X_\z \cbullet \bigl(3\,X_\z \cbullet Y_\v + Y_\z \cbullet X_\v\bigr)
        - (X_\v \spinprod Y_\v) \cbullet X_\v
     \Bigr) \\
  &\quad
    + \frac14\Bigl(
        - 3\,X_\z \cbullet \bigl(X_\z \cbullet Y_\v + Y_\z \cbullet X_\v\bigr)
        + (X_\v \spinprod Y_\v) \cbullet X_\v
        + X_\v \spinprod \bigl(X_\z \cbullet Y_\v + Y_\z \cbullet X_\v\bigr)
     \Bigr) \\
  &= -\frac12\Bigl(
       X_\v \spinprod (X_\z \cbullet Y_\v)
       + X_\z \cbullet Y_\z \cbullet X_\v
     \Bigr)
\end{align*}
Thus,
\begin{align*}
  \Bigl(\frac{\nabla}{\d t}\Big|_{t=0} \C_\gamma(t)\Bigr)Y
  =
  -\frac12\Bigl(
    Y_\z \cbullet X_\z \cbullet X_\v
    + Y_\v \spinprod (X_\z \cbullet X_\v)
    + X_\v \spinprod (X_\z \cbullet Y_\v)
    + X_\z \cbullet Y_\z \cbullet X_\v
  \Bigr)
\end{align*}
To show that this term vanishes, note that
\[
  Y_\z \cbullet X_\z \cbullet X_\v
  + X_\z \cbullet Y_\z \cbullet X_\v
  = - 2\,\langle X_\z, Y_\z \rangle X_\v
  = 0
\]
since \(Y_\z \in \h_X = (X_\z)^\perp\) for \(Y \in \H_X\). It
remains to show that
\[
  Y_\v \spinprod (X_\z \cbullet X_\v)
  + X_\v \spinprod (X_\z \cbullet Y_\v)
  = 0
\]
For \(Z \in \z\),
\begin{align*}
  \langle
    Y_\v \spinprod (X_\z \cbullet X_\v)
    + X_\v \spinprod (X_\z \cbullet Y_\v),
    Z
  \rangle
  &=
    \langle X_\z \cbullet X_\v, Z \cbullet Y_\v \rangle
    + \langle X_\z \cbullet Y_\v, Z \cbullet X_\v \rangle \\
  &\stackrel{\eqref{eq:Clifford_module_1}}{=}
     -\langle Z \cbullet X_\z \cbullet X_\v, Y_\v \rangle
     -\langle Y_\v, X_\z \cbullet Z \cbullet X_\v \rangle \\
  &\stackrel{\eqref{eq:Clifford_module_3}}{=}
    2\,\langle Z, X_\z \rangle\,\langle X_\v, Y_\v \rangle
    \underbrace{-\langle Z \cbullet X_\z \cbullet X_\v, Y_\v \rangle
     + \langle Y_\v, Z \cbullet X_\z \cbullet X_\v \rangle}_{=\,0}\\
  &= 2\,\langle Z, X_\z \rangle\,\langle X_\v, Y_\v \rangle = 0
\end{align*}
since \(X_\v \in (\n_3)_X\) and \(Y \in \H_X\), so
\(\langle X_\v, Y_\v \rangle = 0\). Hence \(Y_\v \spinprod (X_\z \cbullet X_\v)
 + X_\v \spinprod (X_\z \cbullet Y_\v) = 0\), and therefore
\[
  \frac{\nabla}{\d t}\Big|_{t=0} \C_\gamma(t) = 0
\]
The above holds for all geodesics \(\gamma \colon \bbR \to N\) such that
\(\gamma(0) = e_N\) and \(X \coloneqq \dot\gamma(0)\) satisfies
\(X_\z \neq 0\) and \(X_\v \neq 0\). Moreover, by~\eqref{eq:geodesic}
the \(\n\)-component \(X(s)\) of \(\dot\gamma(s)\) in the left
trivialization \(T N \cong N \times \n\) again satisfies
\(X(s)_\z \neq 0\) and \(X(s)_\v \neq 0\) for all \(s \in \bbR\).
Therefore we can apply the previous argument to the geodesic
\(\gamma_s \colon \bbR \to N\) with \(\gamma_s(0) = e_N\)
and \(\dot\gamma_s(0) = X(s)\). Since the Levi-Civita connection
is left-invariant (and hence the geodesic equation is left-invariant
as well), we conclude that
\[
  \frac{\nabla}{\d t}\Big|_{t=s}\,\C_\gamma(t)
  = \rmL_{\gamma(s)*} \circ
    \Bigl(\frac{\nabla}{\d t}\Big|_{t=0}\C_{\gamma_s}(t)\Bigr)
    \circ \rmL_{\gamma(s)*}^{-1}
  = 0
\]
for all \(s \in \bbR\). Thus \(\C_\gamma\) is parallel along \(\gamma\).
\end{proof}

For each \(X \in \n\) with \(X_\z \neq 0\) and \(X_\v \neq 0\), we
consider the endomorphism \(\C(X)\) of \(\n\) defined by \eqref{eq:C_0_structure_1}.
Let  \(\Theta(X)\in\End{\n}\) be defined by
\begin{equation}\label{eq:Theta_0}
  \Theta(X)\colon \n \longrightarrow \n,\qquad Y \longmapsto
  \begin{cases}
    X_\z \cbullet Y_\v & \text{if } Y \in \H_X \\
    -\,X_\z \cbullet Y_\v & \text{if } Y \in (\n_3)_X \oplus (\Z_X)^\perp
  \end{cases}
\end{equation}
Then
\begin{equation}\label{eq:X_prime}
  X' \coloneqq \Theta(X)X = -\,X_\z \cbullet X_\v
\end{equation}
since \(X \in (\n_3)_X\).

\begin{lemma}\label{le:Theta}
We have
\begin{equation}\label{eq:Theta_1}
  [\Theta(X), \scrR(X)] = 2\,\scrR(X, X')
\end{equation}
where \(X' \coloneqq \Theta(X)X\) according to \eqref{eq:X_prime} and
the right-hand side is defined by polarization:
\[
  2\,\scrR(X, X') \coloneqq \scrR(X + X') - \scrR(X) - \scrR(X')
\]
\end{lemma}

\begin{proof}
  Recall that both \(\C(X)\) and \(\scrR(X)\) leave each of
  \((\n_3)_X\), \(\H_X\), and \((\Z_X)^\perp\) invariant (see
  \eqref{eq:R_X_1}--\eqref{eq:C_X_3}, \eqref{eq:R_X_4}--\eqref{eq:C_X_4},
  \eqref{eq:R_X_5}--\eqref{eq:C_X_7}). The same holds for \(\Theta(X)\).
  Thus it suffices to verify \eqref{eq:Theta_1} for \(Y\) in each of the
  subspaces \((\n_3)_X\), \(\H_X\), and \((\Z_X)^\perp\).

 We begin with the case \(Y \in \H_X\). By linearity, it suffices
 to consider \(Y \in \{\,Z,\, Z \cbullet X_\v,\, Z \cbullet X_\z \cbullet
 X_\v\,\}\) for some \(Z\in \h_X \coloneqq X_\z^\perp \subseteq \z\).
 Then \(\Theta(X)Z = 0\), and by~\eqref{eq:Theta_0} and~\eqref{eq:R_X_5}
 we have
 \begin{align*}
  \Theta(X)\bigl(\scrR(X)Z\bigr) &= \frac{3}{4}\,\norm{X_\z}^2\,
    Z \cbullet X_\v \\
  \scrR(X)\bigl(\Theta(X)Z\bigr) &= 0
 \end{align*}
 Thus
 \[
  [\Theta(X), \scrR(X)]\,Z
  = \frac{3}{4} \norm{X_\z}^2\, Z \cbullet X_\v
 \]

 For \(2\,\scrR(X, X')Z\), using~\eqref{eq:sym_curv_3} with \(X'_\z = 0\)
 and \(X'_\v = -\,X_\z \cbullet X_\v\) (so \(\langle X_\v, X'_\v \rangle
 = 0\)) gives
 \begin{align*}
  2\,\scrR(X, X')Z
  &= \frac{3}{4}\, Z \cbullet X_\z \cbullet X'_\v \\
  &\stackrel{\eqref{eq:X_prime}}{=} \frac{3}{4} \norm{X_\z}^2\,
    Z \cbullet X_\v
 \end{align*}
 We conclude
 \[
  [\Theta(X), \scrR(X)]\,Z = 2\,\scrR(X, X')Z
 \]

 Furthermore, since \(Z \cbullet X_\v \in \H_X\),
 \[
  \Theta(X)\bigl(Z \cbullet X_\v\bigr)
  = X_\z \cbullet Z \cbullet X_\v
  = -\,Z \cbullet X_\z \cbullet X_\v
 \]
 By~\eqref{eq:R_X_6} and~\eqref{eq:R_X_7}, and since \(\Theta(X)\) vanishes
 on \(\z\) (so \(\Theta(X)\overK(X)Z = 0\)), we have
 \begin{align*}
  \Theta(X)\bigl(\scrR(X)(Z \cbullet X_\v)\bigr)
  &= -\,\frac{1}{4}\bigl( \norm{X_\z}^2 - 3\,\norm{X_\v}^2 \bigr)
      Z \cbullet X_\z \cbullet X_\v \\
  \scrR(X)\bigl(\Theta(X)(Z \cbullet X_\v)\bigr)
  &= -\,\frac{3}{4} \norm{X_\v}^2\,\norm{X_\z}^2\, Z
     + \frac{3}{4} \overK(X)Z \cbullet X_\v
     - \frac{1}{4} \norm{X_\z}^2\, Z \cbullet X_\z \cbullet X_\v
 \end{align*}
 Therefore
 \begin{equation}\label{eq:commutator_2}
 \begin{aligned}
  [\Theta(X), \scrR(X)]\bigl(Z \cbullet X_\v\bigr)
  &= \frac{3}{4} \norm{X_\v}^2\, Z \cbullet X_\z \cbullet X_\v
     + \frac{3}{4} \norm{X_\v}^2\,\norm{X_\z}^2\, Z \\
  &\quad - \frac{3}{4} \overK(X)Z \cbullet X_\v
 \end{aligned}
 \end{equation}

 Let us compare~\eqref{eq:commutator_2} with
 \(2\,\scrR(X, X')\bigl(Z \cbullet X_\v\bigr)\):
 \begin{align*}
  2\,\scrR(X, X')\bigl(Z \cbullet X_\v\bigr)
  &\stackrel{\eqref{eq:sym_curv_3}}{=}
    -\,\frac{3}{4} \Bigl(
      \bigl(X'_\v \spinprod (Z \cbullet X_\v)\bigr)\cbullet X_\v
      + \bigl(X_\v \spinprod (Z \cbullet X_\v)\bigr)\cbullet X'_\v
    \Bigr) \\
  &\quad -\,\frac{3}{4}\, (Z \cbullet X_\v) \spinprod (X_\z \cbullet
    X'_\v) + \frac{1}{2} \langle X'_\v, Z \cbullet X_\v \rangle\, X_\z
 \end{align*}
 We compute each term:
 \begin{align*}
  \bigl(X'_\v \spinprod (Z \cbullet X_\v)\bigr)\cbullet X_\v
    &\stackrel{\eqref{eq:de_K_4}}{=} \overK(X)Z \cbullet X_\v \\
  \bigl(X_\v \spinprod (Z \cbullet X_\v)\bigr)\cbullet X'_\v
    &\stackrel{\eqref{eq:Clifford_module_5}}{=}
      -\,\norm{X_\v}^2\, Z \cbullet X_\z \cbullet X_\v \\
  (Z \cbullet X_\v) \spinprod (X_\z \cbullet X'_\v)
    &\stackrel{\eqref{eq:Clifford_module_5}}{=}
      -\,\norm{X_\v}^2\,\norm{X_\z}^2\,Z \\
  \langle X'_\v, Z \cbullet X_\v \rangle\,X_\z &= 0
 \end{align*}
 Thus
 \begin{align*}
  2\,\scrR(X, X')\bigl(Z \cbullet X_\v\bigr)
  &= -\,\frac{3}{4}\,\overK(X)Z \cbullet X_\v
     + \frac{3}{4}\,\norm{X_\v}^2\, Z \cbullet X_\z \cbullet X_\v \\
  &\quad + \frac{3}{4}\,\norm{X_\v}^2\,\norm{X_\z}^2\, Z \\
  &\stackrel{\eqref{eq:commutator_2}}{=}
     [\Theta(X), \scrR(X)]\bigl(Z \cbullet X_\v\bigr)
 \end{align*}
 Hence~\eqref{eq:Theta_1} holds on \(Z \cbullet X_\v\).

 Similarly, for
 \([\Theta(X), \scrR(X)]\bigl(Z \cbullet X_\z \cbullet X_\v\bigr)\),
 we have:
 \begin{align*}
  \Theta(X)\bigl(\scrR(X)(Z \cbullet X_\z\cbullet X_\v)\bigr)
  &\stackrel{\eqref{eq:R_X_7}}{=}
    \frac{3}{4}\,\overK(X)Z \cbullet X_\z\cbullet X_\v
    +\frac{1}{4}\,\norm{X_\z}^4\,Z \cbullet X_\v \\
  \scrR(X)\bigl(\Theta(X)(Z \cbullet X_\z\cbullet X_\v)\bigr)
  &=\norm{X_\z}^2\,\scrR(X)\bigl(Z \cbullet X_\v\bigr) \\
  &\stackrel{\eqref{eq:R_X_6}}{=}
    \norm{X_\z}^2\!\left[
    \frac{1}{4}\bigl(\norm{X_\z}^2-3\,\norm{X_\v}^2\bigr)Z \cbullet X_\v
    - \frac{3}{4}\,\overK(X)Z
    \right] \\
  &= \frac{1}{4}\bigl(\norm{X_\z}^2-3\,\norm{X_\v}^2\bigr)\norm{X_\z}^2\,
     Z \cbullet X_\v
    - \frac{3}{4}\,\norm{X_\z}^2\,\overK(X)Z
 \end{align*}
 Therefore
 \begin{equation}\label{eq:commutator_3}
  \begin{aligned}
   [\Theta(X), \scrR(X)]\bigl(Z \cbullet X_\z \cbullet X_\v\bigr)
   &= \frac{3}{4}\,\norm{X_\v}^2\,\norm{X_\z}^2\, Z \cbullet X_\v
     + \frac{3}{4}\,\overK(X)Z \cbullet X_\z \cbullet X_\v \\
   &\quad + \frac{3}{4}\,\norm{X_\z}^2\,\overK(X)Z
  \end{aligned}
 \end{equation}

 On the other hand,
 \begin{align*}
  2\,\scrR(X, X')\bigl(Z \cbullet X_\z \cbullet X_\v\bigr)
  &\stackrel{\eqref{eq:sym_curv_3}}{=}
    -\,\frac{3}{4}\Bigl(
      \bigl(X'_\v \spinprod (Z \cbullet X_\z \cbullet X_\v)\bigr)\cbullet
      X_\v
      + \bigl(X_\v \spinprod (Z \cbullet X_\z \cbullet X_\v)\bigr)\cbullet
      X'_\v
    \Bigr) \\
  &\quad -\,\frac{3}{4}\,\bigl(Z \cbullet X_\z \cbullet X_\v\bigr) \spinprod
    (X_\z \cbullet X'_\v)
    + \frac{1}{2}\,\langle X'_\v, Z \cbullet X_\z \cbullet X_\v \rangle\,
    X_\z
 \end{align*}
 where:
 \begin{align*}
  \bigl(X'_\v \spinprod (Z \cbullet X_\z \cbullet X_\v)\bigr)\cbullet
   X_\v &\stackrel{\eqref{eq:def_spinor_product},\eqref{eq:X_prime}}{=}
          - \norm{X_\v}^2\,\norm{X_\z}^2Z \cbullet X_\v \\
  \bigl(X_\v \spinprod (Z \cbullet X_\z \cbullet X_\v)\bigr)\cbullet
  X'_\v
    &\stackrel{\eqref{eq:def_K_2},\eqref{eq:X_prime}}{=}
      -\,\overK(X)Z \cbullet X_\z \cbullet X_\v \\
  \bigl(Z \cbullet X_\z \cbullet X_\v\bigr) \spinprod (X_\z \cbullet X'_\v)
    &\stackrel{\eqref{eq:def_K_2},\eqref{eq:X_prime}}{=}
      -\,\norm{X_\z}^2\,\overK(X)Z  \\
  \langle X'_\v, Z \cbullet X_\z \cbullet X_\v \rangle\,X_\z
    &\stackrel{\eqref{eq:X_prime}}{=}
      -\langle X'_\v, Z \cbullet X'_\v \rangle\,X_\z
      \stackrel{\eqref{eq:Clifford_module_1}}{=} 0
 \end{align*}
 Thus
 \begin{align*}
  2\,\scrR(X,X')\bigl(Z \cbullet X_\z \cbullet X_\v\bigr)
  &= \frac{3}{4}\,\norm{X_\v}^2\,\norm{X_\z}^2\,Z \cbullet X_\v
     + \frac{3}{4}\,\overK(X)Z \cbullet X_\z \cbullet X_\v \\
  &\quad + \frac{3}{4}\,\norm{X_\z}^2\,\overK(X)Z \\
  &\stackrel{\eqref{eq:commutator_3}}{=}
     [\Theta(X), \scrR(X)]\bigl(Z \cbullet X_\z \cbullet X_\v\bigr)
 \end{align*}
 This proves that~\eqref{eq:Theta_1} holds on \(\H_X\).

 Next, we will verify that~\eqref{eq:Theta_1} holds on \((\n_3)_X\) as
 well:
 \begin{align*}
  \scrR(X)\bigl(\Theta(X)X_\v\bigr)
    &\stackrel{\eqref{eq:R_X_2},\,\eqref{eq:Theta_0}}{=}
      \left( \frac{3}{4}\,\norm{X_\v}^2 - \frac{1}{4}\,\norm{X_\z}^2
      \right) (X_\z \cbullet X_\v) \\
  \scrR(X)\bigl(\Theta(X)(X_\z \cbullet X_\v)\bigr)
    &\stackrel{\eqref{eq:R_X_1},\,\eqref{eq:Theta_0}}{=}
      \frac{1}{4}\,\norm{X_\z}^2 \left( \norm{X_\z}^2\, X_\v -
      \norm{X_\v}^2\, X_\z \right) \\
  \scrR(X)\bigl(\Theta(X)X_\z\bigr)
    &\stackrel{\eqref{eq:Theta_0}}{=} 0
 \end{align*}
 and
 \begin{align*}
  \Theta(X)\bigl(\scrR(X)X_\v\bigr)
    &\stackrel{\eqref{eq:R_X_1}}{=}
      -\,\frac{1}{4}\,\norm{X_\z}^2\, (X_\z \cbullet X_\v) \\
  \Theta(X)\bigl(\scrR(X)(X_\z \cbullet X_\v)\bigr)
    &\stackrel{\eqref{eq:R_X_2}}{=}
      \norm{X_\z}^2 \left( -\,\frac{3}{4}\,\norm{X_\v}^2 +
      \frac{1}{4}\,\norm{X_\z}^2 \right) X_\v \\
  \Theta(X)\bigl(\scrR(X)X_\z\bigr)
    &\stackrel{\eqref{eq:R_X_3}}{=}
      \frac{1}{4}\,\norm{X_\z}^2\, (X_\z \cbullet X_\v)
 \end{align*}
 Therefore:
 \begin{align}
  [\Theta(X), \scrR(X)]\,X_\v
    &= -\,\frac{3}{4}\,\norm{X_\v}^2\, (X_\z \cbullet X_\v)
    \label{eq:commutator_4} \\
  [\Theta(X), \scrR(X)]\,(X_\z \cbullet X_\v)
    &= \frac{1}{4}\,\norm{X_\v}^2\,\norm{X_\z}^2\, X_\z
       - \frac{3}{4}\,\norm{X_\v}^2\,\norm{X_\z}^2\, X_\v
       \label{eq:commutator_5} \\
  [\Theta(X), \scrR(X)]\,X_\z
    &= \frac{1}{4}\,\norm{X_\z}^2\, (X_\z \cbullet X_\v)
    \label{eq:commutator_6}
 \end{align}

 On the other hand:
 \begin{align*}
  2\,\scrR(X, X')X_\v
  &\stackrel{\eqref{eq:sym_curv_3}}{=}
    -\,\frac{3}{4}\,\bigl((X'_\v \spinprod X_\v) \cbullet X_\v\bigr)
    - \frac{3}{4}\,\bigl((X_\v \spinprod X_\v) \cbullet X'_\v\bigr) \\
  &\quad -\,\frac{3}{4}\,\bigl(X_\v \spinprod (X_\z \cbullet X'_\v)
    \bigr)
    + \frac{1}{2}\,\langle X'_\v, X_\v \rangle\,X_\z
 \end{align*}
 where
 \begin{align*}
  (X'_\v \spinprod X_\v) \cbullet X_\v
    &\stackrel{\eqref{eq:Clifford_module_5}}{=}
      \norm{X_\v}^2\, (X_\z \cbullet X_\v) \\
  (X_\v \spinprod X_\v) \cbullet X'_\v
    &\stackrel{\eqref{eq:def_spinor_product}}{=} 0 \\
  X_\v \spinprod (X_\z \cbullet X'_\v)
    &\stackrel{\eqref{eq:Clifford_module_2},\,\eqref{eq:def_spinor_product}}
      {=} 0 \\
  \langle X'_\v, X_\v \rangle\,X_\z
    &\stackrel{\eqref{eq:Clifford_module_1}}{=} 0
 \end{align*}
 Hence:
 \[
  2\,\scrR(X, X')X_\v
  = -\,\frac{3}{4}\,\norm{X_\v}^2\, (X_\z \cbullet X_\v)
  \stackrel{\eqref{eq:commutator_4}}{=}
  [\Theta(X), \scrR(X)]\,X_\v
 \]

 Similarly, for \(2\,\scrR(X, X')(X_\z \cbullet X_\v)\):
 \begin{align*}
  2\,\scrR(X, X')(X_\z \cbullet X_\v)
  &\stackrel{\eqref{eq:sym_curv_3}}{=}
    -\,\frac{3}{4}\,\bigl((X'_\v \spinprod (X_\z \cbullet X_\v)) \cbullet
    X_\v\bigr)
    - \frac{3}{4}\,\bigl((X_\v \spinprod (X_\z \cbullet X_\v)) \cbullet
    X'_\v\bigr) \\
  &\quad -\,\frac{3}{4}\,\bigl((X_\z \cbullet X_\v) \spinprod (X_\z
    \cbullet X'_\v)\bigr)
    + \frac{1}{2}\,\langle X'_\v, X_\z \cbullet X_\v \rangle\,X_\z
 \end{align*}
 where
 \begin{align*}
  (X'_\v \spinprod (X_\z \cbullet X_\v)) \cbullet X_\v
    &\stackrel{\eqref{eq:X_prime}}{=} 0, \\
  (X_\v \spinprod (X_\z \cbullet X_\v)) \cbullet X'_\v
    &= \norm{X_\v}^2\,\norm{X_\z}^2\, X_\v \\
  (X_\z \cbullet X_\v) \spinprod (X_\z \cbullet X'_\v)
    &= -\,\norm{X_\v}^2\,\norm{X_\z}^2\, X_\z \\
  \langle X'_\v, X_\z \cbullet X_\v \rangle\,X_\z
    &\stackrel{\eqref{eq:Clifford_module_1},\,\eqref{eq:Clifford_module_2},
      \,\eqref{eq:X_prime}}{=}
      -\,\norm{X_\v}^2\,\norm{X_\z}^2\, X_\z
 \end{align*}
 Thus:
 \begin{align*}
  2\,\scrR(X, X')(X_\z \cbullet X_\v)
  &= -\,\frac{3}{4}\,\norm{X_\v}^2\,\norm{X_\z}^2\, X_\v
     + \frac{3}{4}\,\norm{X_\v}^2\,\norm{X_\z}^2\, X_\z \\
  &\quad -\,\frac{1}{2}\,\norm{X_\v}^2\,\norm{X_\z}^2\, X_\z \\
  &= -\,\frac{3}{4}\,\norm{X_\v}^2\,\norm{X_\z}^2\, X_\v
     + \frac{1}{4}\,\norm{X_\v}^2\,\norm{X_\z}^2\, X_\z \\
  &\stackrel{\eqref{eq:commutator_5}}{=}
     [\Theta(X), \scrR(X)]\,(X_\z \cbullet X_\v)
 \end{align*}

 For \(2\,\scrR(X, X')X_\z\):
 \begin{align*}
  2\,\scrR(X, X')X_\z
  &\stackrel{\eqref{eq:sym_curv_3}}{=}
    \frac{3}{4}\,(X_\z \cbullet X_\z \cbullet X'_\v)
    + \frac{1}{2}\,\norm{X_\z}^2\, X'_\v
    + \frac{1}{2}\,\langle X'_\v, X_\v \rangle\, X_\z \\
  &= -\,\frac{3}{4}\,\norm{X_\z}^2\, X'_\v
     + \frac{1}{2}\,\norm{X_\z}^2\, X'_\v \\
  &= \frac{1}{4}\,\norm{X_\z}^2\, (X_\z \cbullet X_\v) \\
  &\stackrel{\eqref{eq:commutator_6}}{=}
     [\Theta(X), \scrR(X)]\,X_\z
 \end{align*}

 Finally, for \(Y \in (\Z_X)^\perp\), we have \(Y_\z = 0\) and
 \(Y_\v \in \Kern(X_\v \spinprod) \cap \Kern\bigl((X_\z \cbullet X_\v)
 \spinprod\bigr)\) by~\eqref{eq:decomposition_of_n_1}. Then:
 \begin{align*}
  \Theta(X)\bigl(\scrR(X)Y\bigr)
    &\stackrel{\eqref{eq:R_X_4}}{=}
      -\,\frac{1}{4}\,\norm{X_\z}^2\, (X_\z \cbullet Y_\v) \\
  \scrR(X)\bigl(\Theta(X)Y\bigr)
    &\stackrel{\eqref{eq:R_X_4}}{=}
      -\,\frac{1}{4}\,\norm{X_\z}^2\, (X_\z \cbullet Y_\v)
 \end{align*}
 Thus
 \[
  [\Theta(X), \scrR(X)]\,Y = 0
  \stackrel{\eqref{eq:sym_curv_3}}{=}
  2\,\scrR(X, X')Y
 \]
 This shows that~\eqref{eq:Theta_1} also holds on \((\Z_X)^\perp\).
\end{proof}

\paragraph{Proof of Theorem~\ref{th:C0}}
For \(X \in \n\) with \(X_\z \neq 0\) and \(X_\v \neq 0\) let
\(\gamma \colon \bbR \to N\) be the geodesic with \(\gamma(0) = e_N\)
and \(\dot\gamma(0) = X\). Then \(\C_\gamma(0) = \C(X)\), where
\(\C(X) \in \End{\n}\) is defined by~\eqref{eq:C_0_structure_1}. Let
\(A\) denote the tensor defined in \eqref{eq:LC}. Then \(A(X)Y = \nabla_X Y\)
for \(X, Y \in \n\) describes the Levi-Civita connection
in the left-invariant framing. Comparing~\eqref{eq:LC}
with~\eqref{eq:C_0_structure_1} and~\eqref{eq:Theta_0}, we see that
\[A(X)Y = \Theta(X)Y + \C(X)Y\] for all \(X, Y \in \n\). Since both
\(\Theta(X)\) and \(A(X)\) are skew-symmetric endomorphisms of \(\n\),
it follows that \(\C(X)\) is skew-symmetric as well. Therefore, also
\(\C_\gamma(t)\) is skew-symmetric for all \(t \in \bbR\)
by~\eqref{eq:C_0_structure_3}. Moreover, it is now straightforward
to see that \(\scrR^1(X) = [\C(X),\,\scrR(X)]\) holds:
\begin{align*}
  \scrR^1(X)Y
    &= \bigl( A(X) \cdot \scrR \bigr)(X)Y \\
    &= \bigl( \Theta(X) \cdot \scrR \bigr)(X)Y
       + \bigl( \C(X) \cdot \scrR \bigr)(X)Y \\
    &\stackrel{\eqref{eq:Theta_1}}{=}
       \bigl( \C(X) \cdot \scrR \bigr)(X)Y \\
    &= [\C(X),\,\scrR(X)]\,Y
\end{align*}
where the second-to-last equation uses that \(\C(X)X = 0\), as indicated
immediately after the defining equation~\eqref{eq:C_0_structure_1}.
By~\eqref{eq:Jacobi_operator_versus_symmetrized_curvature_tensor} we conclude that
\[
  \scrR^{(1)}_\gamma(0) = [\C_\gamma(0), \scrR_\gamma(0)]
\]
for all geodesics \(\gamma \colon \bbR \to N\) with \(\gamma(0) = e_N\)
such that \(X \coloneqq \dot\gamma(0)\) satisfies \(X_\z \neq 0\) and
\(X_\v \neq 0\). Because the curvature tensor is left-invariant as well,
we can argue as in the proof of Lemma~\ref{le:C_is_parallel} to conclude
that
\[
  \scrR^{(1)}_\gamma(t) = [\C_\gamma(t), \scrR_\gamma(t)]
\]
holds for all \(t \in \bbR\). Together with Lemma~\ref{le:C_is_parallel},
this shows that \(\C_\gamma\) defines a \(\frakC_0\)-structure along
\(\gamma\).\qed

\section{\texorpdfstring{Positivity of the coefficient polynomials
    \(a_i(X)\)}{Positivity of the coefficient polynomials a_i(X)}}
\label{se:positivity}

\begin{lemma}\label{le:positive}
Consider again the polynomial from~\eqref{eq:def_tilde_P_mu},
written as
\begin{equation}\label{eq:P_mu_2}
\tilde P_{\undermu}(z,v)
  = \lambda^7
    + a_2(z,v)\lambda^5
    + a_4(z,v)\lambda^3
    + a_6(\undermu,z,v)\lambda
\end{equation}
where the coefficients are given explicitly in~\eqref{eq:tilde_P_mu}.
Assume \(z\ge 0\), \(v\ge 0\) and \((z,v)\neq(0,0)\), and also
\(\undermu<1\). Then \(a_i(z,v)>0\) for \(i=2,4\) and
\(a_6(\undermu,z,v)>0\).
\end{lemma}

\begin{proof}
It is clear that \(a_i(z,v)>0\) for \(i=2,4\) unless \(z=v=0\).
If \(v=0\) or \(z=0\) (but not both), then also
\(a_6(\undermu,z,v)>0\). Hence we may assume \(z>0\) and \(v>0\).
Since \(\undermu<1\), we have \(a_6(\undermu,z,v)>a_6(1,z,v)\).

Consider~\eqref{eq:P_mu_2} with \(\undermu\coloneqq 1\):
\begin{equation}\label{eq:P_mu_3}
\begin{aligned}
\lambda^7
&+ \left(\frac{27\,z^{2}}{2}+\frac{3\,v^{2}}{2}\right)\lambda^{5}
+ \left(\frac{729}{16}\,z^{4}+\frac{81}{8}\,v^{2}z^{2}
         +\frac{9}{16}\,v^{4}\right)\lambda^{3}\\
&+ \left(\frac{729}{16}\,z^{6}+\frac{243}{16}\,z^{4}v^{2}
         -\frac{135}{64}\,v^{4}z^{2}+\frac{1}{16}\,v^{6}\right)\lambda
\end{aligned}
\end{equation}
By~\eqref{eq:p_over_q}, this factors as
\begin{equation}\label{eq:P_mu_4}
\frac{\lambda\,(4\lambda^{2}+9z^{2}+4v^{2})
\bigl(16\lambda^{4}+180\lambda^{2}z^{2}+8\lambda^{2}v^{2}
      +324z^{4}-36v^{2}z^{2}+v^{4}\bigr)}{64}\,.
\end{equation}
Hence \(a_6(1,z,v)\ge 0\) iff
\(324z^{4}-36v^{2}z^{2}+v^{4}\ge 0\). Note
\[
324z^{4}-36v^{2}z^{2}+v^{4}=(18z^{2}-v^{2})^{2}\,,
\]
equivalently the quadratic form with matrix
\(\bigl(\begin{smallmatrix}1&-18\\-18&324\end{smallmatrix}\bigr)\)
is positive semidefinite (determinant \(0\), upper‐left entry
positive). Therefore \(a_6(\undermu,z,v)>a_6(1,z,v)\ge 0\), proving
the claim.
\end{proof}

\begin{corollary}\label{co:positive}
Suppose \(X\in U(e_N)\). Then the even coefficients \(a_{2i}(X)\) of
\[
P_{\mathrm{min}}(X)=\sum_{i=0}^k a_i(X)\lambda^{k-i}
\]
are strictly positive.
\end{corollary}

\begin{proof}
Let \(X\in U(e_N)\) and \(k\coloneqq \deg P_{\mathrm{min}}(N,g)\).
Clearly \(X\neq 0\). Let \(\undermu\) be a nonconstant eigenvalue
branch of \(-\underK^2\). We claim \(\undermu(X)<1\). Otherwise we
can either remove the factor \(P_{\undermu}(X)\) from
\eqref{eq:P_min_2} completely (if \(\nu=1\) is a global eigenvalue
of \(-\underK^2\)), or replace it by \(P_1^\sharp(X)\) otherwise, to
obtain a polynomial \(P(X)\) of degree at most \(k-4\) with
\(P(X;\ad(\C(X))\scrR^0(X))=0\). Thus
\(\deg P_{\mathrm{min}}(\C(X),\scrR(X))\le k-4\). By
\eqref{eq:U(p)_2} this implies \(X\notin U(e_N)\), a contradiction.

Hence \(\undermu(X)<1\) for all nonconstant eigenvalue branches of
\(-\underK^2\). By Lemma~\ref{le:positive}, all even coefficients
\(a_{2i}(X)\) of \(P_{\mathrm{min}}(X)\) in~\eqref{eq:P_min_2}\;are
strictly positive.
\end{proof}

\section{\texorpdfstring{Expression of~\eqref{eq:characteristic_polynomial}
in terms of the elementary symmetric polynomials
\(\oversigma_j\) defined in~\eqref{eq:over_sigma_j}}%
{Expression of P\_min in terms of the elementary symmetric polynomials sigma\_j}}
\label{se:expression_as_polynomials}

For \(i = 1, \dotsc, \ell\), let \(P_i(X)\) be defined by
\[
  P_i(X)
  =
  \lambda^6 + A(X)\lambda^4 + B(X)\lambda^2 + C_i(X)
\]
where
\begin{align*}
  A(X)
  &=
  \frac{27}{2}\,\norm{X_\z}^2
  + \frac{3}{2}\,\norm{X_\v}^2
  \\
  B(X)
  &=
  \frac{729}{16}\,\norm{X_\z}^4
  + \frac{81}{8}\,\norm{X_\z}^2 \norm{X_\v}^2
  + \frac{9}{16}\,\norm{X_\v}^4
  \\
  C_i(X)
  &=
  \frac{729}{16}\,\norm{X_\z}^6
  + \frac{243}{16}\,\norm{X_\z}^4 \norm{X_\v}^2
  + \frac{27}{16}\,\norm{X_\z}^2 \norm{X_\v}^4
  + \frac{1}{16}\,\norm{X_\v}^6
  - \frac{243}{64}\,\overmu_i(X)
\end{align*}

Set
\[
  Q(X)
  \coloneqq
  \lambda^6 + A(X)\lambda^4 + B(X)\lambda^2 + C(X)
\]
where
\[
  C(X)
  \coloneqq
  \frac{729}{16}\,\norm{X_\z}^6
  + \frac{243}{16}\,\norm{X_\z}^4 \norm{X_\v}^2
  + \frac{27}{16}\,\norm{X_\z}^2 \norm{X_\v}^4
  + \frac{1}{16}\,\norm{X_\v}^6
\]

Then
\[
  P_i(X)
  =
  Q(X) - \alpha\,\overmu_i(X),
  \qquad
  \alpha \coloneqq \frac{243}{64}
\]

\subsection*{Symmetric expansion of the product}

The product of these polynomials is
\[
  \prod_{i = 1}^{\ell} P_i(X)
  =
  \prod_{i = 1}^{\ell}
  \bigl(Q(X) - \alpha\,\overmu_i(X)\bigr)
\]
and hence, by the elementary symmetric expansion,
\[
  \prod_{i = 1}^{\ell} P_i(X)
  =
  \sum_{j = 0}^{\ell}
  (-1)^j\,\alpha^j\,
  \oversigma_j(X)\,Q(X)^{\ell - j}
\]
where \(\oversigma_0 \coloneqq 1\).

For \(m \ge 0\), write
\[
  Q(X)^m
  =
  \sum_{q = 0}^{3m} E_{q,m}(X)\,\lambda^{2q}
\]
where
\[
  E_{q,m}(X)
  =
  \sum_{\substack{
    a_0 + a_1 + a_2 + a_3 = m \\
    a_1 + 2a_2 + 3a_3 = q
  }}
  \binom{m}{a_0,\,a_1,\,a_2,\,a_3}\,
  C(X)^{a_0}\,B(X)^{a_1}\,A(X)^{a_2}
\]

Thus the coefficient of \(\lambda^{2q}\) in the product is
\begin{equation}\label{eq:D_q_l}
  D_{q,\ell}(X)
  \coloneqq
  \sum_{j = 0}^{\ell - \lceil q/3 \rceil}
  (-1)^j
  \left(\frac{243}{64}\right)^j
  \oversigma_j(X)\,
  E_{q,\ell - j}(X)
\end{equation}

\subsection*{Final form of the product}

The product of the factors in~\eqref{eq:P_min_2} contributed by the
nonconstant eigenvalue branches
\(\undermu_1, \ldots, \undermu_\ell\) therefore becomes
\begin{equation}\label{eq:product_of_the_P_undermu_i}
  \prod_{i = 1}^{\ell} P_{\undermu_i}
  =
  \sum_{q = 0}^{3\ell} D_{q,\ell}\,\lambda^{2q}
\end{equation}
This exhibits the polynomial dependence of
\(\prod_{i = 1}^{\ell} P_{\undermu_i}\) on
\(\oversigma_1, \ldots, \oversigma_\ell\).

To obtain \(\left.P_{\min}(N,g)\right|_e\), it remains to multiply
\eqref{eq:product_of_the_P_undermu_i} by \(P_{\n_3}\), possibly by
\(P_0^\sharp\), depending on the parity of \(n\), and possibly also by
\(P_1^\sharp\), depending on whether \(1\) is a global eigenvalue of
\(-\underK^2\).

Moreover, our analysis of the \(\overmu_j\), more precisely of the
corresponding elementary symmetric functions \(\oversigma_j\), in
Section~\ref{se:explicit_formulas} shows that each \(\oversigma_j(X)\)
is already a polynomial function of \(X \in \n\), for all pairs
\((\z,\v)\). This completes our discussion of the implications of
Theorem~\ref{th:JR_on_C_0_spaces} for the coefficients of
\(\left.P_{\min}(N,g)\right|_e\).

\section*{ACKNOWLEDGEMENTS}
I would like to thank Gregor Weingart for helpful discussions.

\bibliographystyle{amsplain}

\end{document}